\documentclass{article}

\usepackage[nonatbib,final]{nips_arxiv}

\def\supplement{1}
\def\cameraready{1}

\usepackage[numbers,sort&compress,square]{natbib}

\PassOptionsToPackage{natbib}{square,numbers,sort&compress}

\usepackage[utf8]{inputenc} \usepackage[T1]{fontenc}    \usepackage{hyperref}       \usepackage{url}            \usepackage{booktabs}       \usepackage{amsfonts}       \usepackage{nicefrac}       \usepackage{microtype}      
\clearpage{}
\usepackage{amsmath,amsfonts,amsbsy,amsgen,amscd,amssymb,amsthm,bm,stmaryrd}
\usepackage{mathtools}

\usepackage[scaled=1.2]{urwchancal}

\usepackage{multicol}

\usepackage[usenames,dvipsnames]{xcolor}
\usepackage{graphicx}
\graphicspath{{art/}}

\definecolor{dark-gray}{gray}{0.3}
\definecolor{dkgray}{rgb}{.4,.4,.4}
\definecolor{dkblue}{rgb}{0,0,.5}
\definecolor{medblue}{rgb}{0,0,.75}
\definecolor{rust}{rgb}{0.5,0.1,0.1}
\definecolor{purple}{rgb}{0.3,0.0,.4}

\usepackage{enumitem}
\usepackage{float}
\usepackage{caption}
\usepackage{subcaption}
\usepackage[font=small,margin=0.25in,labelfont={sc},labelsep={colon}]{caption}

\usepackage[noend]{algpseudocode}
\usepackage{algorithm,algorithmicx}

\algrenewcommand\alglinenumber[1]{\sf\tiny\color{medblue}{#1}\quad}
\algrenewcommand\algorithmicrequire{\textbf{Input:}}
\algrenewcommand\algorithmicensure{\textbf{Output:}}

\hypersetup{urlcolor=rust}
\hypersetup{citecolor=dkblue}
\hypersetup{linkcolor=dkblue}

\newtheorem{theorem}{Theorem}

\newtheorem{proposition}[theorem]{Proposition}

\theoremstyle{definition}
\newtheorem{remark}[theorem]{Remark}

\numberwithin{equation}{section}
\numberwithin{theorem}{section}
\numberwithin{figure}{section}

\newcommand{\R}{\mathbb{R}}
\newcommand{\C}{\mathbb{C}}
\newcommand{\F}{\mathbb{F}}

\newcommand{\eps}{\varepsilon}
\newcommand{\econst}{\mathrm{e}}

\newcommand{\vct}[1]{\bm{#1}}
\newcommand{\mtx}[1]{\bm{#1}}

\newcommand{\Id}{\mathbf{I}}

\newcommand{\rank}{\operatorname{rank}}
\newcommand{\range}{\operatorname{range}}

\newcommand{\diag}{\operatorname{diag}}

\newcommand{\norm}[1]{\Vert #1 \Vert}

\newcommand{\lowrank}[2]{\llbracket {#1} \rrbracket_{#2}}

\newcommand{\Expect}{\operatorname{\mathbb{E}}}

\algdef{SE}[SUBALG]{Indent}{EndIndent}{}{\algorithmicend\ }\algtext*{Indent}
\algtext*{EndIndent}
\clearpage{}

\title{Fixed-Rank Approximation of a \\
Positive-Semidefinite Matrix from Streaming Data \\
\ifdefined\supplement
{\small Including supplementary appendix} \fi}

\author{
  Joel A.~Tropp \\
  Caltech \\
  \hspace{-0.2pc}{\footnotesize \url{jtropp@caltech.edu}}\hspace{-0.2pc}
  \And
  Alp Yurtsever \\
  EPFL \\
  \hspace{-0.2pc}{\footnotesize \url{alp.yurtsever@epfl.ch}}\hspace{-0.2pc}
  \And
  Madeleine Udell \\   Cornell \\
  \hspace{-0.2pc}{\footnotesize \url{mru8@cornell.edu}}\hspace{-0.2pc}
  \And
  Volkan Cevher \\
  EPFL \\
  \hspace{-0.2pc}{\footnotesize \url{volkan.cevher@epfl.ch}}\hspace{-0.2pc}}

\begin{document}

\maketitle

\begin{abstract}
Several important applications, such as streaming PCA and semidefinite programming, involve a large-scale positive-semidefinite (psd) matrix that is presented
as a sequence of linear updates.  Because of storage limitations,
it may only be possible to retain a sketch of the psd matrix.
This paper develops a new algorithm for fixed-rank psd approximation from a sketch.
The approach combines the Nystr{\"o}m approximation with a novel mechanism for rank truncation.
Theoretical analysis establishes that
the proposed method can achieve any prescribed relative
error in the Schatten 1-norm and that it exploits the spectral
decay of the input matrix.  Computer experiments
show that the proposed method dominates alternative
techniques for fixed-rank psd matrix approximation across
a wide range of examples.
\end{abstract}

\section{Motivation}
\label{sec:intro}

In recent years, researchers have studied many applications where
a large positive-semidefinite (psd) matrix is presented
as a series of linear updates.  A recurring theme is that we only have space to store
a small summary of the psd matrix, and we must use this
information to construct an accurate psd approximation
with specified rank.
Here are two important cases where this problem arises.

\textbf{Streaming Covariance Estimation.}
Suppose that we receive a stream $\vct{h}_1, \vct{h}_2, \vct{h}_3, \dots \in \R^n$
of high-dimensional vectors.
The psd sample covariance matrix of these vectors has the linear dynamics
\begin{equation*} \label{eqn:sample-covar-dynamics}
\mtx{A}^{(0)} \gets \mtx{0}
\quad\text{and}\quad
\mtx{A}^{(i)} \gets (1 - i^{-1}) \mtx{A}^{(i-1)} + i^{-1} \vct{h}_i \vct{h}_i^*.
\end{equation*}
When the dimension $n$ and the number of vectors are both large,
it is not possible to store
the vectors or the sample covariance matrix.
Instead, we wish to maintain a small summary that allows us to compute the rank-$r$ psd approximation of the sample covariance
matrix $\mtx{A}^{(i)}$ at a specified instant $i$.
This problem and its variants are often called \emph{streaming PCA}
\cite{GPW12:Sketched-SVD, MCJ13:Memory-Limited, BGKL15:Online-Principal, GLPW16:Frequent-Directions, JJK+16:Streaming-PCA, FVR16:Dimensionality-Reduction}.

\textbf{Convex Low-Rank Matrix Optimization with Optimal Storage.}
A primary application of semidefinite programming (SDP)
is to search for a rank-$r$ psd matrix that satisfies
additional constraints.  Because of storage costs,
SDPs are difficult to solve when the matrix variable is large.
Recently, Yurtsever et al.~\cite{YUTC17:Sketchy-Decisions}
exhibited the first provable algorithm, called SketchyCGM, that produces a
rank-$r$ approximate solution to an SDP \emph{using optimal storage}.

Implicitly, SketchyCGM forms a sequence of approximate psd solutions to the SDP via the iteration \begin{equation*} \label{eqn:sketchy-cgm-dynamics}
\mtx{A}^{(0)} \gets \mtx{0}
\quad\text{and}\quad
\mtx{A}^{(i)} \gets (1 - \eta_i) \mtx{A}^{(i-1)} + \eta_i \vct{h}_i \vct{h}_i^*.
\end{equation*}
The step size $\eta_i = 2/(i+2)$, and the vectors $\vct{h}_i$ do not
depend on the matrices $\mtx{A}^{(i)}$.
In fact, SketchyCGM only maintains a small summary of the evolving solution $\mtx{A}^{(i)}$.
When the iteration terminates, SketchyCGM computes a rank-$r$ psd approximation
of the final iterate using the method described by \citet[Alg.~9]{TYUC17:Randomized-Single-View-TR}.

\subsection{Notation and Background} \label{sec:notation}
The scalar field $\F = \R$ or $\F = \C$. Define $\alpha(\R) = 1$ and $\alpha(\C) = 0$.
The asterisk ${}^*$ is the (conjugate) transpose,
and the dagger ${}^\dagger$ denotes the Moore--Penrose pseudoinverse.
The notation $\mtx{A}^{1/2}$ refers to the unique psd square root of a psd matrix $\mtx{A}$.
For $p \in [1, \infty]$, the Schatten $p$-norm $\norm{ \cdot }_p$ returns the $\ell_p$ norm
of the singular values of a matrix.
As usual, $\sigma_{r}$ refers to the $r$th largest singular value.

For a nonnegative integer $r$,
the phrase ``rank-$r$'' and its variants mean ``rank at most $r$.''
For a matrix $\mtx{M}$,
the symbol $\lowrank{\mtx{M}}{r}$ denotes a (simultaneous)
best rank-$r$ approximation of the matrix $\mtx{M}$ with respect to any Schatten $p$-norm.
We can take $\lowrank{\mtx{M}}{r}$ to be any $r$-truncated
singular value decomposition (SVD) of $\mtx{M}$~\cite[Sec.~6]{Hig89:Matrix-Nearness}.
Every best rank-$r$ approximation of a psd matrix is psd.

\section{Sketching and Fixed-Rank PSD Approximation}

We begin with a streaming data model for
a psd matrix that evolves via a sequence of general linear updates,
and it describes a randomized linear sketch for tracking the psd matrix.
To compute a fixed-rank psd approximation, we develop an algorithm
based on the Nystr{\"o}m method~\cite{WS01:Using-Nystrom},
a technique from the literature on kernel methods.
In contrast to previous approaches, \textbf{our algorithm uses a distinct
mechanism to truncate the rank of the approximation.}

\textbf{The Streaming Model.}
Fix a rank parameter $r$ in the range $1 \leq r \leq n$.
Initially, the psd matrix $\mtx{A} \in \F^{n \times n}$ equals
a known psd matrix $\mtx{A}_{\mathrm{init}} \in \F^{n \times n}$.
Then $\mtx{A}$ evolves via a series of linear updates:
\begin{equation} \label{eqn:linear-update}
\mtx{A} \gets \theta_1 \mtx{A} + \theta_2 \mtx{H}
\quad\text{where}\quad
\theta_i \in \R, \quad
\text{$\mtx{H} \in \F^{n \times n}$ is (conjugate) symmetric.}
\end{equation}
In many applications, the innovation $\mtx{H}$ is low-rank and/or sparse.
We assume that the evolving matrix $\mtx{A}$ always remains psd.
At one given instant, we must produce an accurate rank-$r$ approximation
of the psd matrix $\mtx{A}$ induced by the stream of linear updates.

\textbf{The Sketch.}
Fix a sketch size parameter $k$ in the range $r \leq k \leq n$.
Independent from $\mtx{A}$, we draw and fix a random test matrix
\begin{equation} \label{eqn:test-matrix}
\mtx{\Omega} \in \F^{n \times k}.
\end{equation}
See Sec.~\ref{sec:implementation} for a discussion of possible distributions.
The sketch of the matrix $\mtx{A}$ takes the form
\begin{equation} \label{eqn:sketch}
\mtx{Y} = \mtx{A} \mtx{\Omega} \in \F^{n \times k}.
\end{equation}
The sketch~\eqref{eqn:sketch} supports updates of the form~\eqref{eqn:linear-update}:
\begin{equation} \label{eqn:sketch-update}
\mtx{Y} \gets \theta_1 \mtx{Y} + \theta_2 \mtx{H} \mtx{\Omega}.
\end{equation}
To find a good rank-$r$ approximation, we must set the sketch size $k$ larger
than $r$.  But storage costs and computation also increase with $k$.
One of our main contributions is to clarify the role of $k$.

Under the model~\eqref{eqn:linear-update}, it is more or less
necessary to use a randomized linear sketch to track $\mtx{A}$~\cite{LNW14:Turnstile-Streaming}.
For psd matrices, sketches of the form~\eqref{eqn:test-matrix}--\eqref{eqn:sketch}
appear explicitly in Gittens's work~\cite{Git11:Spectral-Norm,Git13:Topics-Randomized,GM16:Revisiting-Nystrom-JMLR}.
\citet{TYUC17:Randomized-Single-View-TR} relies on
a more complicated sketch developed
in~\cite{WLRT08:Fast-Randomized,CW09:Numerical-Linear}.

\textbf{The Nystr{\"o}m Approximation.}
The Nystr{\"o}m method is a general technique for low-rank psd matrix
approximation.  Various instantiations appear in the papers~\cite{WS01:Using-Nystrom,Pla05:FastMap-MetricMap,
FBCM04:Spectral-Grouping,DM05:Nystrom-Method,HMT11:Finding-Structure,Git11:Spectral-Norm,CD13:Sublinear-Randomized,Git13:Topics-Randomized,GM16:Revisiting-Nystrom-JMLR,LLS+17:Algorithm-971}.

Here is the application to the present situation.
Given the test matrix $\mtx{\Omega}$ and the sketch $\mtx{Y} = \mtx{A\Omega}$,
the Nystr{\"o}m method constructs a rank-$k$ psd approximation of
the psd matrix $\mtx{A}$ via the formula
\begin{equation} \label{eqn:nystrom}
\hat{\mtx{A}}^{\mathrm{nys}} = \mtx{Y} (\mtx{\Omega}^* \mtx{Y})^{\dagger} \mtx{Y}^*.
\end{equation}
In most work on the Nystr{\"o}m method,
the test matrix $\mtx{\Omega}$ depends adaptively on $\mtx{A}$,
so these approaches are not valid in the streaming setting.
Gittens's framework~\cite{Git11:Spectral-Norm,Git13:Topics-Randomized,GM16:Revisiting-Nystrom-JMLR}
covers the streaming case.

\textbf{Fixed-Rank Nystr{\"o}m Approximation: Prior Art.}
To construct a Nystr{\"o}m approximation with exact rank $r$ from a sketch
of size $k$, the standard approach
is to truncate the center matrix to rank $r$:
\begin{equation} \label{eqn:gm}
\hat{\mtx{A}}^{\mathrm{nysfix}}_r = \mtx{Y} ( \lowrank{ \mtx{\Omega}^* \mtx{Y} }{r} )^{\dagger} \mtx{Y}^*.
\end{equation}
The truncated Nystr{\"o}m approximation~\eqref{eqn:gm} appears
in the many papers, including~\cite{Pla05:FastMap-MetricMap,DM05:Nystrom-Method,CD13:Sublinear-Randomized,GM13:Revisiting-Nystrom}.
We have found (Sec.~\ref{sec:numerics}) that the truncation method~\eqref{eqn:gm}
performs poorly in the present setting. This observation motivated us to search for more effective techniques.

\textbf{Fixed-Rank Nystr{\"o}m Approximation: Proposal.}
The purpose of this paper is to develop, analyze, and evaluate a new approach
for fixed-rank approximation of a psd matrix under the streaming model.
We propose a more intuitive rank-$r$ approximation:
\begin{equation} \label{eqn:Ahat-fixed}
\hat{\mtx{A}}_r = \lowrank{ \hat{\mtx{A}}^{\mathrm{nys}} }{r}.
\end{equation}
That is, we report a best rank-$r$ approximation of the full
Nystr{\"o}m approximation~\eqref{eqn:nystrom}.

This ``matrix nearness'' approach to fixed-rank approximation appears in the
papers~\cite{HMT11:Finding-Structure,Gu15:Subspace-Iteration,TYUC17:Randomized-Single-View-TR}.
The combination with the Nystr{\"o}m method~\eqref{eqn:nystrom} seems
totally natural.  Even so, we were unable to find a reference after an
exhaustive literature search and inquiries to experts on this subject.

\textbf{Summary of Contributions.}
This paper contains a number of advances over the prior art:

\vspace{-0.25pc}

\begin{enumerate} \setlength{\itemsep}{-0.1pc}
\item	We propose a distinct technique~\eqref{eqn:Ahat-fixed} for
truncating the Nystr{\"o}m approximation to rank $r$.
This formulation differs from earlier work on fixed-rank Nystr{\"o}m
approximations.

\item	We present a stable numerical implementation of~\eqref{eqn:Ahat-fixed}
based on the best practices outlined in the paper~\cite{LLS+17:Algorithm-971}.
This approach is essential for achieving high precision!  (Sec.~\ref{sec:implementation})

\item	We establish informative error bounds for the method~\eqref{eqn:Ahat-fixed}.
In particular, we prove that it attains
$(1 + \eps)$-relative error in the Schatten 1-norm when $k = \Theta(r/\eps)$.  (Sec.~\ref{sec:theory})

\item	We document numerical experiments on real and synthetic data
to demonstrate that our method dominates existing
techniques~\cite{GM13:Revisiting-Nystrom,TYUC17:Randomized-Single-View-TR} for fixed-rank psd approximation.
(Sec.~\ref{sec:numerics})
\end{enumerate}

\vspace{-0.25pc}

Psd matrix approximation is a ubiquitous problem, so we expect these results to have a broad impact.

\textbf{Related Work.}
Randomized algorithms for \textbf{low-rank matrix approximation} were proposed
in the late 1990s and developed into a technology in the 2000s;
see~\cite{HMT11:Finding-Structure,Mah11:Randomized-Algorithms,Woo14:Sketching-Tool}
for more background.  In the absence of constraints, such as streaming,
we recommend the general-purpose methods from~\cite{HMT11:Finding-Structure, LLS+17:Algorithm-971}.

Algorithms for low-rank matrix approximation in the important \textbf{streaming data} setting
are discussed in~\cite{WLRT08:Fast-Randomized,CW09:Numerical-Linear,HMT11:Finding-Structure,GPW12:Sketched-SVD,
Woo14:Sketching-Tool,CEMMP15:Dimensionality-Reduction,BWZ16:Optimal-Principal-STOC,TYUC17:Randomized-Single-View-TR}.
Few of these methods are designed for psd matrices.

\textbf{Nystr{\"o}m methods} for low-rank psd matrix approximation appear
in~\cite{WS01:Using-Nystrom,Pla05:FastMap-MetricMap,FBCM04:Spectral-Grouping,
DM05:Nystrom-Method,HMT11:Finding-Structure,Git11:Spectral-Norm,KMT12:Sampling-Methods,YLM+12:Nystrom-Method,
Git13:Topics-Randomized,GM16:Revisiting-Nystrom-JMLR,TYUC17:Randomized-Single-View-TR}.
These works mostly concern kernel matrices; they do not focus on the streaming model.

We are only aware of a few papers~\cite{Git11:Spectral-Norm,Git13:Topics-Randomized,GM16:Revisiting-Nystrom-JMLR,TYUC17:Randomized-Single-View-TR} on algorithms for \textbf{psd matrix approximation}
that operate under the \textbf{streaming model}~\eqref{eqn:linear-update}.  These papers form the comparison group.

Finally, let us mention two very recent \textbf{theoretical papers}~\cite{CW17:Low-Rank-PSD,MW17:Sublinear-Time}
that present existential results on algorithms for fixed-rank psd matrix approximation.
The approach in~\cite{CW17:Low-Rank-PSD} is only appropriate for sparse input matrices,
while the work~\cite{MW17:Sublinear-Time} is not valid in the streaming setting.

\section{Implementation}
\label{sec:implementation}

\textbf{Distributions for the Test Matrix.}
To ensure that the sketch is informative,
we must draw the test matrix~\eqref{eqn:test-matrix} at random from
a suitable distribution.  The choice of distribution determines the
computational requirements for the sketch~\eqref{eqn:sketch},
the linear updates~\eqref{eqn:sketch-update}, and the
matrix approximation~\eqref{eqn:Ahat-fixed}.
It also affects the quality of the approximation~\eqref{eqn:Ahat-fixed}.
Let us outline some of the most useful distributions.
An exhaustive discussion is outside the scope of our work, but
see~\cite{Lib09:Accelerated-Dense,HMT11:Finding-Structure,Mah11:Randomized-Algorithms,
Git13:Topics-Randomized,GM16:Revisiting-Nystrom-JMLR,Woo14:Sketching-Tool,TYUC17:Randomized-Single-View-TR}.

\textbf{Isotropic Models.}
Mathematically, the most natural model is to construct
a test matrix $\mtx{\Omega} \in \F^{n \times k}$
whose range is a uniformly random $k$-dimensional subspace in $\F^n$.
There are two approaches:

\vspace{-0.25pc}

\begin{enumerate} \setlength{\itemsep}{-0.1pc}
\item	\textbf{Gaussian.}  Draw each entry of the matrix $\mtx{\Omega} \in \F^{n \times k}$ independently
at random from the standard normal distribution on $\F$.

\item	\textbf{Orthonormal.}  Draw a Gaussian matrix $\mtx{G} \in \F^{n \times k}$,
as above.  Compute a thin orthogonal--triangular factorization $\mtx{G} = \mtx{\Omega R}$
to obtain the test matrix $\mtx{\Omega} \in \F^{n \times k}$.  Discard $\mtx{R}$.
\end{enumerate}

\vspace{-0.25pc}

Gaussian and orthonormal test matrices both require storage of $kn$ floating-point numbers
in $\F$ for the test matrix $\mtx{\Omega}$ and
another $kn$ floating-point numbers for the sketch $\mtx{Y}$.
In both cases, the cost of multiplying
a vector in $\F^n$ into $\mtx{\Omega}$ is $\Theta(kn)$ floating-point operations.

For isotropic models, we can analyze the approximation~\eqref{eqn:Ahat-fixed} in detail.
In exact arithmetic, Gaussian and isotropic test matrices yield identical Nystr{\"o}m
approximations
{\ifdefined\supplement (Proposition~\ref{prop:nys-proj}).
\else (Supplement). \fi}In floating-point arithmetic,
orthonormal matrices are more stable for large $k$,
but we can generate Gaussian matrices with
less arithmetic and communication.  References for isotropic test matrices
include~\cite{MRT11:Randomized-Algorithm,HMT11:Finding-Structure,Gu15:Subspace-Iteration}.

\textbf{Subsampled Scrambled Fourier Transform (SSFT).}
One shortcoming of the isotropic models is the cost of storing the test matrix
and the cost of multiplying a vector into the test matrix.  We can often reduce these
costs using an SSFT test matrix.  An SSFT takes the form
\begin{equation} \label{eqn:ssft}
\mtx{\Omega} = \mtx{\Pi}_1 \mtx{F} \mtx{\Pi}_2 \mtx{F} \mtx{R} \in \F^{n \times k}.
\end{equation}
The $\mtx{\Pi}_i \in \F^{n \times n}$ are independent, signed permutation matrices,\footnote{A signed permutation has exactly one nonzero entry in each row and column;
the nonzero has modulus one.}
chosen uniformly at random.  The matrix $\mtx{F} \in \F^{n \times n}$
is a discrete Fourier transform ($\F = \C$) or a discrete cosine transform ($\F = \R$).
The matrix $\mtx{R} \in \F^{n \times k}$ is a restriction to $k$ coordinates,
chosen uniformly at random.

An SSFT $\mtx{\Omega}$ requires only $\Theta(n)$ storage, but the sketch $\mtx{Y}$
still requires storage of $kn$ numbers.  We can multiply a vector in $\F^n$
into $\mtx{\Omega}$ using $\Theta(n \log n)$ arithmetic operations via an FFT
or FCT algorithm.  Thus, for most choices of sketch size $k$, the SSFT improves
over the isotropic models.

In practice, the SSFT yields matrix approximations whose quality is identical to
those we obtain with an isotropic test matrix (Sec.~\ref{sec:numerics}). Although the analysis for SSFTs is less complete, the empirical evidence confirms that the theory for isotropic models also
offers excellent guidance for SSFTs.
References for SSFTs and related test matrices include~\cite{WLRT08:Fast-Randomized,AC09:Fast-Johnson-Lindenstrauss,Lib09:Accelerated-Dense,HMT11:Finding-Structure,Tro11:Improved-Analysis,BG13:Improved-Matrix,CNW16:Optimal-Approximate}.

\begin{algorithm}[tb]
  \caption{\textsl{Sketch Initialization.}  Implements~\eqref{eqn:test-matrix}--\eqref{eqn:sketch} with a random orthonormal test matrix.
  \label{alg:sketch}}
  \begin{algorithmic}[1]
    \Require{Positive-semidefinite input matrix $\mtx{A} \in \F^{n \times n}$;
    sketch size parameter $k$}
    \Ensure{Constructs test matrix $\mtx{\Omega} \in \F^{n \times k}$
    and sketch $\mtx{Y} = \mtx{A\Omega} \in \F^{n \times k}$}
\vspace{0.5pc}

	\State \textbf{local:} $\mtx{\Omega}, \mtx{Y}$
		\Comment{Internal variables for \textsc{NystromSketch}}
	\Function{NystromSketch}{$\mtx{A}; k$}
		\Comment{Constructor}
    \If{$\F = \R$}
    \State	$\mtx{\Omega} \gets \texttt{randn}(n, k)$
			\EndIf
	\If{$\F = \C$}
    \State	$\mtx{\Omega} \gets \texttt{randn}(n, k) + \texttt{i} * \texttt{randn}(n, k)$
			\EndIf
	\State	$\mtx{\Omega} \gets \texttt{orth}(\mtx{\Omega})$
		\Comment{Improve numerical stability}
	\State	$\mtx{Y} \gets \mtx{A\Omega}$
			\EndFunction

	\vspace{0.25pc}

\end{algorithmic}
\end{algorithm}

\begin{algorithm}[tb]
  \caption{\textsl{Linear Update.}  Implements~\eqref{eqn:sketch-update}.
  \label{alg:update}}
  \begin{algorithmic}[1]
    \Require{Scalars $\theta_1, \theta_2 \in \R$ and conjugate symmetric $\mtx{H} \in \F^{n \times n}$}
    \Ensure{Updates sketch to reflect linear innovation $\mtx{A} \gets \theta_1 \mtx{A} + \theta_2\mtx{H}$}
\vspace{0.5pc}

	\State \textbf{local:} $\mtx{\Omega}, \mtx{Y}$
		\Comment{Internal variables for \textsc{NystromSketch}}

	\Function{LinearUpdate}{$\theta_1, \theta_2, \mtx{H}$}
	\State	$\mtx{Y} \gets \theta_1 \mtx{Y} + \theta_2 \mtx{H} \mtx{\Omega}$
	\EndFunction

	\vspace{0.25pc}

\end{algorithmic}
\end{algorithm}

\begin{algorithm}[t!]
  \caption{\textsl{Fixed-Rank PSD Approximation.}  Implements~\eqref{eqn:Ahat-fixed}.
  \label{alg:low-rank-recon}}
  \begin{algorithmic}[1]
  	\Require{Matrix $\mtx{A}$ in sketch must be psd; rank parameter $1 \leq r \leq k$}
    \Ensure{Returns factors $\mtx{U} \in \F^{n \times r}$ with orthonormal columns and nonnegative, diagonal $\mtx{\Lambda} \in \F^{r \times r}$ that form a rank-$r$ psd approximation $\hat{\mtx{A}}_{r} = \mtx{U\Lambda U}^*$ of the sketched matrix $\mtx{A}$}
\vspace{0.5pc}

	\State \textbf{local:} $\mtx{\Omega}, \mtx{Y}$
		\Comment{Internal variables for \textsc{NystromSketch}}

	\Function{FixedRankPSDApprox}{$r$}
	\State	$\nu \gets \mu \, \texttt{norm}(\mtx{Y})$
		\Comment	$\mu = 2.2 \cdot 10^{-16}$ in double precision

	\State	$\mtx{Y} \gets \mtx{Y} + \nu \mtx{\Omega}$
		\Comment	Sketch of shifted matrix $\mtx{A} + \nu \Id$

	\State	$\mtx{B} \gets \mtx{\Omega}^* \mtx{Y}$

	\State	$\mtx{C} \gets \texttt{chol}( (\mtx{B} + \mtx{B}^*) / 2 )$
		\Comment	Force symmetry 
	\State	$(\mtx{U}, \mtx{\Sigma}, \sim) \gets \texttt{svd}( \mtx{Y} / \mtx{C}, \texttt{'econ'} )$
		\Comment	Solve least squares problem; form thin SVD

	\State	$\mtx{U} \gets \mtx{U}(\texttt{:, 1:r})$ and $\mtx{\Sigma} \gets \mtx{\Sigma}(\texttt{1:r, 1:r})$
		\Comment	Truncate to rank $r$

	\State	$\mtx{\Lambda} \gets \max\{0, \mtx{\Sigma}^2 - \nu \Id\}$
		\Comment	Square to get eigenvalues; remove shift

	\State \Return{$(\mtx{U}, \mtx{\Lambda})$}
	\EndFunction

	\vspace{0.25pc}

\end{algorithmic}
\end{algorithm}

\textbf{Numerically Stable Implementation.}
It requires care to compute the fixed-rank approximation~\eqref{eqn:Ahat-fixed}.
{\ifdefined\supplement App.~\ref{sec:numerics-extra}
\else The supplement \fi}shows that a poor implementation may produce an approximation
with 100\% error!

Let us outline a numerically stable and very accurate implementation of~\eqref{eqn:Ahat-fixed},
based on an idea from~\cite{Tyg14:Matlab-Routines,LLS+17:Algorithm-971}.
Fix a small parameter $\nu > 0$.
Instead of approximating the psd matrix $\mtx{A}$ directly,
we approximate the shifted matrix $\mtx{A}_{\nu} = \mtx{A} + \nu \Id$ and then remove the shift.
Here are the steps:
\vspace{-0.25pc}
\begin{enumerate} \setlength{\itemsep}{-0.1pc}
\item	Construct the shifted sketch $\mtx{Y}_{\nu} = \mtx{Y} + \nu \mtx{\Omega}$.

\item	Form the matrix $\mtx{B} = \mtx{\Omega}^* \mtx{Y}_{\nu}$.

\item	Compute a Cholesky decomposition $\mtx{B} = \mtx{CC}^*$.

\item	Compute $\mtx{E} = \mtx{Y}_{\nu} \mtx{C}^{-1}$ by back-substitution.

\item	Compute the (thin) singular value decomposition $\mtx{E} = \mtx{U \Sigma V}^*$.

\item	Form $\hat{\mtx{A}}_r = \mtx{U} \lowrank{\mtx{\Sigma}^2 - \nu \Id}{r} \mtx{U}^*$.
\end{enumerate}
\vspace{-0.25pc}
\noindent
The pseudocode addresses some additional implementation details.
Related, but distinct, methods were proposed by Williams \& Seeger~\cite{WS01:Using-Nystrom}
and analyzed in Gittens's thesis~\cite{Git13:Topics-Randomized}.

\textbf{Pseudocode.}
We present detailed pseudocode for the sketch~\eqref{eqn:test-matrix}--\eqref{eqn:sketch-update}
and the implementation of the fixed-rank psd approximation~\eqref{eqn:Ahat-fixed} described above.
For simplicity, we only elaborate the case of a random orthonormal
test matrix; we have also developed an SSFT implementation for empirical testing.
The pseudocode uses both mathematical notation and \textsc{Matlab 2017a} functions.

\textbf{Algorithms and Computational Costs.}
Algorithm~\ref{alg:sketch} constructs a random orthonormal test matrix,
and computes the sketch~\eqref{eqn:sketch} of an input matrix.
The test matrix and sketch require the storage of $2kn$ floating-point
numbers.  Owing to the orthogonalization step, the construction of the
test matrix requires $\Theta(k^2 n)$ floating-point operations.
For a general input matrix, the sketch requires $\Theta(k n^2)$
floating-point operations; this cost can be removed by initializing
the input matrix to zero.

Algorithm~\ref{alg:update} implements the linear update~\eqref{eqn:sketch-update}
to the sketch.  Nominally, the computation requires $\Theta(kn^2)$ arithmetic
operations, but this cost can be reduced when $\mtx{H}$ has structure
(e.g., low rank).  Using the SSFT test matrix~\eqref{eqn:ssft} also reduces this cost.

Algorithm~\ref{alg:low-rank-recon} computes the rank-$r$ psd approximation~\eqref{eqn:Ahat-fixed}.
This method requires additional storage of $\Theta(kn)$.
The arithmetic cost is $\Theta(k^2 n)$ operations,
which is dominated by the SVD of the matrix $\mtx{E}$.

\section{Theoretical Results}
\label{sec:theory}

\textbf{Relative Error Bound.}
Our first result is an accurate bound for the expected Schatten 1-norm
error in the fixed-rank psd approximation~\eqref{eqn:Ahat-fixed}.

\begin{theorem}[Fixed-Rank Nystr{\"o}m: Relative Error] \label{thm:error-fixed}
Assume $1 \leq r < k \leq n$.   Let $\mtx{A} \in \F^{n \times n}$ be a psd matrix.
Draw a test matrix $\mtx{\Omega} \in \F^{n \times k}$ from the Gaussian or orthonormal distribution,
and form the sketch $\mtx{Y} = \mtx{A\Omega}$.  Then the approximation $\hat{\mtx{A}}_r$ given
by~\eqref{eqn:nystrom} and~\eqref{eqn:Ahat-fixed} satisfies \begin{align}
\Expect \norm{ \mtx{A} - \hat{\mtx{A}}_r }_{1}
	&\leq \left(1 + \frac{r}{k - r - \alpha}\right) \cdot \norm{ \mtx{A} - \lowrank{\mtx{A}}{r} }_1;
	\label{eqn:error-fixed-S1} \\
\Expect \norm{ \mtx{A} - \hat{\mtx{A}}_r }_{\infty}
	&\leq \norm{\mtx{A} - \lowrank{\mtx{A}}{r}}_{\infty} + \frac{r}{k - r - \alpha} \cdot \norm{ \mtx{A} - \lowrank{\mtx{A}}{r} }_1.
	\label{eqn:error-fixed-Sinf}
	\end{align}

The quantity $\alpha(\R) = 1$ and $\alpha(\C) = 0$.  Similar results hold with high probability.
\end{theorem}

\noindent
The proof of Theorem~\ref{thm:error-fixed} appears in
{\ifdefined\supplement App.~\ref{sec:proofs}.
\else the supplement. \fi}
In contrast to previous analyses of Nystr{\"o}m methods,
Theorem~\ref{thm:error-fixed} yields explicit, sharp constants.
As a consequence, the formulae~\eqref{eqn:error-fixed-S1}--\eqref{eqn:error-fixed-Sinf} offer
an \emph{a priori} mechanism for selecting the sketch size $k$
to achieve a desired error bound.  In particular, for each $\eps > 0$,
$$
k = (1+\eps^{-1})r + \alpha
\quad\text{implies}\quad
\Expect \norm{ \mtx{A} - \hat{\mtx{A}}_r }_{1}
	\leq (1 + \eps) \cdot \norm{\mtx{A} - \lowrank{\mtx{A}}{r}}_1. $$
Thus, we can attain an arbitrarily small relative error
in the Schatten 1-norm.  In the streaming setting, the scaling $k = \Theta(r/\eps)$
is optimal for this result~\cite[Thm.~4.2]{GLPW16:Frequent-Directions}.
Furthermore, it is impossible~\cite[Sec.~6.2]{Woo14:Sketching-Tool} 
to obtain ``pure'' relative error bounds in the Schatten $\infty$-norm unless $k = \Omega(n)$.

\textbf{The Role of Spectral Decay.}
To circumvent these limitations, it is necessary to develop
a different kind of error bound.
Our second result shows that the fixed-rank psd approximation~\eqref{eqn:Ahat-fixed}
automatically exploits decay in the spectrum of the input matrix.

\begin{theorem}[Fixed-Rank Nystr{\"o}m: Spectral Decay] \label{thm:error-fixed-2}
Instate the notation and assumptions of Theorem~\ref{thm:error-fixed}.
Then
\begin{align}
\Expect \norm{ \mtx{A} - \hat{\mtx{A}}_r }_{1}
	&\leq \norm{\mtx{A} - \lowrank{\mtx{A}}{r}}_1
	+ 2 \min_{\varrho < k - \alpha}\left[ \left(1 + \frac{\varrho}{k - \varrho - \alpha}\right)
	\cdot \norm{\mtx{A} - \lowrank{\mtx{A}}{\varrho}}_1 \right];
	 \label{eqn:error-fixed-2-S1} \\
\Expect \norm{ \mtx{A} - \hat{\mtx{A}}_r }_{\infty}
	&\leq \norm{\mtx{A} - \lowrank{\mtx{A}}{r}}_{\infty}
	+ 2 \min_{\varrho < k - \alpha}\left[ \left(1 + \frac{\varrho}{k - \varrho - \alpha}\right)
	\cdot \norm{\mtx{A} - \lowrank{\mtx{A}}{\varrho}}_1 \right].
	 \label{eqn:error-fixed-2-Sinf}
\end{align}
The index $\varrho$ ranges over the natural numbers.
\end{theorem}

\noindent
The proof of Theorem~\ref{thm:error-fixed-2} appears in
{\ifdefined\supplement App.~\ref{sec:proofs}.
\else the supplement. \fi}
Here is one way to understand this result.
As the index $\varrho$ increases, the quantity $\varrho / (k - \varrho - \alpha)$ increases
while the rank-$\varrho$ approximation error decreases. Theorem~\ref{thm:error-fixed-2} states that the approximation~\eqref{eqn:Ahat-fixed} automatically achieves the best tradeoff between these two terms.
When the spectrum of $\mtx{A}$ decays, the rank-$\varrho$ approximation
error may be far smaller than the rank-$r$ approximation error.  In this case,
Theorem~\ref{thm:error-fixed-2} is tighter than Theorem~\ref{thm:error-fixed},
although the prediction is more qualitative.

\textbf{Additional Results.}
The proofs can be extended to obtain high-probability bounds, as well as results
for other Schatten norms or for other test matrices
{\ifdefined\supplement (App.~\ref{sec:proofs}).
\else (Supplement). \fi}

\section{Numerical Performance}
\label{sec:numerics}

\textbf{Experimental Setup.}
In many streaming applications, such as~\cite{YUTC17:Sketchy-Decisions},
it is essential that the sketch uses as little memory as
possible and that the psd approximation achieves the best possible
error.  For the methods we consider, the arithmetic costs of
linear updates and psd approximation are roughly comparable.
Therefore, we only assess storage and accuracy.

For the numerical experiments,
the field $\F = \C$ except when noted explicitly.
Choose a psd input matrix $\mtx{A} \in \F^{n \times n}$
and a target rank $r$.  Then fix a sketch size parameter $k$ with $r \leq k \leq n$.
For each trial, draw the test matrix $\mtx{\Omega}$ from the orthonormal or the SSFT distribution, and form the sketch $\mtx{Y} = \mtx{A \Omega}$ of the input matrix.
Using Algorithm~\ref{alg:low-rank-recon}, compute the rank-$r$ psd approximation $\hat{\mtx{A}}_r$
defined in~\eqref{eqn:Ahat-fixed}.
We evaluate the performance using the relative error metric:
\begin{equation} \label{eqn:relative-error}
\text{Schatten $p$-norm relative error} \quad = \quad
\frac{\norm{ \mtx{A} - \hat{\mtx{A}}_{r}}_p}{\norm{\mtx{A} - \lowrank{\mtx{A}}{r}}_p}
	- 1.
\end{equation}
We perform 20 independent trials and report the average error.

We compare our method~\eqref{eqn:Ahat-fixed} with the standard truncated Nystr{\"om} approximation~\eqref{eqn:gm};
the best reference for this type of approach is~\cite[Sec.~2.2]{GM13:Revisiting-Nystrom}.
The approximation~\eqref{eqn:gm} is constructed from the same sketch as~\eqref{eqn:Ahat-fixed},
so the experimental procedure is identical.

We also consider the sketching method and psd approximation
algorithm~\cite[Alg.~9]{TYUC17:Randomized-Single-View-TR} based on earlier work
from~\cite{WLRT08:Fast-Randomized,CW09:Numerical-Linear,HMT11:Finding-Structure}.
We implemented this sketch with orthonormal matrices and also with SSFT matrices.
The sketch has two different parameters $(k, \ell)$,
so we select the parameters that result in the minimum relative error.
Otherwise, the experimental procedure is the same.

We apply the methods to representative input matrices;
{\ifdefined\supplement see Figure~\ref{fig:spectra} for the spectra.
\else see the Supplement for plots of the spectra. \fi}

\textbf{Synthetic Examples.}
The synthetic examples are \textbf{diagonal} with dimension $n = 10^3$;
results for larger and non-diagonal matrices are similar.
These matrices are parameterized by an effective rank parameter $R$,
which takes values in $\{ 5, 10, 20 \}$.  We compute approximations
with rank $r = 10$.

\vspace{-0.25pc}

\begin{enumerate} \item	\textbf{Low-Rank + PSD Noise.}  These matrices take the form
\begin{equation*}
\mtx{A} = \diag( \underbrace{1, \dots, 1}_R, 0, \dots, 0 )
	+ \xi n^{-1} \mtx{W} \in \F^{n \times n}.
\end{equation*}
The matrix $\mtx{W} \in \F^{n \times n}$ has the $\textsc{Wishart}(n, n; \F)$ distribution;
that is, $\mtx{W} = \mtx{GG}^*$ where $\mtx{G} \in \F^{n \times n}$ is standard normal.
The parameter $\xi$ controls the signal-to-noise ratio.  We consider three examples:
\texttt{LowRankLowNoise} ($\xi = 10^{-4}$),
\texttt{LowRankMedNoise} ($\xi = 10^{-2}$),
\texttt{LowRankHiNoise} ($\xi = 10^{-1}$).

\item	\textbf{Polynomial Decay.}  These matrices take the form
\begin{equation*}
\mtx{A} = \diag( \underbrace{1, \dots, 1}_R, 2^{-p}, 3^{-p}, \dots, (n - R + 1)^{-p} )
	\in \F^{n \times n}.
\end{equation*}
The parameter $p > 0$ controls the rate of polynomial decay.  We consider three examples:
\texttt{PolyDecaySlow} ($p = 0.5$),
\texttt{PolyDecayMed} ($p = 1$),
\texttt{PolyDecayFast} ($p = 2$).

\item	\textbf{Exponential Decay.}  These matrices take the form
\begin{equation*}
\mtx{A} = \diag( \underbrace{1, \dots, 1}_R, 10^{-q}, 10^{-2q}, \dots, 10^{-(n - R)q} )
	\in \F^{n \times n}.
\end{equation*}
The parameter $q > 0$ controls the rate of exponential decay.  We consider three examples:
\texttt{ExpDecaySlow} ($q = 0.1$),
\texttt{ExpDecayMed} ($q = 0.25$),
\texttt{ExpDecayFast} ($q = 1$).

\end{enumerate}

\vspace{-0.25pc}

\textbf{Application Examples.}
We also consider \textbf{non-diagonal} matrices inspired by the SDP algorithm~\cite{YUTC17:Sketchy-Decisions}.

\vspace{-0.25pc}

\begin{enumerate}
\item	\texttt{MaxCut}: This is a \textbf{real-valued} psd matrix with dimension $n = 2 \, 000$, and its effective rank $R = 14$.  We form approximations with rank $r \in \{1, 14\}$.  The matrix is an approximate solution to the \textsc{MaxCut} SDP~\cite{GW95:Improved-Approximation} for the sparse graph \texttt{G40}~\cite{DH11:University-Florida}.

\item	\texttt{PhaseRetrieval}: This is a psd matrix with dimension $n = 25 \, 921$.
It has exact rank $250$, but its effective rank $R = 5$.
We form approximations with rank $r \in \{1, 5\}$.  The matrix is an approximate solution to a phase
retrieval SDP; it was provided by the authors of~\cite{YUTC17:Sketchy-Decisions}.
\end{enumerate}

\vspace{-0.25pc}

\begin{figure}[t]
\begin{center}
\begin{subfigure}{.35\textwidth}
\begin{center}
\includegraphics[height=1.5in]{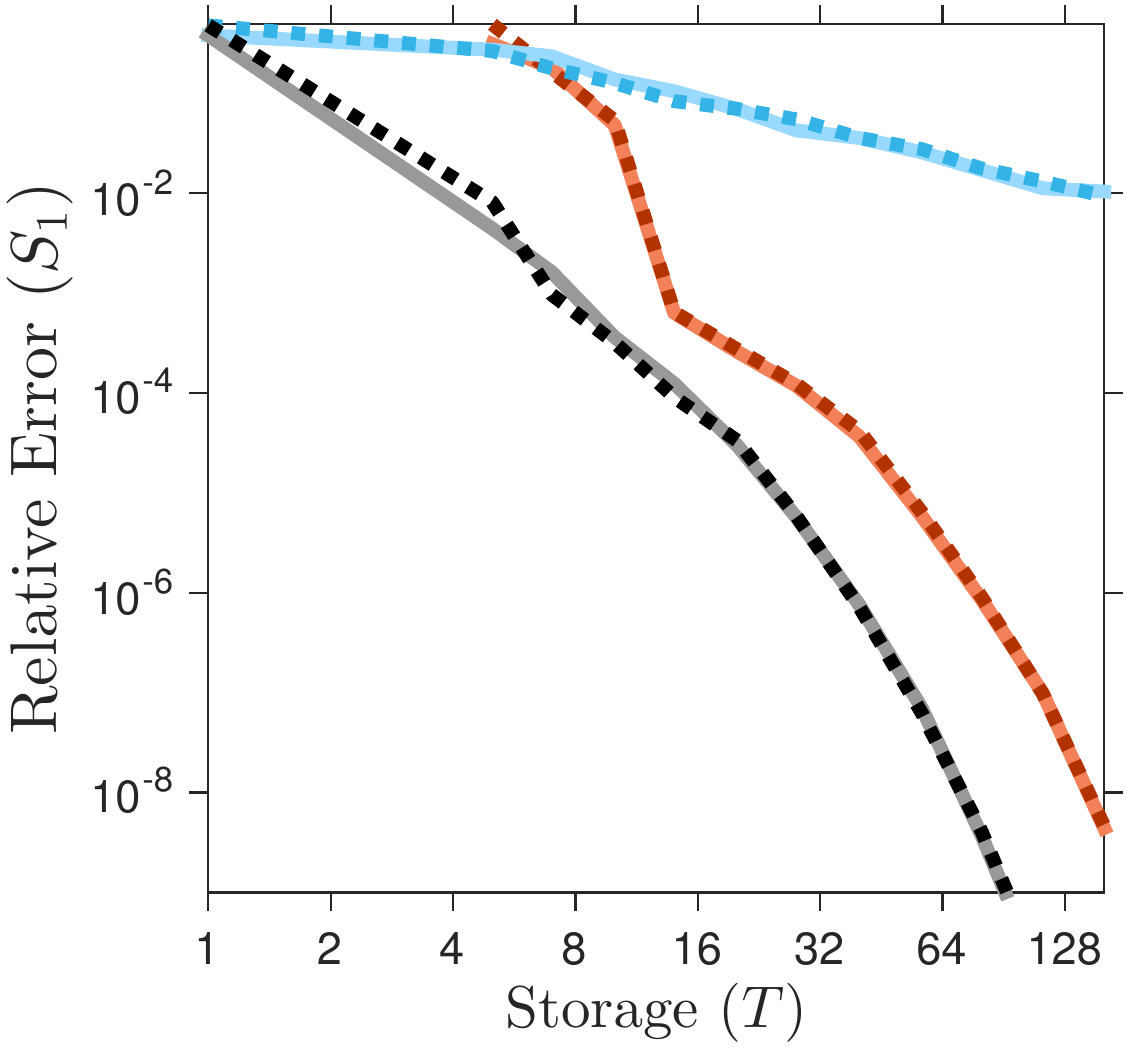}
\caption{\texttt{PhaseRetrieval} $(r = 1)$}
\end{center}
\end{subfigure}
\begin{subfigure}{.35\textwidth}
\begin{center}
\includegraphics[height=1.5in]{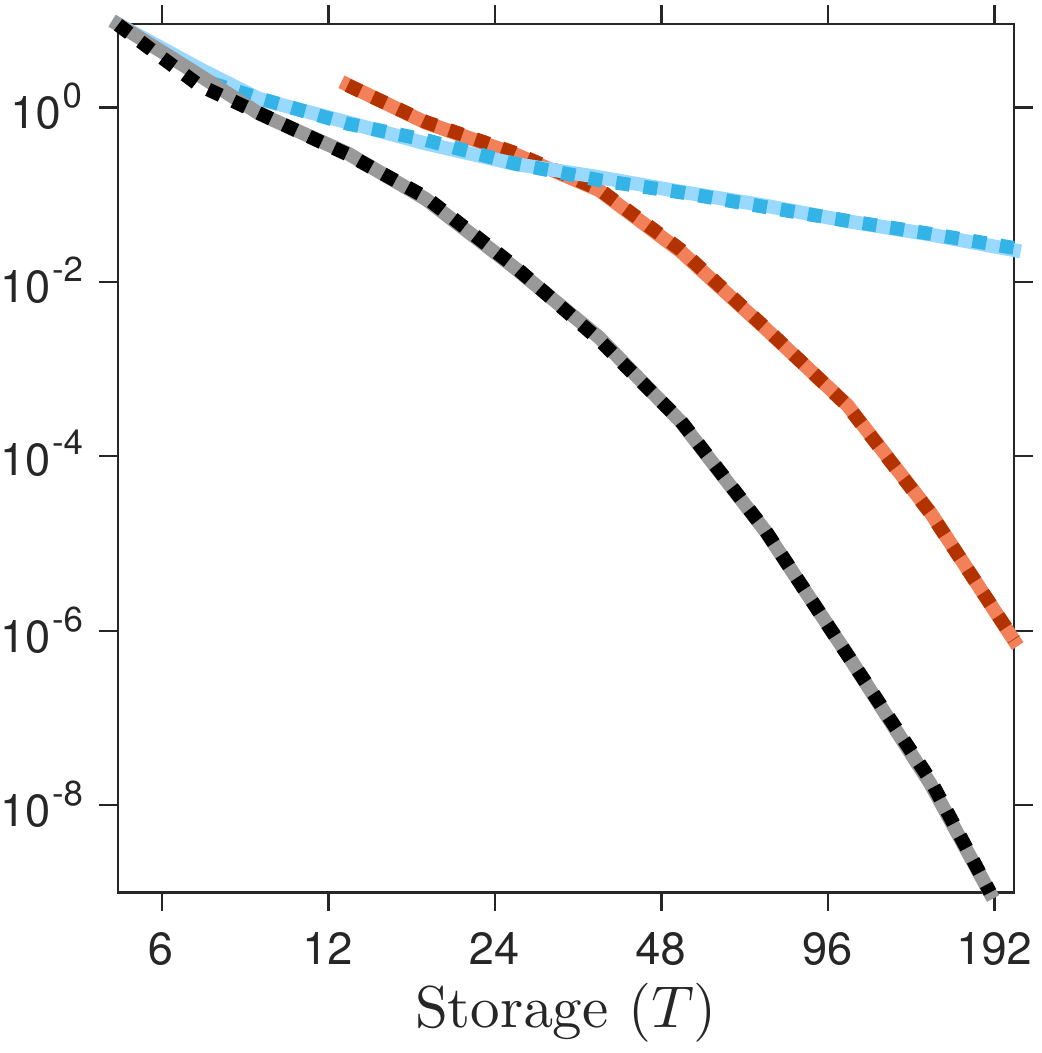}
\caption{\texttt{PhaseRetrieval} $(r = 5)$}
\end{center}
\end{subfigure}

\vspace{1pc}

\begin{subfigure}{.35\textwidth}
\begin{center}
\includegraphics[height=1.5in]{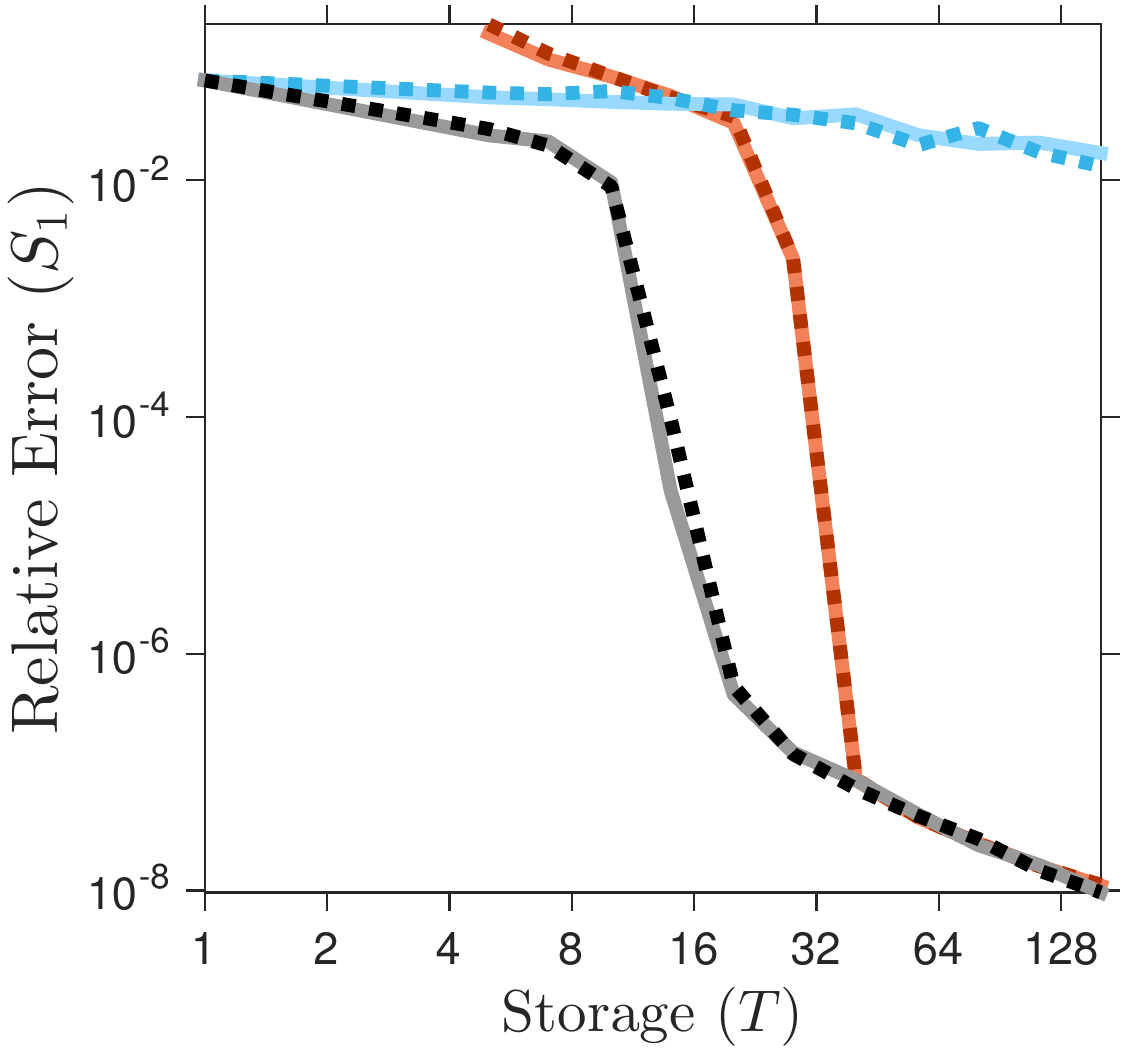}
\caption{\texttt{MaxCut} $(r = 1)$}
\end{center}
\end{subfigure}
\begin{subfigure}{.35\textwidth}
\begin{center}
\includegraphics[height=1.5in]{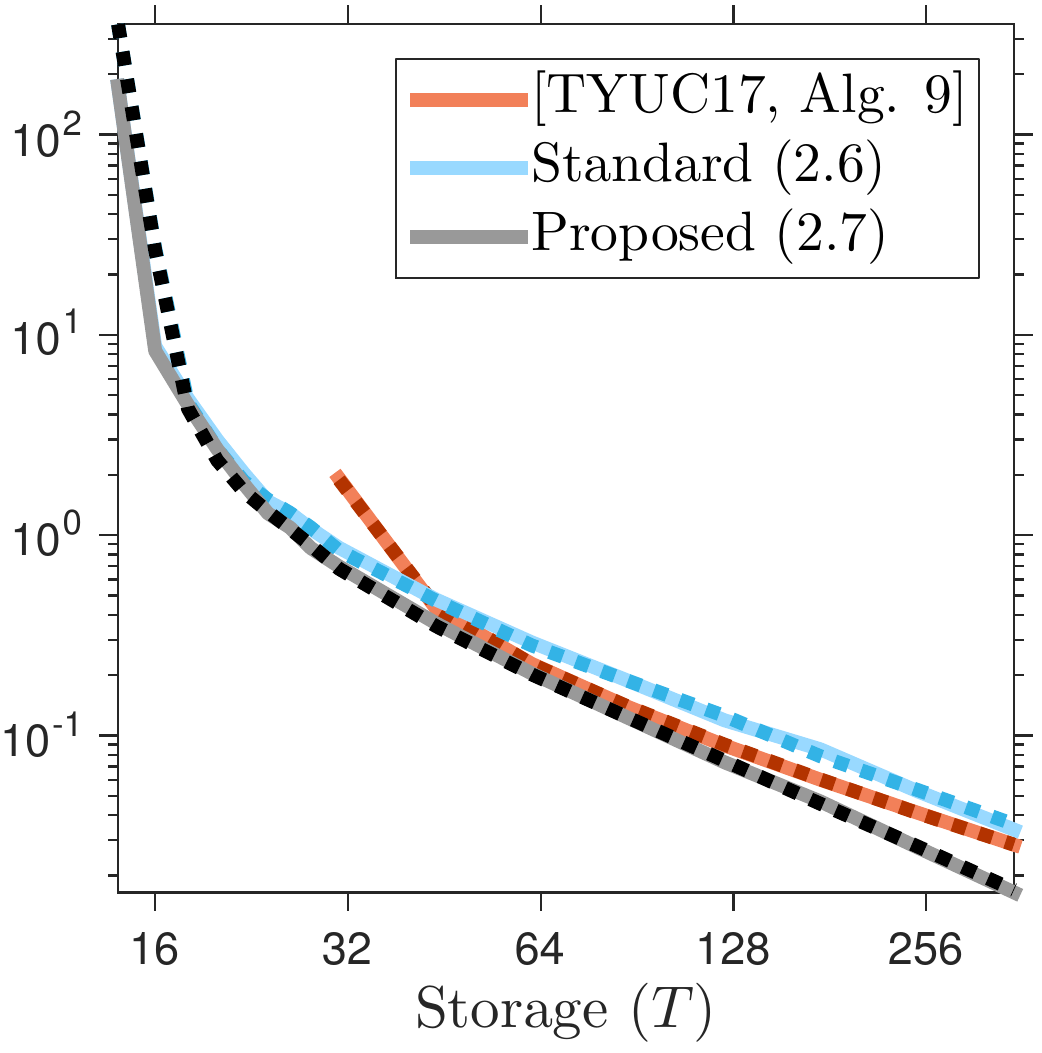}
\caption{\texttt{MaxCut} $(r = 14)$}
\end{center}
\end{subfigure}
\end{center}

\caption{\textbf{Application Examples, Approximation Rank $r$, Schatten $1$-Norm Error.}
The data series show the performance
of three algorithms for rank-$r$ psd approximation.
\textbf{Solid lines} are generated from the Gaussian sketch;
\textbf{dashed lines} are from the SSFT sketch.
Each panel displays the Schatten 1-norm relative error~\eqref{eqn:relative-error}
as a function of storage cost $T$.  See Sec.~\ref{sec:numerics}
for details.}
\label{fig:data-S1}
\end{figure}

\textbf{Experimental Results.}
Figures~\ref{fig:data-S1}--\ref{fig:synthetic-S1-R10} display the performance
of the three fixed-rank psd approximation methods for a subcollection of the input matrices.
The vertical axis is the Schatten $1$-norm relative error~\eqref{eqn:relative-error}.
The variable $T$ on the horizontal axis is proportional to the storage required for
the sketch only.  For the Nystr{\"o}m-based approximations~\eqref{eqn:gm}--\eqref{eqn:Ahat-fixed},
we have the correspondence $T = k$.
For the approximation~\cite[Alg.~9]{TYUC17:Randomized-Single-View-TR}, we set $T = k + \ell$.

The experiments demonstrate that the proposed method~\eqref{eqn:Ahat-fixed}
has a significant benefit over the alternatives for input matrices that admit a
good low-rank approximation.  It equals or improves on the competitors for almost
all other examples and storage budgets.
{\ifdefined\supplement App.~\ref{sec:numerics-extra}
\else The supplement \fi}contains additional numerical results;
these experiments only reinforce the message of Figures~\ref{fig:data-S1}--\ref{fig:synthetic-S1-R10}.

\begin{figure}[t]
\begin{center}
\begin{subfigure}{.325\textwidth}
\begin{center}
\includegraphics[height=1.5in]{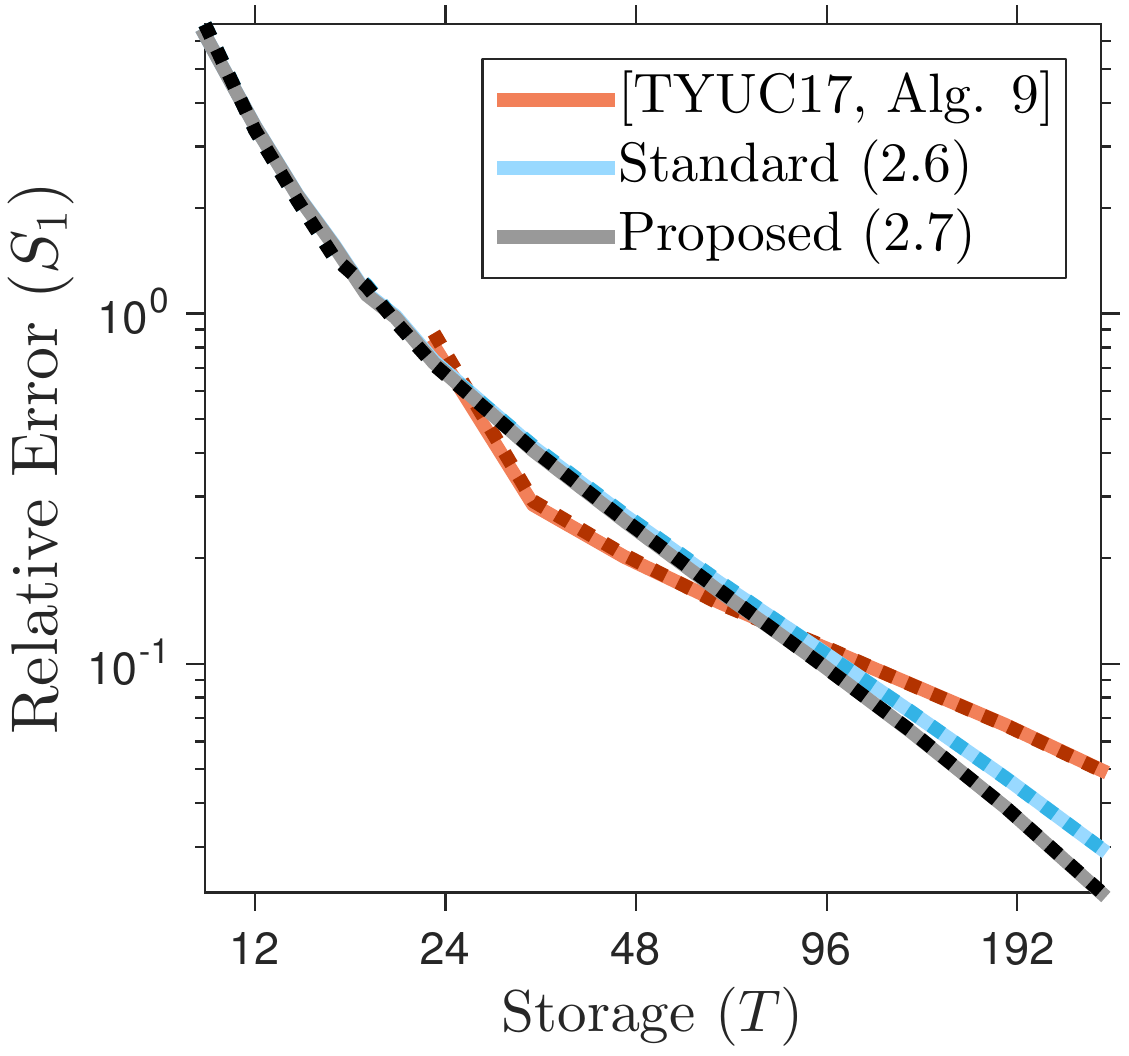}
\caption{\texttt{LowRankLowNoise}}
\label{fig:LR-algs}
\end{center}
\end{subfigure}
\begin{subfigure}{.325\textwidth}
\begin{center}
\includegraphics[height=1.5in]{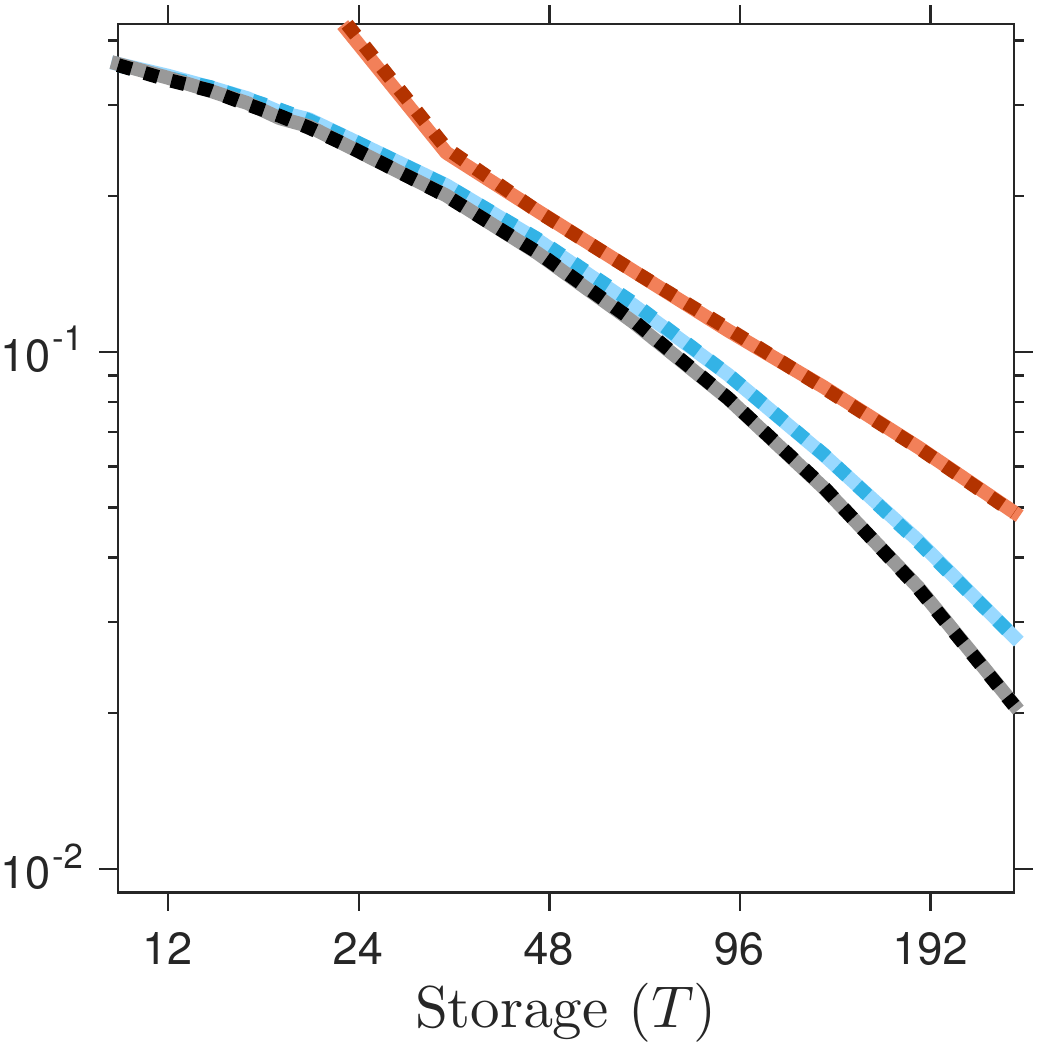}
\caption{\texttt{LowRankMedNoise}}
\label{fig:MED-theory}
\end{center}
\end{subfigure}
\begin{subfigure}{.325\textwidth}
\begin{center}
\includegraphics[height=1.5in]{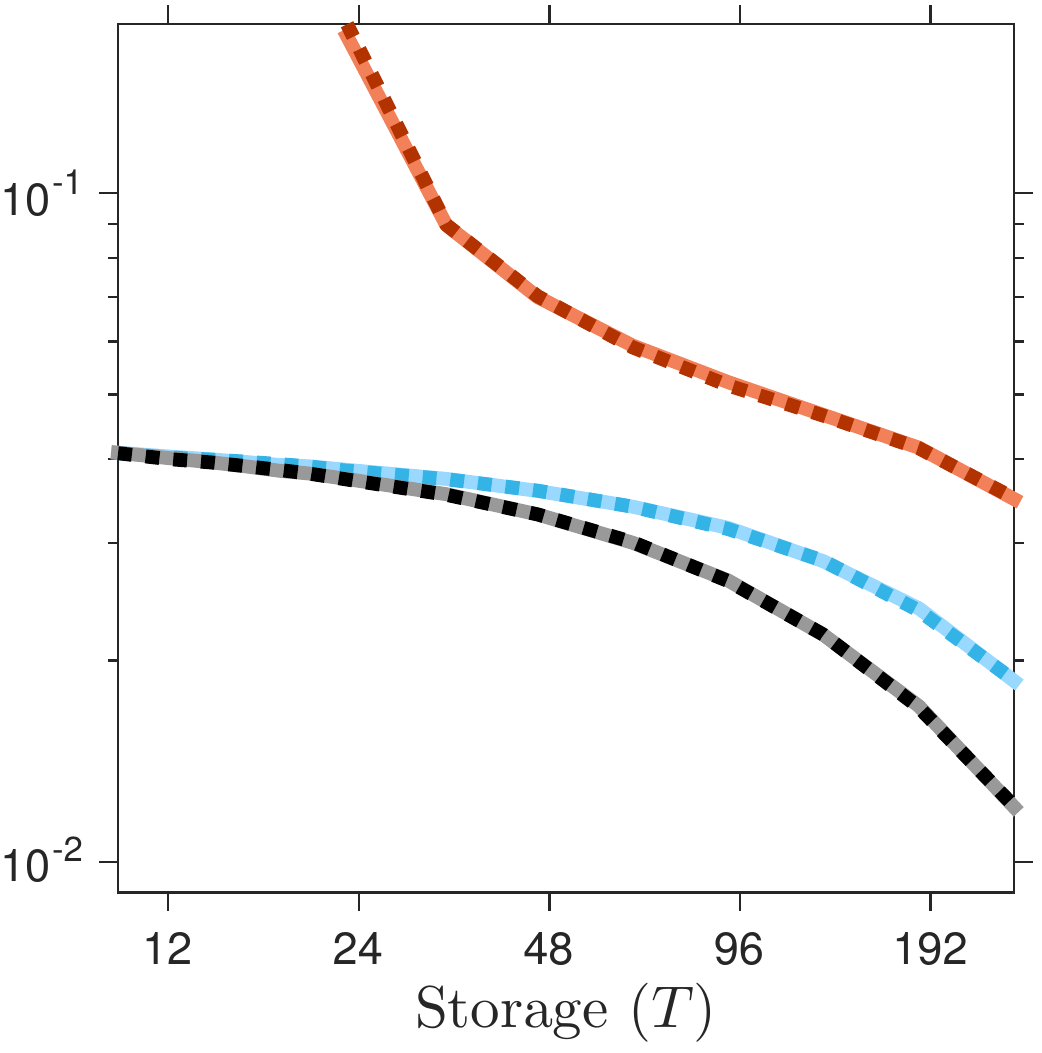}
\caption{\texttt{LowRankHiNoise}}
\label{fig:HI-theory}
\end{center}
\end{subfigure}
\end{center}

\vspace{.5em}

\begin{center}
\begin{subfigure}{.325\textwidth}
\begin{center}
\includegraphics[height=1.5in]{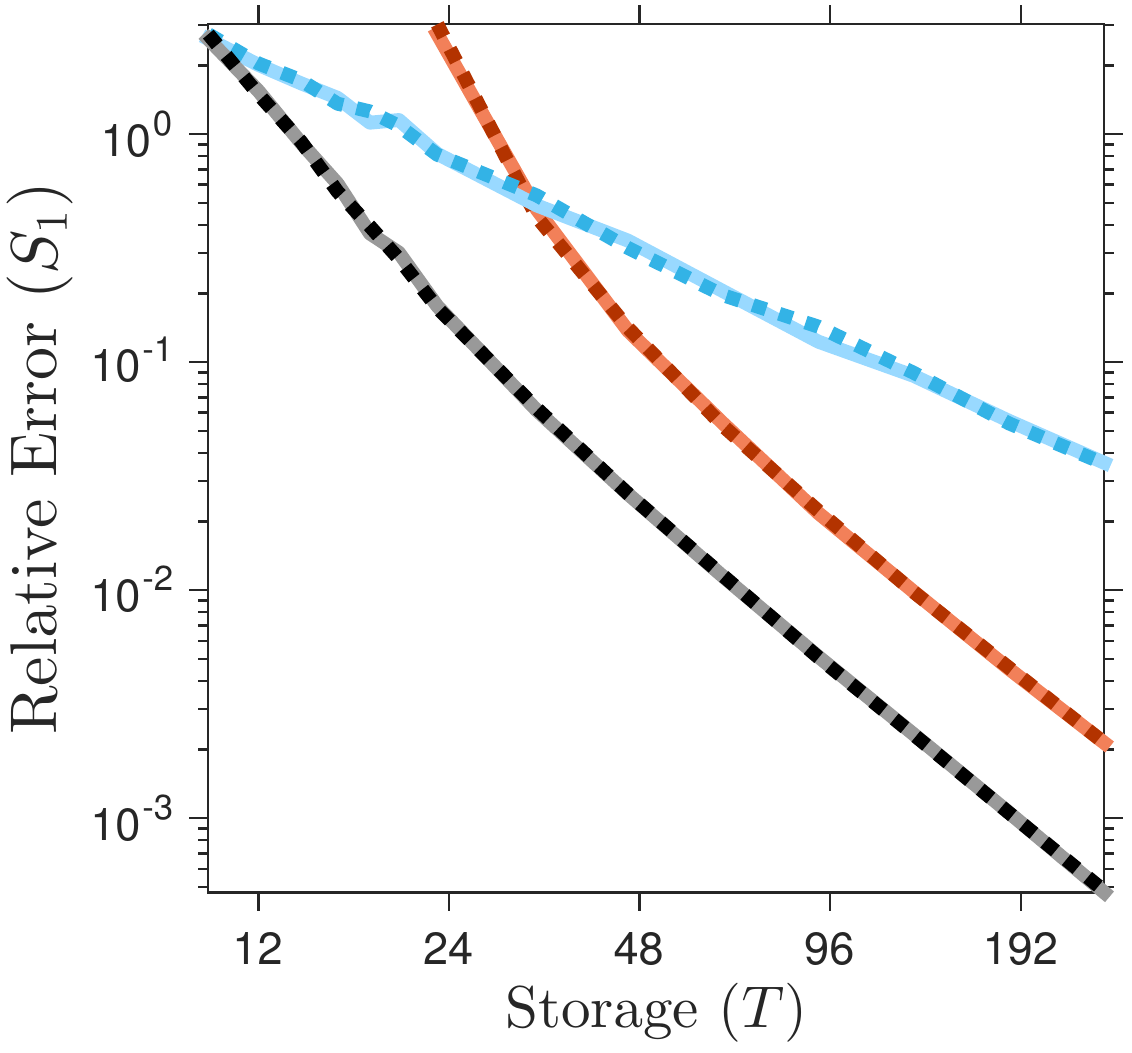}
\caption{\texttt{PolyDecayFast}}
\end{center}
\end{subfigure}
\begin{subfigure}{.325\textwidth}
\begin{center}
\includegraphics[height=1.5in]{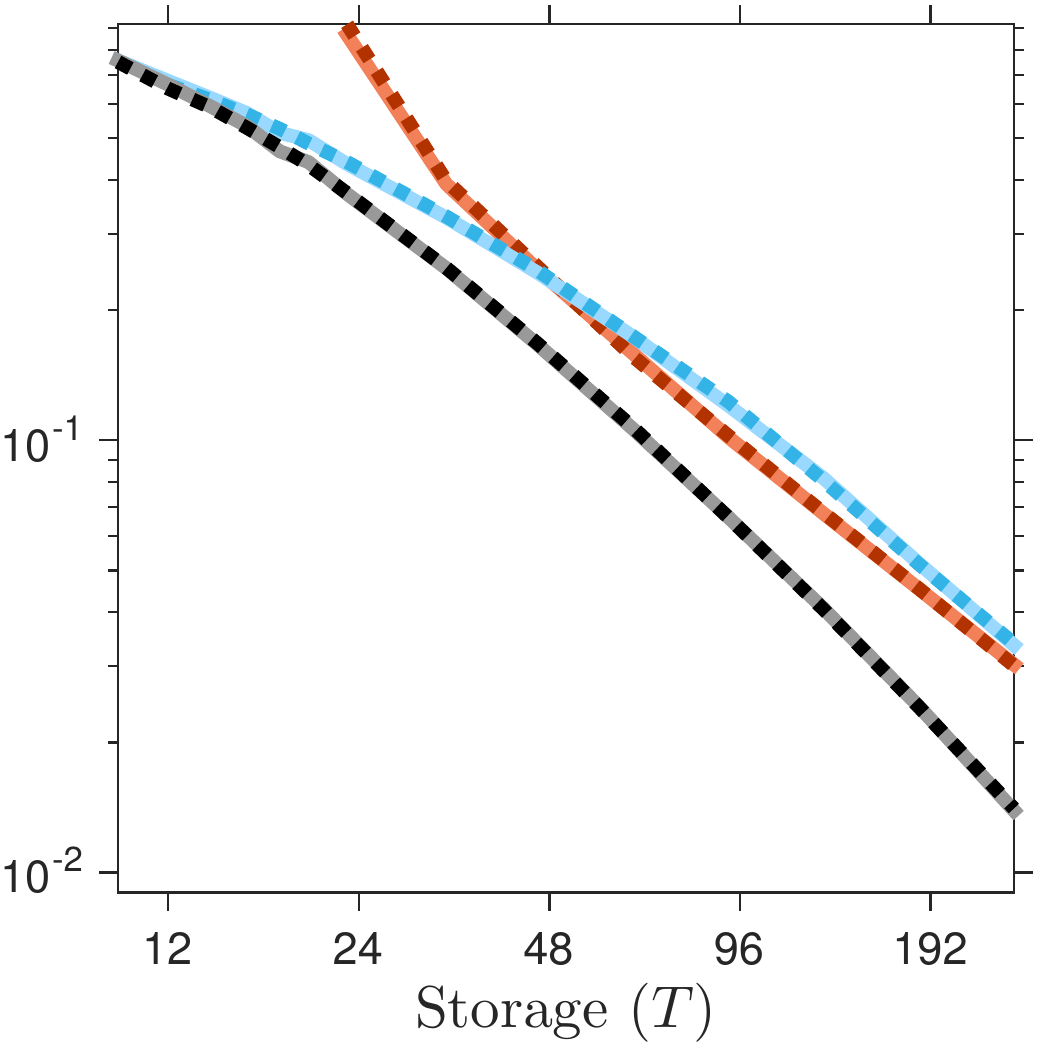}
\caption{\texttt{PolyDecayMed}}
\end{center}
\end{subfigure}
\begin{subfigure}{.325\textwidth}
\begin{center}
\includegraphics[height=1.5in]{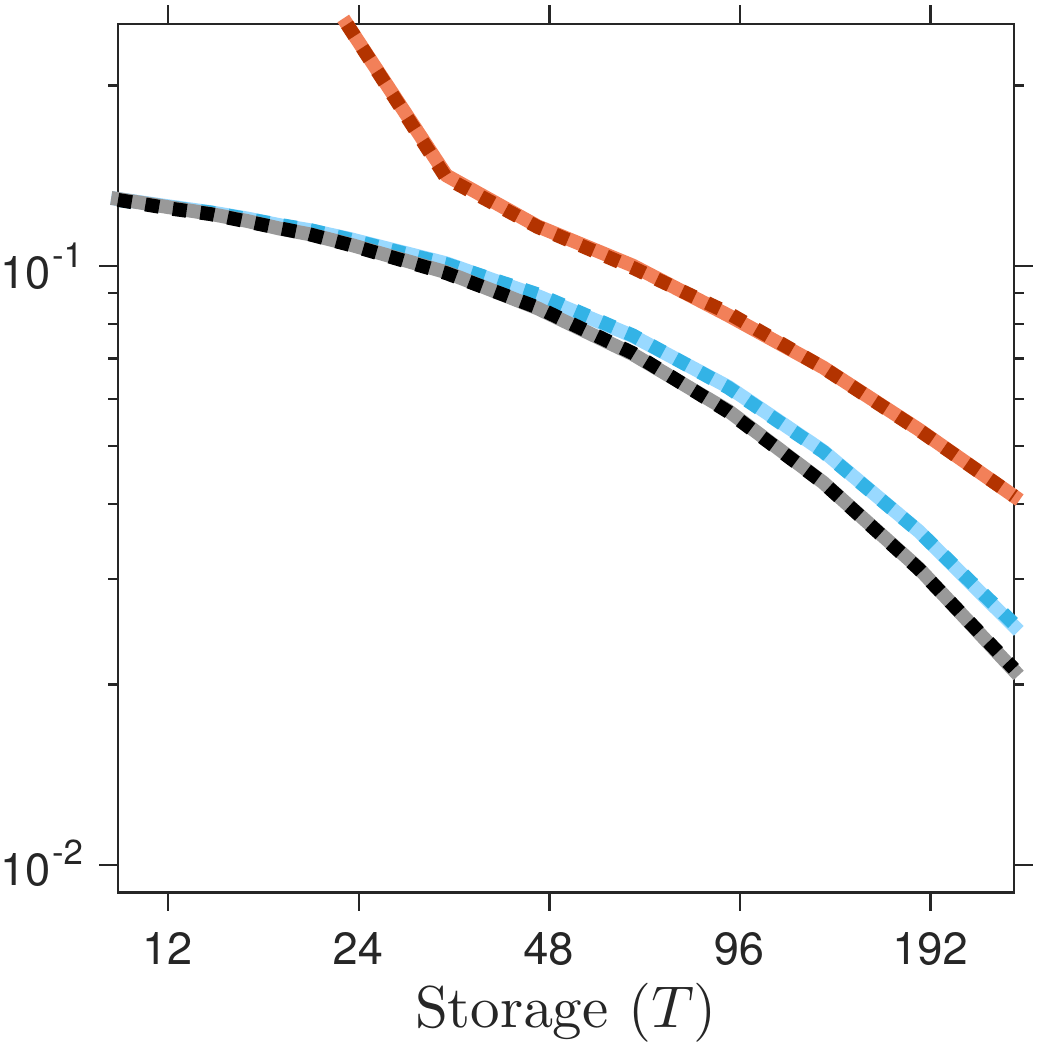}
\caption{\texttt{PolyDecaySlow}}
\end{center}
\end{subfigure}
\end{center}

\vspace{0.5em}

\begin{center}
\begin{subfigure}{.325\textwidth}
\begin{center}
\includegraphics[height=1.5in]{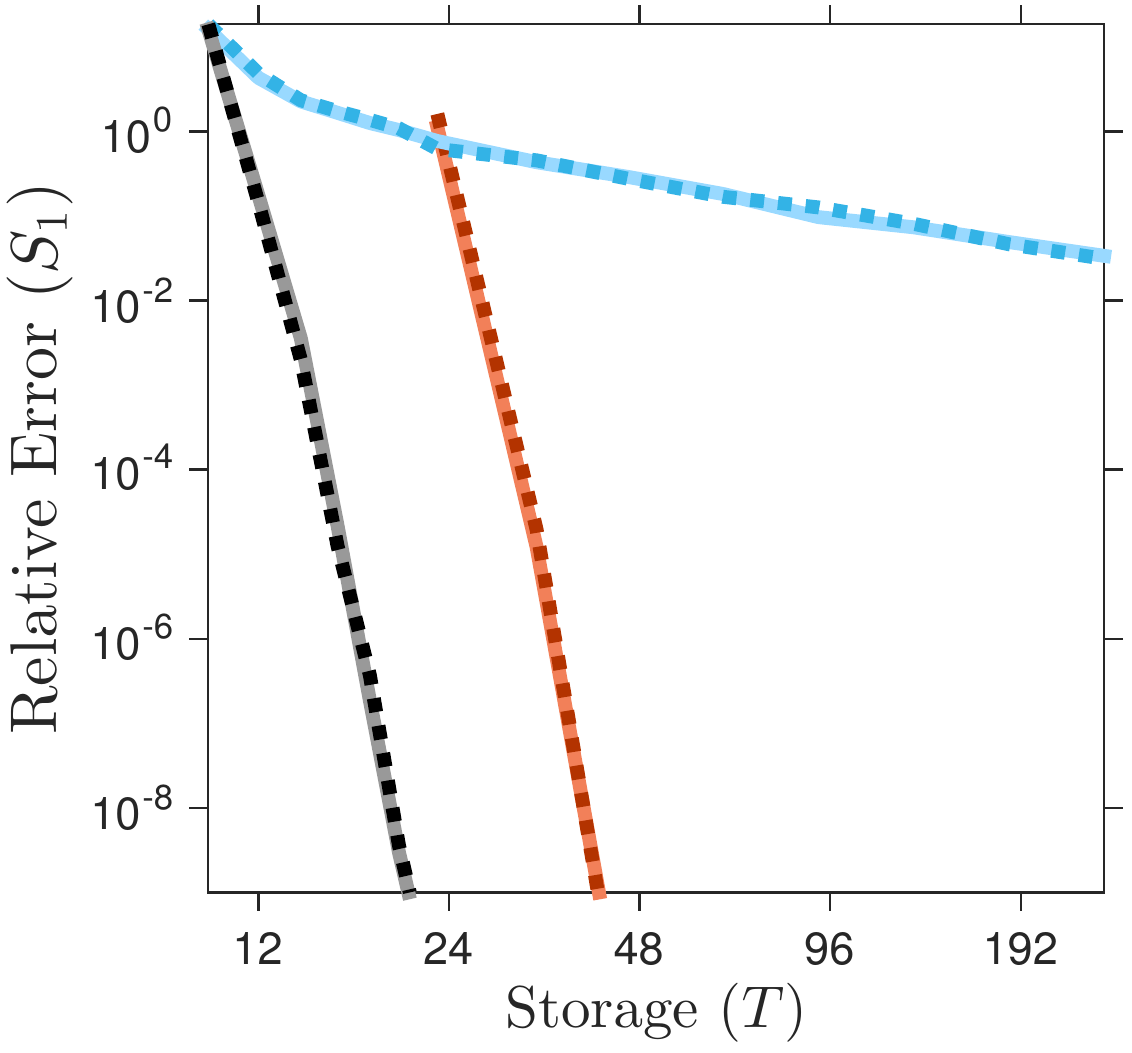}
\caption{\texttt{ExpDecayFast}}
\end{center}
\end{subfigure}
\begin{subfigure}{.325\textwidth}
\begin{center}
\includegraphics[height=1.5in]{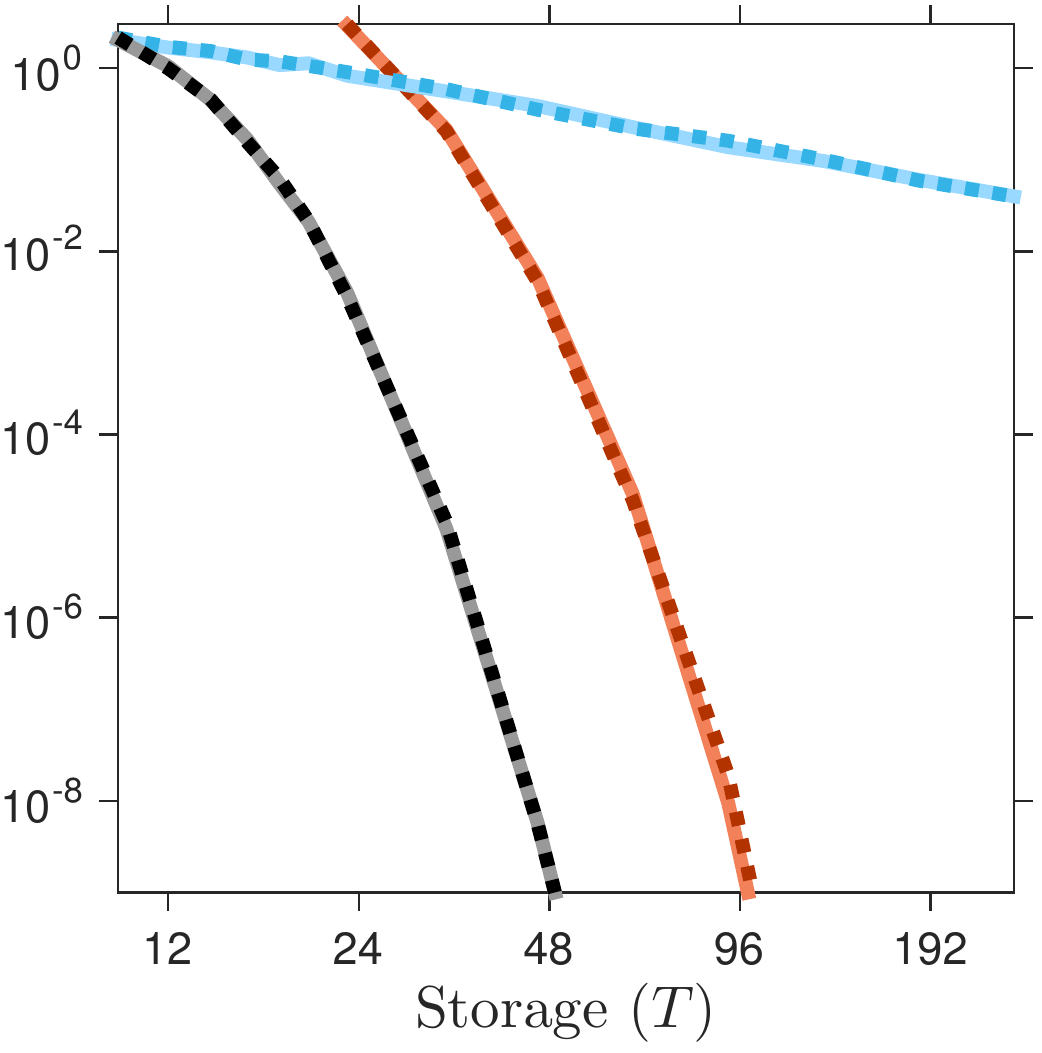}
\caption{\texttt{ExpDecayMed}}
\end{center}
\end{subfigure}
\begin{subfigure}{.325\textwidth}
\begin{center}
\includegraphics[height=1.5in]{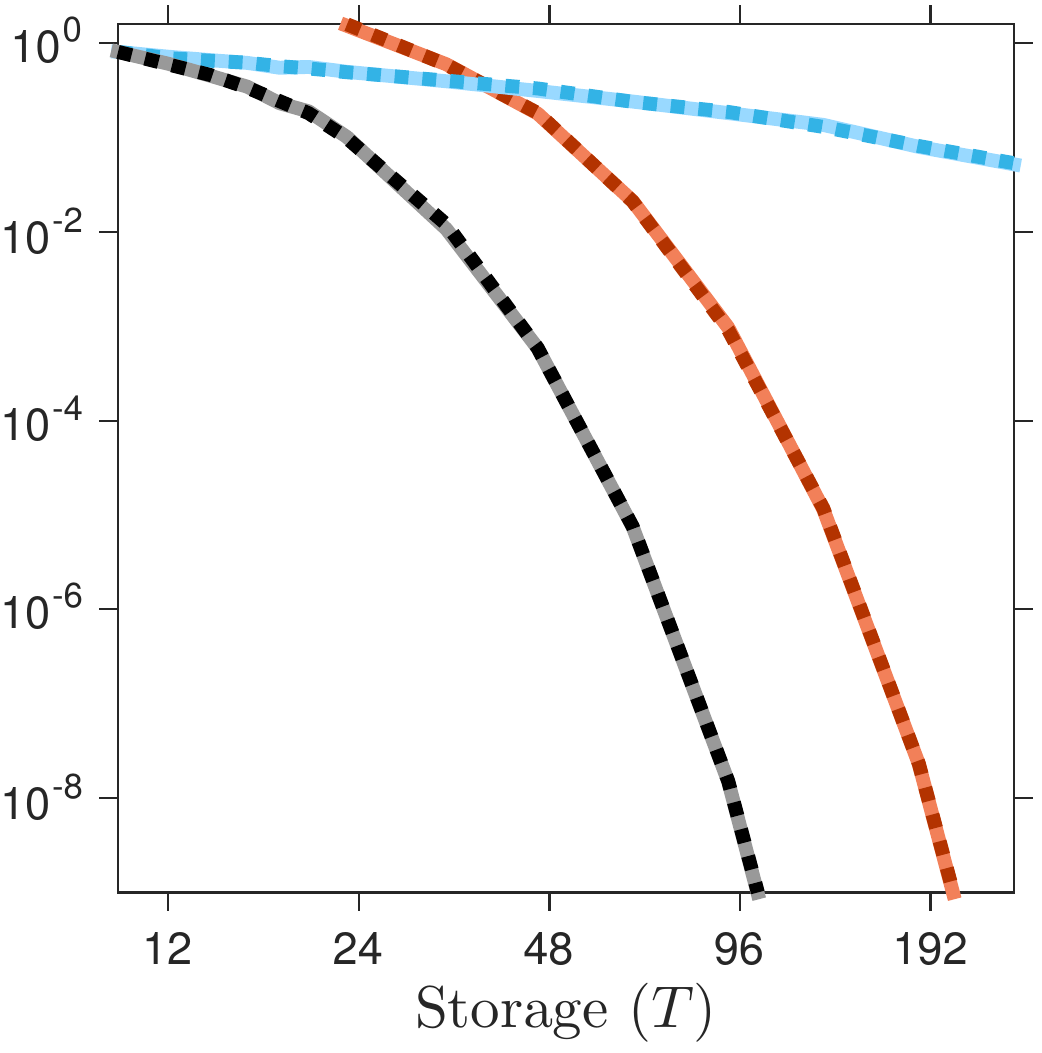}
\caption{\texttt{ExpDecaySlow}}
\end{center}
\end{subfigure}
\end{center}

\caption{\textbf{Synthetic Examples with Effective Rank $R = 10$, Approximation Rank $r = 10$, Schatten $1$-Norm Error.}
The data series show the performance of three algorithms for rank-$r$ psd approximation with $r = 10$.
\textbf{Solid lines} are generated from the Gaussian sketch;
\textbf{dashed lines} are from the SSFT sketch.
Each panel displays the Schatten 1-norm relative error~\eqref{eqn:relative-error}
as a function of storage cost $T$.}\label{fig:synthetic-S1-R10}
\end{figure}

\textbf{Conclusions.}
This paper makes the case for using the
proposed fixed-rank psd approximation~\eqref{eqn:Ahat-fixed}
in lieu of the alternatives~\eqref{eqn:gm} or~\cite[Alg.~9]{TYUC17:Randomized-Single-View-TR}.
Theorem~\ref{thm:error-fixed} shows that the proposed fixed-rank psd
approximation~\eqref{eqn:Ahat-fixed} can attain any prescribed
relative error, and Theorem~\ref{thm:error-fixed-2} shows that it can exploit spectral decay.
Furthermore, our numerical work demonstrates that the proposed approximation
improves (almost) uniformly over the competitors for a range of examples.
These results are timely because of the recent arrival of compelling applications,
such as~\cite{YUTC17:Sketchy-Decisions}, for sketching psd matrices.

{\ifdefined\cameraready

\textbf{Acknowledgments.}
The authors wish to thank Mark Tygert and Alex Gittens for helpful feedback
on preliminary versions of this work.  JAT gratefully acknowledges partial
support from ONR Award N00014-17-1-2146 and the Gordon \& Betty Moore Foundation.
VC and AY were supported in part by the European Commission under Grant ERC Future Proof, SNF 200021-146750, and SNF CRSII2-147633.  MU was supported in part by DARPA Award FA8750-17-2-0101.

\fi}

{\ifdefined\supplement

\clearpage
\appendix

\section{Details of the Theoretical Analysis}
\label{sec:proofs}

This appendix contains a new theoretical analysis of the
simple Nystr{\"o}m approximation~\eqref{eqn:nystrom} and
the proposed fixed-rank Nystr{\"o}m approximation~\eqref{eqn:Ahat-fixed}.

\subsection{Best Approximation in Schatten Norms}

Let us introduce compact notation for the optimal rank-$r$ approximation error
in the Schatten $p$-norm:
\begin{equation} \label{eqn:tail-sum}
\sigma_{r+1}^{(p)}(\mtx{M}) = \norm{ \mtx{M} - \lowrank{\mtx{M}}{r} }_{p}
	= \left[ \sum\nolimits_{i > r} \sigma_i(\mtx{M})^p \right]^{1/p}.
\end{equation}
Ordinary singular values correspond to the case $p = \infty$.

\subsection{Analysis of the Nystr{\"o}m Approximation}

The first result gives a very accurate error bound for
the basic Nystr{\"o}m approximation $\hat{\mtx{A}}^{\mathrm{nys}}$
with respect to the Schatten 1-norm.  This estimate
is the key ingredient in the proof of Theorem~\ref{thm:error-fixed-2}.

\begin{theorem}[Error in Nystr{\"o}m Approximation] \label{thm:error-lr} \rm
Assume $1 \leq r \leq k \leq n$.   Let $\mtx{A} \in \F^{n \times n}$ be a psd matrix.
Draw the test matrix $\mtx{\Omega} \in \F^{n \times k}$ from the Gaussian or orthonormal distribution,
and form the sketch $\mtx{Y} = \mtx{A\Omega}$.  Then the rank-$k$ Nystr{\"o}m
approximation $\hat{\mtx{A}}^{\mathrm{nys}}$ determined by~\eqref{eqn:nystrom}
satisfies the error bound
\begin{equation} \label{eqn:error-lr}
\Expect \norm{ \mtx{A} - \hat{\mtx{A}}^{\mathrm{nys}} }_{1}
	\leq \min_{\varrho < k - \alpha} \left[ \left(1 + \frac{\varrho}{k - \varrho - \alpha}\right)
	\sigma^{(1)}_{\varrho + 1}(\mtx{A}) \right].
\end{equation}
The index $\varrho$ ranges over natural numbers.
The quantity $\alpha(\R) = 1$ and $\alpha(\C) = 0$.
The optimal rank-$\varrho$ Schatten 1-norm approximation error is defined in~\eqref{eqn:tail-sum}.
\end{theorem}

\noindent
To the best of our knowledge, Theorem~\ref{thm:error-lr} is new.
The proof appears below in App.~\ref{sec:proof-lr}.

Let us situate Theorem~\ref{thm:error-lr} with respect to
the results in Gittens's work~\cite{Git13:Topics-Randomized,GM16:Revisiting-Nystrom-JMLR}.
Gittens develops error bounds for the Nystr{\"o}m approximation~\eqref{eqn:nystrom}
that hold with high probability, rather than in expectation.
He measures errors in the Schatten $p$-norm for $p = 1, 2, \infty$.
He also obtains results for several types of test matrices,
including isotropic models and a relative of the SSFT.
In contrast to Theorem~\ref{thm:error-lr}, Gittens's
bounds are more complicated, and the constants are much larger.

\subsection{Proof of Theorem~\ref{thm:error-lr}}
\label{sec:proof-lr}

We begin with the proof of Theorem~\ref{thm:error-lr}.
Gittens~\cite{Git11:Spectral-Norm,Git13:Topics-Randomized,GM16:Revisiting-Nystrom-JMLR}
uses a related argument to obtain bounds on the \emph{probability}
that the Nystr{\"o}m approximation achieves a given error.

The first step is to write the Nystr{\"o}m approximation
in terms of an orthogonal projector.  This expression allows us
to exploit the analysis from~\cite{HMT11:Finding-Structure,TYUC17:Randomized-Single-View-TR}.

\begin{proposition}[Representation of Nystr{\"o}m Approximation] \label{prop:nys-proj}
Let $\mtx{P}$ be the orthogonal projector onto $\range(\mtx{A}^{1/2} \mtx{\Omega})$:
\begin{equation} \label{eqn:proj-form}
\mtx{P} = (\mtx{A}^{1/2}\mtx{\Omega}) (\mtx{\Omega}^* \mtx{A} \mtx{\Omega})^\dagger (\mtx{A}^{1/2}\mtx{\Omega})^*.
\end{equation}
Then the Nystr{\"o}m approximation~\eqref{eqn:nystrom} can be expressed as
\begin{equation} \label{eqn:nys-proj}
\hat{\mtx{A}}^{\mathrm{nys}} = \mtx{A}^{1/2} \mtx{P} \mtx{A}^{1/2}
\end{equation}
In particular, the Nystr{\"o}m approximation only depends on $\mtx{\Omega}$
through $\range(\mtx{\Omega})$.
\end{proposition}

We believe that Proposition~\ref{prop:nys-proj} first appeared explicitly
in the work of Gittens~\cite{Git11:Spectral-Norm}.

\begin{proof}
This argument follows from a direct calculation:
\begin{equation*} \begin{aligned}
\hat{\mtx{A}}^{\mathrm{nys}} &= \mtx{A} \mtx{\Omega} (\mtx{\Omega}^* \mtx{A} \mtx{\Omega})^{\dagger} \mtx{\Omega}^* \mtx{A} \\
	&= \mtx{A}^{1/2} (\mtx{A}^{1/2} \mtx{\Omega})
	\big[ (\mtx{A}^{1/2} \mtx{\Omega})^*  (\mtx{A}^{1/2} \mtx{\Omega}) \big]^\dagger
	(\mtx{A}^{1/2} \mtx{\Omega})^* \mtx{A}^{1/2} \\
	&= \mtx{A}^{1/2} \mtx{P} \mtx{A}^{1/2}.
\end{aligned}
\end{equation*}
To reach the last line, we identified the orthogonal projector~\eqref{eqn:proj-form}.
\end{proof}

With Proposition~\ref{prop:nys-proj} at hand,
the proof of Theorem~\ref{thm:error-lr} is straightforward.

We may assume that $\mtx{\Omega}$ is a Gaussian matrix because
the reconstruction $\hat{\mtx{A}}$ only depends on $\range(\mtx{\Omega})$.
The range of a random orthonormal matrix has the same distribution
as a Gaussian matrix up to a set of measure zero.

Let $\mtx{P}$ be the orthogonal projector~\eqref{eqn:proj-form}.
In view of the formula~\eqref{eqn:nys-proj} for $\hat{\mtx{A}}^{\mathrm{nys}}$, we have
\begin{equation} \label{eqn:nys-error}
\mtx{A} - \hat{\mtx{A}}^{\mathrm{nys}} = \mtx{A}^{1/2}(\Id - \mtx{P}) \mtx{A}^{1/2}.
\end{equation}
We can now express the Schatten 1-norm of the error in terms of the Schatten 2-norm:
\begin{equation*}
\norm{ \mtx{A} - \hat{\mtx{A}}^{\mathrm{nys}} }_{1}
	= \norm{ \mtx{A}^{1/2} (\Id - \mtx{P}) (\Id - \mtx{P}) \mtx{A}^{1/2} }_1
	= \norm{ (\Id - \mtx{P}) \mtx{A}^{1/2} }_{2}^2.
\end{equation*}
The first identity follows from~\eqref{eqn:nys-error}
and the fact that the orthogonal projector $\Id - \mtx{P}$ is idempotent.

Fix a natural number $\varrho < k - \alpha$.  We can use established
results from the literature to control the expectation of the error.
In particular, we invoke a slight generalization~\cite[Fact.~8.3]{TYUC17:Randomized-Single-View-TR}
of a result~\cite[Thm.~10.5]{HMT11:Finding-Structure} of Halko et al.
We arrive at the bound
\begin{equation*}
\begin{aligned}
\Expect \norm{ (\Id - \mtx{P}) \mtx{A}^{1/2} }_{2}^2
	&\leq \left( 1 + \frac{ \varrho }{k - \varrho - \alpha} \right) \sum_{i > \varrho} \sigma_i(\mtx{A}^{1/2})^2 \\
	&= \left( 1 + \frac{ \varrho }{k - \varrho - \alpha} \right) \sum_{i > \varrho} \sigma_i(\mtx{A})
	= \left( 1 + \frac{ \varrho }{k - \varrho - \alpha} \right) \sigma^{(1)}_{\varrho+1}(\mtx{A}).
\end{aligned}
\end{equation*}
Combine the last two displays and minimize over eligible $\varrho$ to complete the argument.

\begin{remark}[Spectral-Norm Error]
When $\F = \R$, we can also obtain a spectral-norm error bound by
combining this argument with another result~\cite[Thm.~10.6]{HMT11:Finding-Structure}
of Halko et al.:
$$
\Expect \sqrt{\norm{ \mtx{A} - \hat{\mtx{A}}^{\mathrm{nys}} } } 	\leq \min_{\varrho < k - 1} \left[ \left( 1 + \sqrt{\frac{\varrho}{k - \varrho - 1}} \right) \sqrt{\sigma_{\varrho+1}(\mtx{A})}
	+ \frac{\econst \sqrt{k}}{k - \varrho} \sqrt{ \sigma^{(1)}_{\varrho + 1}(\mtx{A}) } \right].
$$
It takes a surprising amount of additional work to obtain an accurate bound
for the first moment of the error (instead of the 1/2 moment).  We have chosen
not to include this argument.
\end{remark}

\begin{remark}[High-Probability Bounds]
As noted by Gittens~\cite{Git13:Topics-Randomized,GM16:Revisiting-Nystrom-JMLR},
we can obtain high-probability error bounds in the real setting
by combining the approach here with results~\cite[Thms.~10.7--10.8]{HMT11:Finding-Structure} from Halko et al.
We omit the details.
\end{remark}

\begin{remark}[Other Test Matrices]
As noted by Gittens~\cite{Git13:Topics-Randomized,GM16:Revisiting-Nystrom-JMLR},
we can obtain results for other types of test matrices by replacing parts
of the analysis that depend on Gaussian matrices.  These changes result in
bounds that are quantitatively and qualitatively worse.
The numerical evidence suggests that many types of test matrices
have the same empirical performance, so we omit this development.
\end{remark}

\subsection{Theorem~\ref{thm:error-fixed-2}: Schatten 1-Norm Bound}

Let us continue with the proof of the Schatten 1-norm bound~\eqref{eqn:error-fixed-2-S1}
from Theorem~\ref{thm:error-fixed-2}.
We require a basic result on rank-$r$ approximation
adapted from~\cite[Prop.~7.1]{TYUC17:Randomized-Single-View-TR}.

\begin{proposition}[Fixed-Rank Projection] \label{prop:fixed-rank}
Let $\mtx{A} \in \F^{n \times n}$ and $\hat{\mtx{A}} \in \F^{n \times n}$
be arbitrary matrices.  For each natural number $r$ and number $p \in [1, \infty]$,
$$
\norm{ \mtx{A} - \lowrank{\hat{\mtx{A}}}{r} }_p
	\leq \sigma_{r+1}^{(p)}(\mtx{A}) + 2 \norm{ \mtx{A} - \hat{\mtx{A}} }_p.
$$
\end{proposition}

\begin{proof}
The argument follows from a short calculation based on the triangle inequality:
$$
\begin{aligned}
\norm{ \mtx{A} - \lowrank{\hat{\mtx{A}}}{r} }_p
	&\leq \norm{ \mtx{A} - \hat{\mtx{A}} }_p + \norm{ \hat{\mtx{A}} - \lowrank{\hat{\mtx{A}}}{r} }_p \\
	&\leq \norm{ \mtx{A} - \hat{\mtx{A}} }_p + \norm{ \hat{\mtx{A}} - \lowrank{\mtx{A}}{r} }_p \\
	&\leq 2 \norm{ \mtx{A} - \hat{\mtx{A}} }_p + \norm{ \mtx{A} - \lowrank{\mtx{A}}{r} }_p.
\end{aligned}
$$
In the second line, we have used the fact that $\lowrank{\hat{\mtx{A}}}{r}$ is a \emph{best} rank-$r$
approximation of $\hat{\mtx{A}}$.  To complete the argument, we identify the last term~\eqref{eqn:tail-sum}
as the best rank-$r$ approximation error in the Schatten $p$-norm.
\end{proof}

The bound~\eqref{eqn:error-fixed-2-S1} from Theorem~\ref{thm:error-fixed-2}
is now an immediate consequence of Theorem~\ref{thm:error-lr} and Proposition~\ref{prop:fixed-rank}:
$$
\begin{aligned}
\Expect \norm{ \mtx{A} - \hat{\mtx{A}}_r }_1
	&\leq \sigma_{r+1}^{(1)}(\mtx{A}) + 2 \Expect \norm{ \mtx{A} - \hat{\mtx{A}}^{\mathrm{nys} } }_1 \\
	&\leq \sigma_{r+1}^{(1)}(\mtx{A}) + 2 \min_{\varrho < k - \alpha} \left(1 + \frac{\varrho}{k - \varrho - \alpha} \right) \sigma^{(1)}_{\varrho + 1}(\mtx{A}). \\
\end{aligned}
$$
We have used the definition~\eqref{eqn:Ahat-fixed} of our fixed-rank approximation:
$\hat{\mtx{A}}_r = \lowrank{ \hat{\mtx{A}}^{\mathrm{nys}} }{r}$.

\begin{remark}[Extensions]
Given a bound on the error in the Nystr{\"o}m approximation~\eqref{eqn:nystrom}
in the Schatten $p$-norm for any test matrix, this approach automatically yields
an estimate for the associated fixed-rank psd approximation~\eqref{eqn:Ahat-fixed}.
\end{remark}

\subsection{Theorem~\ref{thm:error-fixed}: Schatten 1-Norm Bound}
\label{sec:proof-fixed}

Next, we turn to the proof of the Schatten 1-norm bound~\eqref{eqn:error-fixed-S1}
from Theorem~\ref{thm:error-fixed}.
This argument is based on the same approach as Theorem~\ref{thm:error-lr},
but we require several additional ingredients
from~\cite{HMT11:Finding-Structure,Git13:Topics-Randomized,Gu15:Subspace-Iteration,TYUC17:Randomized-Single-View-TR}.

As before, we may assume that $\mtx{\Omega}$ is Gaussian.
With probability one, the nonzero eigenvalues of $\hat{\mtx{A}}^{\mathrm{nys}}$
are all distinct, so the best rank-$r$ approximation
$\hat{\mtx{A}}_r$ of $\hat{\mtx{A}}^{\mathrm{nys}}$ is determined uniquely.

Let $\mtx{P}$ be the orthogonal projector~\eqref{eqn:proj-form}. According to~\eqref{eqn:nys-proj}, the Nystr{\"o}m approximation takes the form
\begin{equation*}
\hat{\mtx{A}}^{\mathrm{nys}} = \mtx{A}^{1/2} \mtx{P} \mtx{A}^{1/2}
	= (\mtx{A}^{1/2} \mtx{P})(\mtx{P} \mtx{A}^{1/2}).
\end{equation*}
Let $\mtx{Q}$ denote the orthogonal projector onto the range of
$\lowrank{\mtx{P} \mtx{A}^{1/2}}{r}$.
Using the (truncated) SVD of the matrix $\mtx{PA}^{1/2}$, we can verify that
the best rank-$r$ approximation $\hat{\mtx{A}}_r$
of $\hat{\mtx{A}}^{\mathrm{nys}}$ satisfies
\begin{equation*}
\hat{\mtx{A}}_r = \lowrank{ \mtx{A}^{1/2} \mtx{P} }{r} \lowrank{ \mtx{P} \mtx{A}^{1/2} }{r}
	= \mtx{A}^{1/2} \mtx{P} \mtx{Q} \mtx{P} \mtx{A}^{1/2}
\end{equation*}
As in the proof of Theorem~\ref{thm:error-lr}, the Schatten 1-norm of the error satisfies
$$
\norm{ \mtx{A} - \hat{\mtx{A}}_r }_1
	= \norm{ \mtx{A} - \mtx{A}^{1/2} \mtx{PQP} \mtx{A}^{1/2} }_1
	= \norm{ (\Id - \mtx{QP}) \mtx{A}^{1/2} }_2^2.
$$
Since $\range(\mtx{Q}) \subset \range(\mtx{P})$, we can rewrite this expression as
$$
\norm{ (\Id - \mtx{QP})\mtx{A}^{1/2} }_2^2
	= \norm{ (\Id - \mtx{PQP})\mtx{A}^{1/2} }_2^2
	= \norm{ \mtx{A}^{1/2} - \mtx{P} \lowrank{\mtx{PA}^{1/2}}{r} }_2^2.
$$
The last identity holds because $\mtx{QPA}^{1/2} = \lowrank{\mtx{PA}^{1/2}}{r}$.
A direct application of Gu's result~\cite[Thm.~3.5]{Gu15:Subspace-Iteration}
yields
$$
\norm{ \mtx{A}^{1/2} - \mtx{P} \lowrank{\mtx{PA}^{1/2}}{r} }_2^2
	\leq \norm{ (\Id - \mtx{P}) \lowrank{\mtx{A}^{1/2}}{r} }_2^2 + \sum_{i > r} \sigma_i(\mtx{A}^{1/2})^2.
$$
A direct application of the result~\cite[Prop.~9.2]{TYUC17:Randomized-Single-View-TR}
shows that
$$
\Expect \norm{ (\Id - \mtx{P}) \lowrank{\mtx{A}^{1/2}}{r} }_2^2
	= \frac{r}{k - r - \alpha} \sum_{i > r} \sigma_i(\mtx{A}^{1/2})^2.
$$
As before, we note that
$$
\sum_{i > r} \sigma_i(\mtx{A}^{1/2})^2 = \sigma^{(1)}_{r+1}(\mtx{A}).
$$
Taking an expectation and sequencing these displays, we arrive at
$$
\Expect \norm{ \mtx{A} - \hat{\mtx{A}}_r }_1
	\leq \left( 1 + \frac{r}{k-r-\alpha} \right) \sigma^{(1)}_{r+1}(\mtx{A}).
$$
This is the stated result~\eqref{eqn:error-fixed-S1}.

\subsection{Theorems~\ref{thm:error-fixed} and~\ref{thm:error-fixed-2}: Schatten $\infty$-Norm Bounds}

Last, we develop the bounds~\eqref{eqn:error-fixed-Sinf} and~\eqref{eqn:error-fixed-2-Sinf}
on the Schatten $\infty$-norm of the fixed-rank psd approximation~\eqref{eqn:Ahat-fixed}
using a formal argument.  We require the following result.

\begin{proposition}[Reversed Eckart--Young] \label{prop:rev-ey}
Let $\mtx{A}, \mtx{B} \in \F^{n \times n}$ be matrices, and assume that $\rank(\mtx{B}) \leq r$.
Then
$$
\norm{ \mtx{A} - \mtx{B} }_{\infty}
	\leq \sigma_{r+1}(\mtx{A}) + \left[ \norm{\mtx{A} - \mtx{B}}_1 - \sigma_{r+1}^{(1)}(\mtx{A}) \right].
$$
\end{proposition}

The proof of Proposition~\ref{prop:rev-ey} follows from a minor
change to~\cite[Thm.~3.4]{Gu15:Subspace-Iteration}.

\begin{proof}
As a consequence of Weyl's inequalities~\cite[Thm.~III.2.1]{Bha97:Matrix-Analysis},
we have the bound
\begin{equation} \label{eqn:my-weyl}
\sigma_{i + r}(\mtx{A}) \leq \sigma_i(\mtx{A} - \mtx{B}) + \sigma_{r+1}(\mtx{B}) = \sigma_i(\mtx{A} - \mtx{B}).
\end{equation}
The last identity holds because $\rank(\mtx{B}) \leq r$.  It follows that
$$
\begin{aligned}
\norm{\mtx{A} - \mtx{B}}_1 &= \sum_{i \geq 1} \sigma_i(\mtx{A} - \mtx{B}) \\
	&= \sigma_1(\mtx{A} - \mtx{B}) + \sum_{i \geq 2} \sigma_i(\mtx{A} - \mtx{B}) \\
	&\geq \norm{\mtx{A} - \mtx{B}}_{\infty} + \sum_{i \geq 2} \sigma_{r+i}(\mtx{A}) \\
	&= \norm{\mtx{A} - \mtx{B}}_{\infty} - \sigma_{r+1}(\mtx{A}) + \sigma^{(1)}_{r+1}(\mtx{A}).
\end{aligned}
$$
The first expression is the representation of the Schatten 1-norm
in terms of singular values.  The inequality is~\eqref{eqn:my-weyl}.
Finally, we identify the best Schatten 1-norm error from~\eqref{eqn:tail-sum}.
\end{proof}

To obtain the Schatten $\infty$-norm bound~\eqref{eqn:error-fixed-Sinf},
we combine Proposition~\ref{prop:rev-ey} with the Schatten 1-norm bound~\eqref{eqn:error-fixed-S1}:
\begin{equation*}
\begin{aligned}
\Expect \norm{ \mtx{A} - \hat{\mtx{A}}_r }_{\infty}
	&\leq \sigma_{r+1}(\mtx{A}) + \left[ \Expect \norm{ \mtx{A} -\hat{\mtx{A}}_r }_1 - \sigma^{(1)}_{r+1}(\mtx{A}) \right] \\
	&\leq \sigma_{r+1}(\mtx{A}) + \frac{r}{k - r - \alpha} \cdot \sigma^{(1)}_{r+1}(\mtx{A}).
\end{aligned}
\end{equation*}
Similarly, to obtain the Schatten $\infty$-norm bound~\eqref{eqn:error-fixed-2-Sinf},
we combine Proposition~\ref{prop:rev-ey} with the Schatten 1-norm bound~\eqref{eqn:error-fixed-2-S1}.

\section{Supplemental Numerics}
\label{sec:numerics-extra}

This appendix documents additional numerical work.
These experiments provide a more complete picture
of the performance of the psd approximation methods.

\begin{itemize}
\item	Figure~\ref{fig:spectra} contains a plot of the singular-value spectrum
of each input matrix described in Sec.~\ref{sec:numerics}.

\item	Figures~\ref{fig:data-S2}--\ref{fig:synthetic-Sinf-R10}
document the results of numerical experiments for the remaining parameter
regimes outlined in Sec.~\ref{sec:numerics}.  In particular, we consider
all Schatten $p$-norm relative error measures for $p \in \{1, 2, \infty\}$
and all effective rank parameters $R \in \{5, 10, 20\}$ for the synthetic data.
We omit the case $p = \infty, R = 20$ because the plots are uninformative.

\item	Figure~\ref{fig:bad-numerics} gives evidence about the numerical challenges
involved in implementing Nystr{\"o}m approximations, such as~\eqref{eqn:Ahat-fixed}.
Our implementation in Algorithm~\ref{alg:low-rank-recon} is based on the
Nystr{\"o}m approximation routine \texttt{eigenn}
released by Tygert~\cite{Tyg14:Matlab-Routines} to accompany
the paper~\cite{LLS+17:Algorithm-971}.  We compare with another implementation strategy
described in the text of the same paper~\cite[Eqn.~(13)]{LLS+17:Algorithm-971}.
It is surprising to discover very different levels of precision in two implementations
designed by professional numerical analysts.

\end{itemize}

\begin{figure}[htp!]
\begin{center}
\begin{subfigure}{.45\textwidth}
\begin{center}
\includegraphics[height=2in]{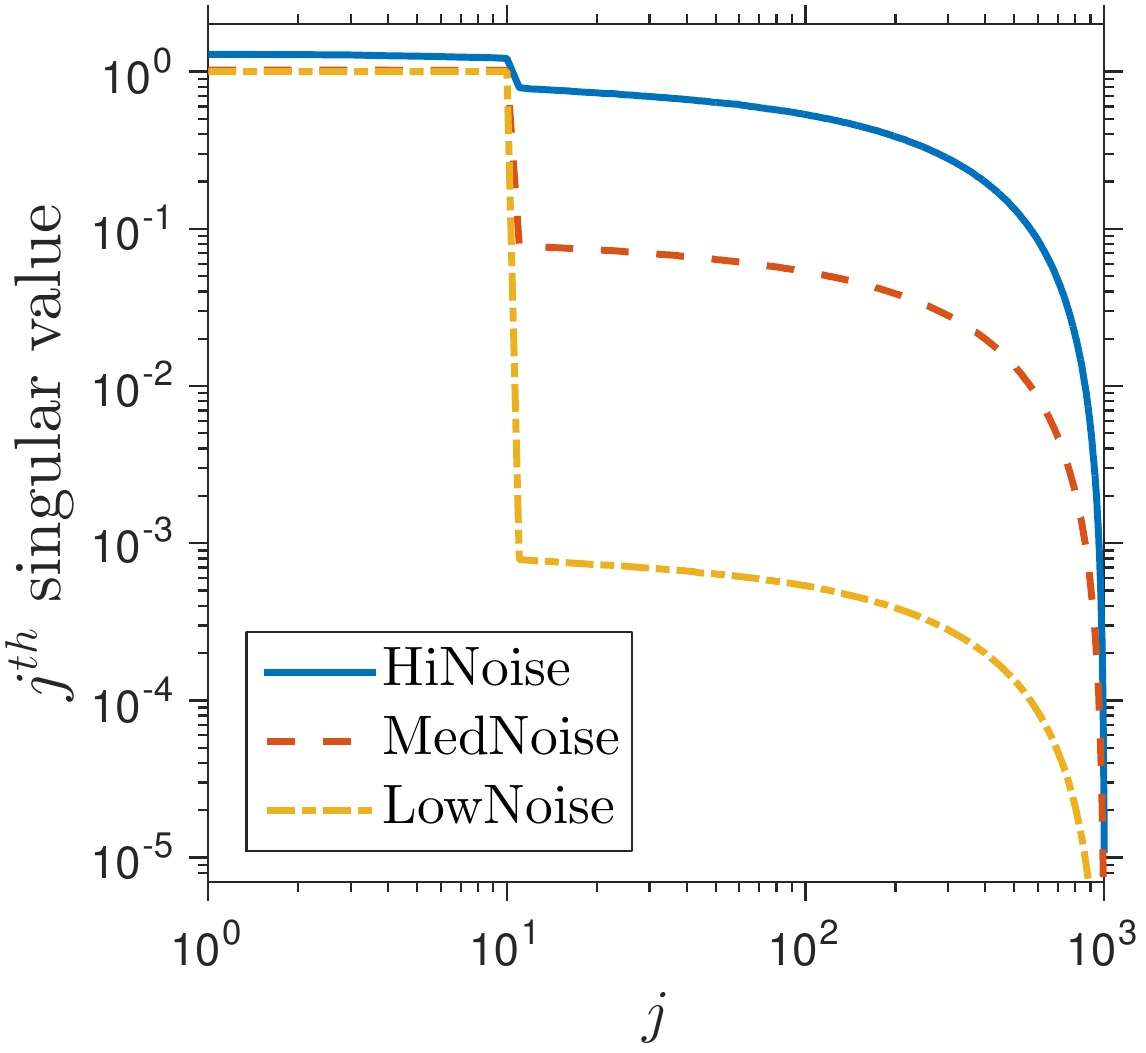}
\caption{Low-Rank + PSD Noise}
\end{center}
\end{subfigure}
\begin{subfigure}{.45\textwidth}
\begin{center}
\includegraphics[height=2in]{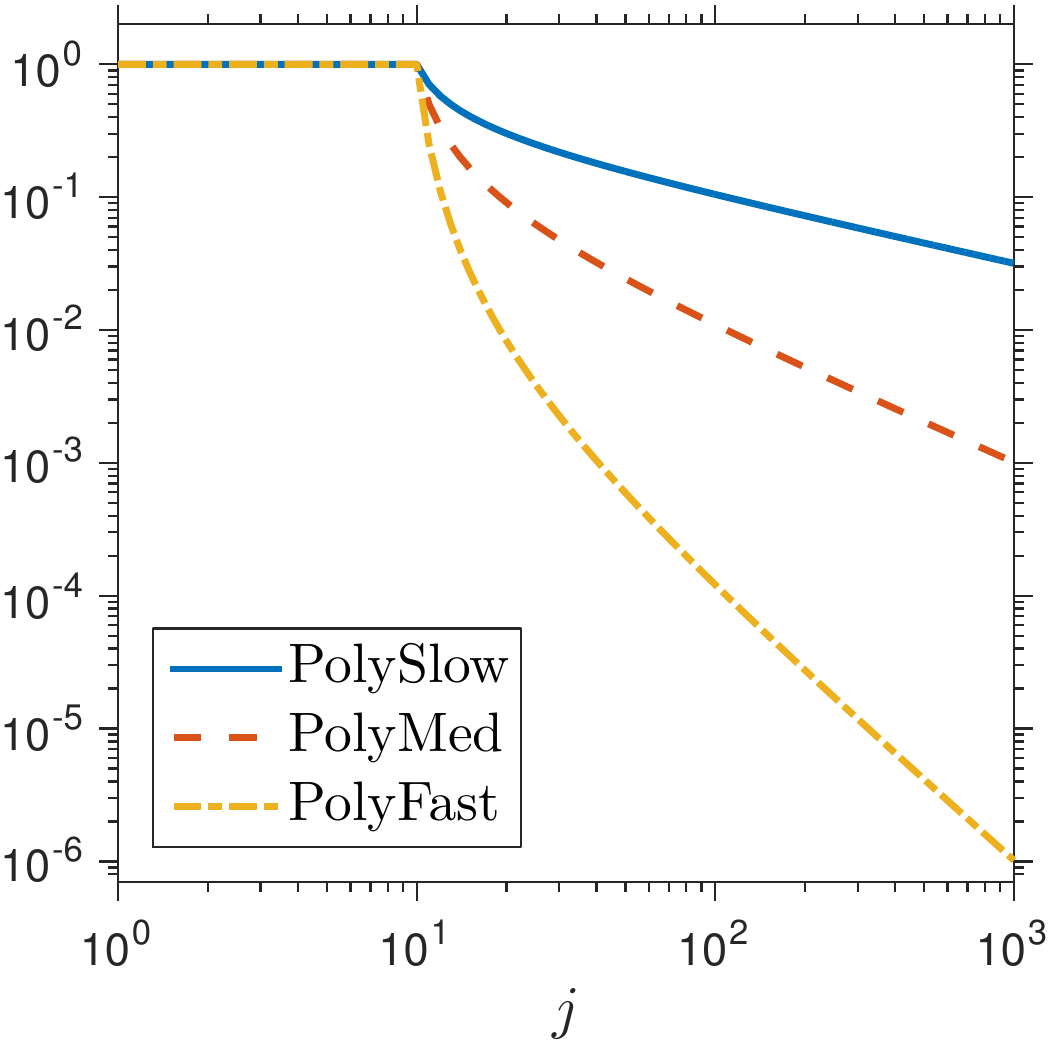}
\caption{Polynomial Decay}
\end{center}
\end{subfigure}

\vspace{2pc}

\begin{subfigure}{.45\textwidth}
\begin{center}
\includegraphics[height=2in]{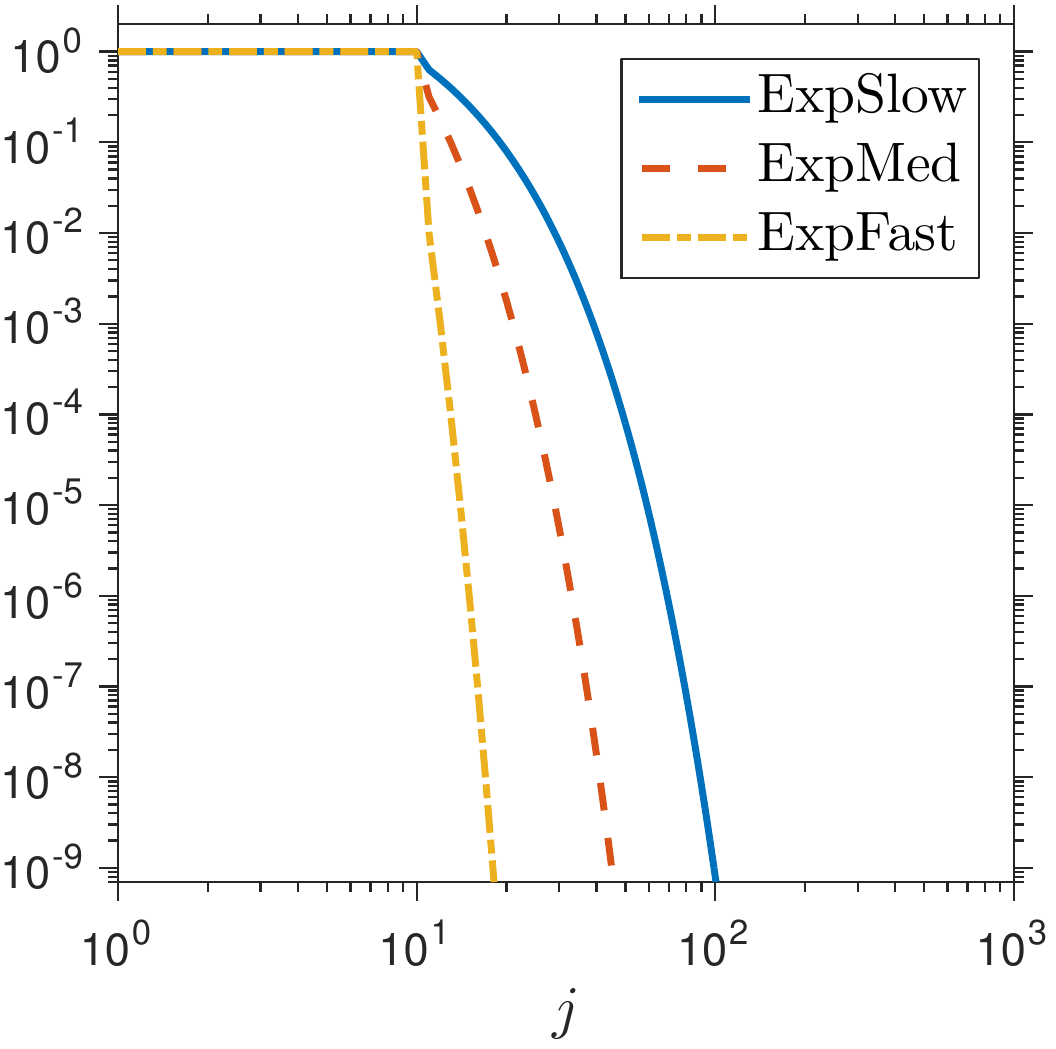}
\caption{Exponential Decay}
\end{center}
\end{subfigure}
\begin{subfigure}{.45\textwidth}
\begin{center}
\includegraphics[height=2in]{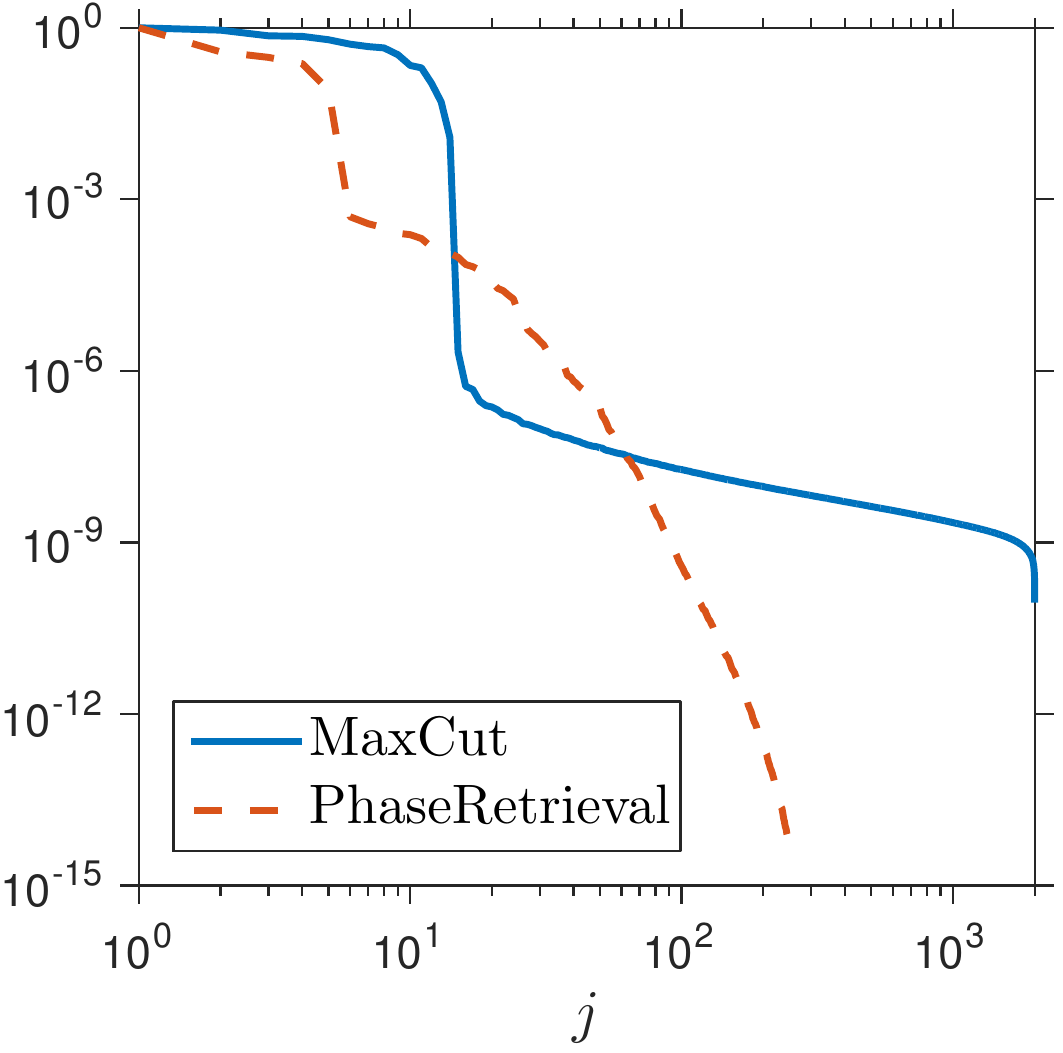}
\caption{\texttt{MaxCut} and \texttt{PhaseRetrieval}}
\end{center}
\end{subfigure}
\end{center}

\caption{\textbf{Singular Values of Input Matrices.}
These plots display the singular value spectra of the
input matrices that appear in the experiments.
See Sec.~\ref{sec:numerics} for descriptions of the matrices.} \label{fig:spectra}
\end{figure}

\begin{figure}[htp!]
\begin{center}
\begin{subfigure}{.45\textwidth}
\begin{center}
\includegraphics[height=2in]{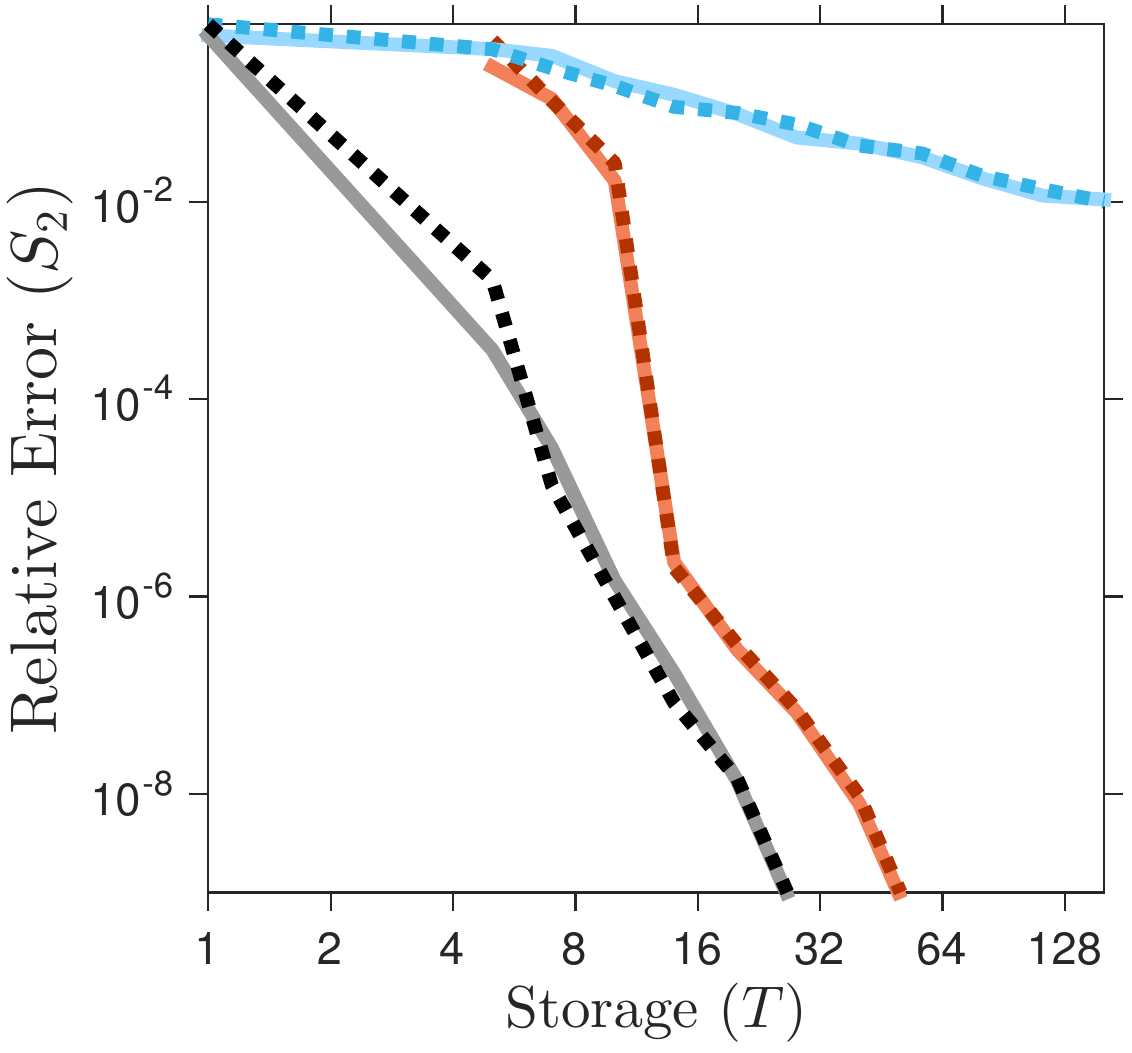}
\caption{\texttt{PhaseRetrieval} $(r = 1)$}
\end{center}
\end{subfigure}
\begin{subfigure}{.45\textwidth}
\begin{center}
\includegraphics[height=2in]{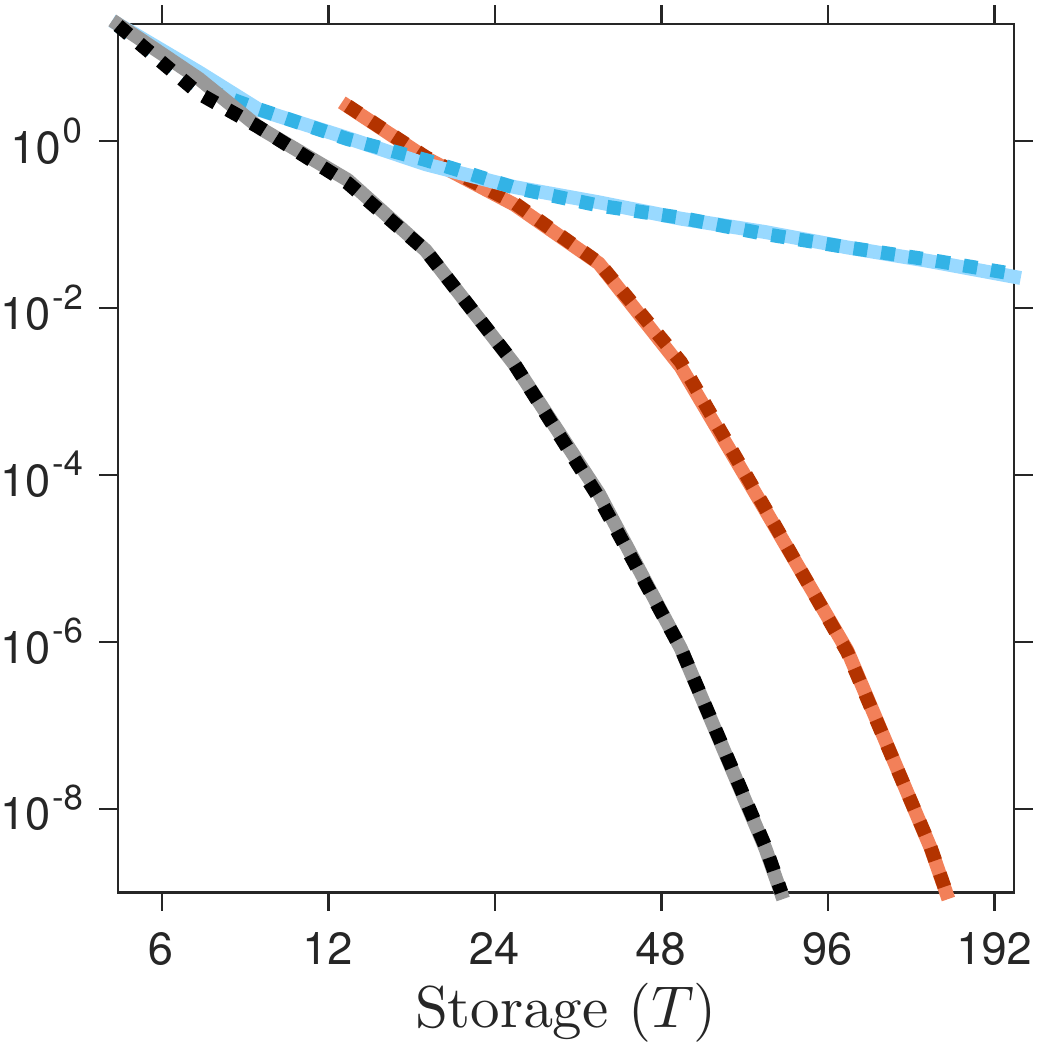}
\caption{\texttt{PhaseRetrieval} $(r = 5)$}
\end{center}
\end{subfigure}

\vspace{2pc}

\begin{subfigure}{.45\textwidth}
\begin{center}
\includegraphics[height=2in]{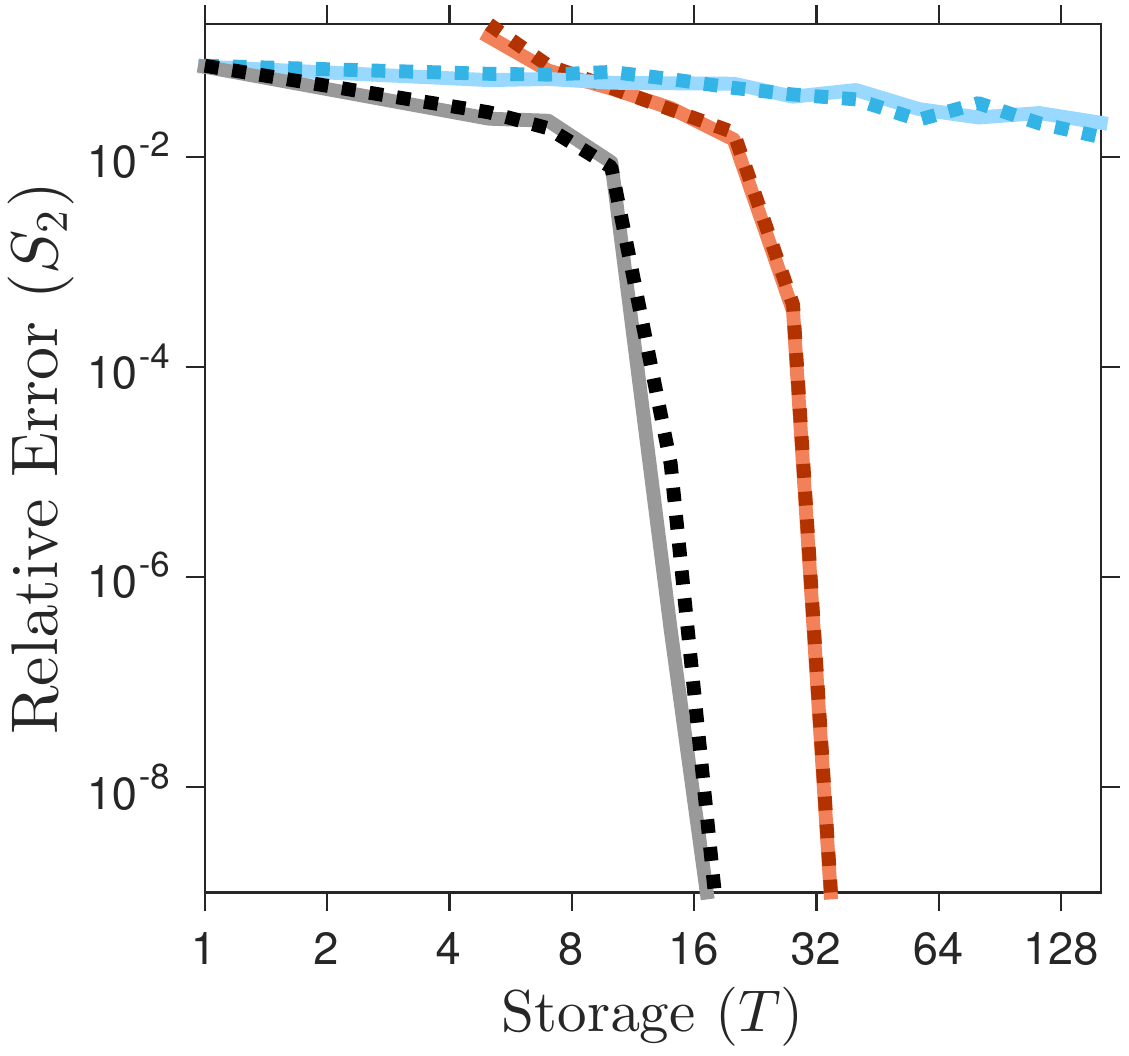}
\caption{\texttt{MaxCut} $(r = 1)$}
\end{center}
\end{subfigure}
\begin{subfigure}{.45\textwidth}
\begin{center}
\includegraphics[height=2in]{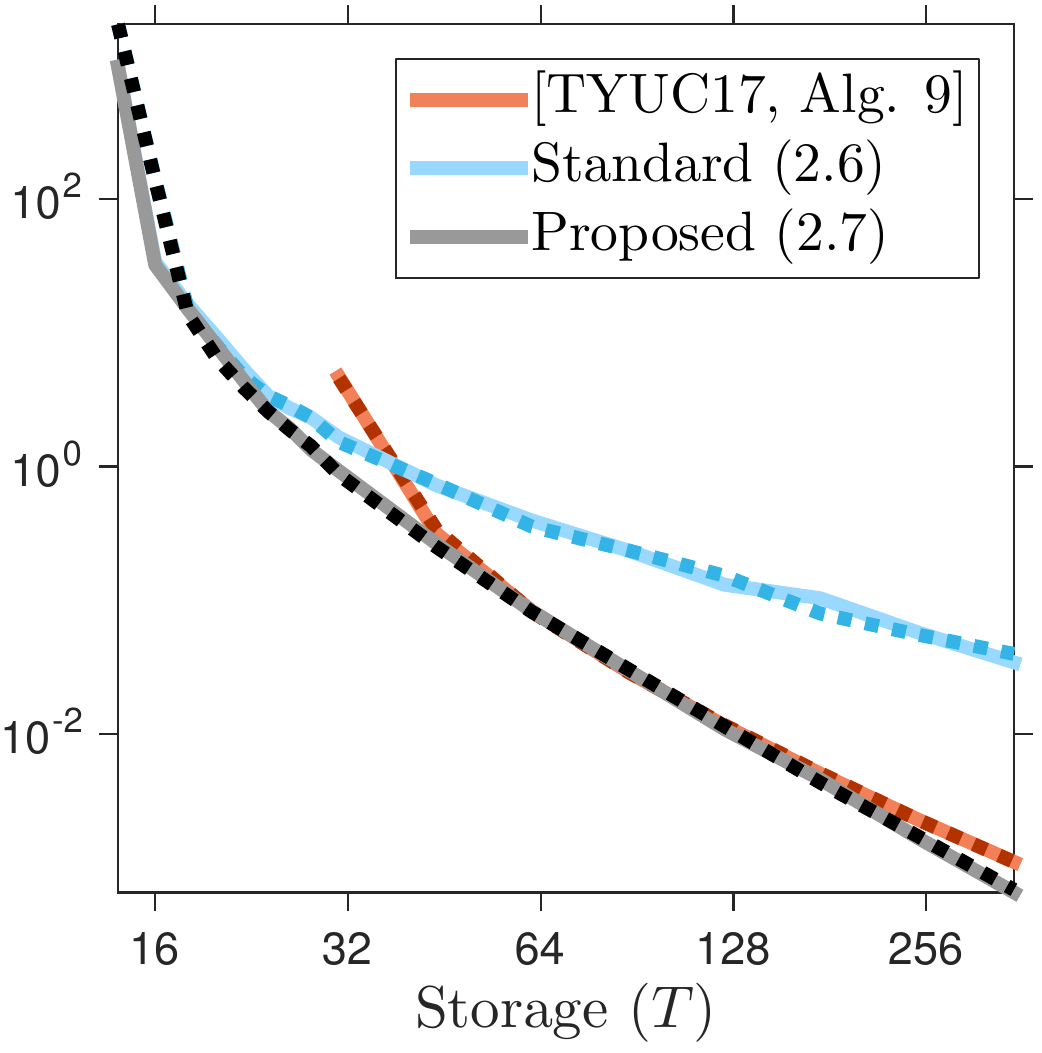}
\caption{\texttt{MaxCut} $(r = 14)$}
\end{center}
\end{subfigure}
\end{center}

\caption{\textbf{Application Examples, Approximation Rank $r$, Schatten $2$-Norm Error.}
The data series are
generated by three algorithms for rank-$r$ psd approximation.
\textbf{Solid lines} are generated from the Gaussian sketch;
\textbf{dashed lines} are from the SSFT sketch.
Each panel displays the Schatten 2-norm relative error~\eqref{eqn:relative-error}
as a function of storage cost $T$.  See Sec.~\ref{sec:numerics}
for details.} \label{fig:data-S2}
\end{figure}

\begin{figure}[htp!]
\begin{center}
\begin{subfigure}{.45\textwidth}
\begin{center}
\includegraphics[height=2in]{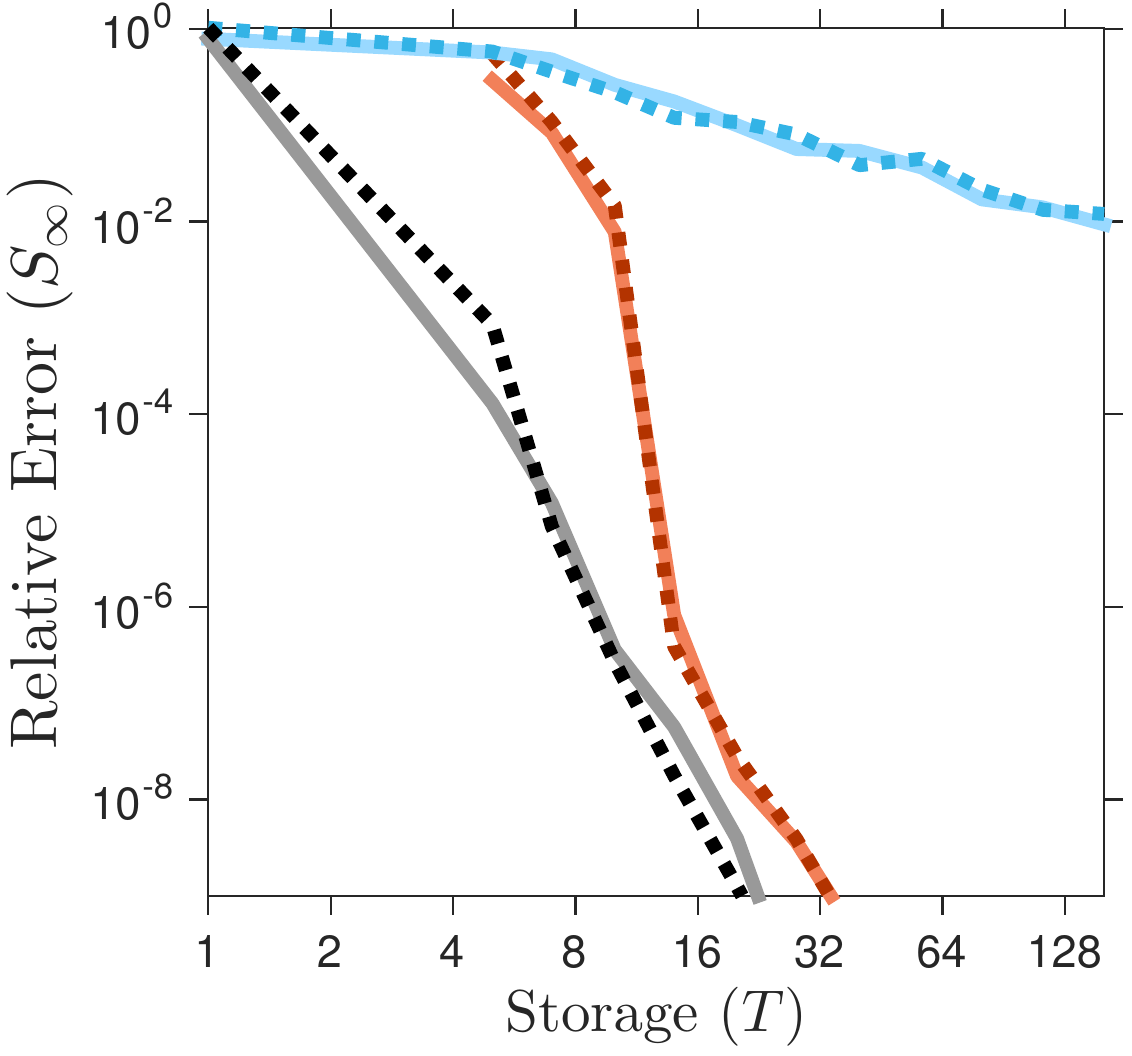}
\caption{\texttt{PhaseRetrieval} $(r = 1)$}
\end{center}
\end{subfigure}
\begin{subfigure}{.45\textwidth}
\begin{center}
\includegraphics[height=2in]{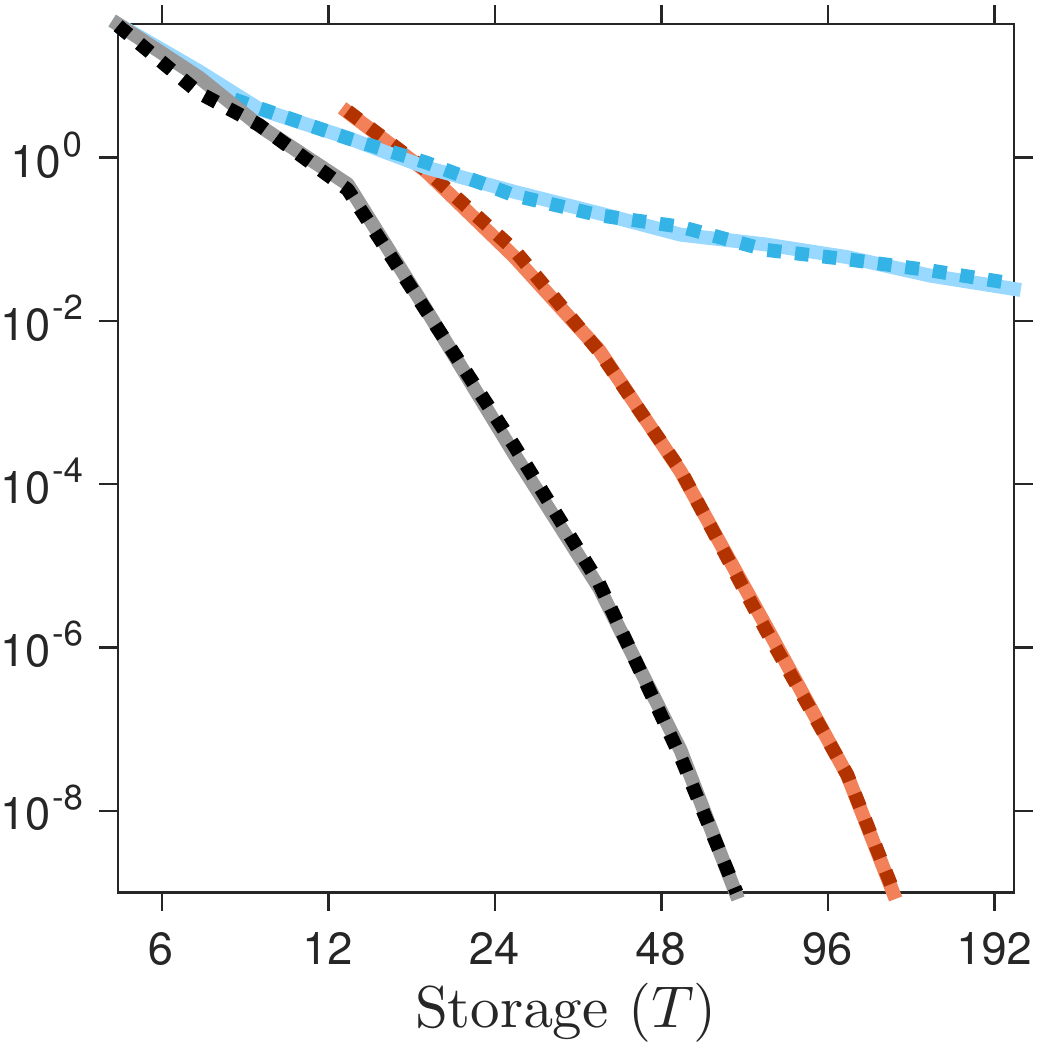}
\caption{\texttt{PhaseRetrieval} $(r = 5)$}
\end{center}
\end{subfigure}

\vspace{2pc}

\begin{subfigure}{.45\textwidth}
\begin{center}
\includegraphics[height=2in]{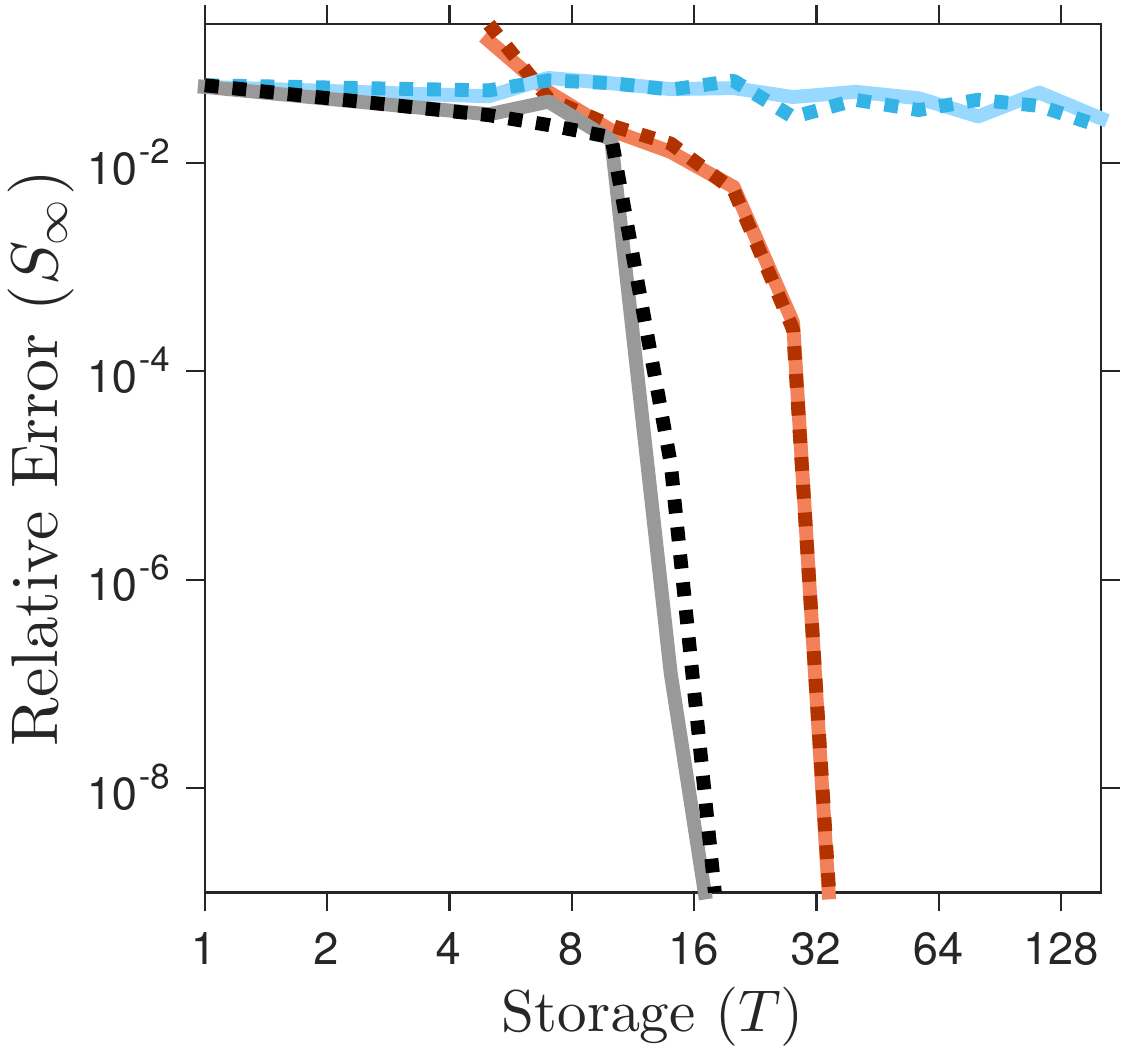}
\caption{\texttt{MaxCut} $(r = 1)$}
\end{center}
\end{subfigure}
\begin{subfigure}{.45\textwidth}
\begin{center}
\includegraphics[height=2in]{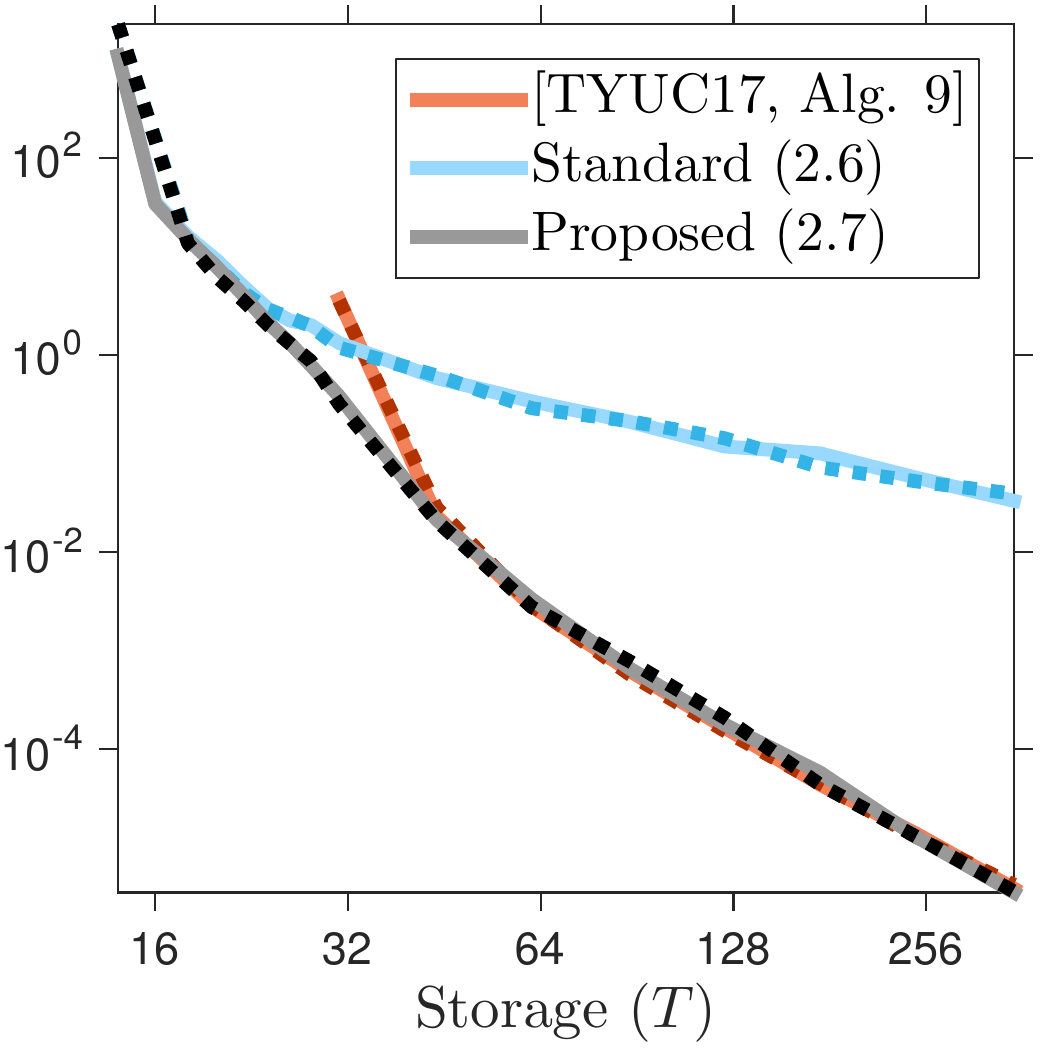}
\caption{\texttt{MaxCut} $(r = 14)$}
\end{center}
\end{subfigure}
\end{center}

\caption{\textbf{Application Examples, Approximation Rank $r$, Schatten $\infty$-Norm Error.}
The data series are generated by three algorithms for rank-$r$ psd approximation.
\textbf{Solid lines} are generated from the Gaussian sketch;
\textbf{dashed lines} are from the SSFT sketch.
Each panel displays the Schatten $\infty$-norm relative error~\eqref{eqn:relative-error}
as a function of storage cost $T$.  See Sec.~\ref{sec:numerics}
for details.} \label{fig:data-Sinf}
\end{figure}

\begin{figure}[htp!]
\begin{center}
\begin{subfigure}{.325\textwidth}
\begin{center}
\includegraphics[height=1.5in]{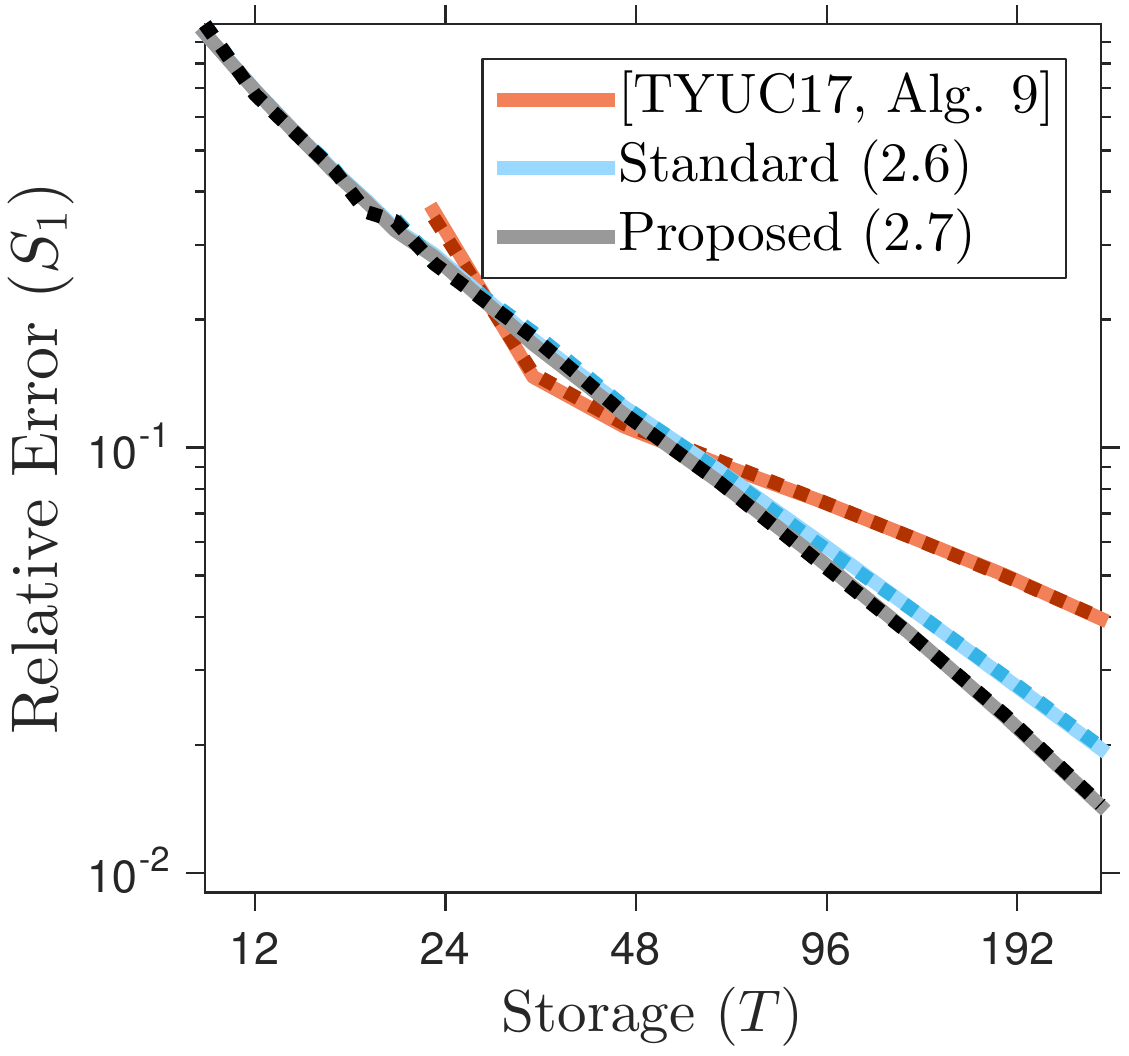}
\caption{\texttt{LowRankLowNoise}}
\end{center}
\end{subfigure}
\begin{subfigure}{.325\textwidth}
\begin{center}
\includegraphics[height=1.5in]{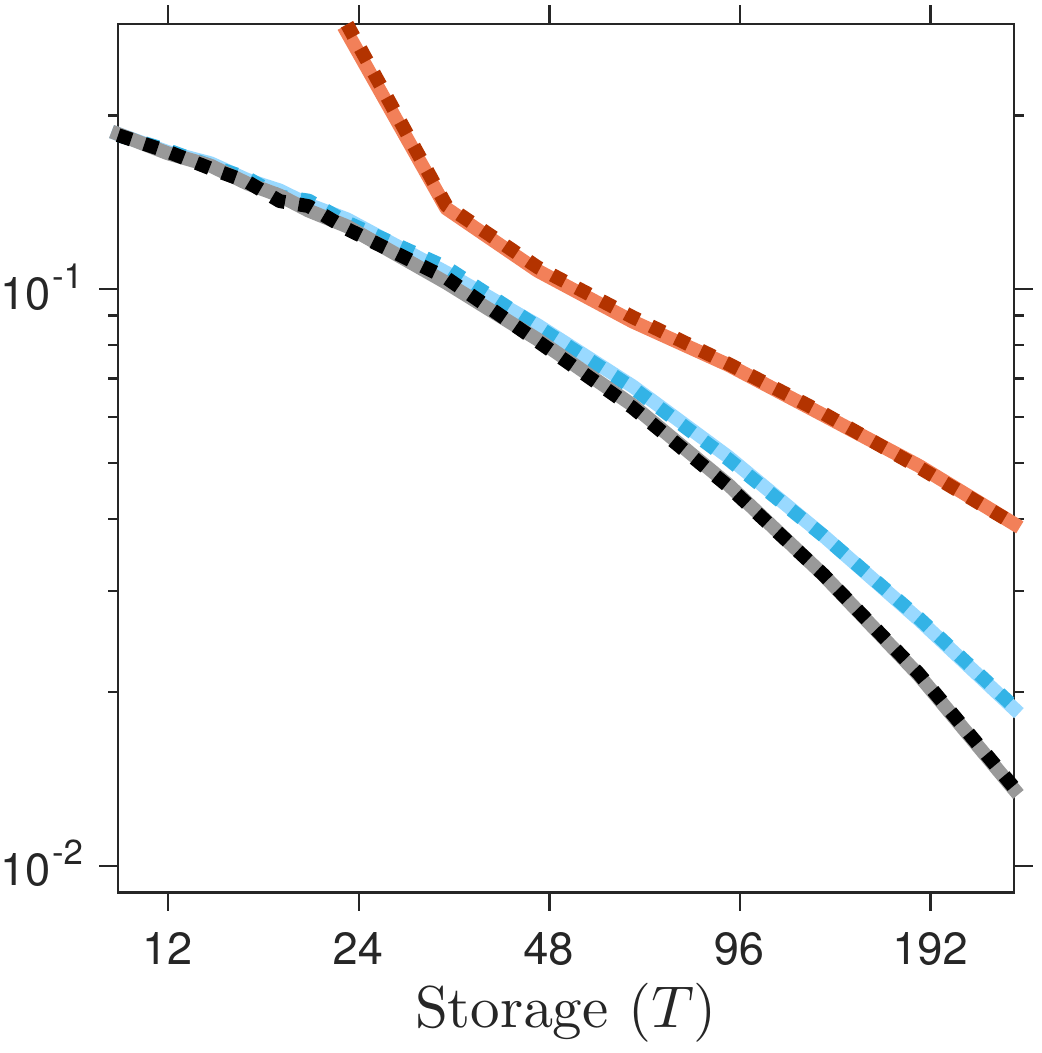}
\caption{\texttt{LowRankMedNoise}}
\end{center}
\end{subfigure}
\begin{subfigure}{.325\textwidth}
\begin{center}
\includegraphics[height=1.5in]{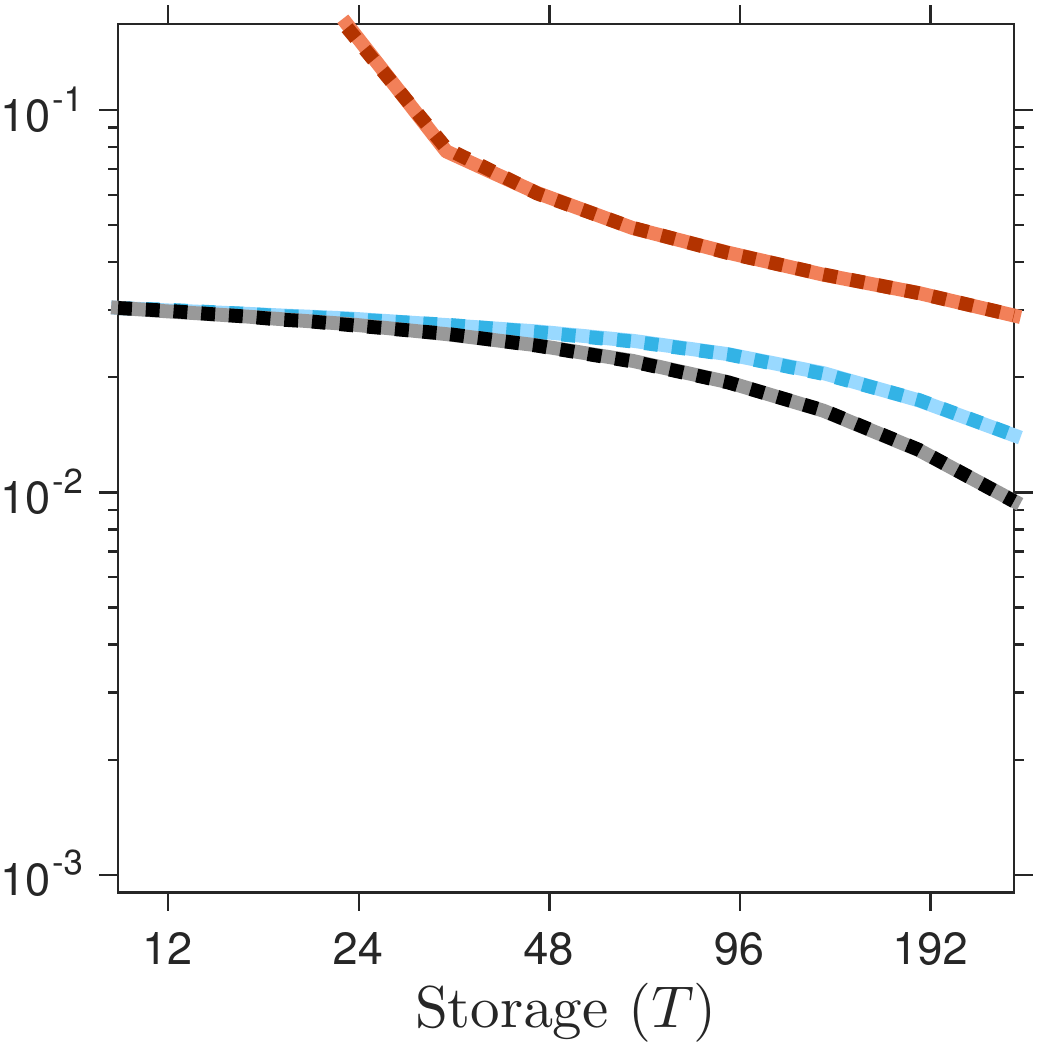}
\caption{\texttt{LowRankHiNoise}}
\end{center}
\end{subfigure}
\end{center}

\vspace{.5em}

\begin{center}
\begin{subfigure}{.325\textwidth}
\begin{center}
\includegraphics[height=1.5in]{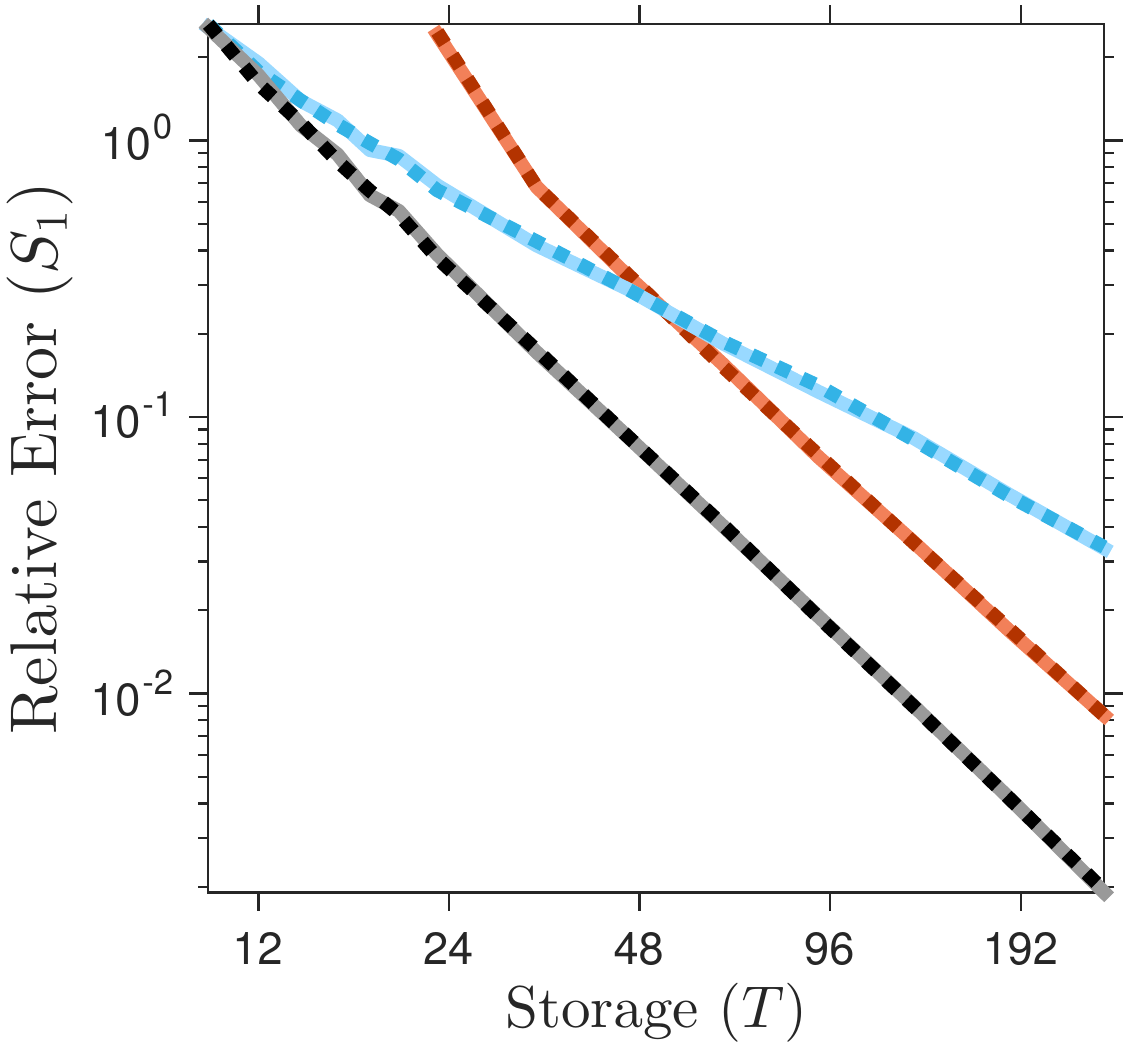}
\caption{\texttt{PolyDecayFast}}
\end{center}
\end{subfigure}
\begin{subfigure}{.325\textwidth}
\begin{center}
\includegraphics[height=1.5in]{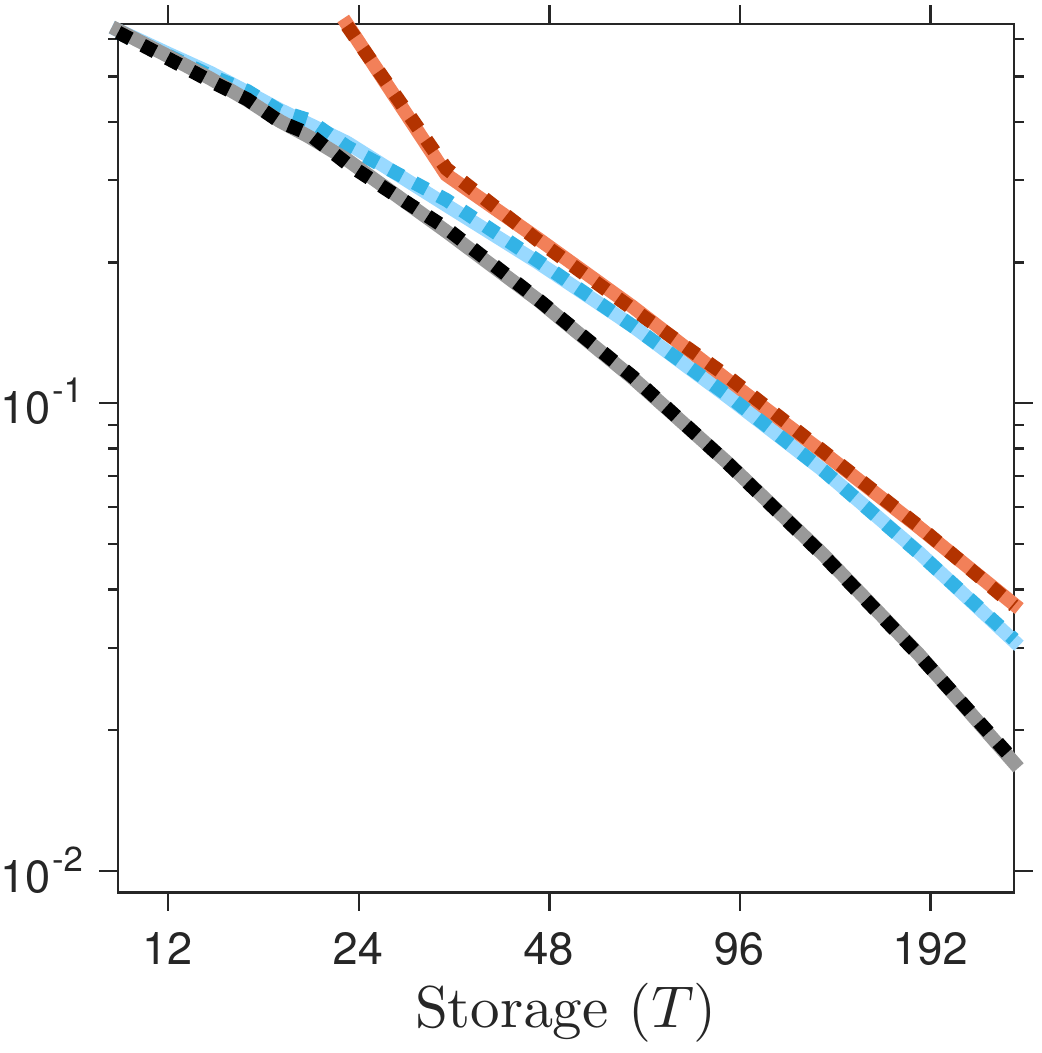}
\caption{\texttt{PolyDecayMed}}
\end{center}
\end{subfigure}
\begin{subfigure}{.325\textwidth}
\begin{center}
\includegraphics[height=1.5in]{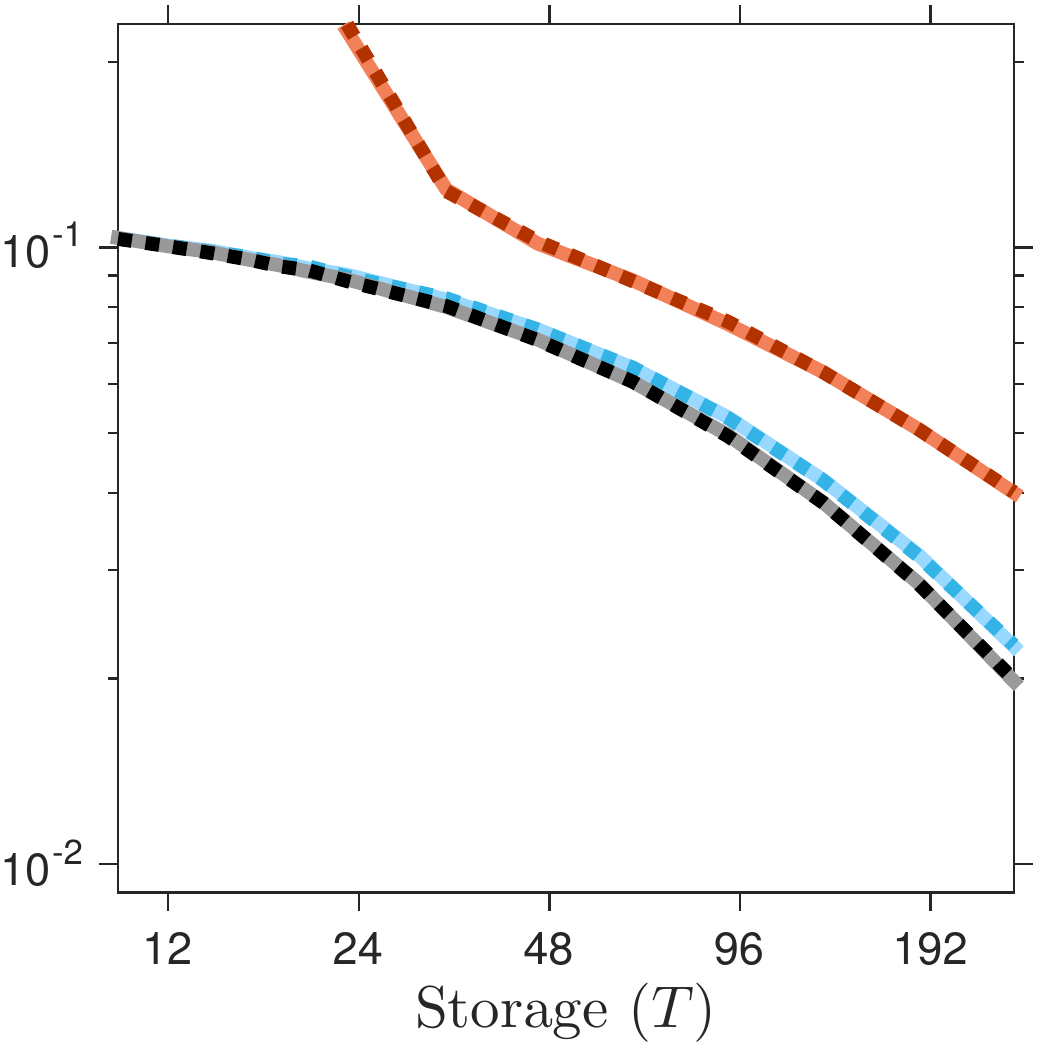}
\caption{\texttt{PolyDecaySlow}}
\end{center}
\end{subfigure}
\end{center}

\vspace{0.5em}

\begin{center}
\begin{subfigure}{.325\textwidth}
\begin{center}
\includegraphics[height=1.5in]{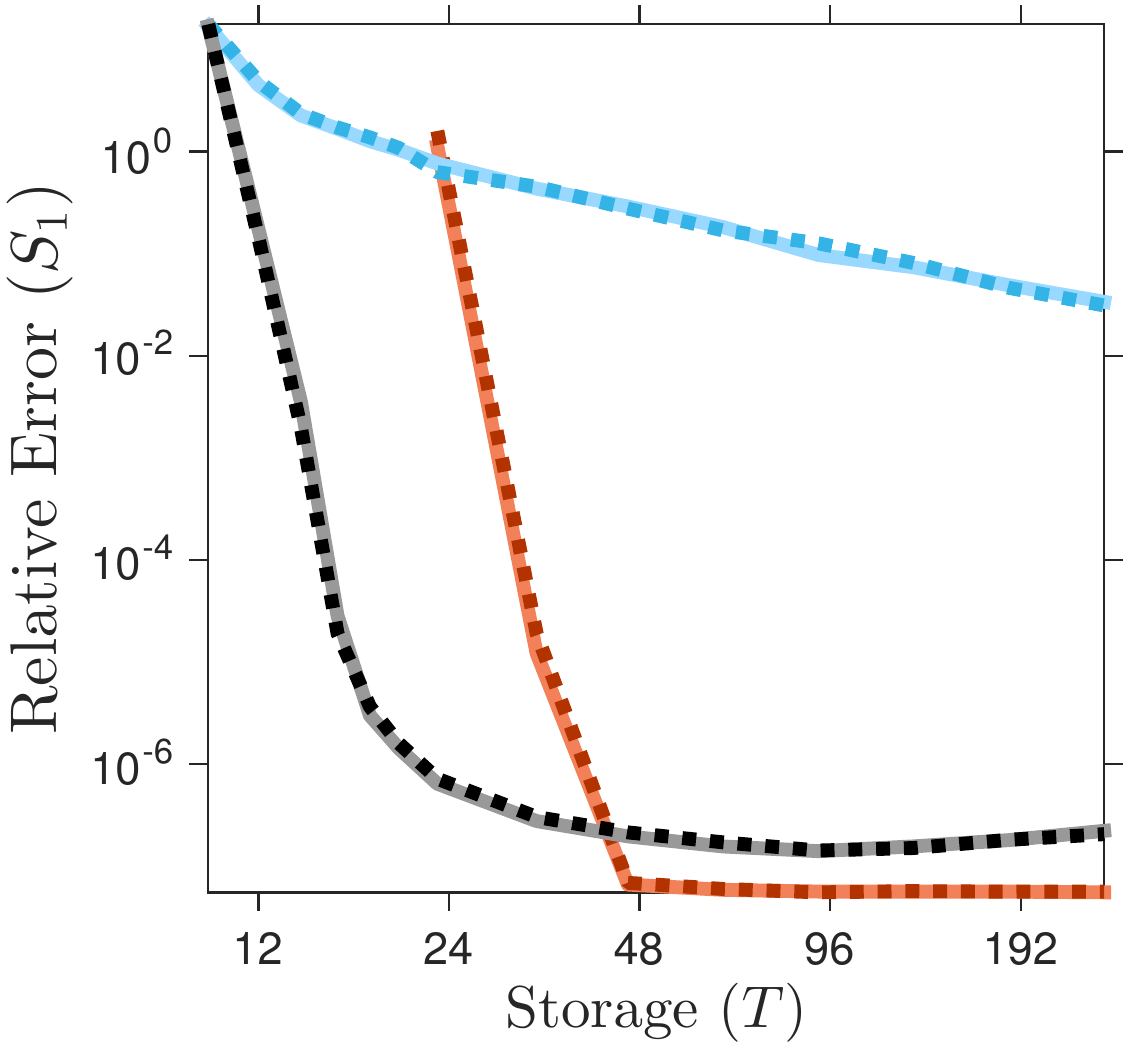}
\caption{\texttt{ExpDecayFast}}
\end{center}
\end{subfigure}
\begin{subfigure}{.325\textwidth}
\begin{center}
\includegraphics[height=1.5in]{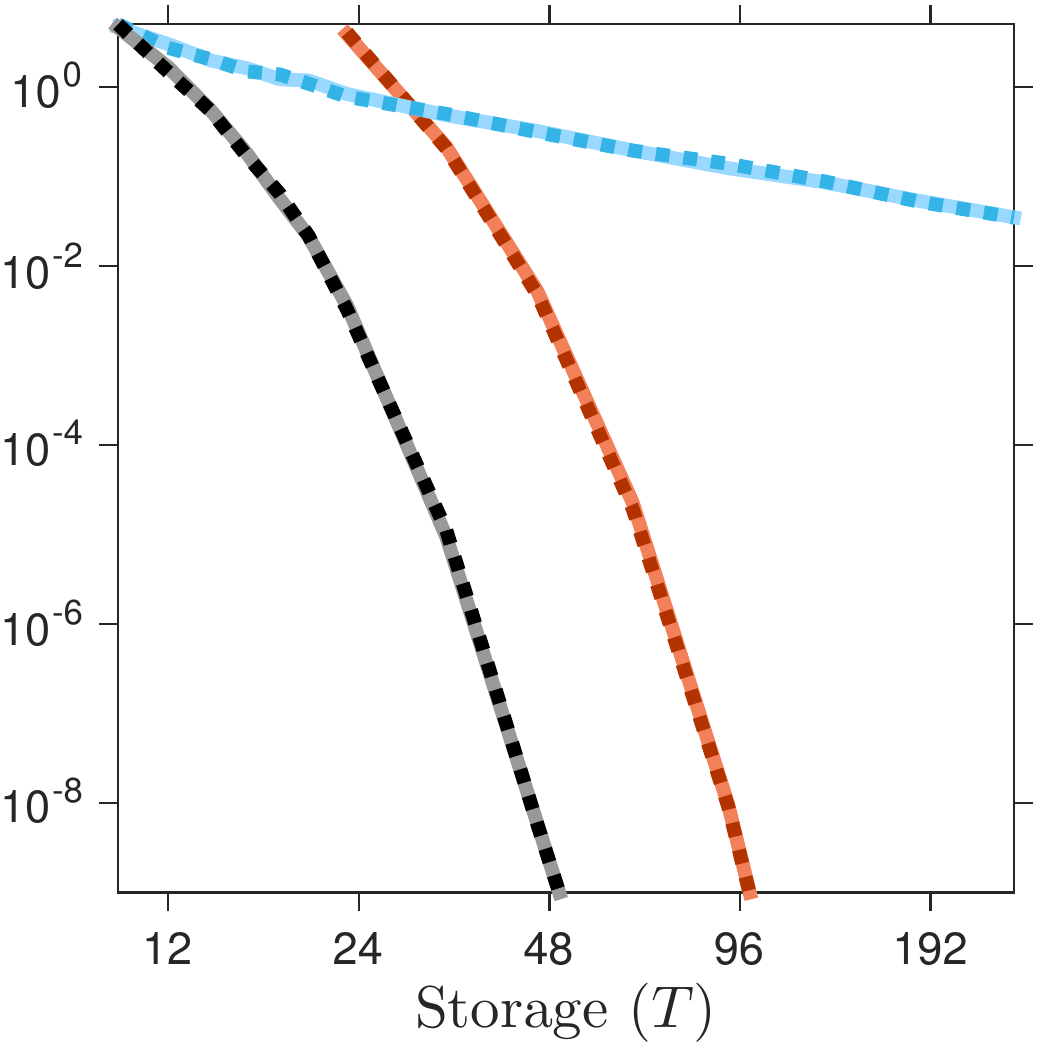}
\caption{\texttt{ExpDecayMed}}
\end{center}
\end{subfigure}
\begin{subfigure}{.325\textwidth}
\begin{center}
\includegraphics[height=1.5in]{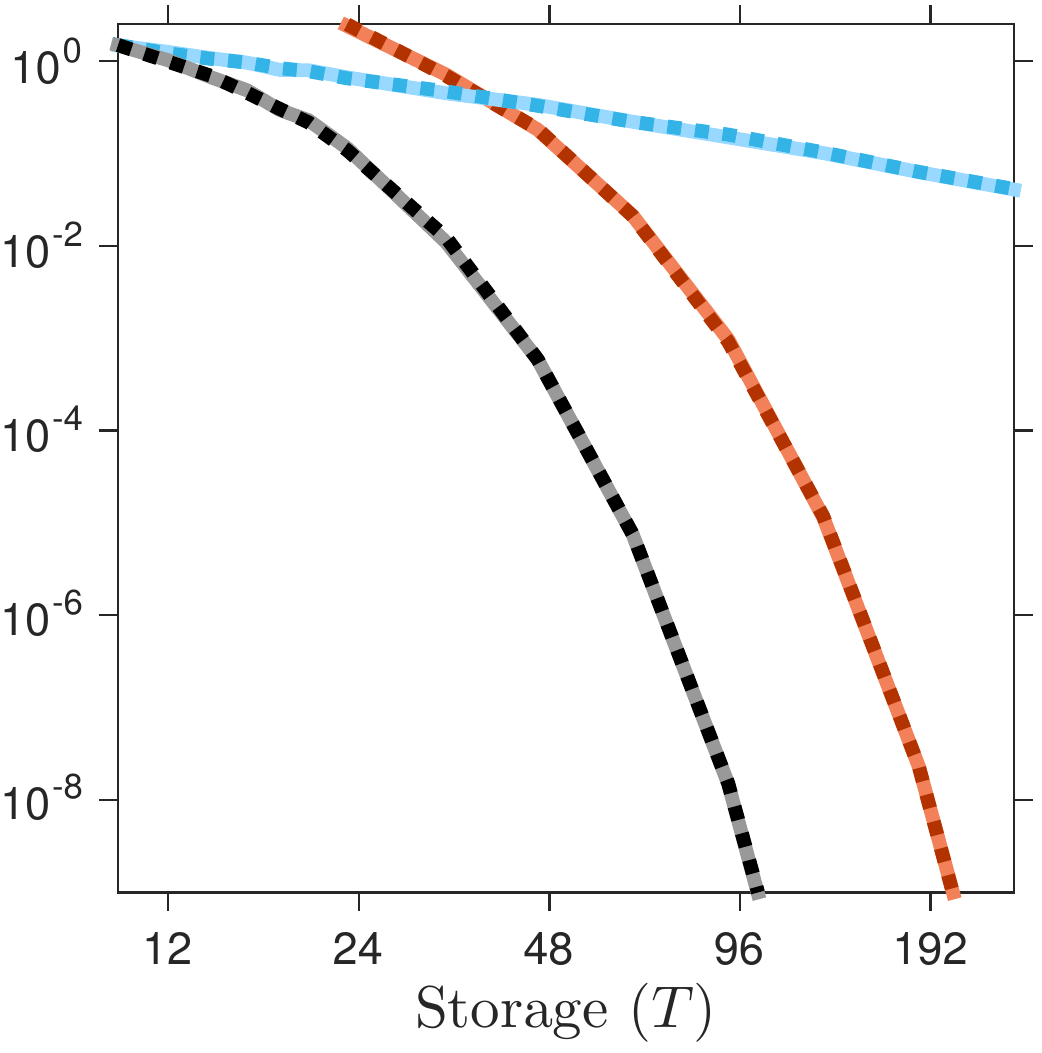}
\caption{\texttt{ExpDecaySlow}}
\end{center}
\end{subfigure}
\end{center}

\caption{\textbf{Synthetic Examples with Effective Rank $R = 5$, Approximation Rank $r = 10$, Schatten $1$-Norm Error.}
The series are
generated by three algorithms for rank-$r$ psd approximation with $r = 10$.
\textbf{Solid lines} are generated from the Gaussian sketch;
\textbf{dashed lines} are from the SSFT sketch.
Each panel displays the Schatten 1-norm relative error~\eqref{eqn:relative-error}
as a function of storage cost $T$.  See Sec.~\ref{sec:numerics}
for details.}
\label{fig:synthetic-S1-R5}
\end{figure}

\begin{figure}[htp!]
\begin{center}
\begin{subfigure}{.325\textwidth}
\begin{center}
\includegraphics[height=1.5in]{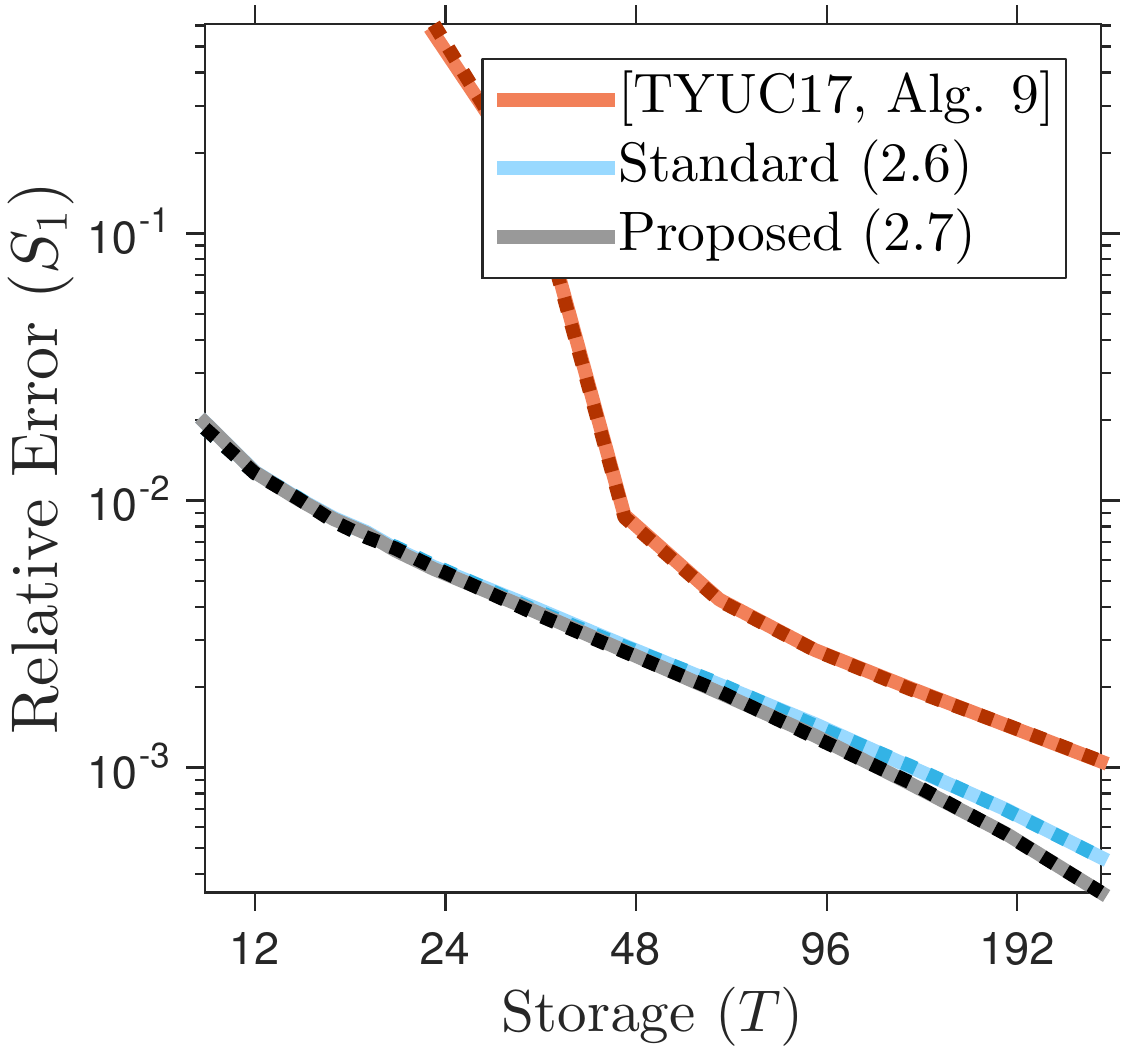}
\caption{\texttt{LowRankLowNoise}}
\end{center}
\end{subfigure}
\begin{subfigure}{.325\textwidth}
\begin{center}
\includegraphics[height=1.5in]{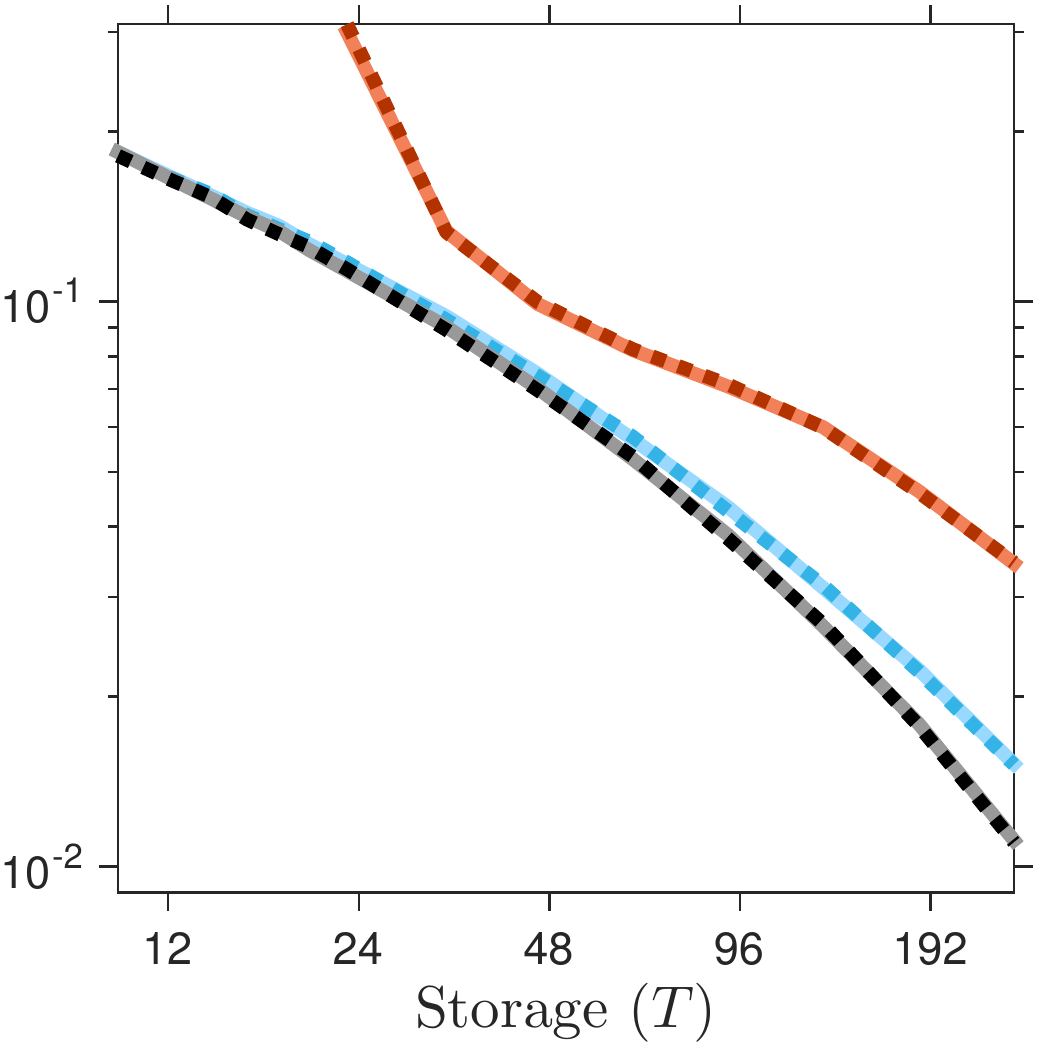}
\caption{\texttt{LowRankMedNoise}}
\end{center}
\end{subfigure}
\begin{subfigure}{.325\textwidth}
\begin{center}
\includegraphics[height=1.5in]{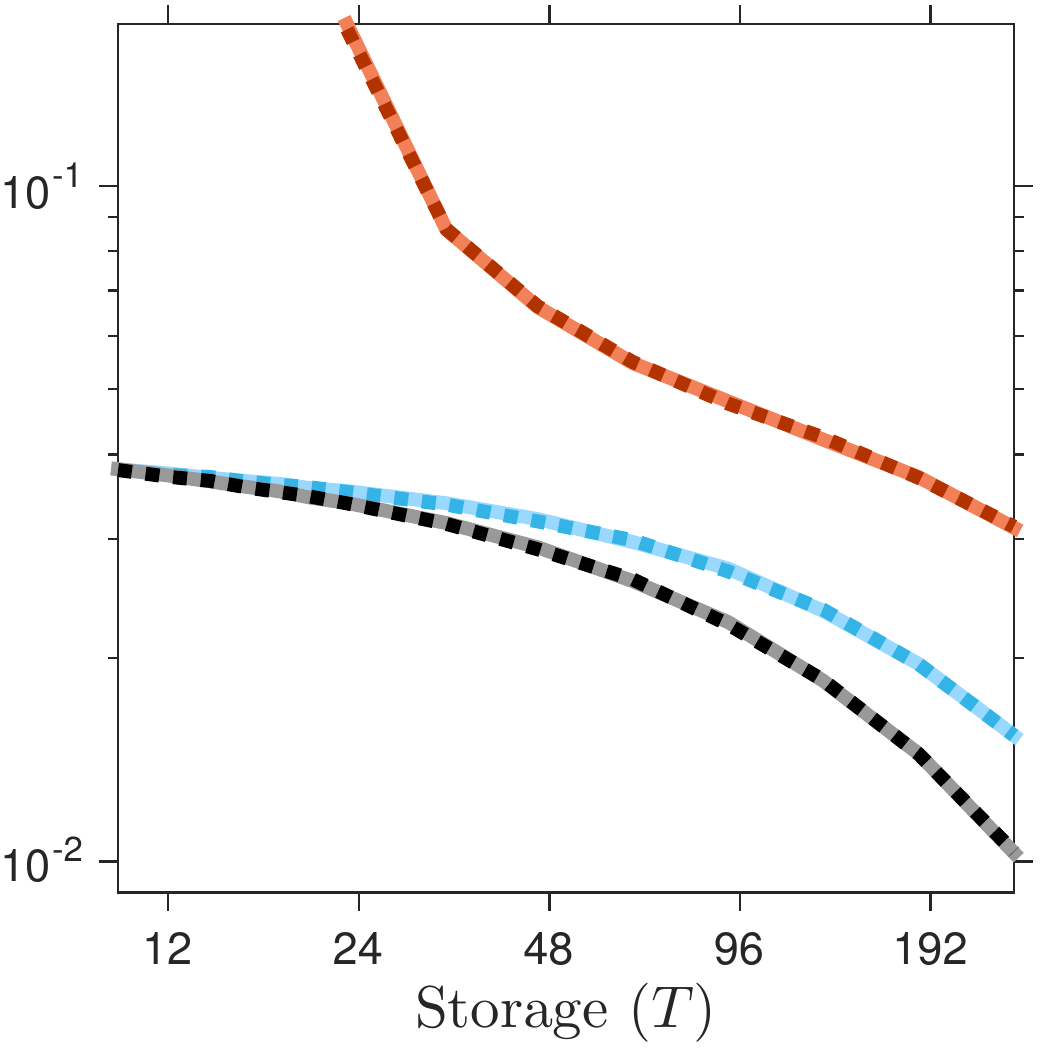}
\caption{\texttt{LowRankHiNoise}}
\end{center}
\end{subfigure}
\end{center}

\vspace{.5em}

\begin{center}
\begin{subfigure}{.325\textwidth}
\begin{center}
\includegraphics[height=1.5in]{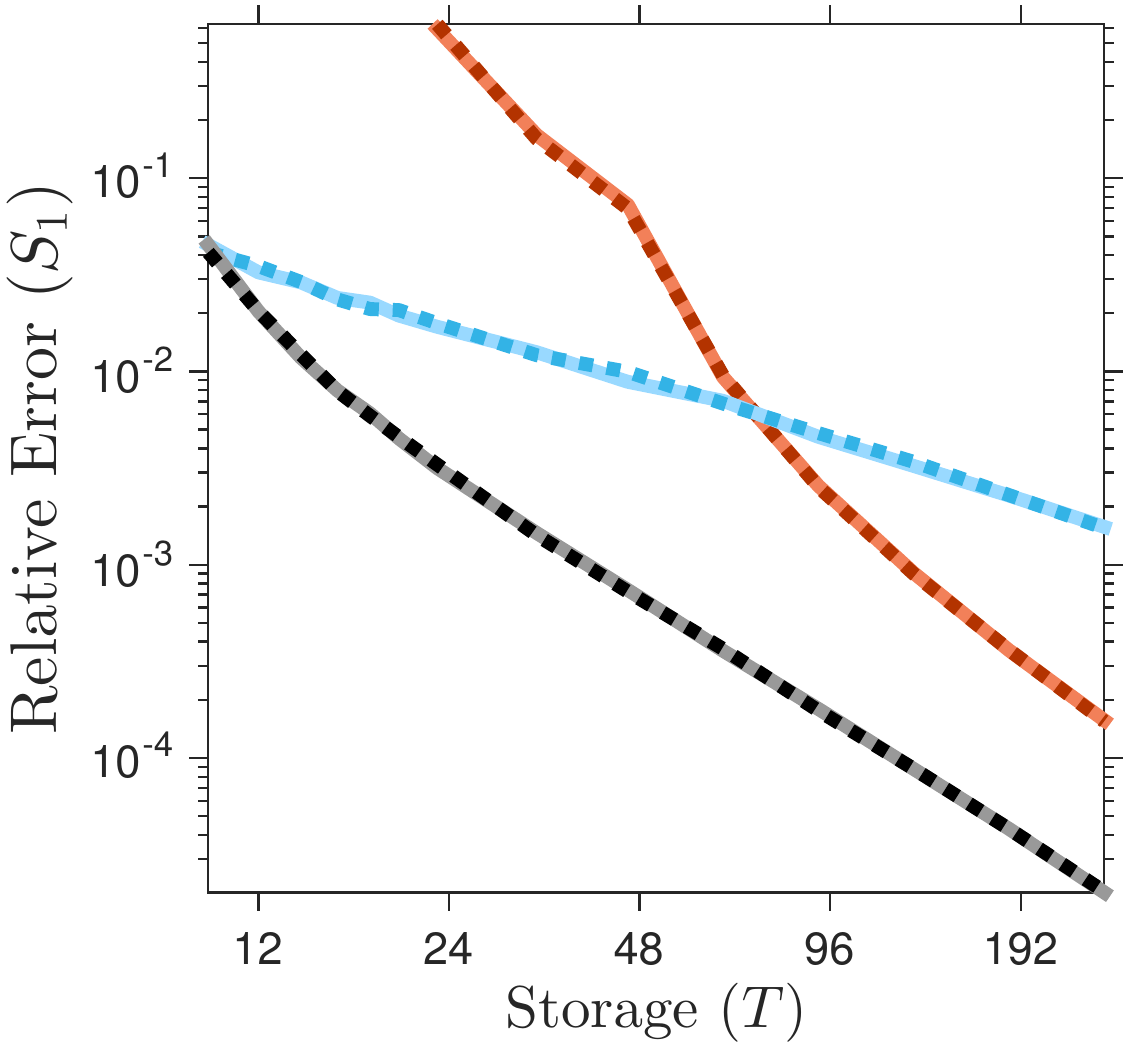}
\caption{\texttt{PolyDecayFast}}
\end{center}
\end{subfigure}
\begin{subfigure}{.325\textwidth}
\begin{center}
\includegraphics[height=1.5in]{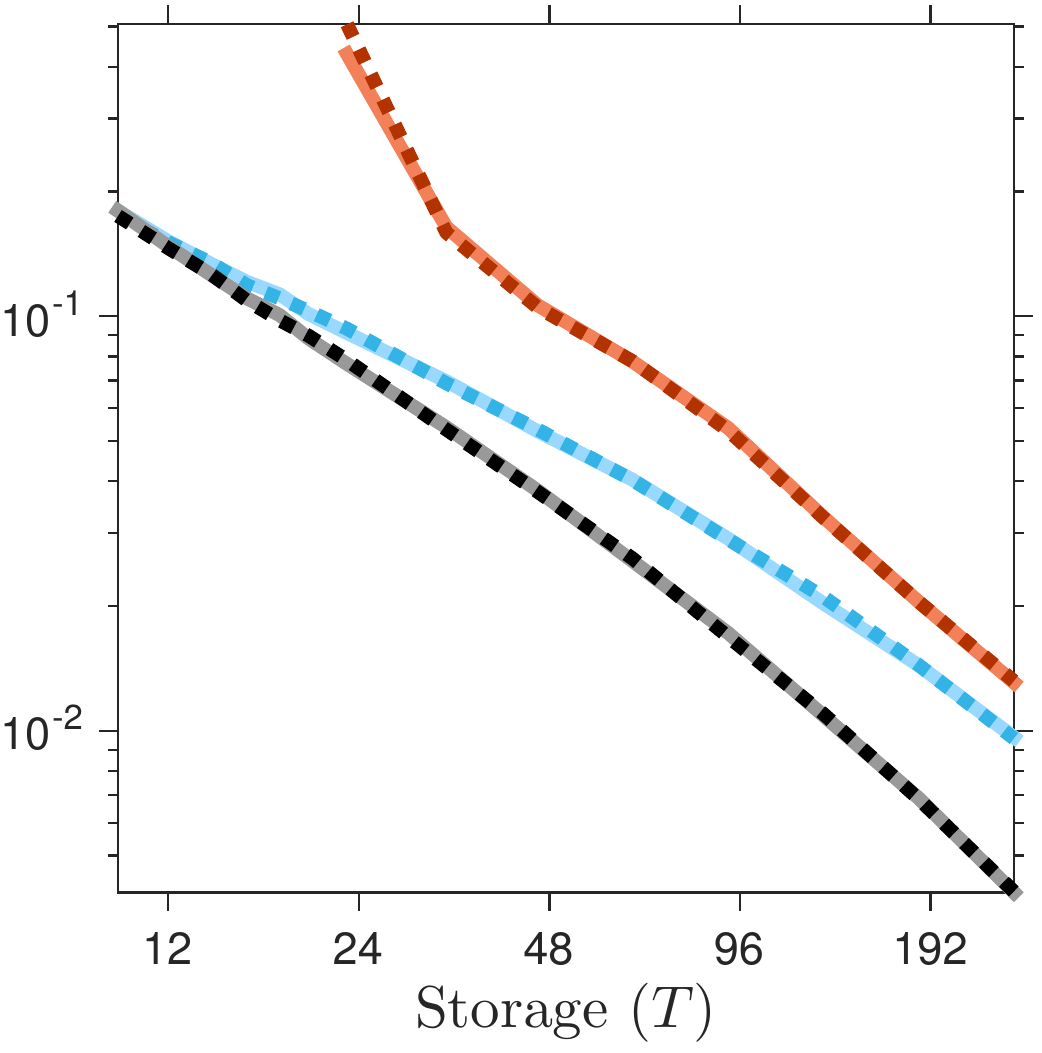}
\caption{\texttt{PolyDecayMed}}
\end{center}
\end{subfigure}
\begin{subfigure}{.325\textwidth}
\begin{center}
\includegraphics[height=1.5in]{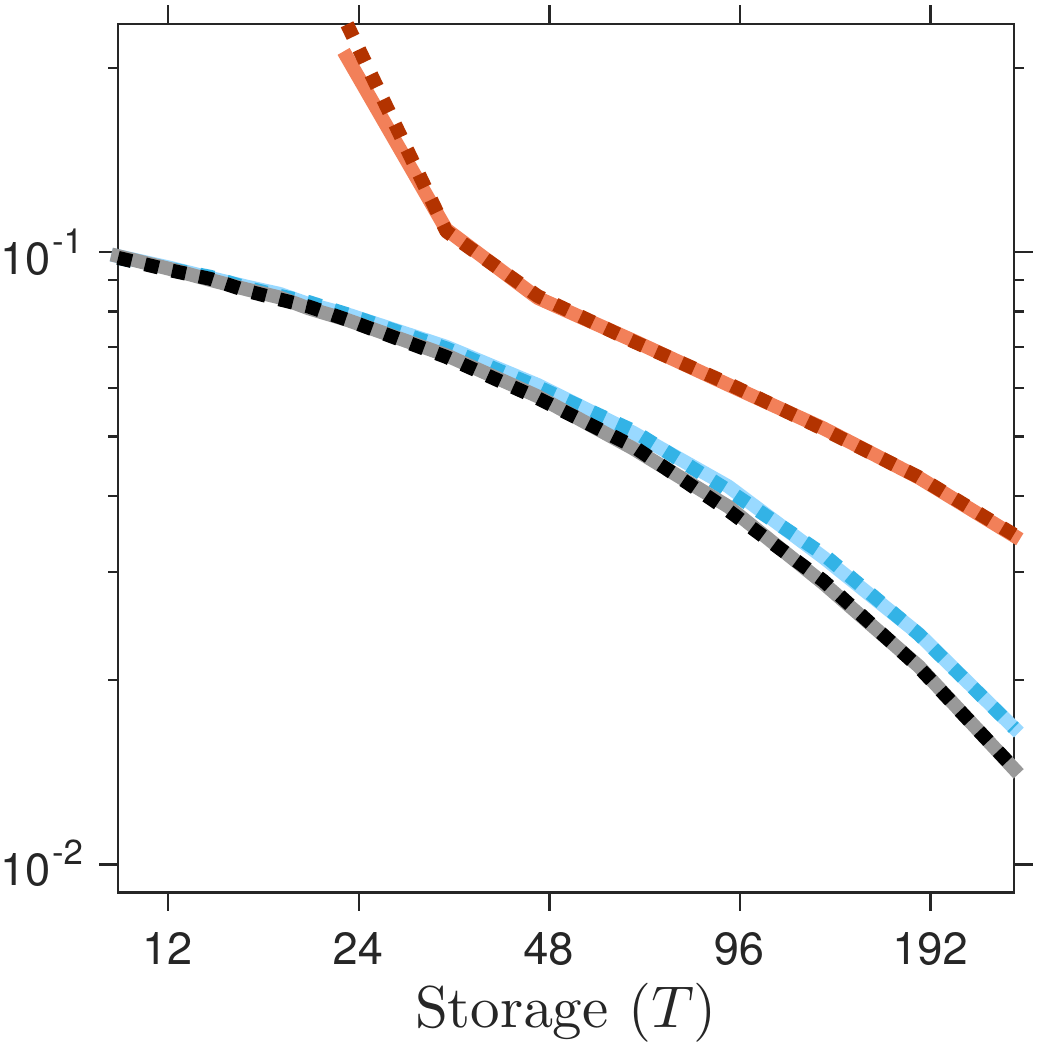}
\caption{\texttt{PolyDecaySlow}}
\end{center}
\end{subfigure}
\end{center}

\vspace{0.5em}

\begin{center}
\begin{subfigure}{.325\textwidth}
\begin{center}
\includegraphics[height=1.5in]{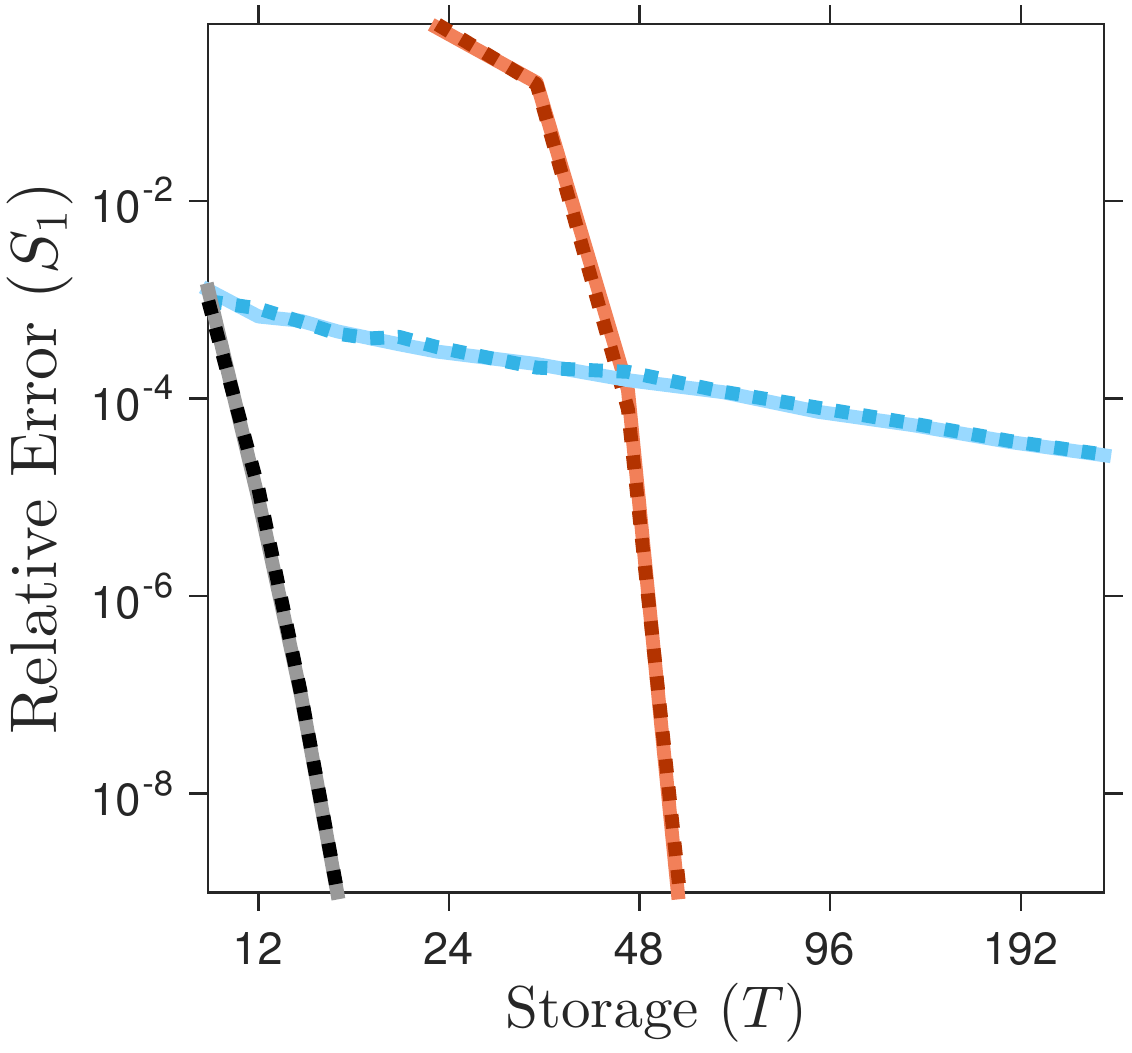}
\caption{\texttt{ExpDecayFast}}
\end{center}
\end{subfigure}
\begin{subfigure}{.325\textwidth}
\begin{center}
\includegraphics[height=1.5in]{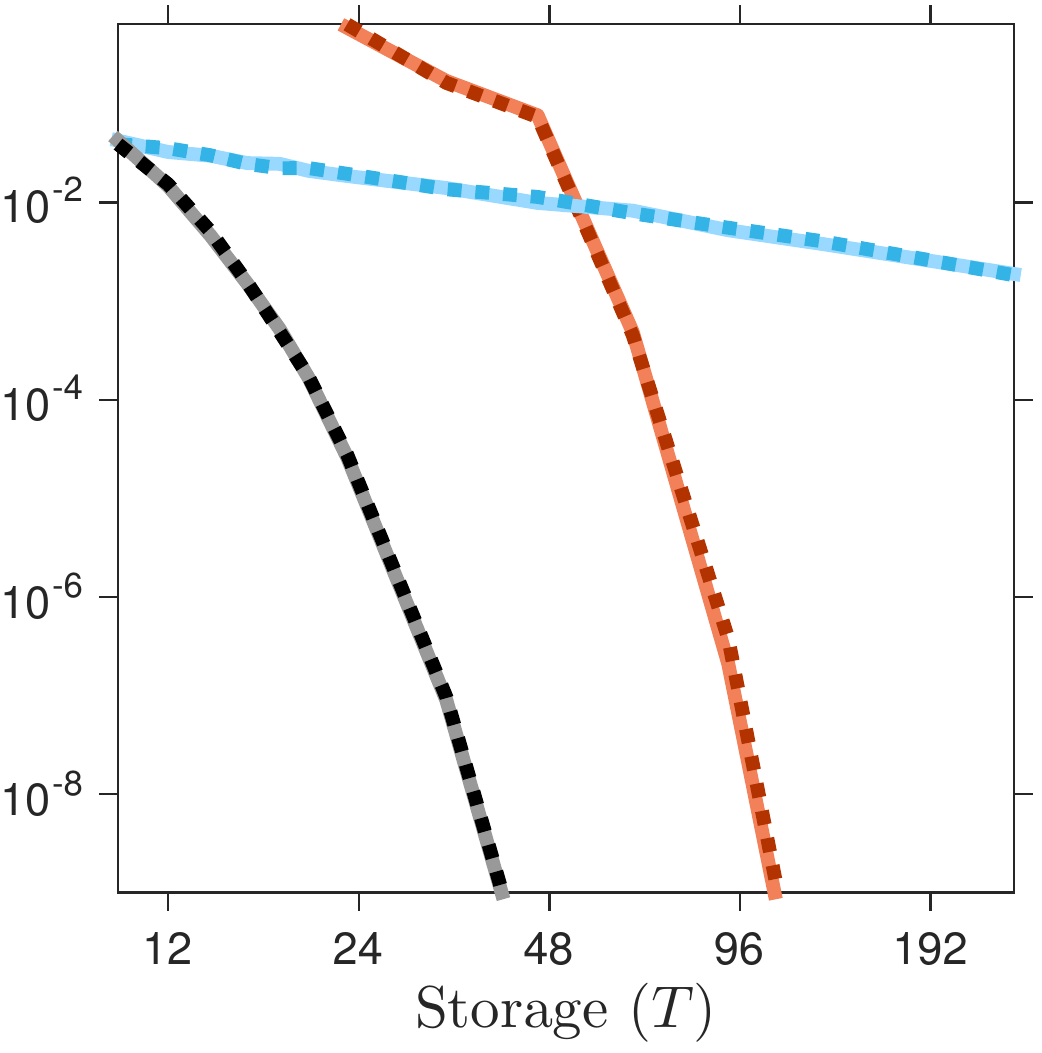}
\caption{\texttt{ExpDecayMed}}
\end{center}
\end{subfigure}
\begin{subfigure}{.325\textwidth}
\begin{center}
\includegraphics[height=1.5in]{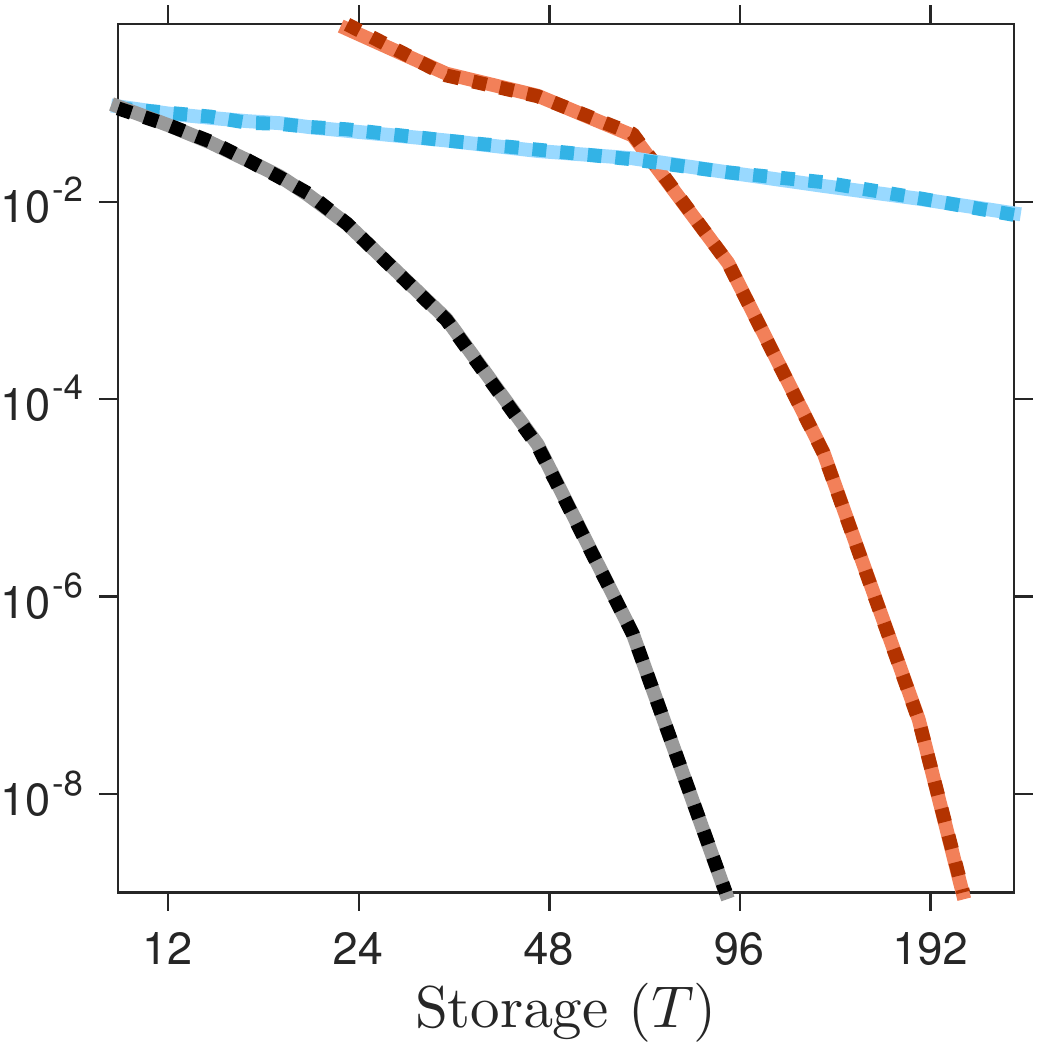}
\caption{\texttt{ExpDecaySlow}}
\end{center}
\end{subfigure}
\end{center}

\caption{\textbf{Synthetic Examples with Effective Rank $R = 20$, Approximation Rank $r = 10$, Schatten $1$-Norm Error.}
The series are generated by three algorithms for rank-$r$ psd approximation with $r = 10$.
\textbf{Solid lines} are generated from the Gaussian sketch;
\textbf{dashed lines} are from the SSFT sketch.
Each panel displays the  Schatten 1-norm relative error~\eqref{eqn:relative-error}
as a function of storage cost $T$.  See Sec.~\ref{sec:numerics}
for details.}
\label{fig:synthetic-S1-R25}
\end{figure}

\begin{figure}[htp!]
\begin{center}
\begin{subfigure}{.325\textwidth}
\begin{center}
\includegraphics[height=1.5in]{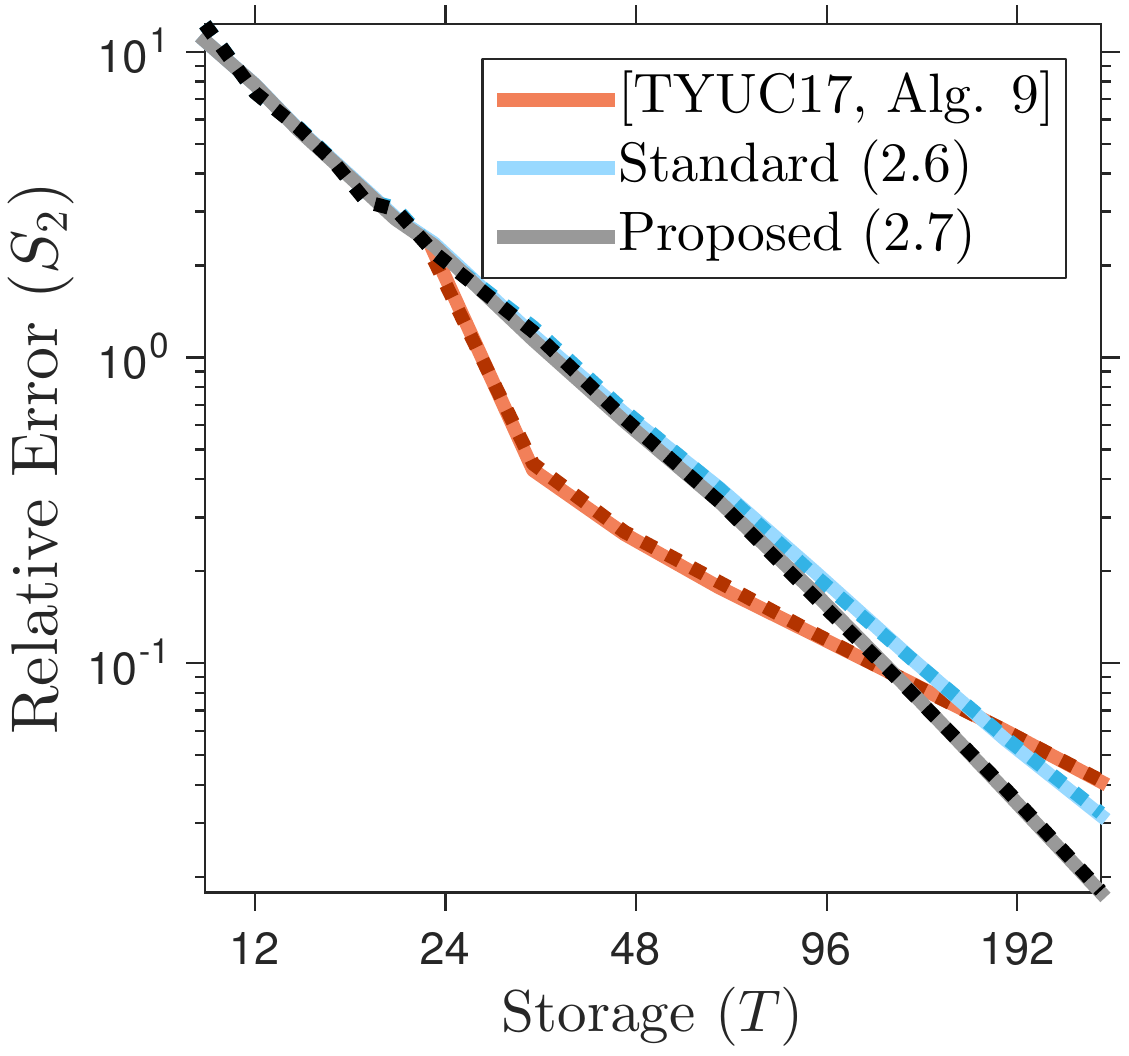}
\caption{\texttt{LowRankLowNoise}}
\end{center}
\end{subfigure}
\begin{subfigure}{.325\textwidth}
\begin{center}
\includegraphics[height=1.5in]{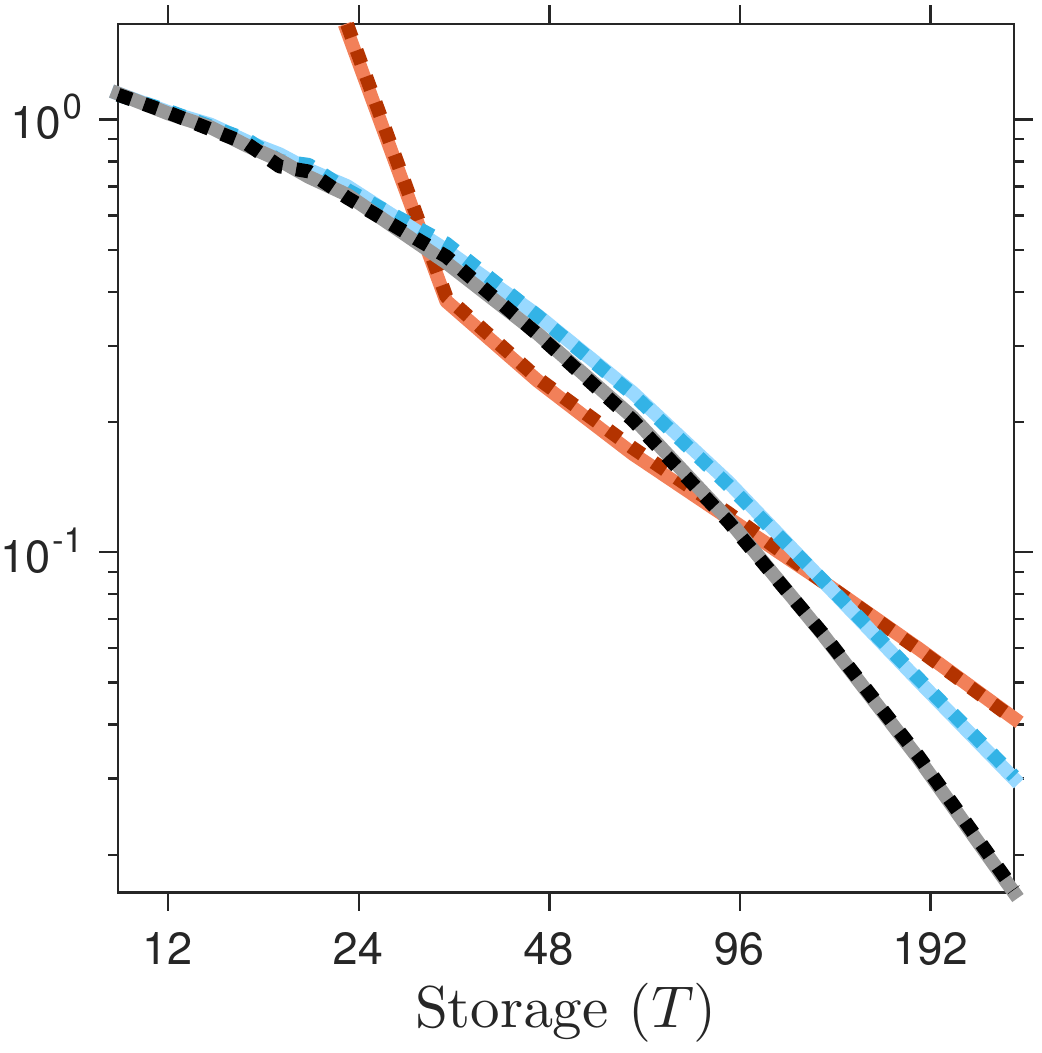}
\caption{\texttt{LowRankMedNoise}}
\end{center}
\end{subfigure}
\begin{subfigure}{.325\textwidth}
\begin{center}
\includegraphics[height=1.5in]{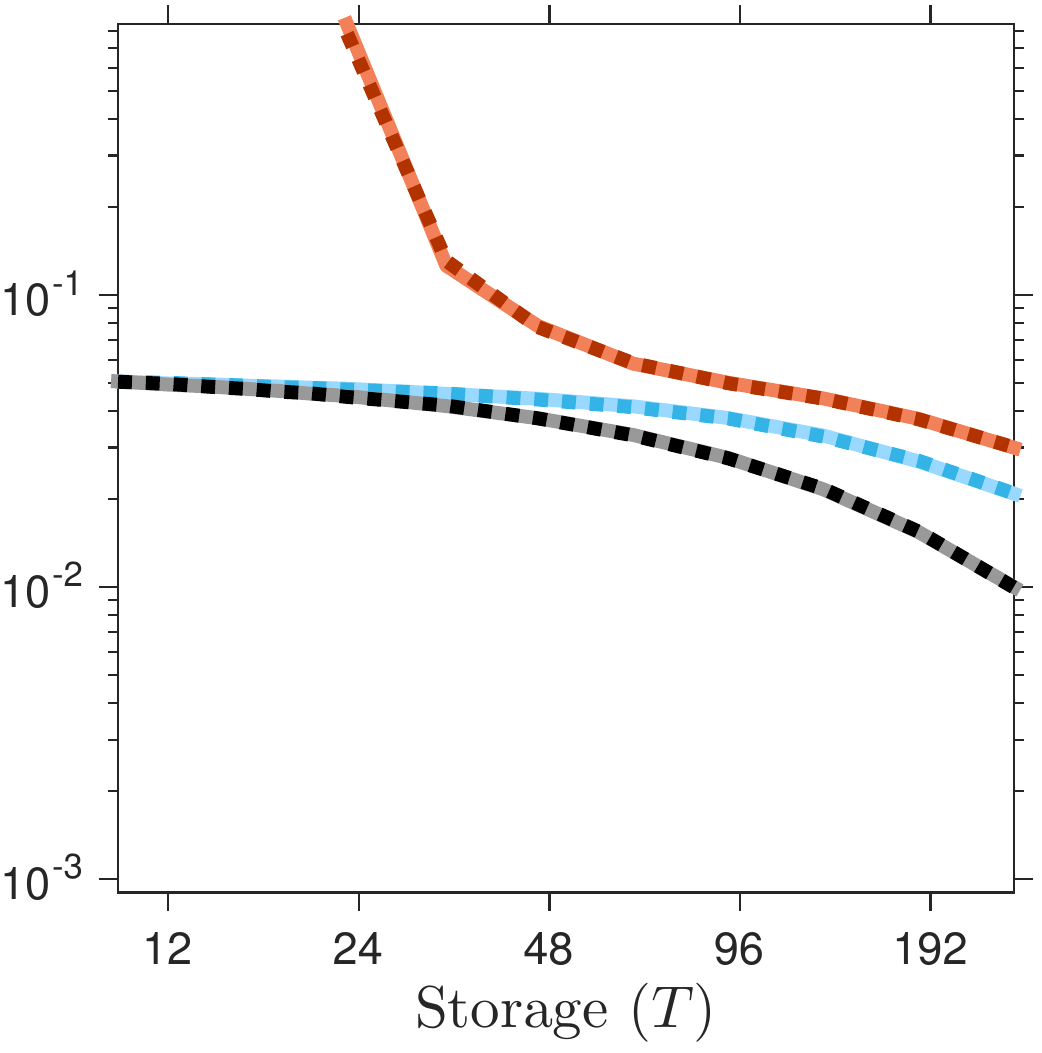}
\caption{\texttt{LowRankHiNoise}}
\end{center}
\end{subfigure}
\end{center}

\vspace{.5em}

\begin{center}
\begin{subfigure}{.325\textwidth}
\begin{center}
\includegraphics[height=1.5in]{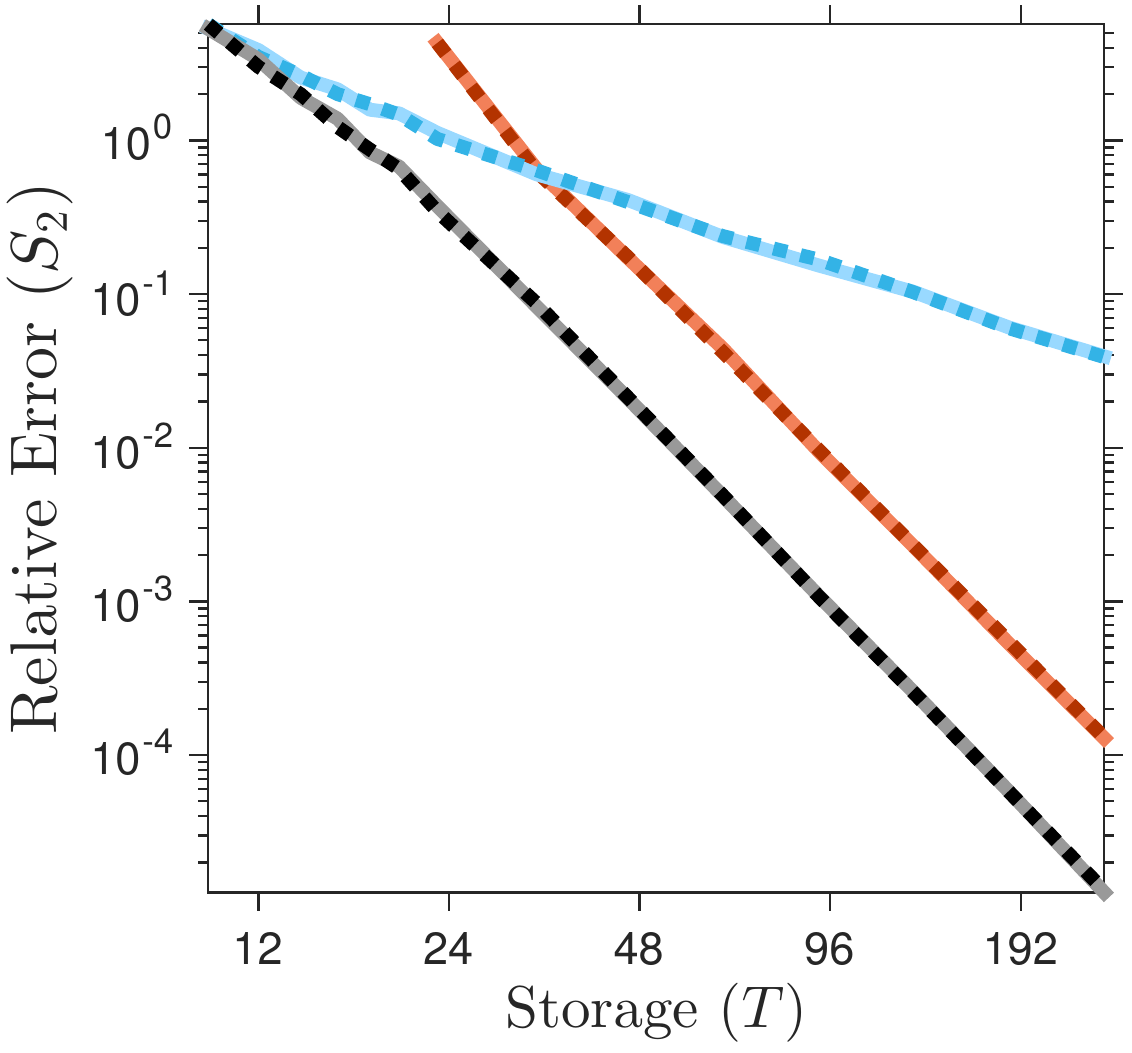}
\caption{\texttt{PolyDecayFast}}
\end{center}
\end{subfigure}
\begin{subfigure}{.325\textwidth}
\begin{center}
\includegraphics[height=1.5in]{art/S1/R5/PolyDecayMed_R5_S1.pdf}
\caption{\texttt{PolyDecayMed}}
\end{center}
\end{subfigure}
\begin{subfigure}{.325\textwidth}
\begin{center}
\includegraphics[height=1.5in]{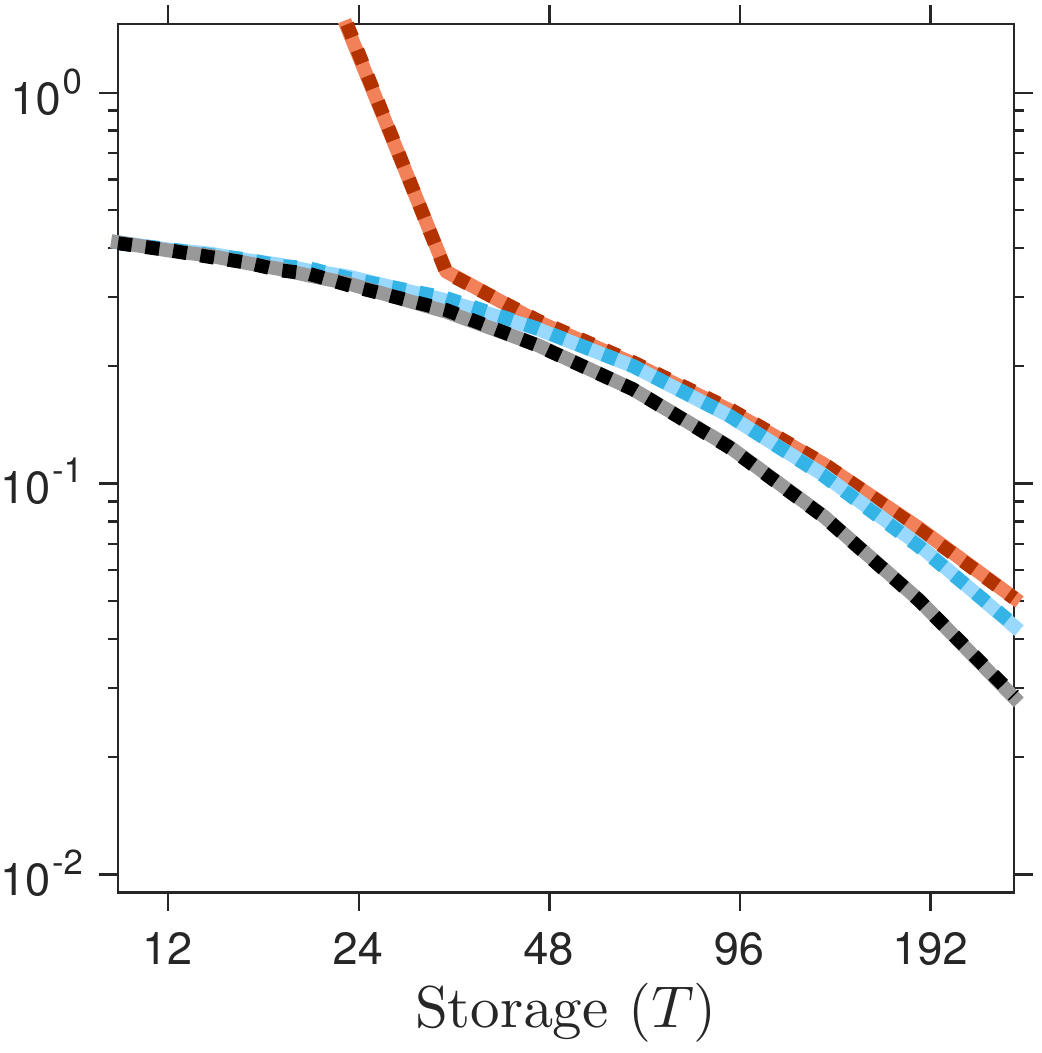}
\caption{\texttt{PolyDecaySlow}}
\end{center}
\end{subfigure}
\end{center}

\vspace{0.5em}

\begin{center}
\begin{subfigure}{.325\textwidth}
\begin{center}
\includegraphics[height=1.5in]{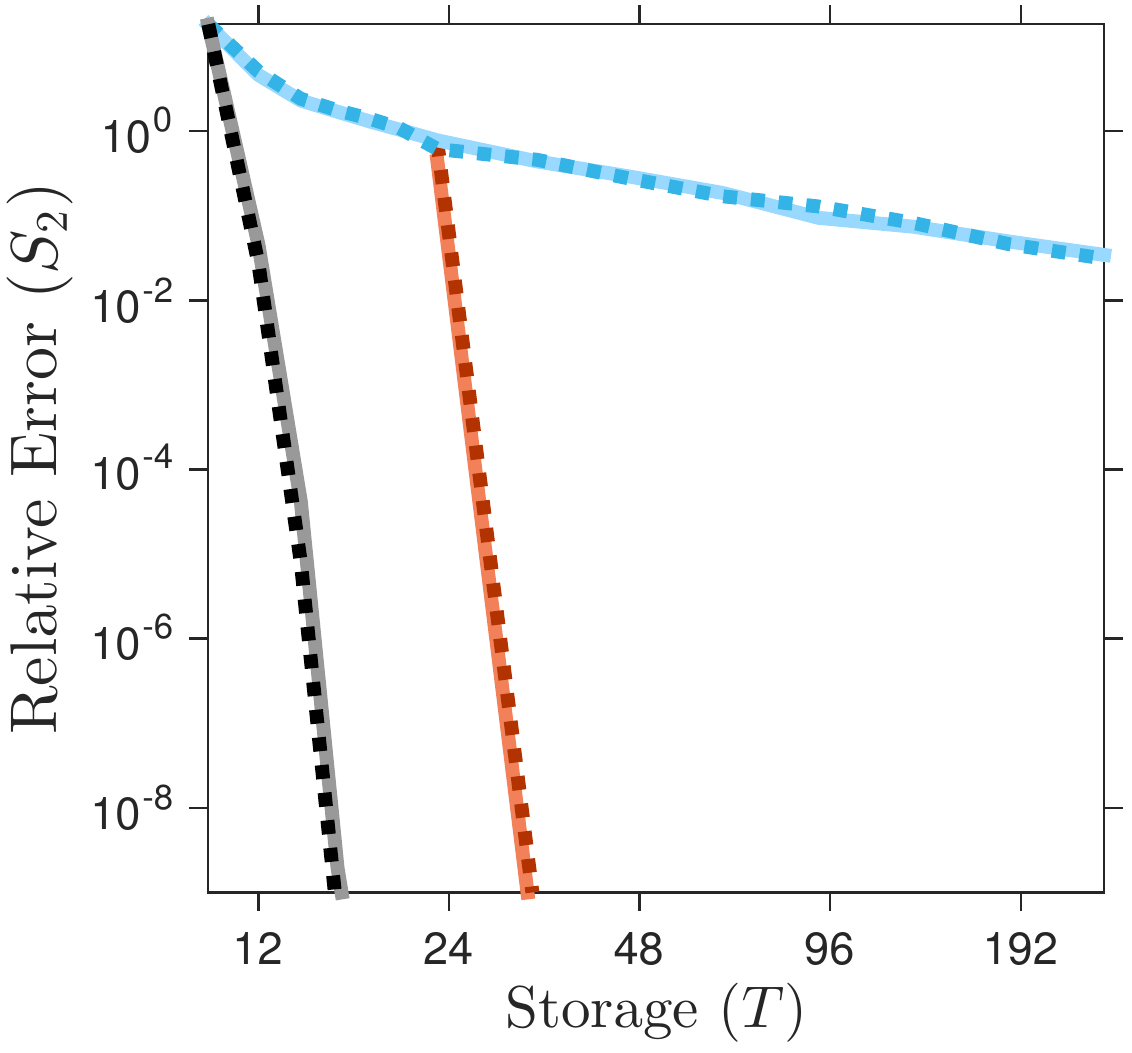}
\caption{\texttt{ExpDecayFast}}
\end{center}
\end{subfigure}
\begin{subfigure}{.325\textwidth}
\begin{center}
\includegraphics[height=1.5in]{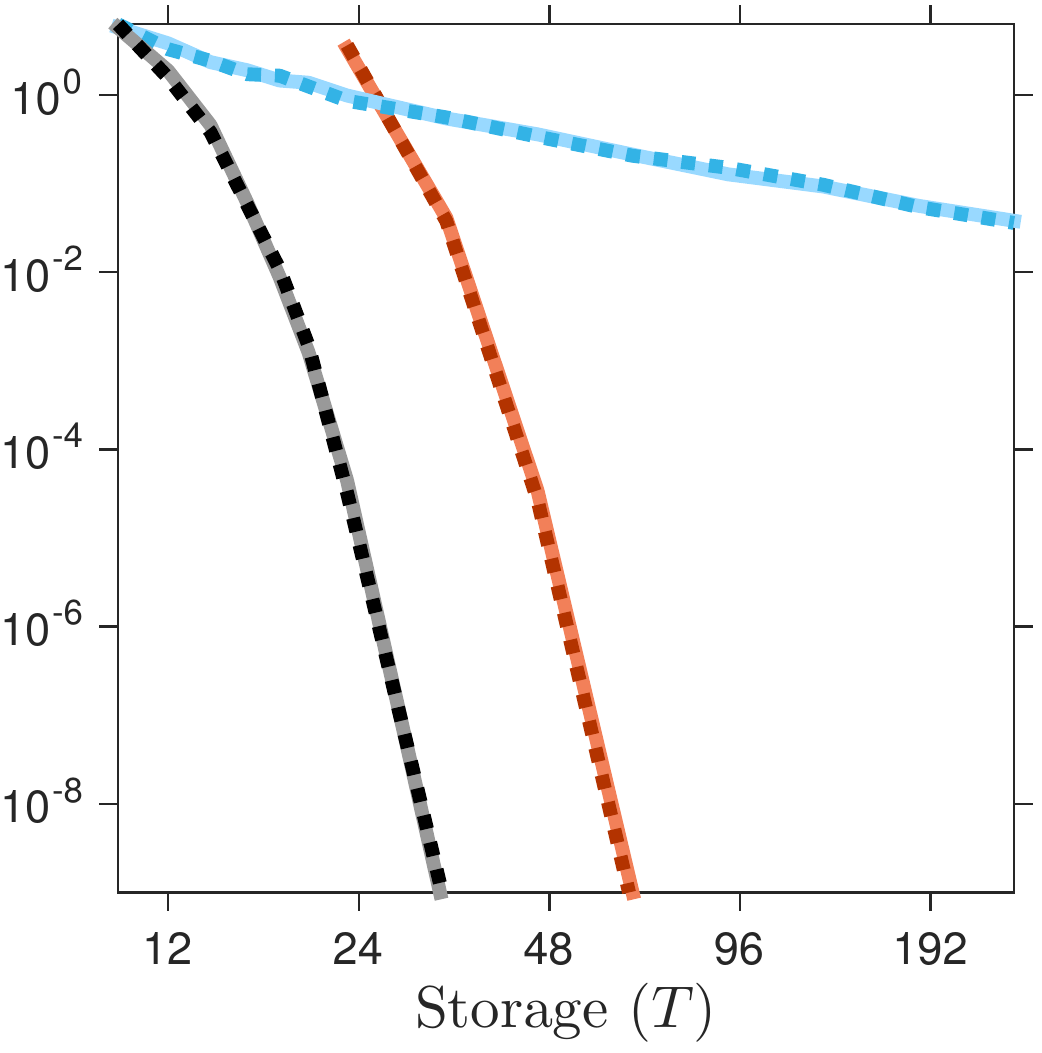}
\caption{\texttt{ExpDecayMed}}
\end{center}
\end{subfigure}
\begin{subfigure}{.325\textwidth}
\begin{center}
\includegraphics[height=1.5in]{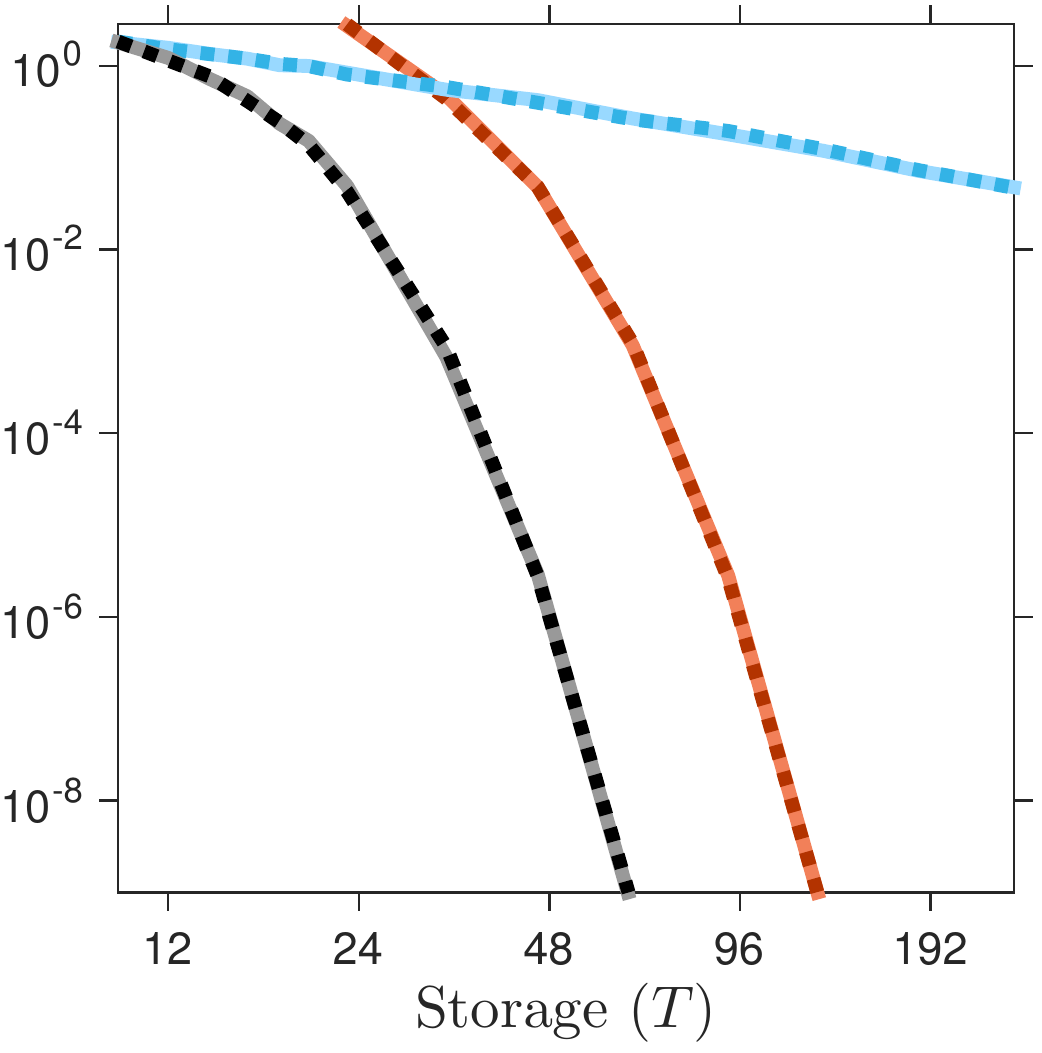}
\caption{\texttt{ExpDecaySlow}}
\end{center}
\end{subfigure}
\end{center}

\caption{\textbf{Synthetic Examples with Effective Rank $R = 5$, Approximation Rank $r = 10$, Schatten $2$-Norm Error.}
The series are
generated by three algorithms for rank-$r$ psd approximation with $r = 10$.
\textbf{Solid lines} are generated from the Gaussian sketch;
\textbf{dashed lines} are from the SSFT sketch.
Each panel displays the  Schatten 2-norm relative error~\eqref{eqn:relative-error}
as a function of storage cost $T$.
See Sec.~\ref{sec:numerics}
for details.}
\label{fig:synthetic-S2-R5}
\end{figure}

\begin{figure}[htp!]
\begin{center}
\begin{subfigure}{.325\textwidth}
\begin{center}
\includegraphics[height=1.5in]{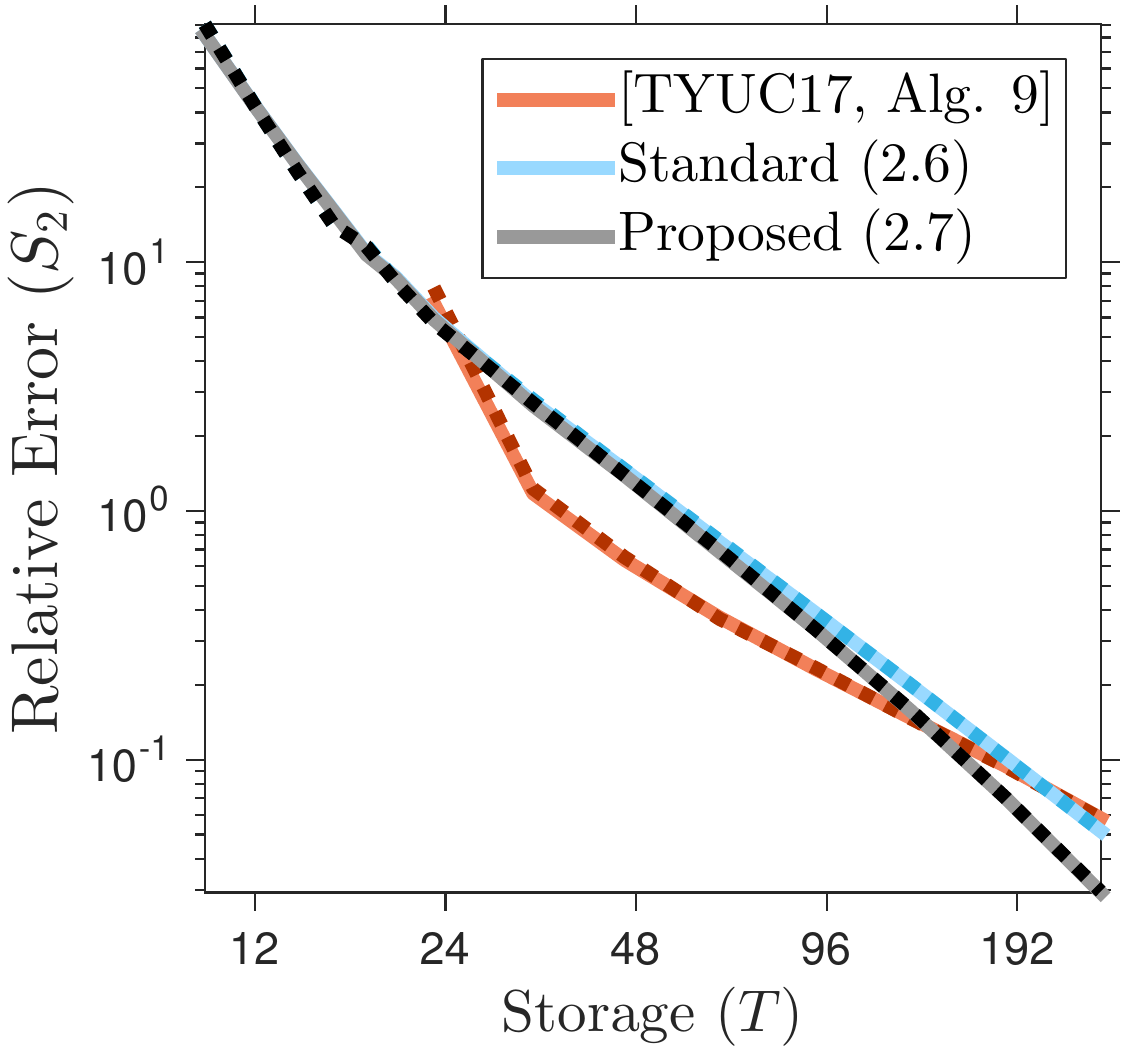}
\caption{\texttt{LowRankLowNoise}}
\end{center}
\end{subfigure}
\begin{subfigure}{.325\textwidth}
\begin{center}
\includegraphics[height=1.5in]{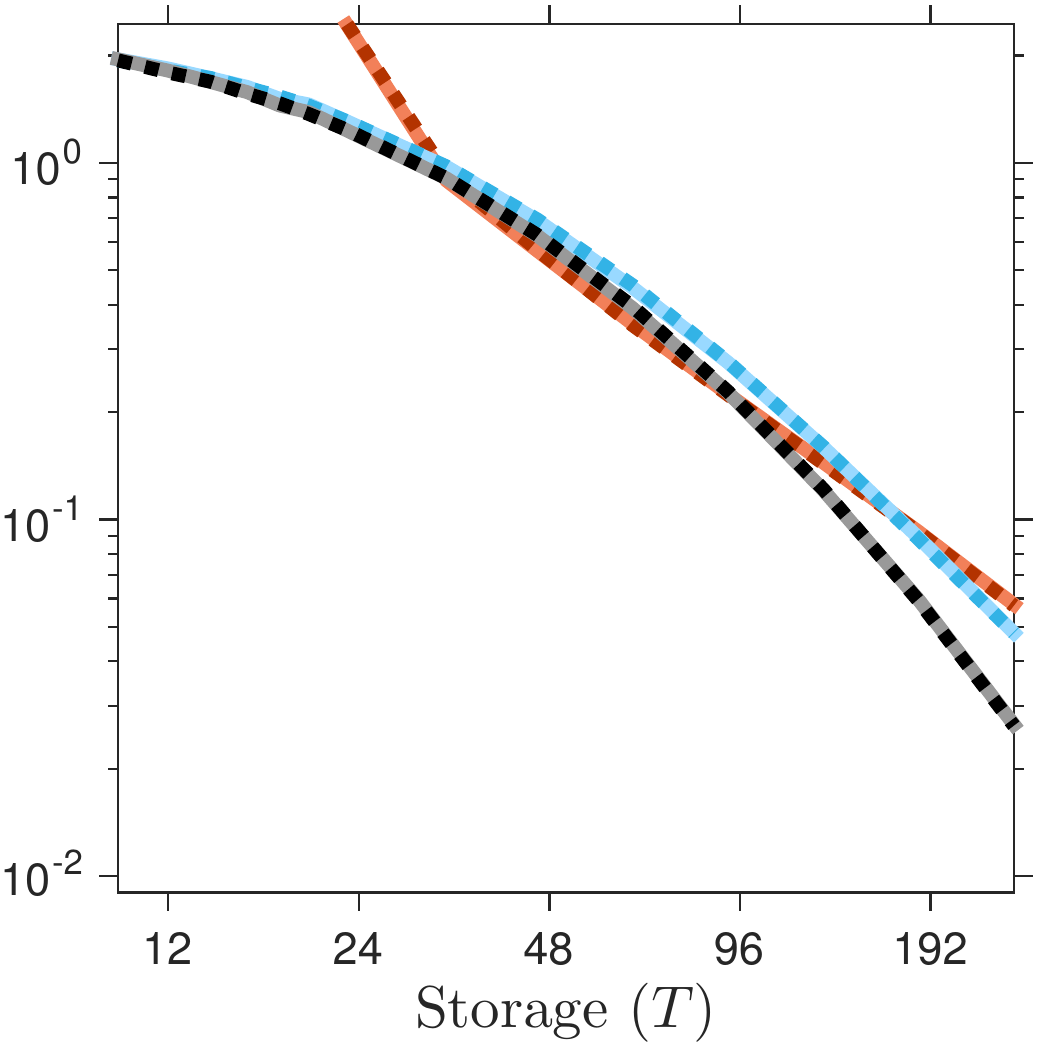}
\caption{\texttt{LowRankMedNoise}}
\end{center}
\end{subfigure}
\begin{subfigure}{.325\textwidth}
\begin{center}
\includegraphics[height=1.5in]{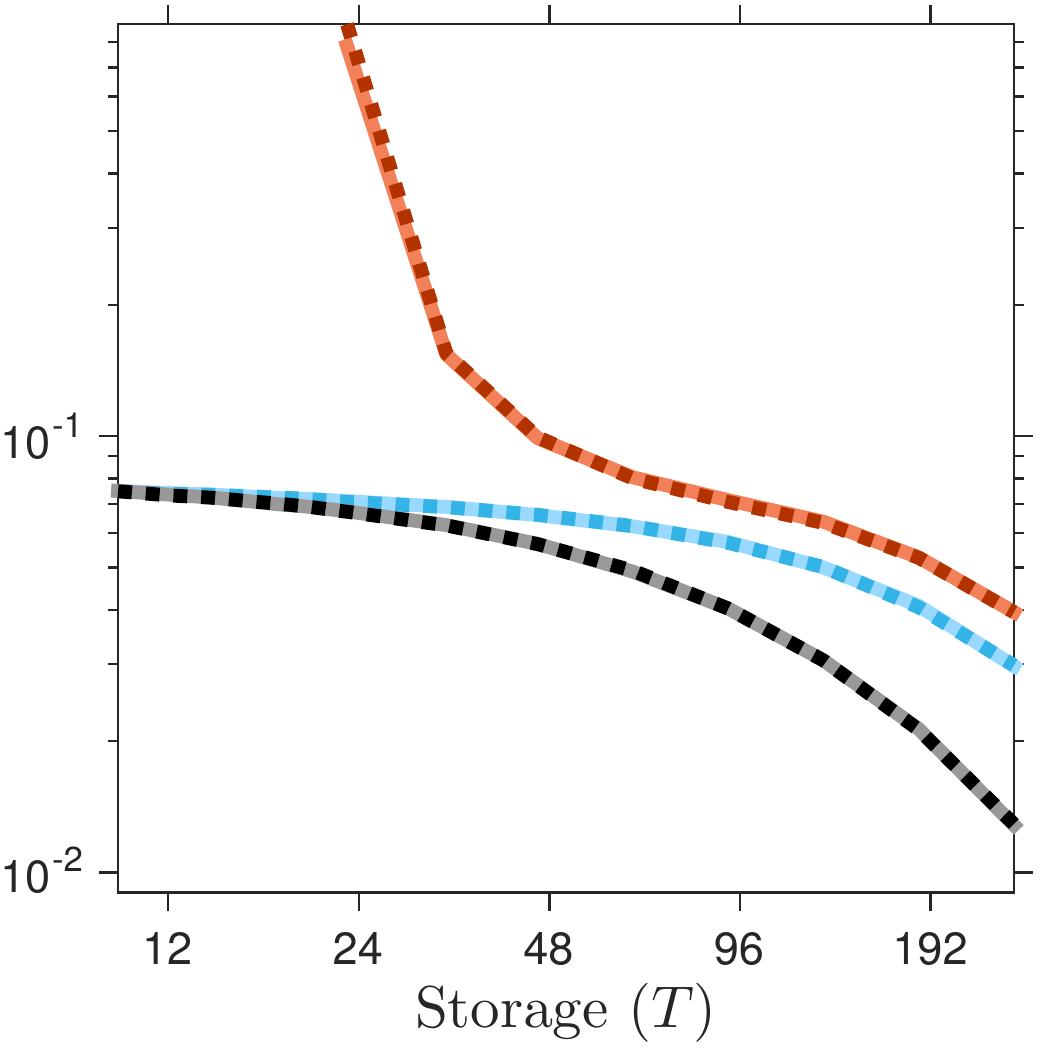}
\caption{\texttt{LowRankHiNoise}}
\end{center}
\end{subfigure}
\end{center}

\vspace{.5em}

\begin{center}
\begin{subfigure}{.325\textwidth}
\begin{center}
\includegraphics[height=1.5in]{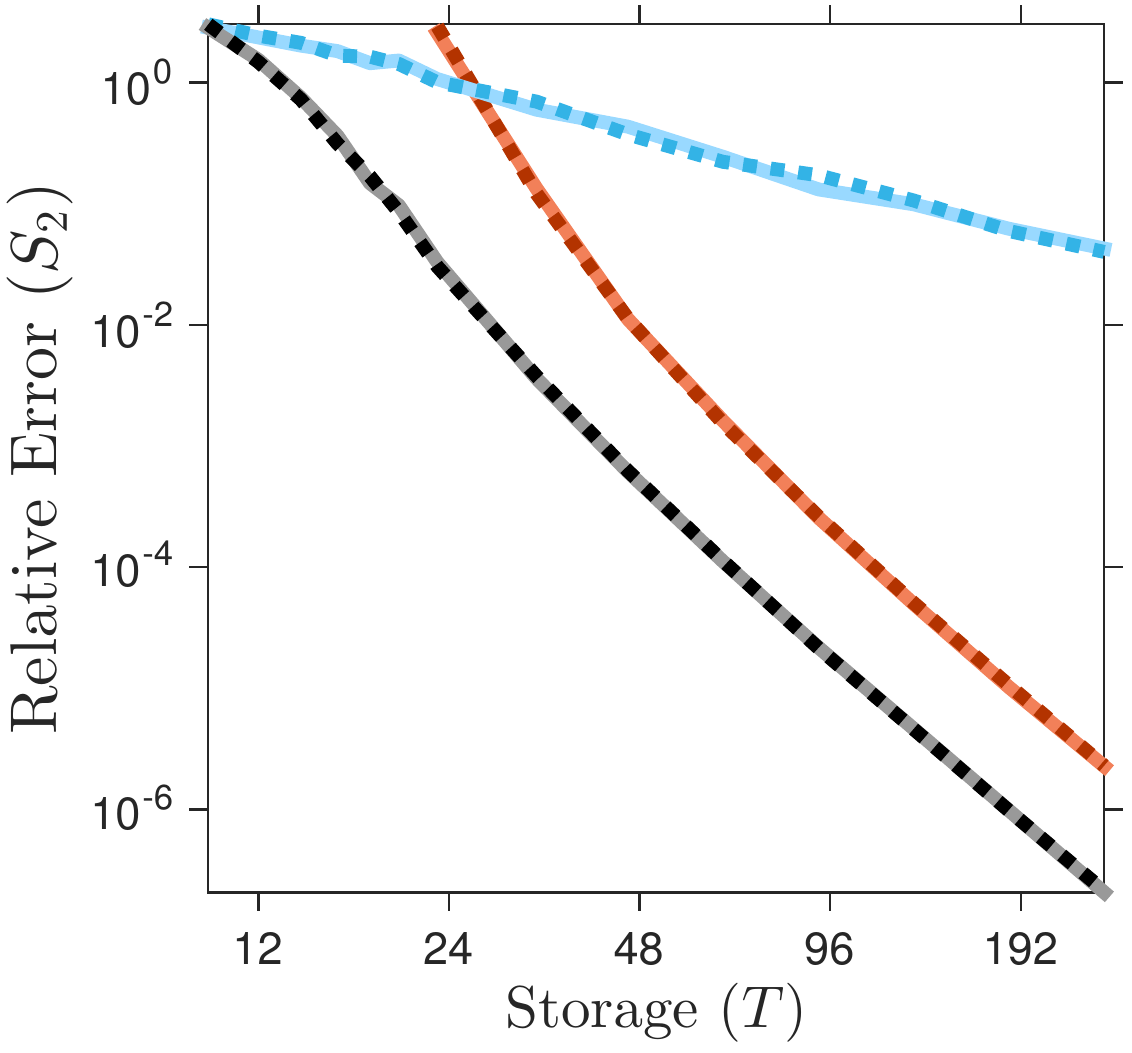}
\caption{\texttt{PolyDecayFast}}
\end{center}
\end{subfigure}
\begin{subfigure}{.325\textwidth}
\begin{center}
\includegraphics[height=1.5in]{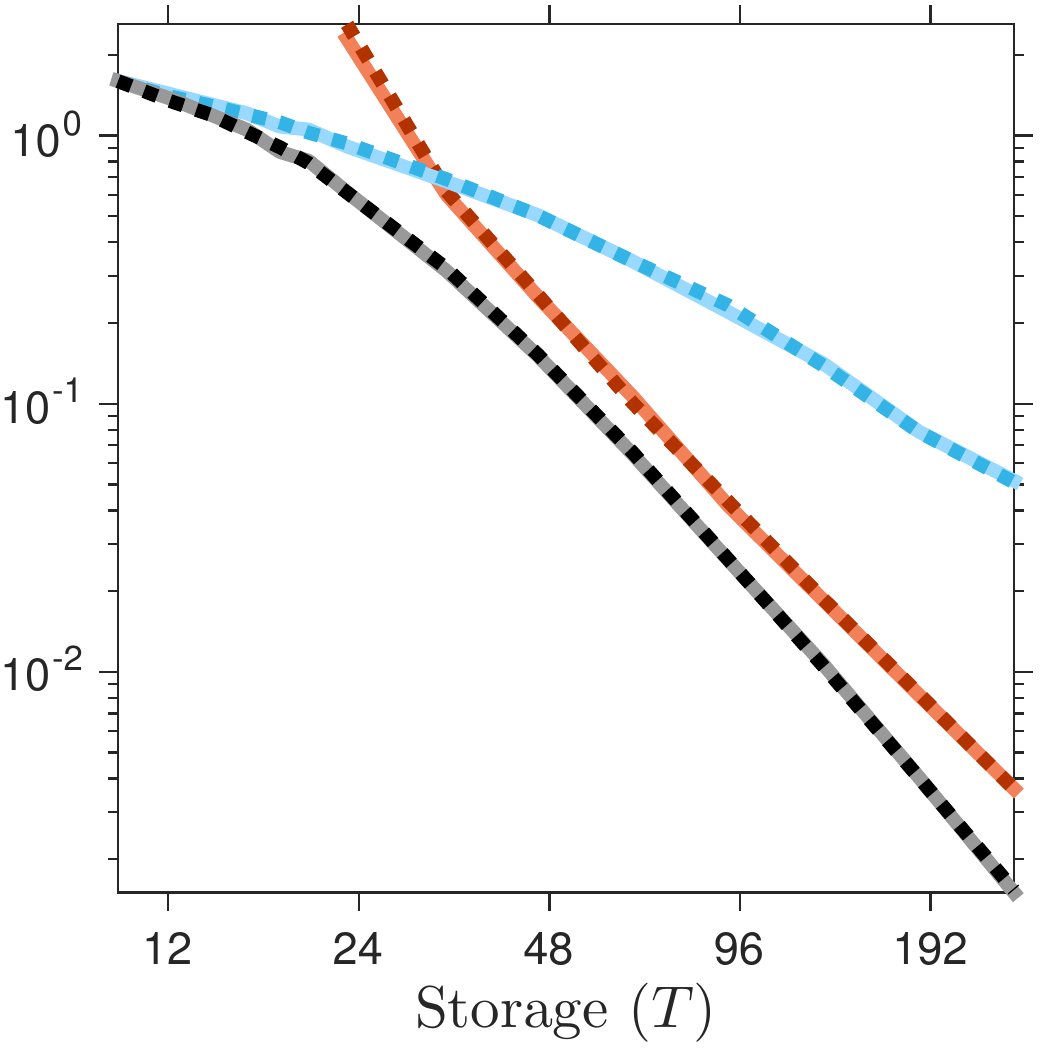}
\caption{\texttt{PolyDecayMed}}
\end{center}
\end{subfigure}
\begin{subfigure}{.325\textwidth}
\begin{center}
\includegraphics[height=1.5in]{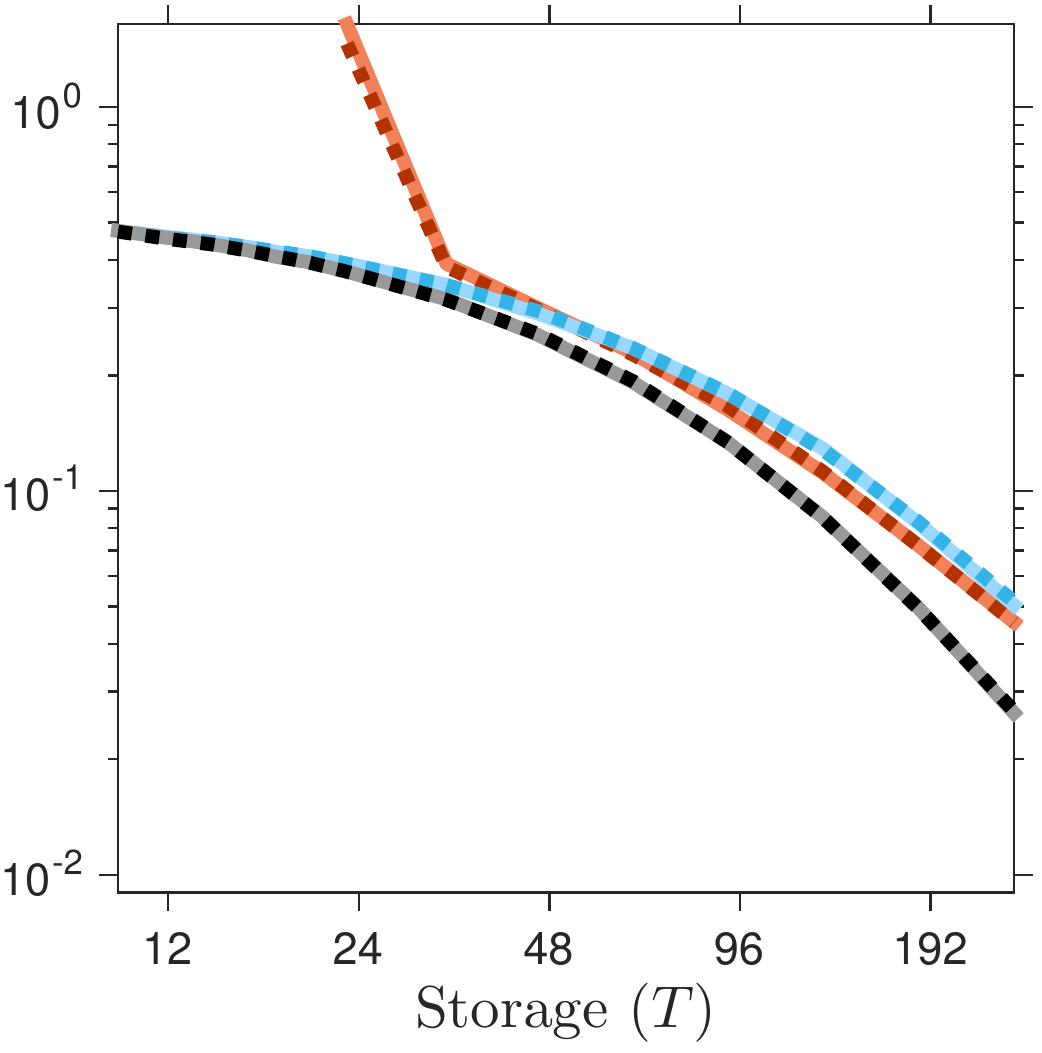}
\caption{\texttt{PolyDecaySlow}}
\end{center}
\end{subfigure}
\end{center}

\vspace{0.5em}

\begin{center}
\begin{subfigure}{.325\textwidth}
\begin{center}
\includegraphics[height=1.5in]{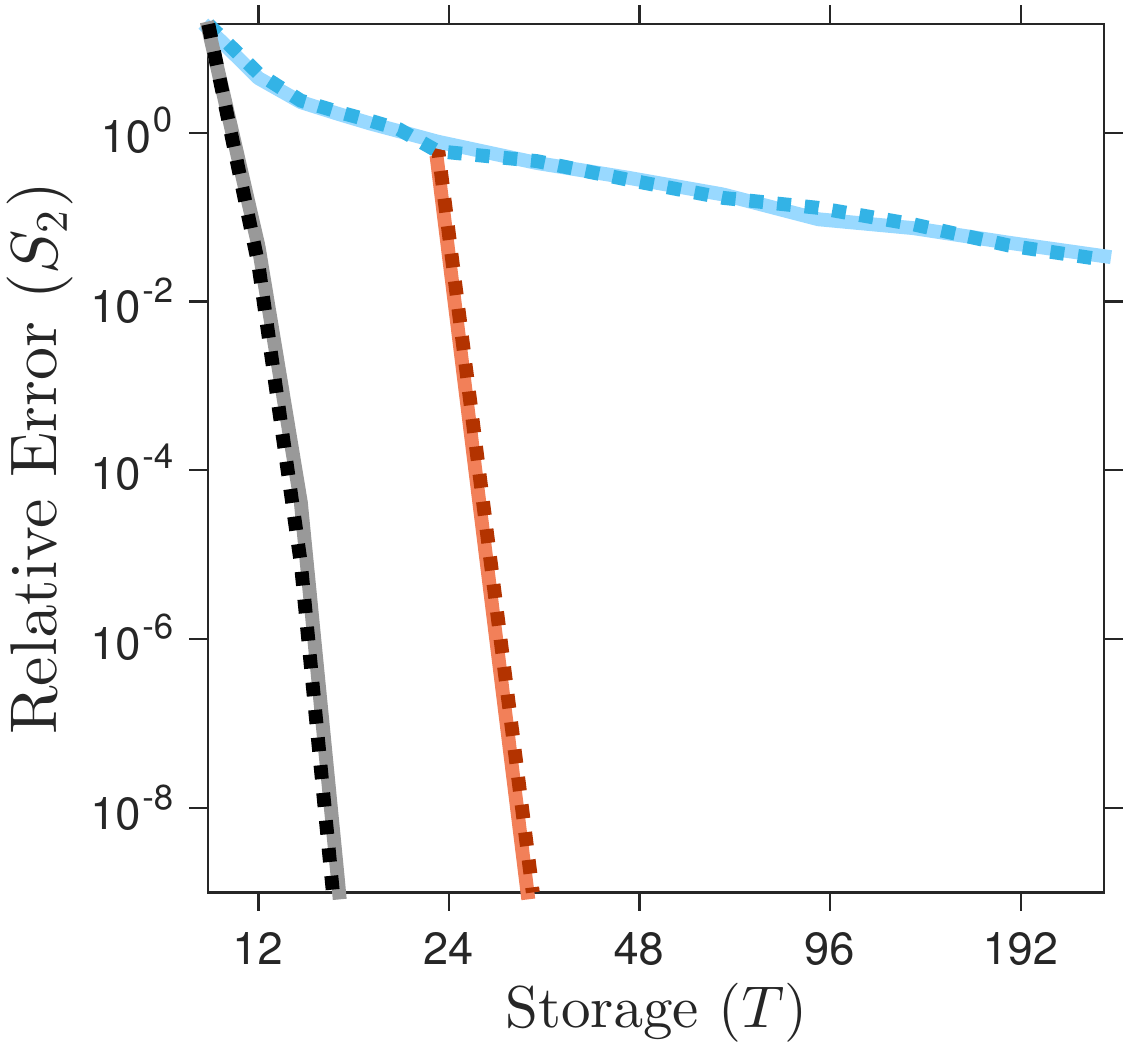}
\caption{\texttt{ExpDecayFast}}
\end{center}
\end{subfigure}
\begin{subfigure}{.325\textwidth}
\begin{center}
\includegraphics[height=1.5in]{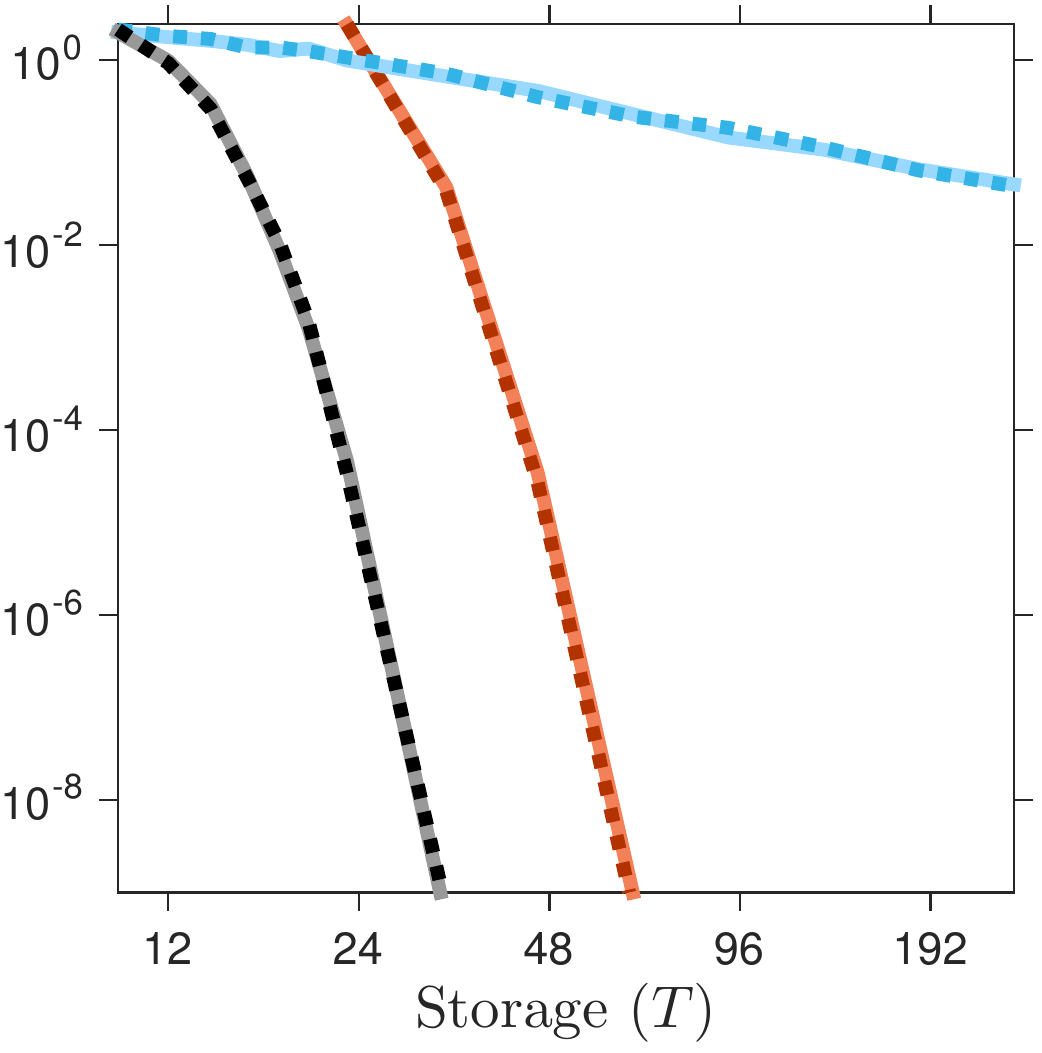}
\caption{\texttt{ExpDecayMed}}
\end{center}
\end{subfigure}
\begin{subfigure}{.325\textwidth}
\begin{center}
\includegraphics[height=1.5in]{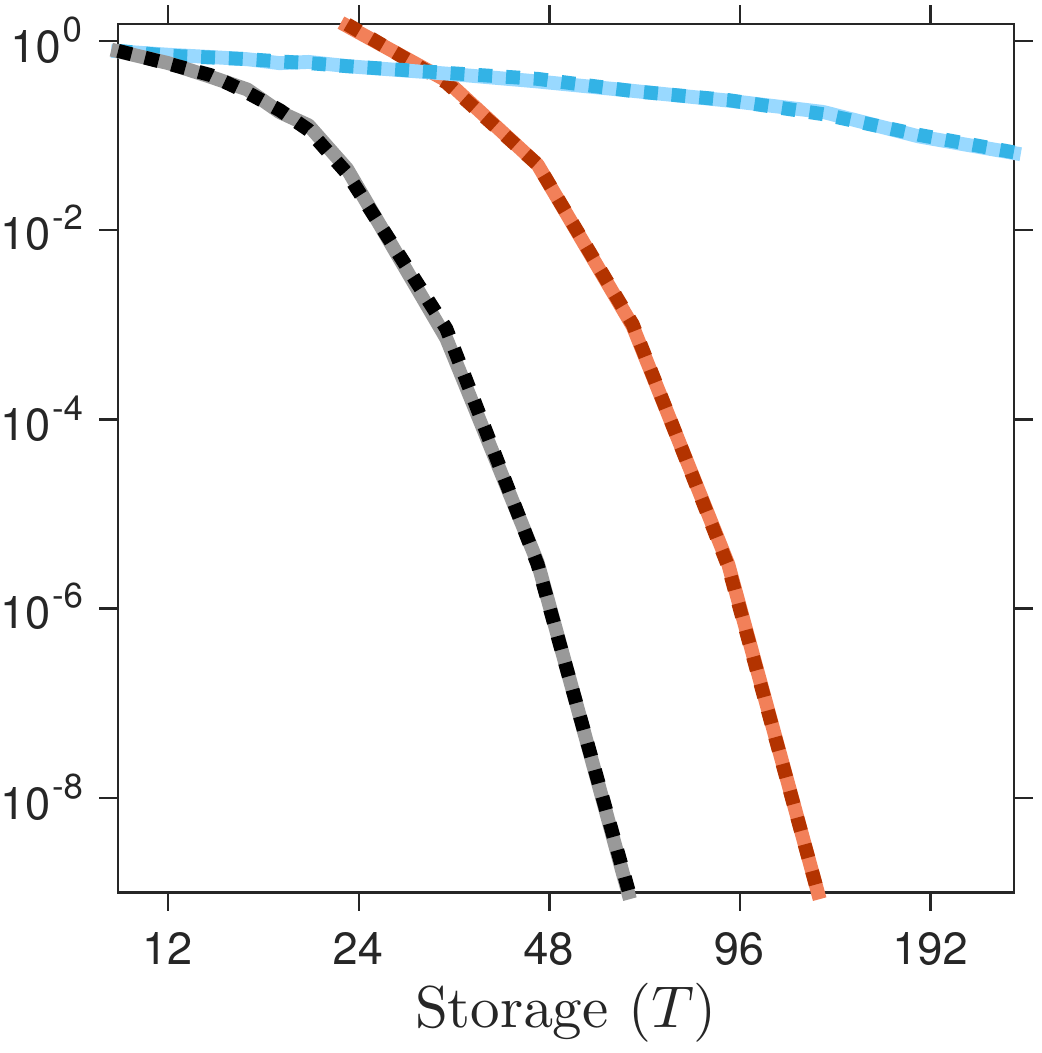}
\caption{\texttt{ExpDecaySlow}}
\end{center}
\end{subfigure}
\end{center}

\caption{\textbf{Synthetic Examples with Effective Rank $R = 10$, Approximation Rank $r = 10$, Schatten $2$-Norm Error.}
The series are
generated by three algorithms for rank-$r$ psd approximation with $r = 10$.
\textbf{Solid lines} are generated from the Gaussian sketch;
\textbf{dashed lines} are from the SSFT sketch.
Each panel displays the  Schatten 2-norm relative error~\eqref{eqn:relative-error}
as a function of storage cost $T$.  See Sec.~\ref{sec:numerics}
for details.}
\label{fig:synthetic-S2-R10}
\end{figure}

\begin{figure}[htp!]
\begin{center}
\begin{subfigure}{.325\textwidth}
\begin{center}
\includegraphics[height=1.5in]{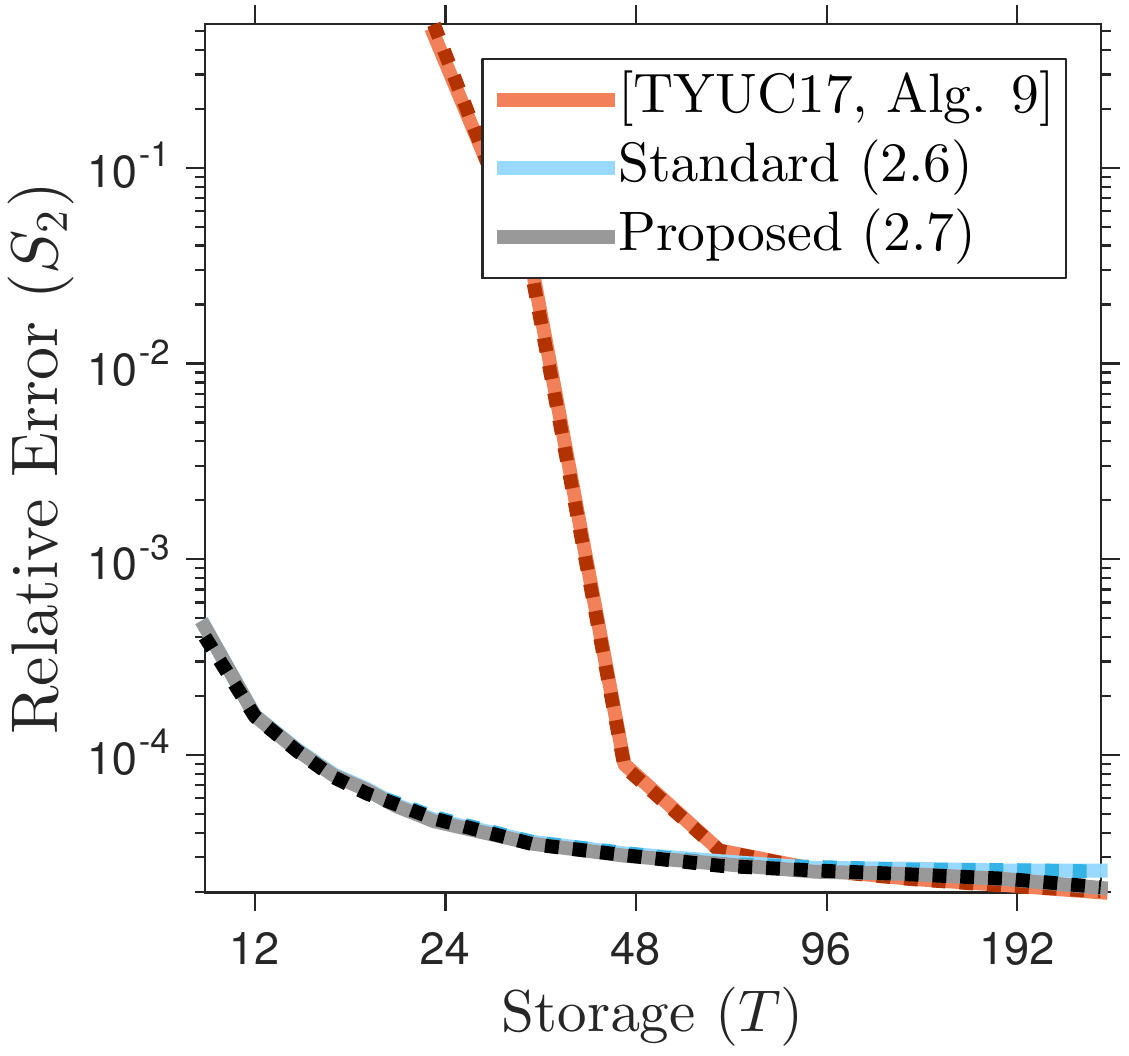}
\caption{\texttt{LowRankLowNoise}}
\end{center}
\end{subfigure}
\begin{subfigure}{.325\textwidth}
\begin{center}
\includegraphics[height=1.5in]{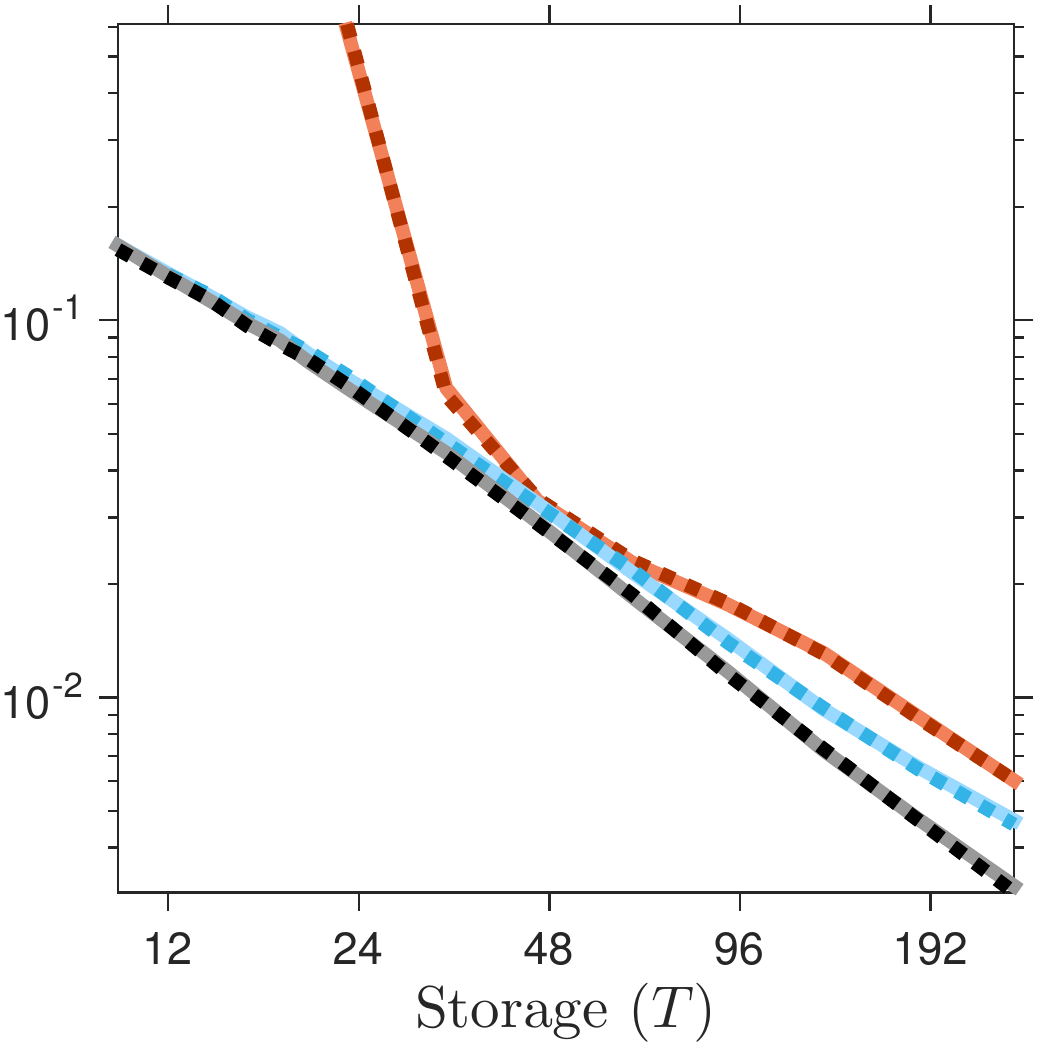}
\caption{\texttt{LowRankMedNoise}}
\end{center}
\end{subfigure}
\begin{subfigure}{.325\textwidth}
\begin{center}
\includegraphics[height=1.5in]{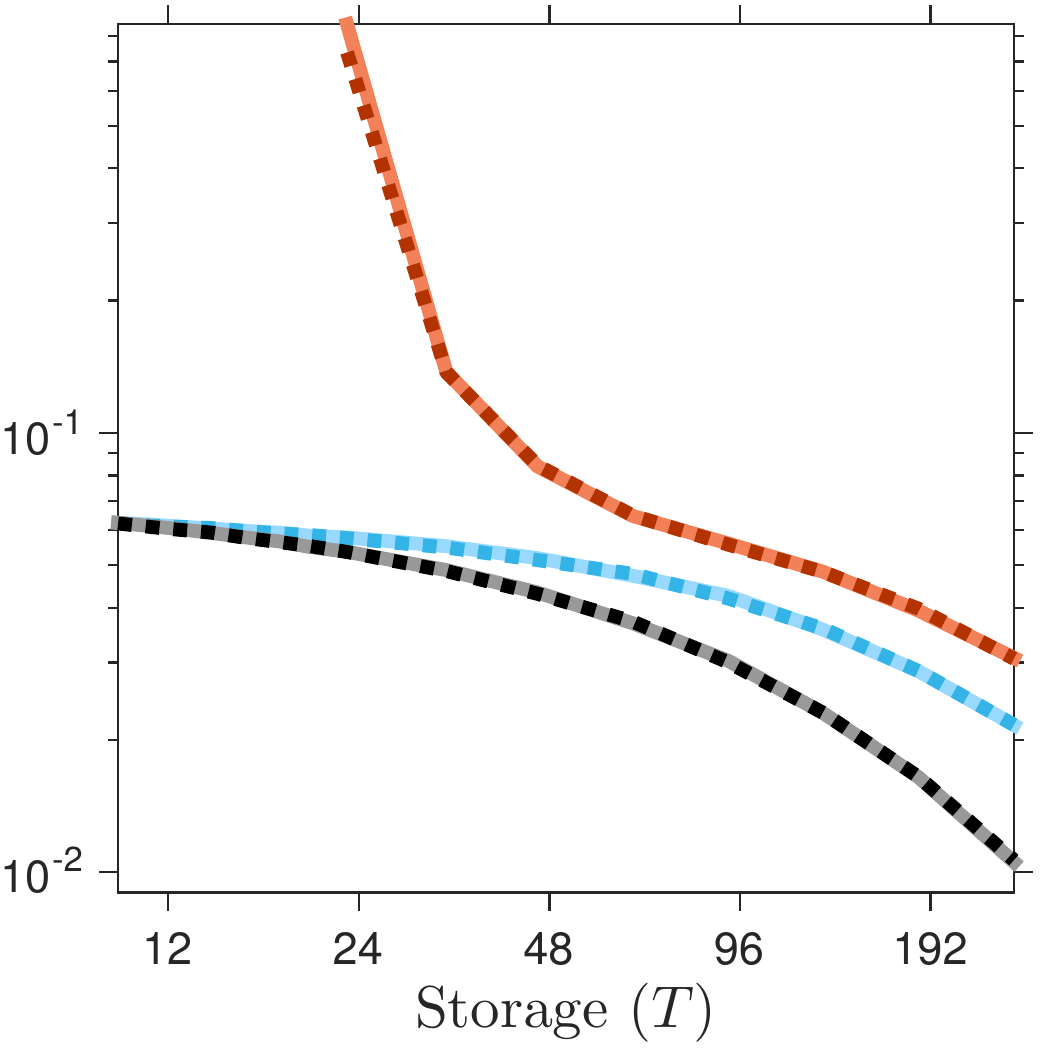}
\caption{\texttt{LowRankHiNoise}}
\end{center}
\end{subfigure}
\end{center}

\vspace{.5em}

\begin{center}
\begin{subfigure}{.325\textwidth}
\begin{center}
\includegraphics[height=1.5in]{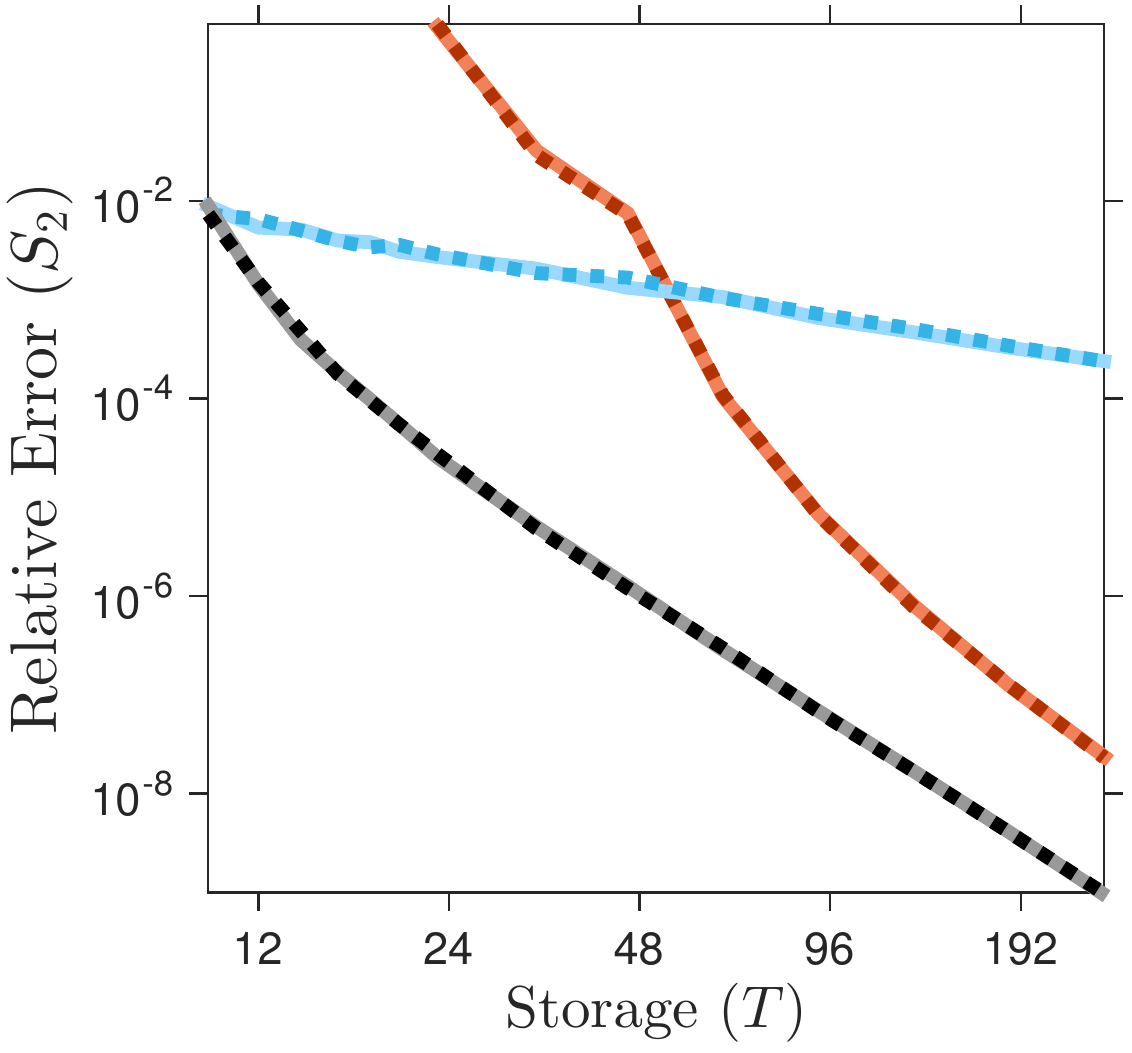}
\caption{\texttt{PolyDecayFast}}
\end{center}
\end{subfigure}
\begin{subfigure}{.325\textwidth}
\begin{center}
\includegraphics[height=1.5in]{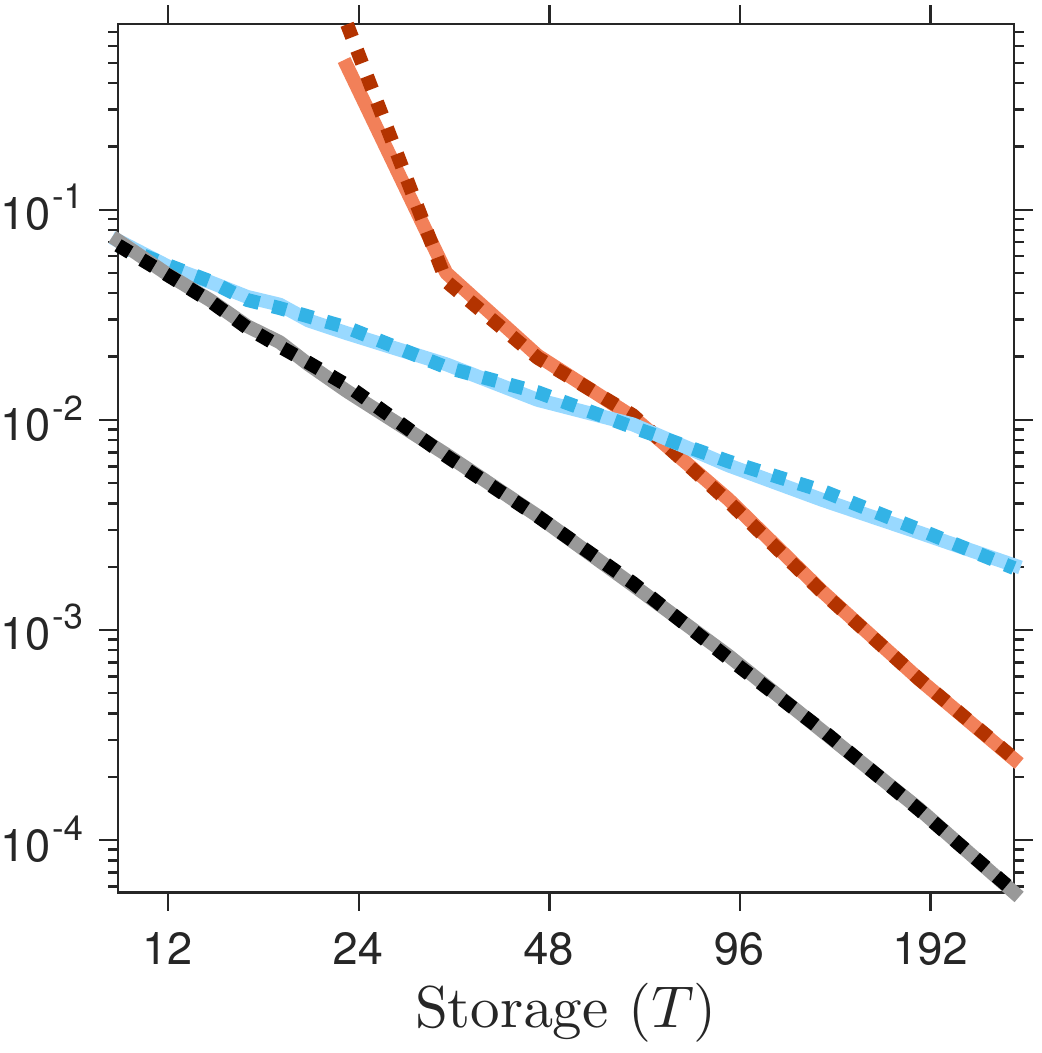}
\caption{\texttt{PolyDecayMed}}
\end{center}
\end{subfigure}
\begin{subfigure}{.325\textwidth}
\begin{center}
\includegraphics[height=1.5in]{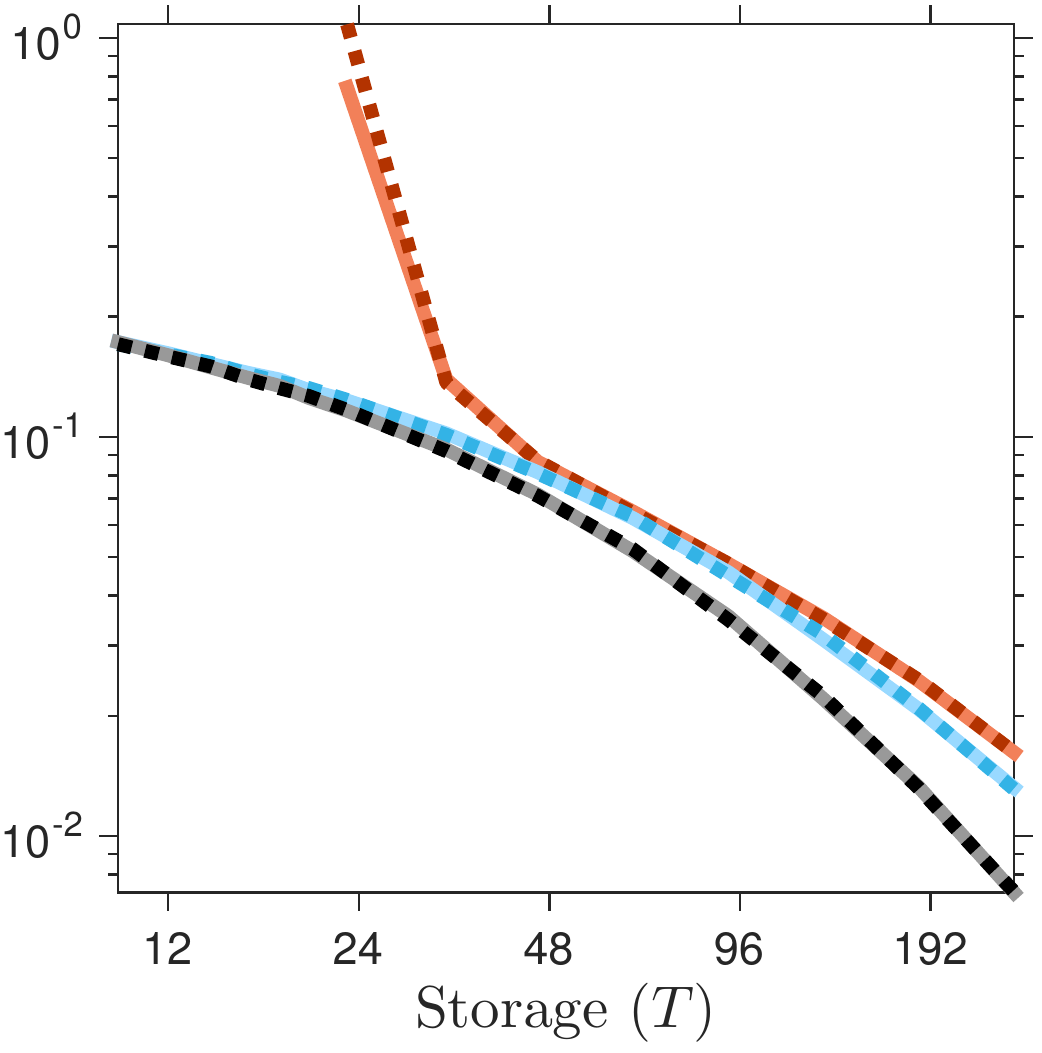}
\caption{\texttt{PolyDecaySlow}}
\end{center}
\end{subfigure}
\end{center}

\vspace{0.5em}

\begin{center}
\begin{subfigure}{.325\textwidth}
\begin{center}
\includegraphics[height=1.5in]{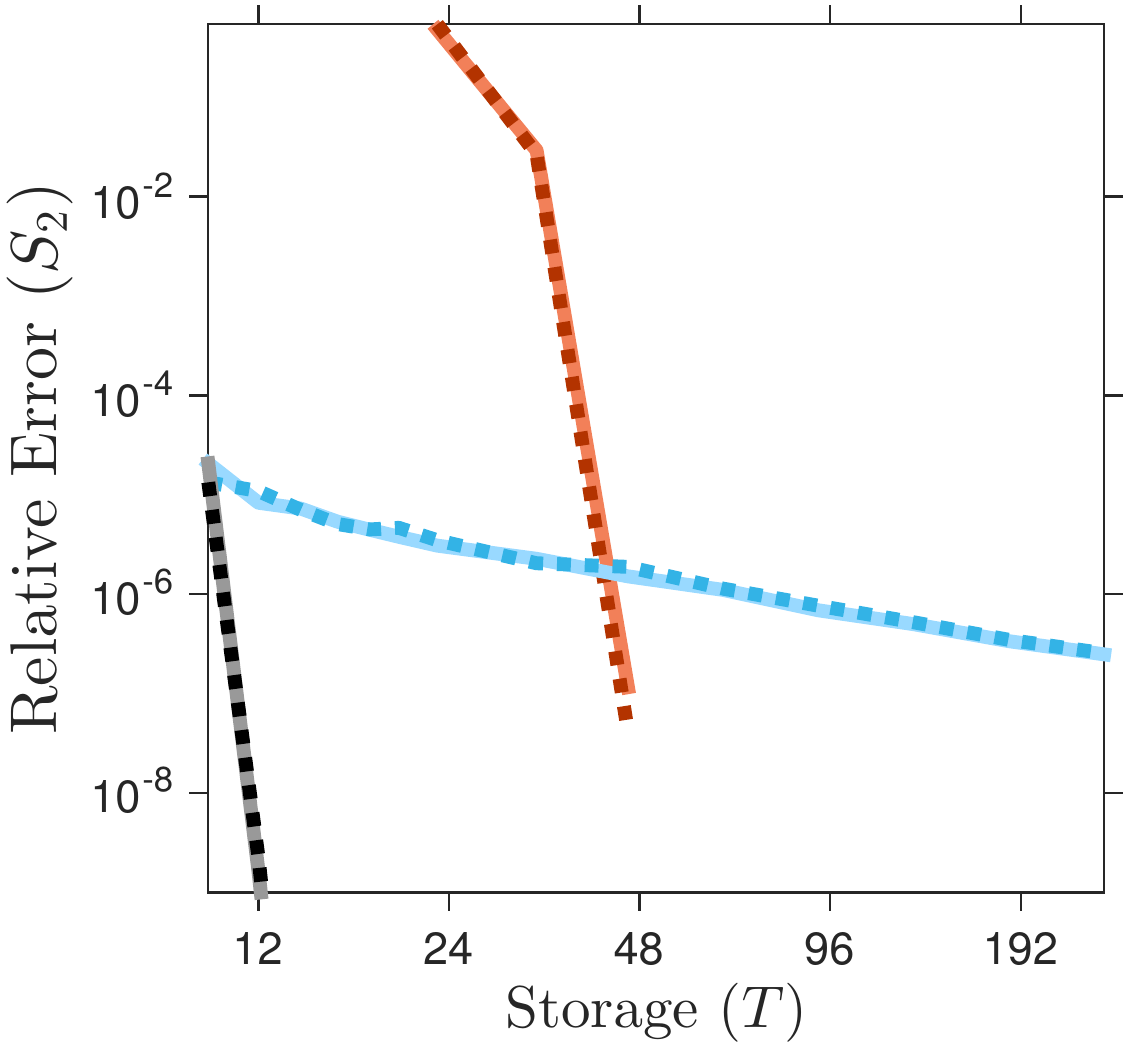}
\caption{\texttt{ExpDecayFast}}
\end{center}
\end{subfigure}
\begin{subfigure}{.325\textwidth}
\begin{center}
\includegraphics[height=1.5in]{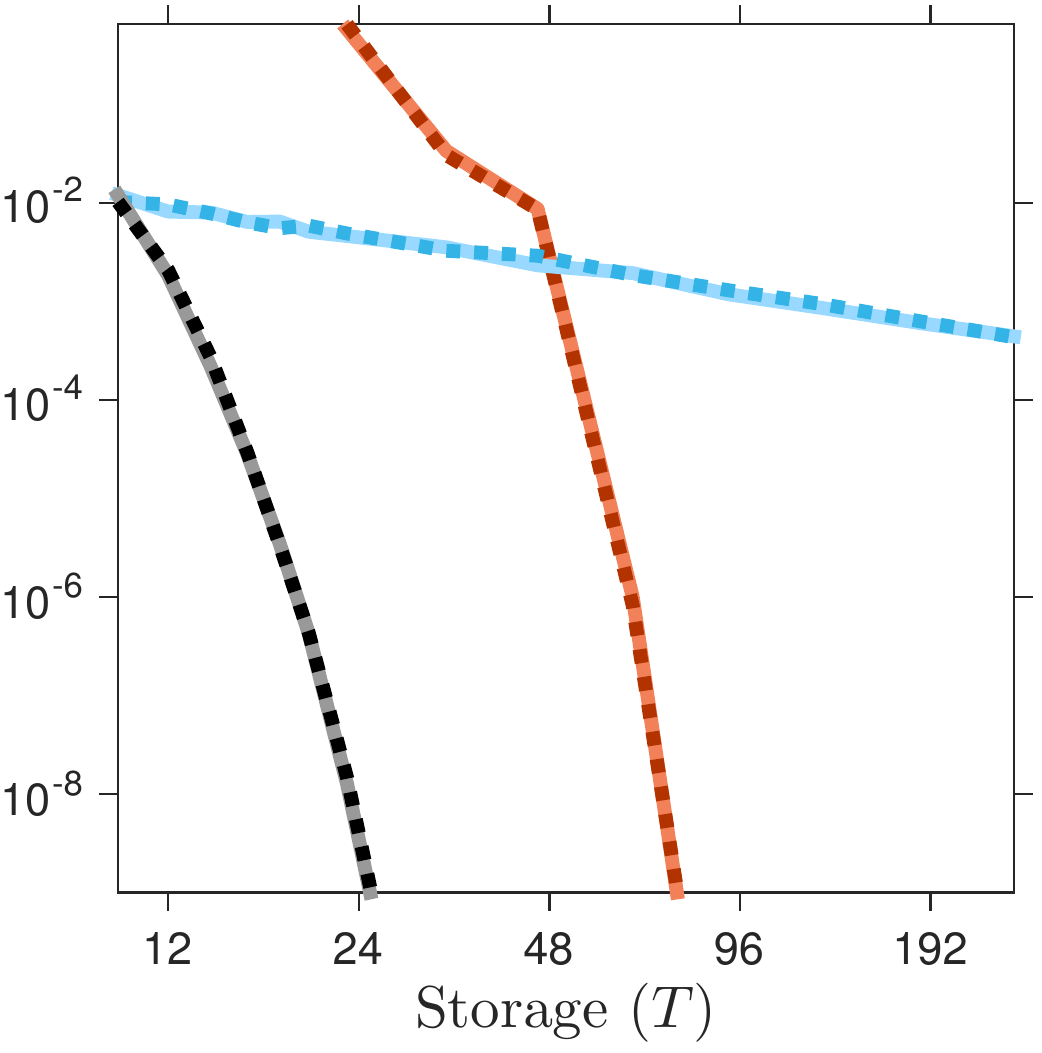}
\caption{\texttt{ExpDecayMed}}
\end{center}
\end{subfigure}
\begin{subfigure}{.325\textwidth}
\begin{center}
\includegraphics[height=1.5in]{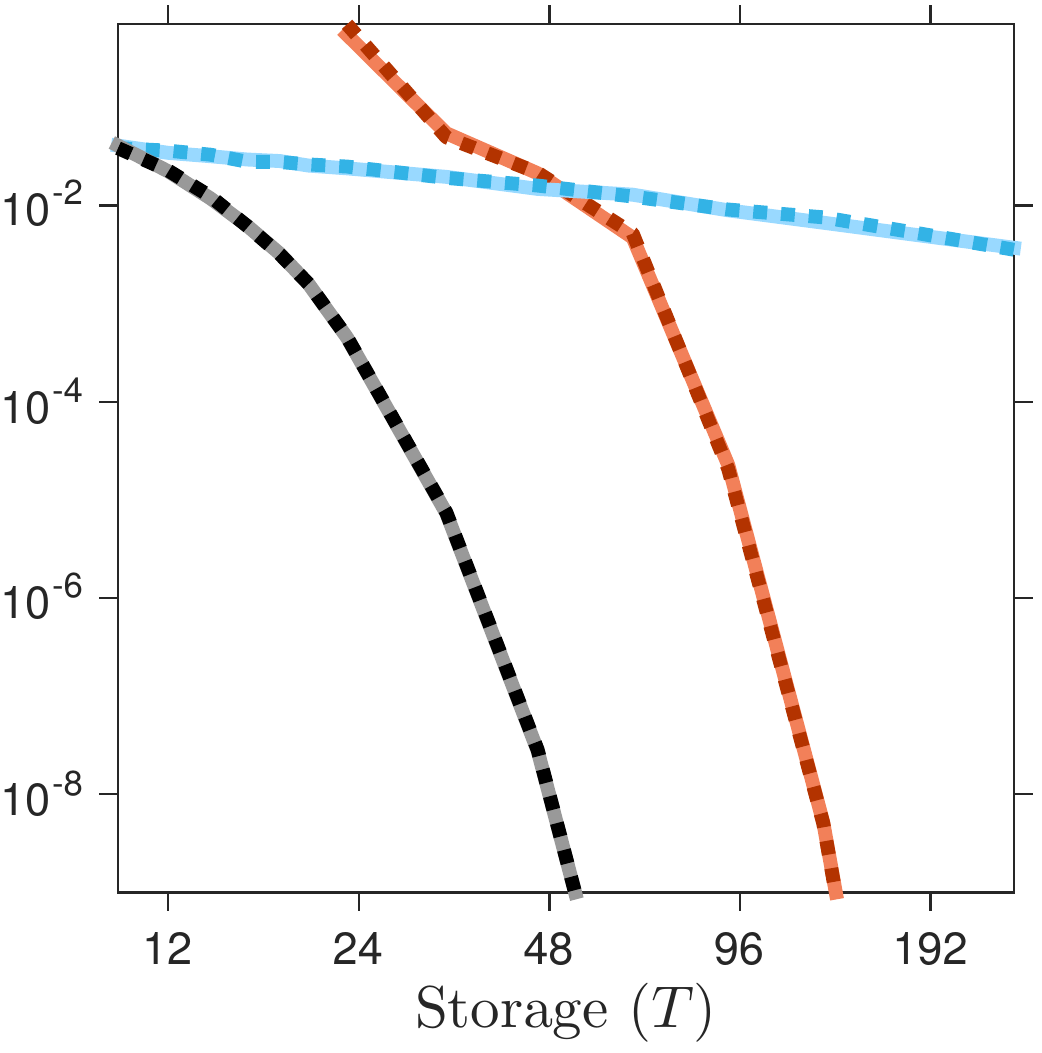}
\caption{\texttt{ExpDecaySlow}}
\end{center}
\end{subfigure}
\end{center}

\caption{\textbf{Synthetic Examples with Effective Rank $R = 20$, Approximation Rank $r = 10$, Schatten $2$-Norm Error.}
The series are
generated by three algorithms for rank-$r$ psd approximation with $r = 10$.
\textbf{Solid lines} are generated from the Gaussian sketch;
\textbf{dashed lines} are from the SSFT sketch.
Each panel displays the Schatten 2-norm relative error~\eqref{eqn:relative-error}
as a function of storage cost $T$.  See Sec.~\ref{sec:numerics}
for details.}
\label{fig:synthetic-S2-R20}
\end{figure}

\begin{figure}[htp!]
\begin{center}
\begin{subfigure}{.325\textwidth}
\begin{center}
\includegraphics[height=1.5in]{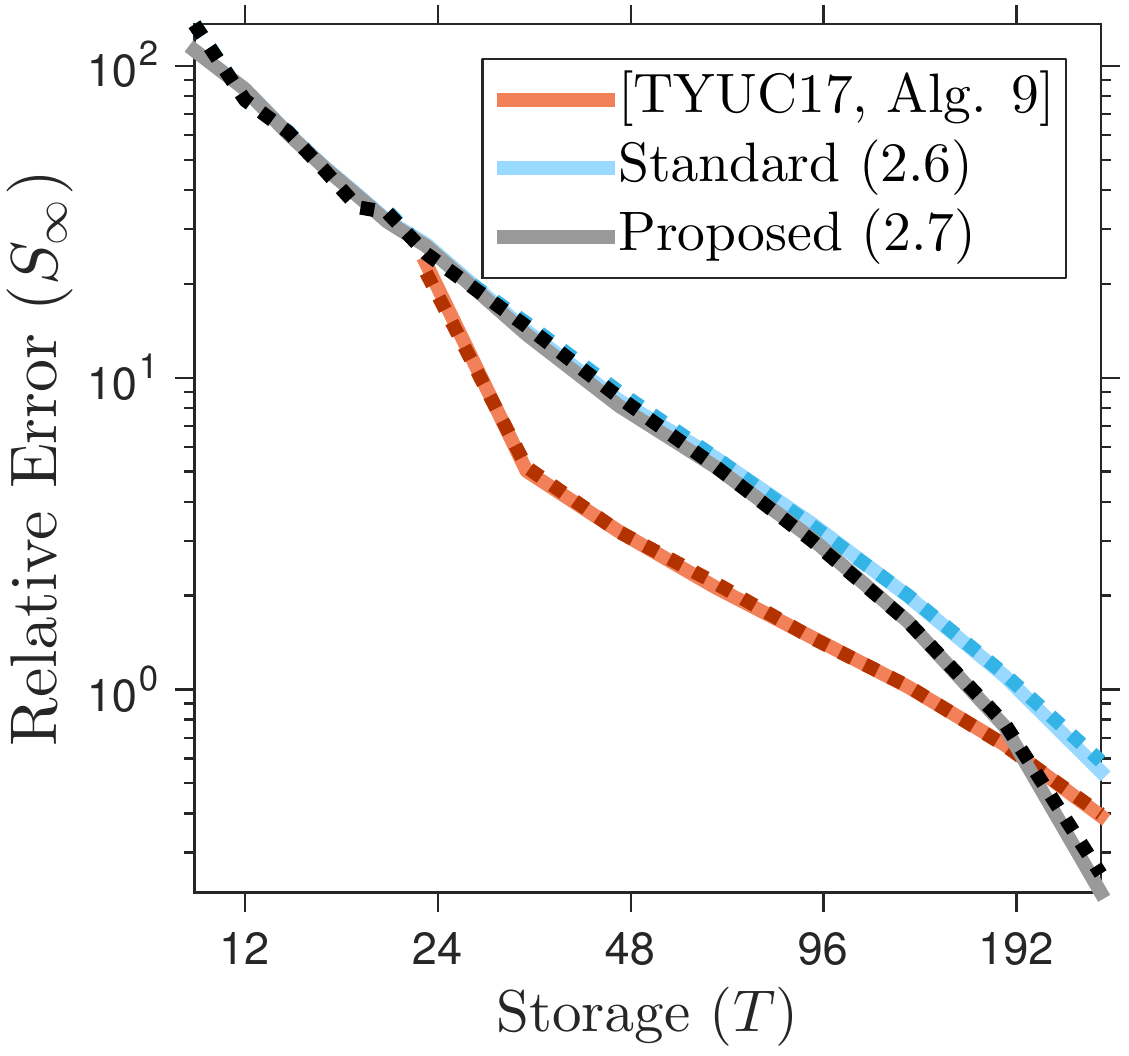}
\caption{\texttt{LowRankLowNoise}}
\end{center}
\end{subfigure}
\begin{subfigure}{.325\textwidth}
\begin{center}
\includegraphics[height=1.5in]{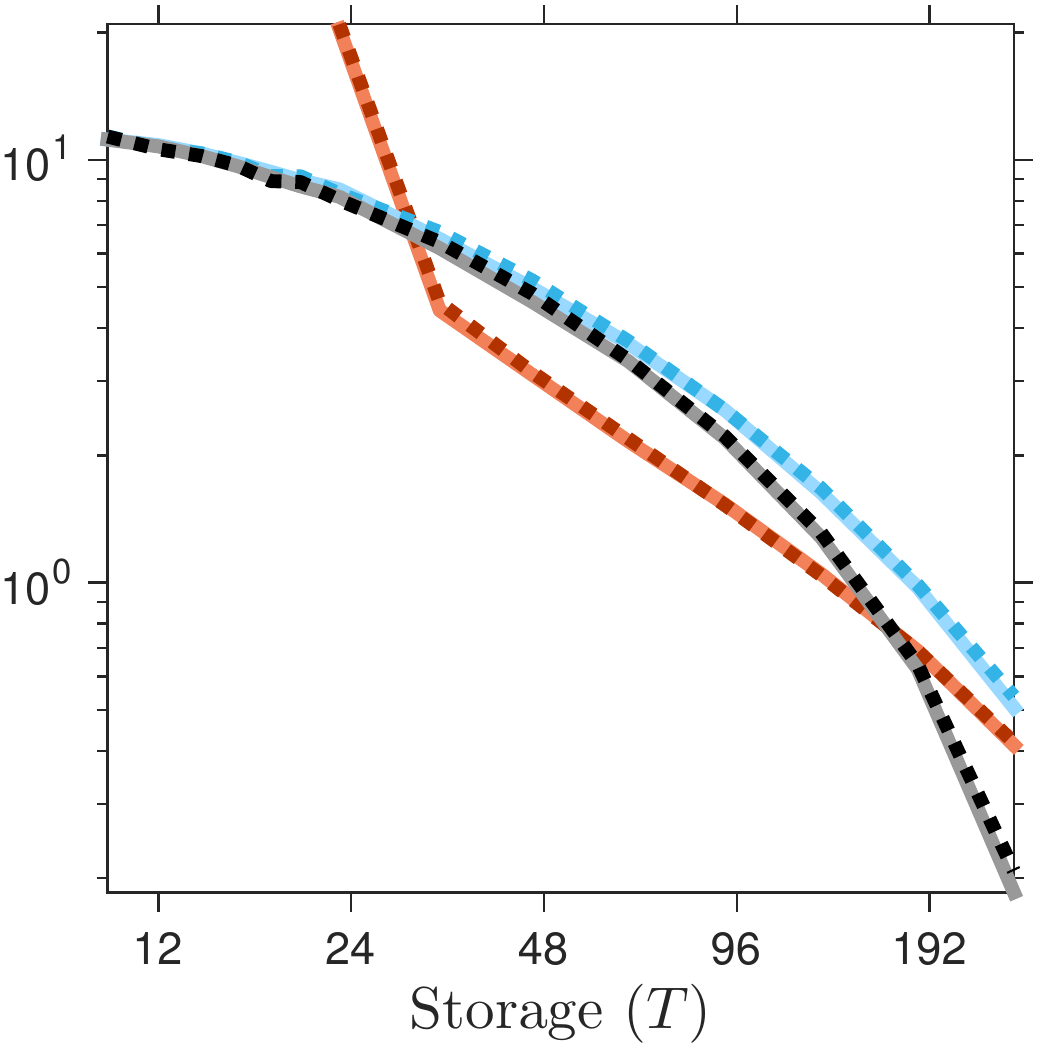}
\caption{\texttt{LowRankMedNoise}}
\end{center}
\end{subfigure}
\begin{subfigure}{.325\textwidth}
\begin{center}
\includegraphics[height=1.5in]{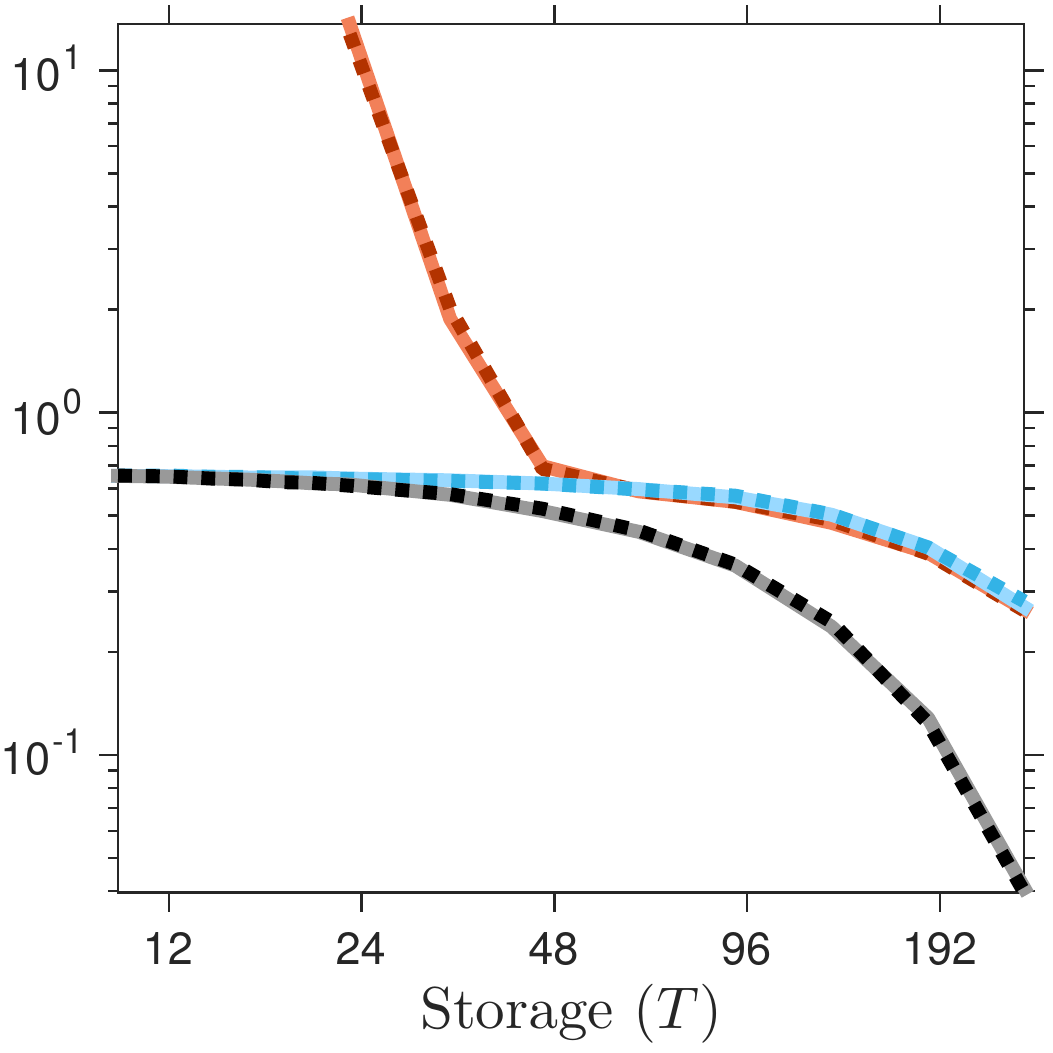}
\caption{\texttt{LowRankHiNoise}}
\end{center}
\end{subfigure}
\end{center}

\vspace{.5em}

\begin{center}
\begin{subfigure}{.325\textwidth}
\begin{center}
\includegraphics[height=1.5in]{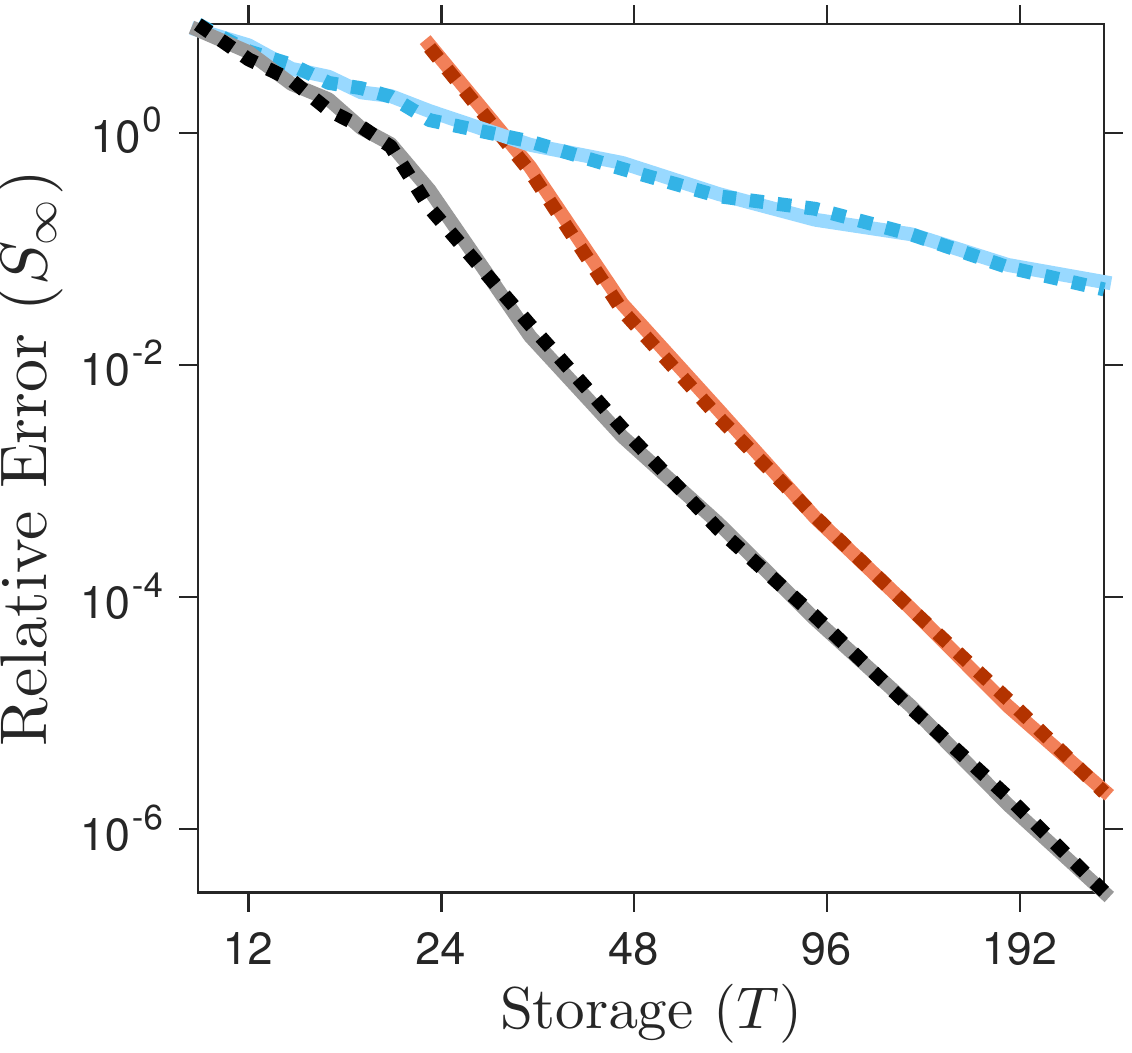}
\caption{\texttt{PolyDecayFast}}
\end{center}
\end{subfigure}
\begin{subfigure}{.325\textwidth}
\begin{center}
\includegraphics[height=1.5in]{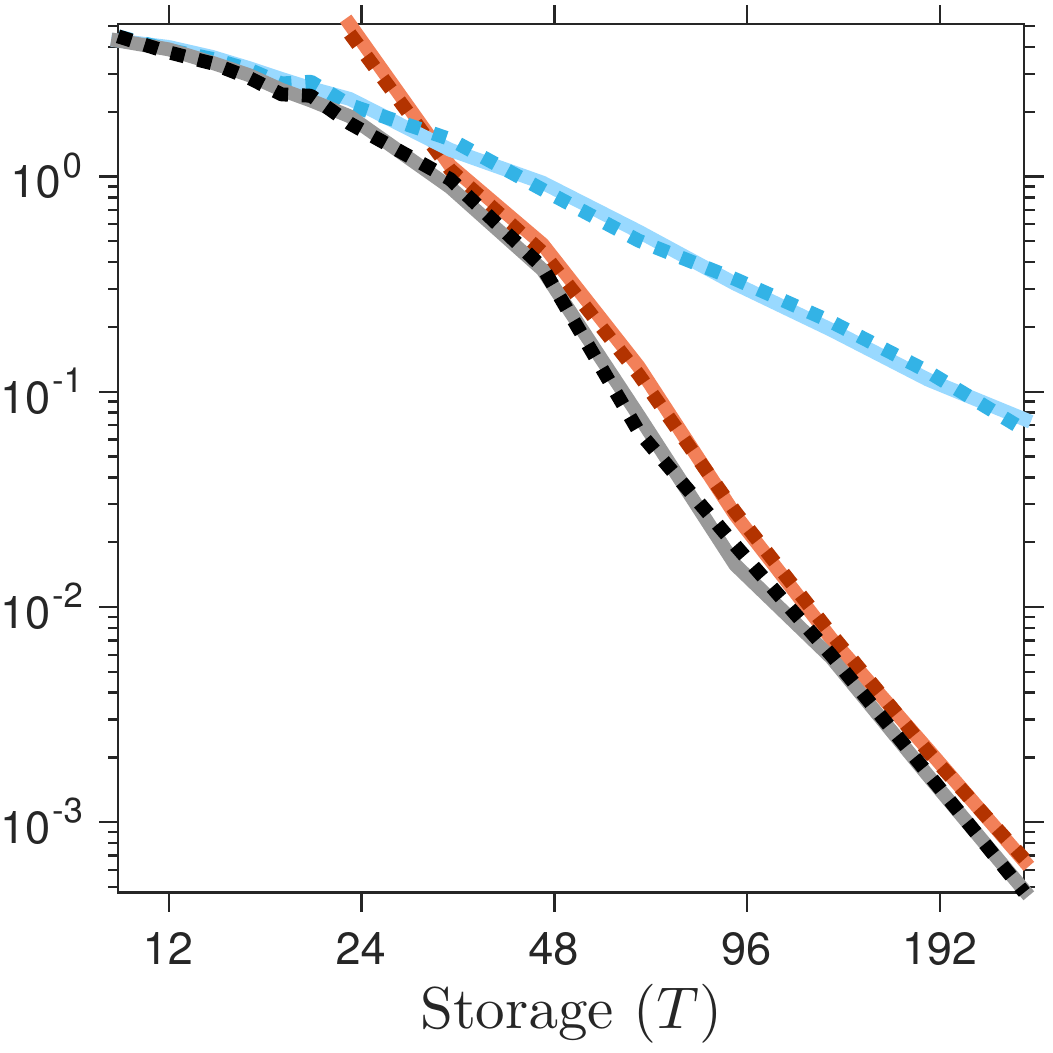}
\caption{\texttt{PolyDecayMed}}
\end{center}
\end{subfigure}
\begin{subfigure}{.325\textwidth}
\begin{center}
\includegraphics[height=1.5in]{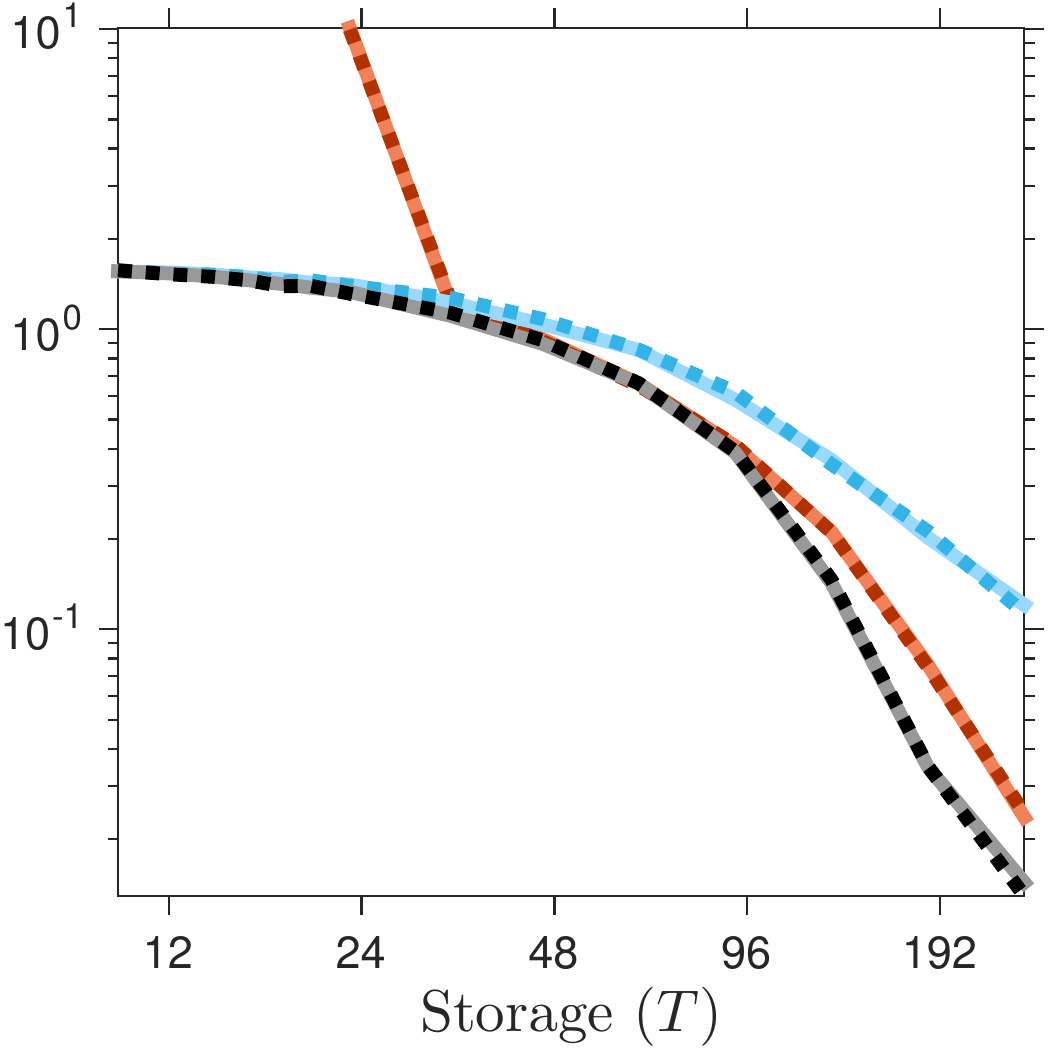}
\caption{\texttt{PolyDecaySlow}}
\end{center}
\end{subfigure}
\end{center}

\vspace{0.5em}

\begin{center}
\begin{subfigure}{.325\textwidth}
\begin{center}
\includegraphics[height=1.5in]{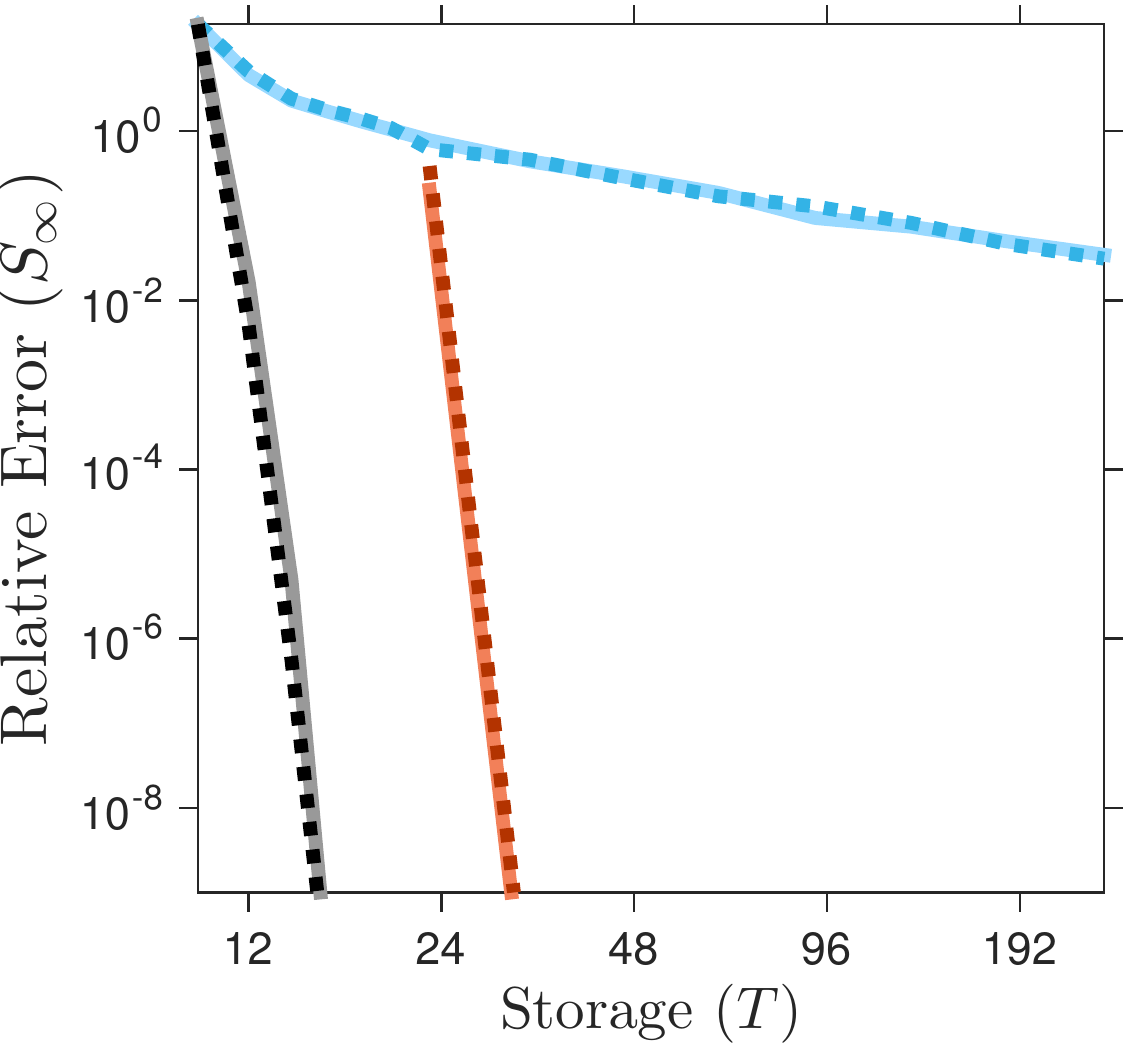}
\caption{\texttt{ExpDecayFast}}
\end{center}
\end{subfigure}
\begin{subfigure}{.325\textwidth}
\begin{center}
\includegraphics[height=1.5in]{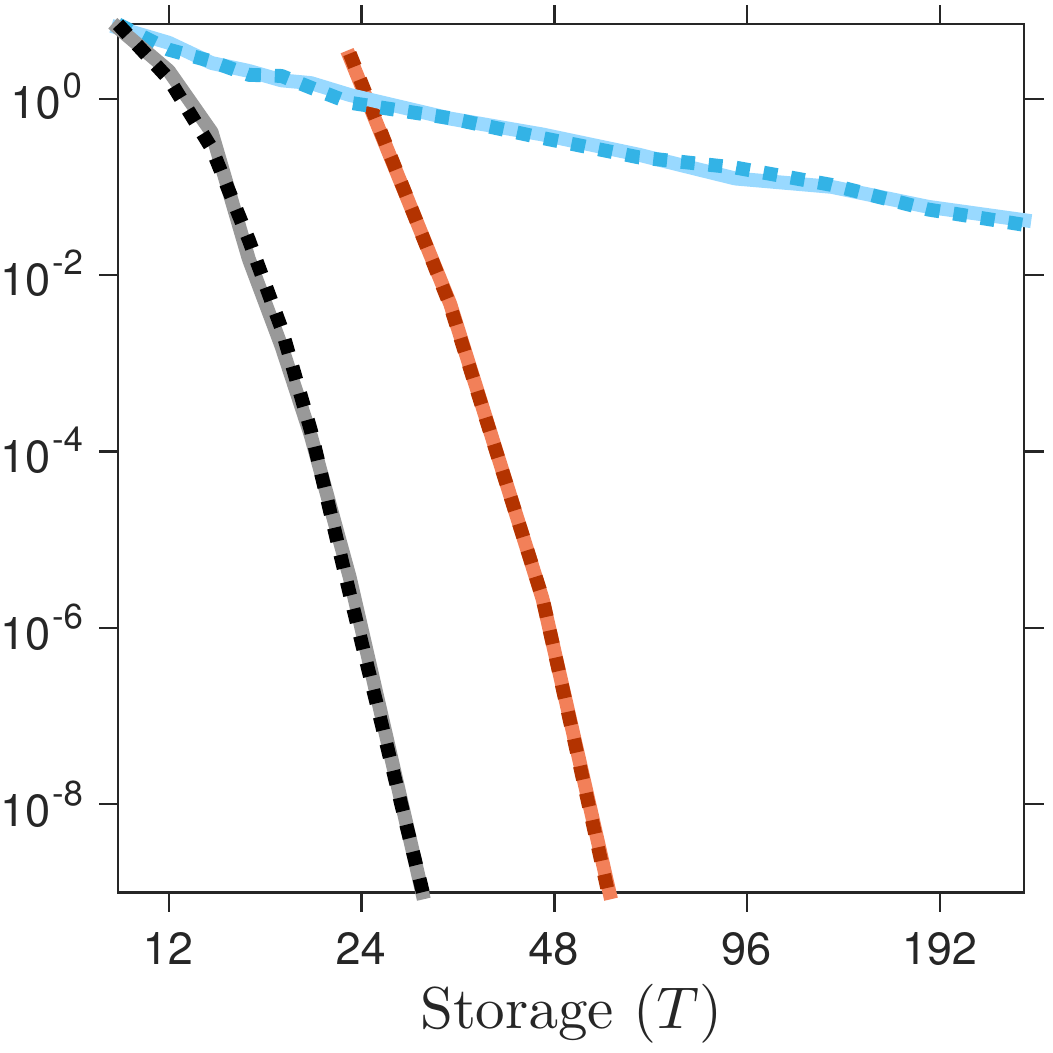}
\caption{\texttt{ExpDecayMed}}
\end{center}
\end{subfigure}
\begin{subfigure}{.325\textwidth}
\begin{center}
\includegraphics[height=1.5in]{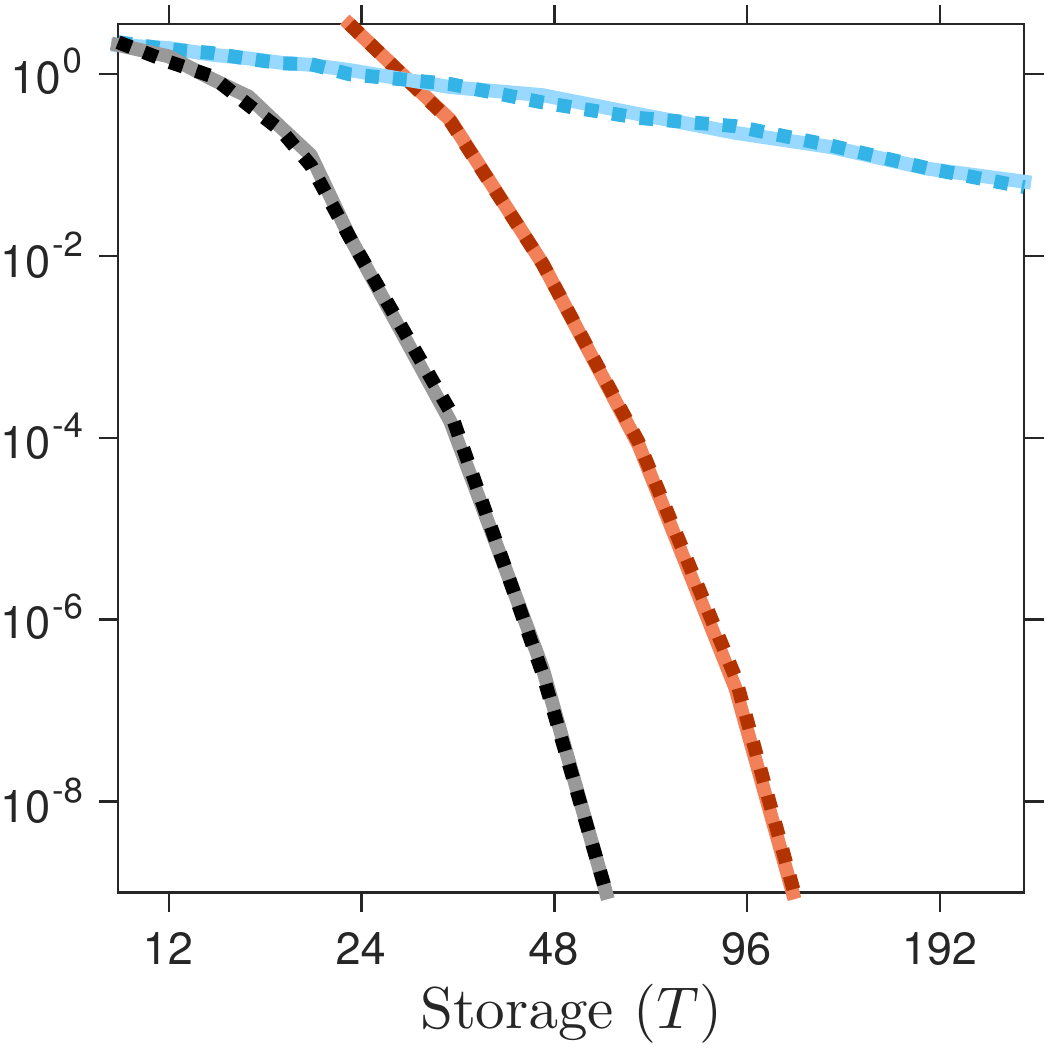}
\caption{\texttt{ExpDecaySlow}}
\end{center}
\end{subfigure}
\end{center}

\caption{\textbf{Synthetic Examples with Effective Rank $R = 5$, Approximation Rank $r = 10$, Schatten $\infty$-Norm Error.}
The series are
generated by three algorithms for rank-$r$ psd approximation with $r = 10$.
\textbf{Solid lines} are generated from the Gaussian sketch;
\textbf{dashed lines} are from the SSFT sketch.
Each panel displays the  Schatten $\infty$-norm relative error~\eqref{eqn:relative-error}
as a function of storage cost $T$.  See Sec.~\ref{sec:numerics}
for details.}
\label{fig:synthetic-Sinf-R5}
\end{figure}

\begin{figure}[htp!]
\begin{center}
\begin{subfigure}{.325\textwidth}
\begin{center}
\includegraphics[height=1.5in]{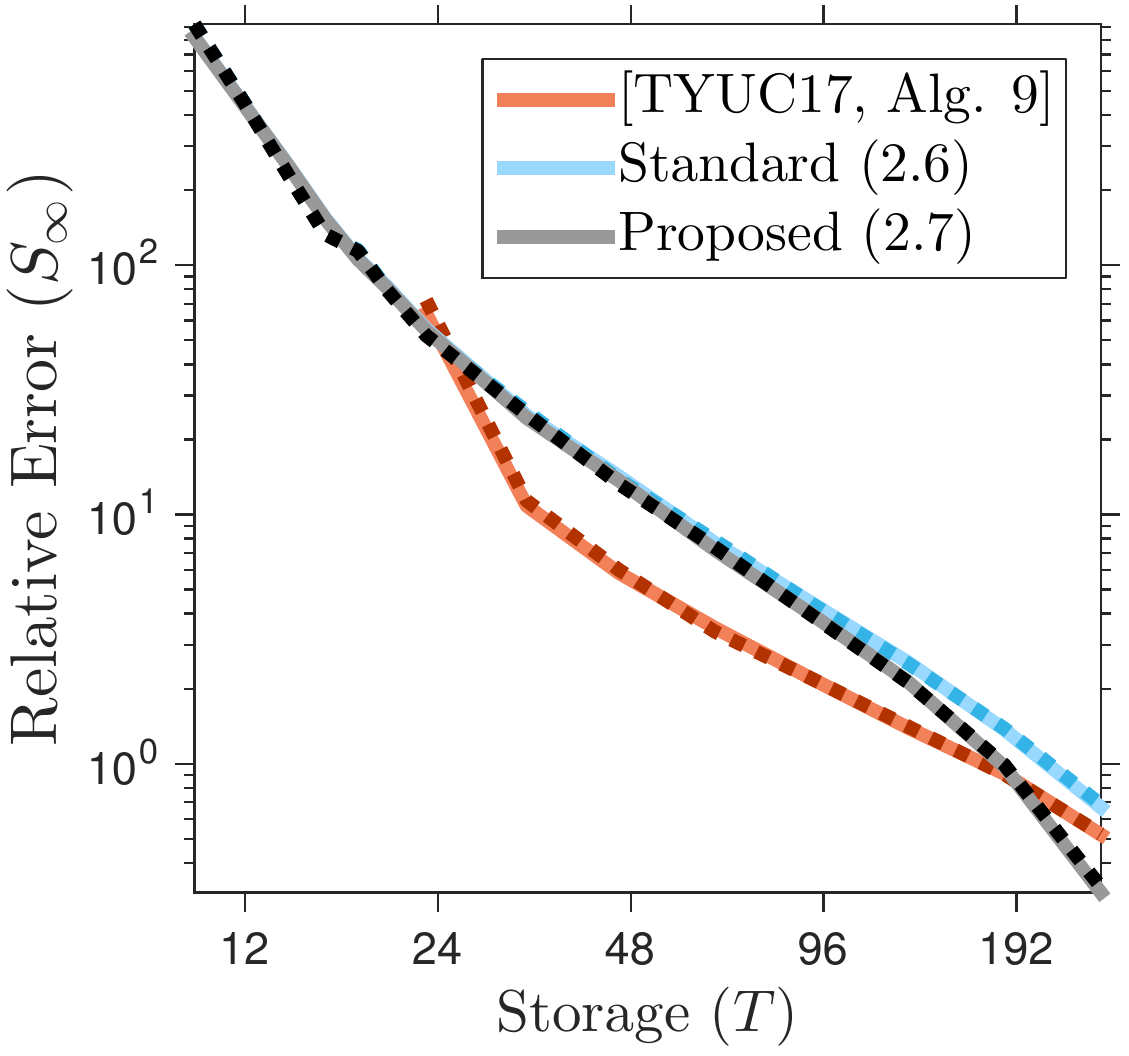}
\caption{\texttt{LowRankLowNoise}}
\end{center}
\end{subfigure}
\begin{subfigure}{.325\textwidth}
\begin{center}
\includegraphics[height=1.5in]{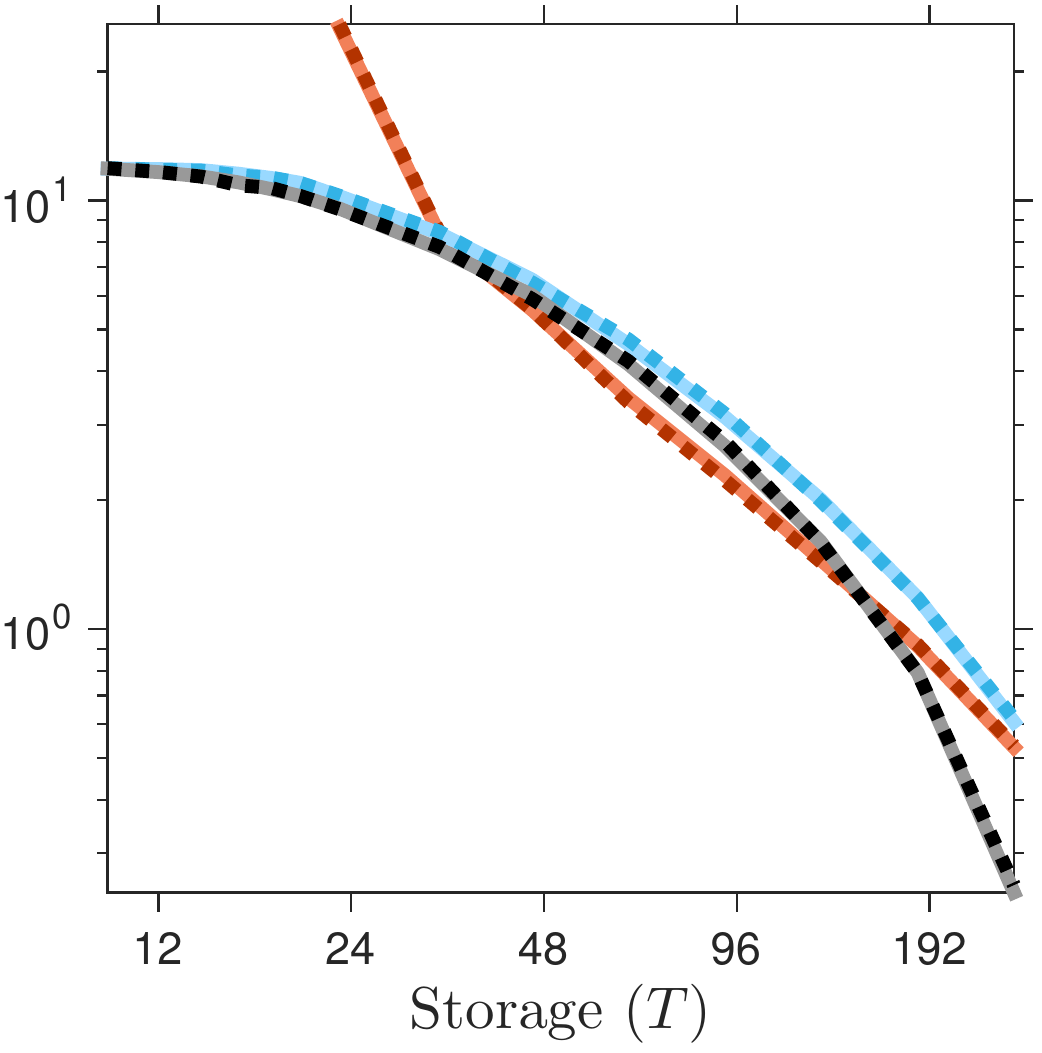}
\caption{\texttt{LowRankMedNoise}}
\end{center}
\end{subfigure}
\begin{subfigure}{.325\textwidth}
\begin{center}
\includegraphics[height=1.5in]{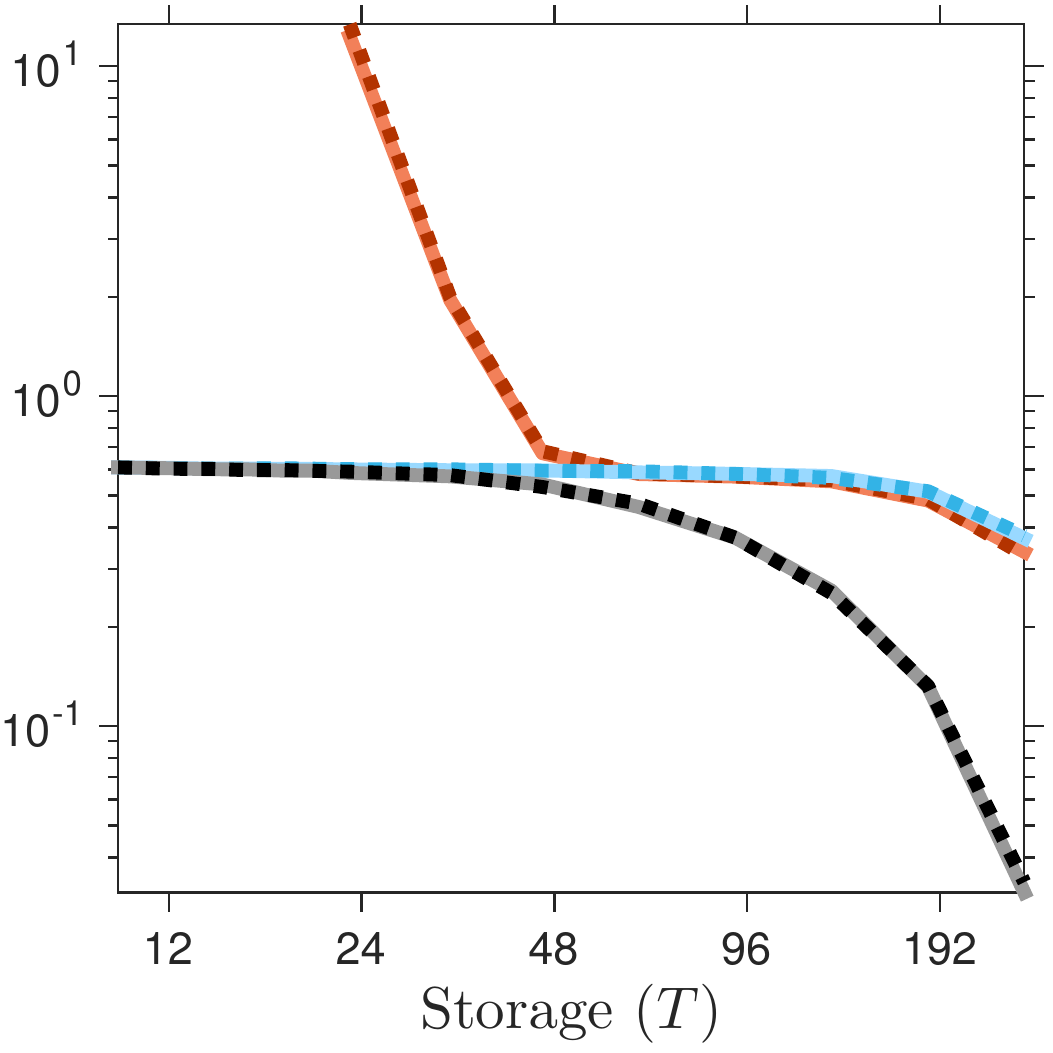}
\caption{\texttt{LowRankHiNoise}}
\end{center}
\end{subfigure}
\end{center}

\vspace{.5em}

\begin{center}
\begin{subfigure}{.325\textwidth}
\begin{center}
\includegraphics[height=1.5in]{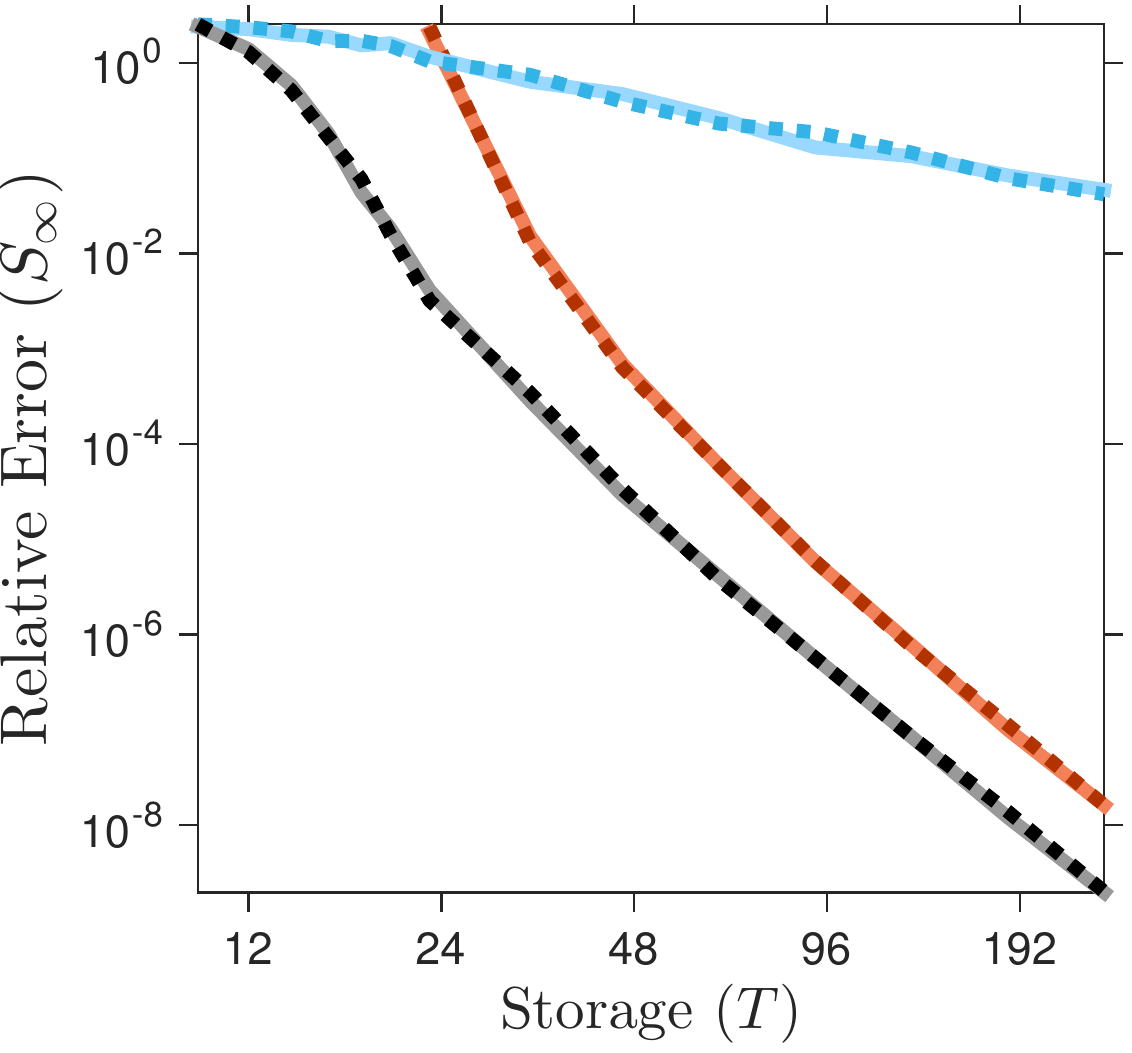}
\caption{\texttt{PolyDecayFast}}
\end{center}
\end{subfigure}
\begin{subfigure}{.325\textwidth}
\begin{center}
\includegraphics[height=1.5in]{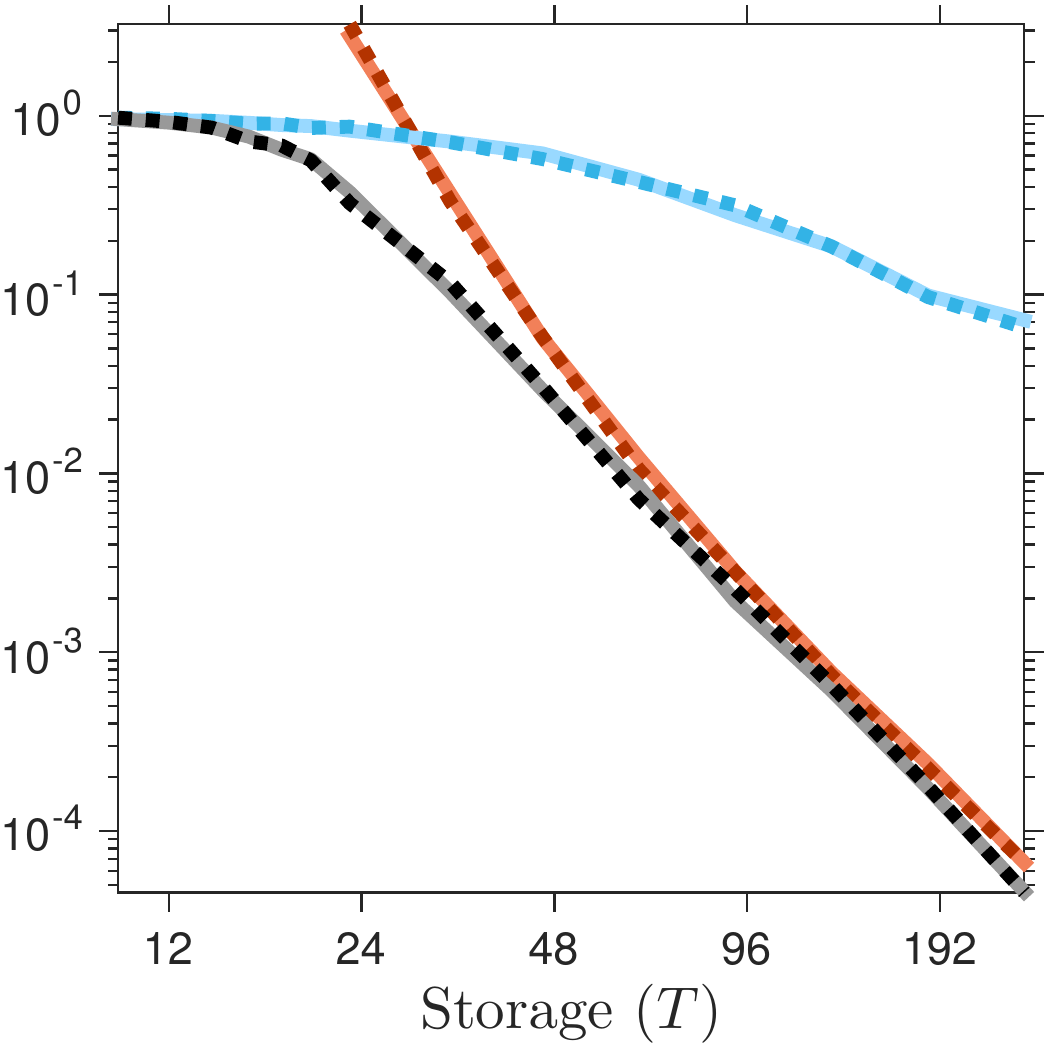}
\caption{\texttt{PolyDecayMed}}
\end{center}
\end{subfigure}
\begin{subfigure}{.325\textwidth}
\begin{center}
\includegraphics[height=1.5in]{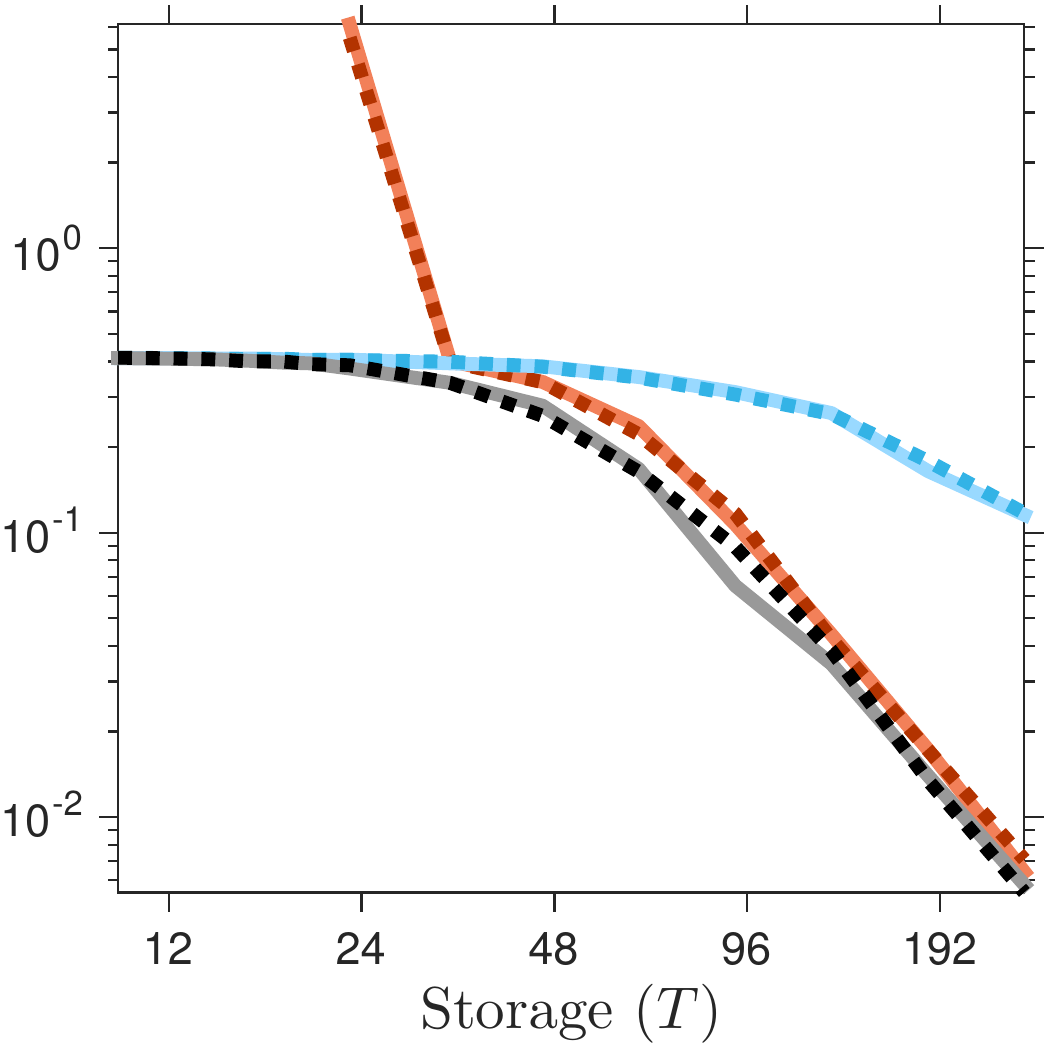}
\caption{\texttt{PolyDecaySlow}}
\end{center}
\end{subfigure}
\end{center}

\vspace{0.5em}

\begin{center}
\begin{subfigure}{.325\textwidth}
\begin{center}
\includegraphics[height=1.5in]{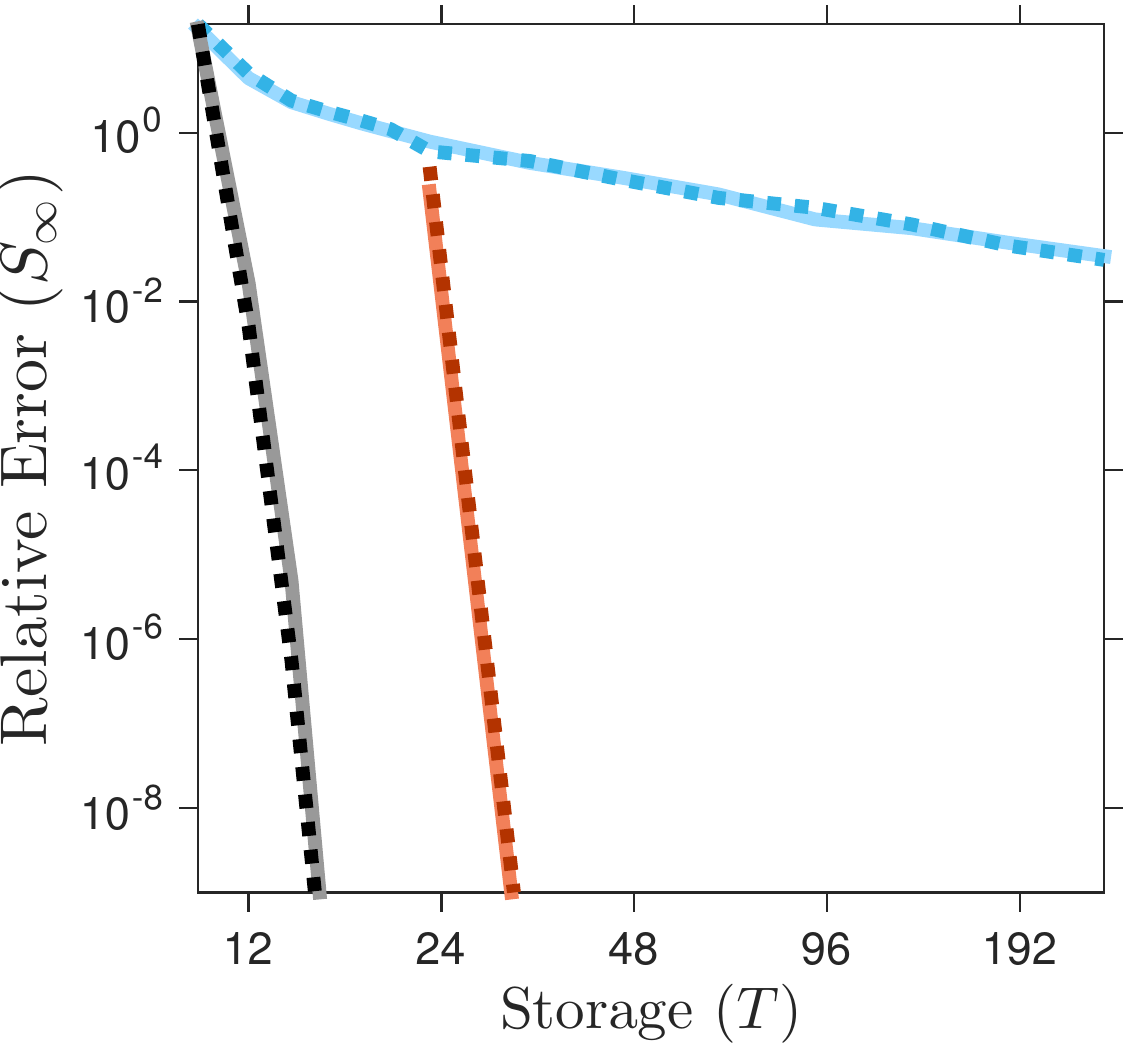}
\caption{\texttt{ExpDecayFast}}
\end{center}
\end{subfigure}
\begin{subfigure}{.325\textwidth}
\begin{center}
\includegraphics[height=1.5in]{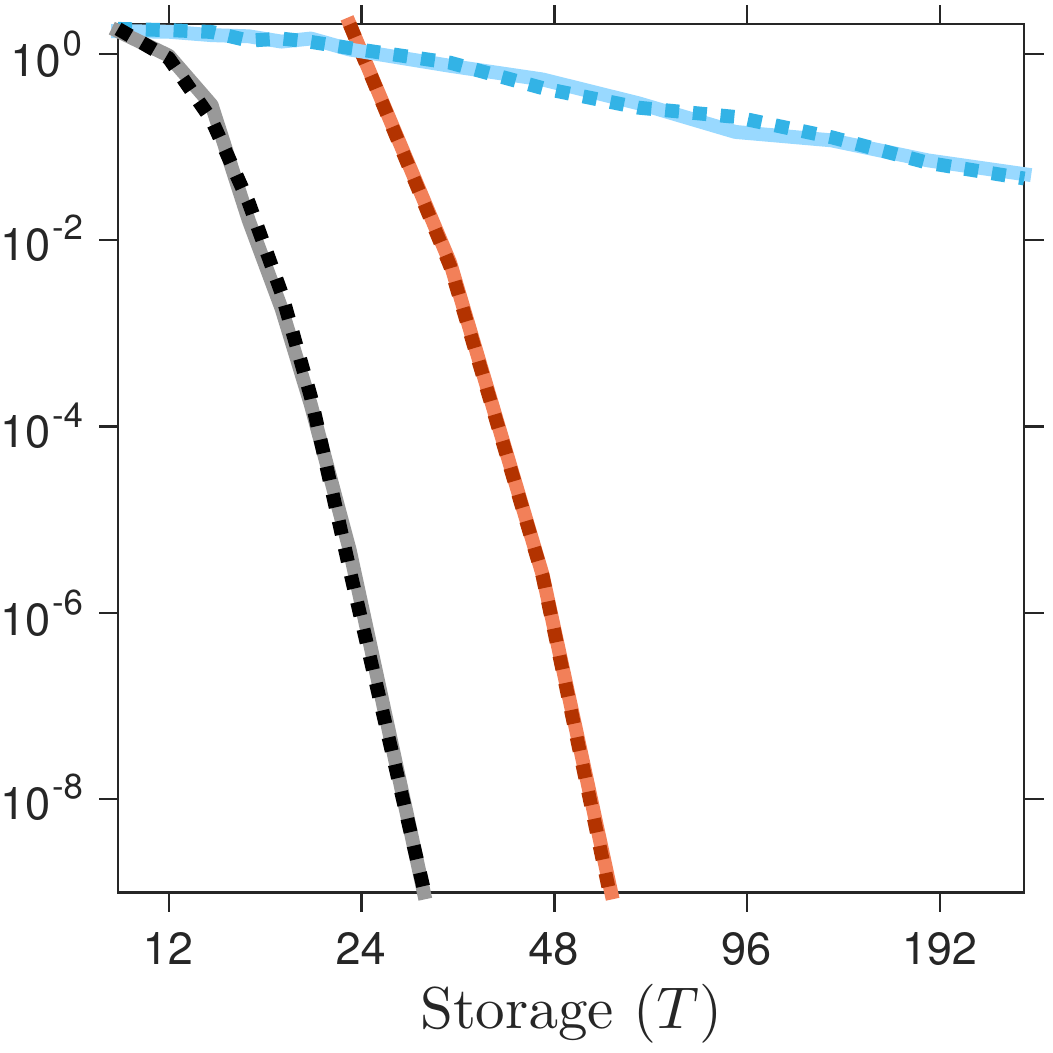}
\caption{\texttt{ExpDecayMed}}
\end{center}
\end{subfigure}
\begin{subfigure}{.325\textwidth}
\begin{center}
\includegraphics[height=1.5in]{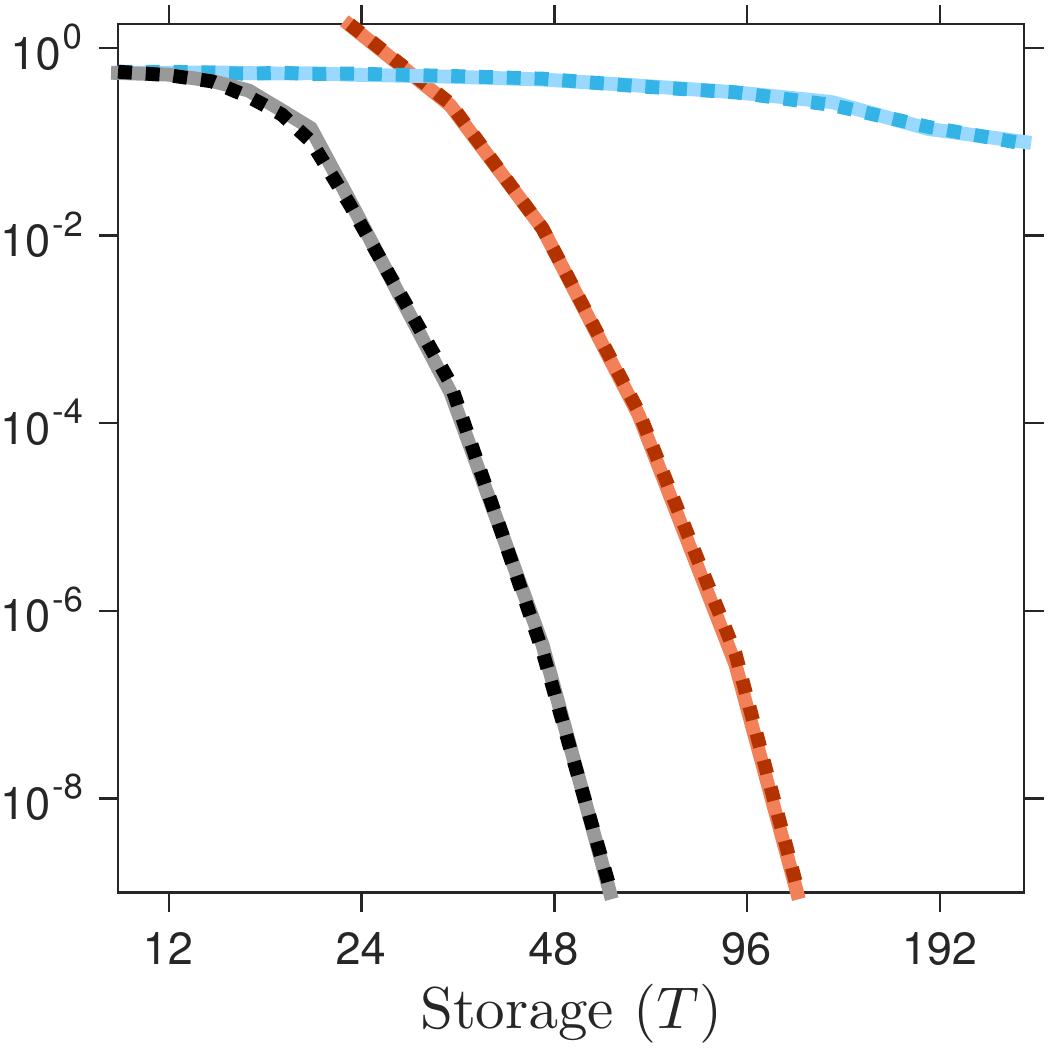}
\caption{\texttt{ExpDecaySlow}}
\end{center}
\end{subfigure}
\end{center}

\caption{\textbf{Synthetic Examples with Effective Rank $R = 10$, Approximation Rank $r = 10$, Schatten $\infty$-Norm Error.}
The series are
generated by three algorithms for rank-$r$ psd approximation with $r = 10$.
\textbf{Solid lines} are generated from the Gaussian sketch;
\textbf{dashed lines} are from the SSFT sketch.
Each panel displays the  Schatten $\infty$-norm relative error~\eqref{eqn:relative-error}
as a function of storage cost $T$.  See Sec.~\ref{sec:numerics}
for details.}
\label{fig:synthetic-Sinf-R10}
\end{figure}

\begin{figure}[htp!]
\begin{center}
\begin{subfigure}{.325\textwidth}
\begin{center}
\includegraphics[height=1.5in]{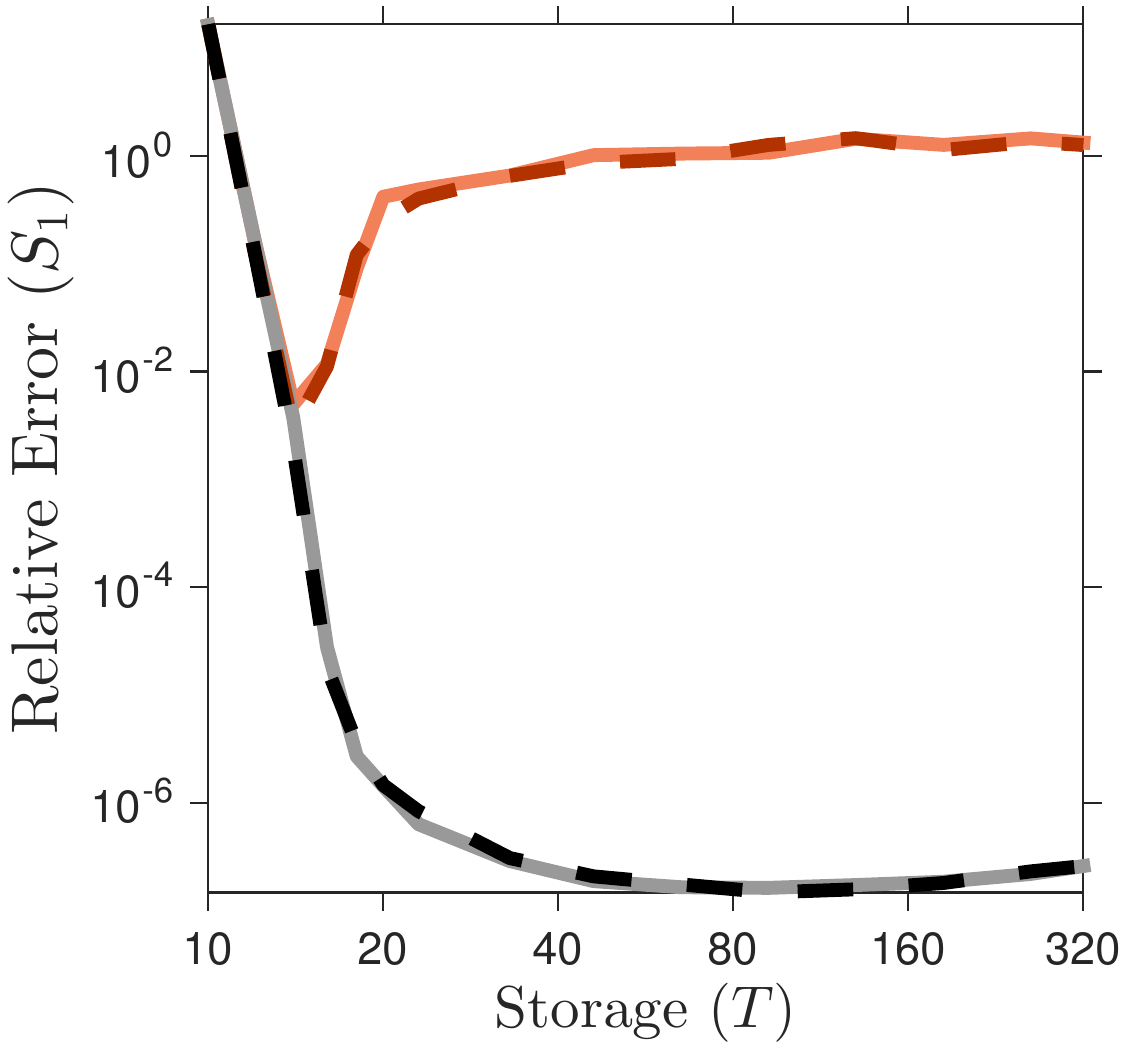}
\caption{\texttt{ExpDecayFast}, $R = 5$}
\end{center}
\end{subfigure}
\begin{subfigure}{.325\textwidth}
\begin{center}
\includegraphics[height=1.5in]{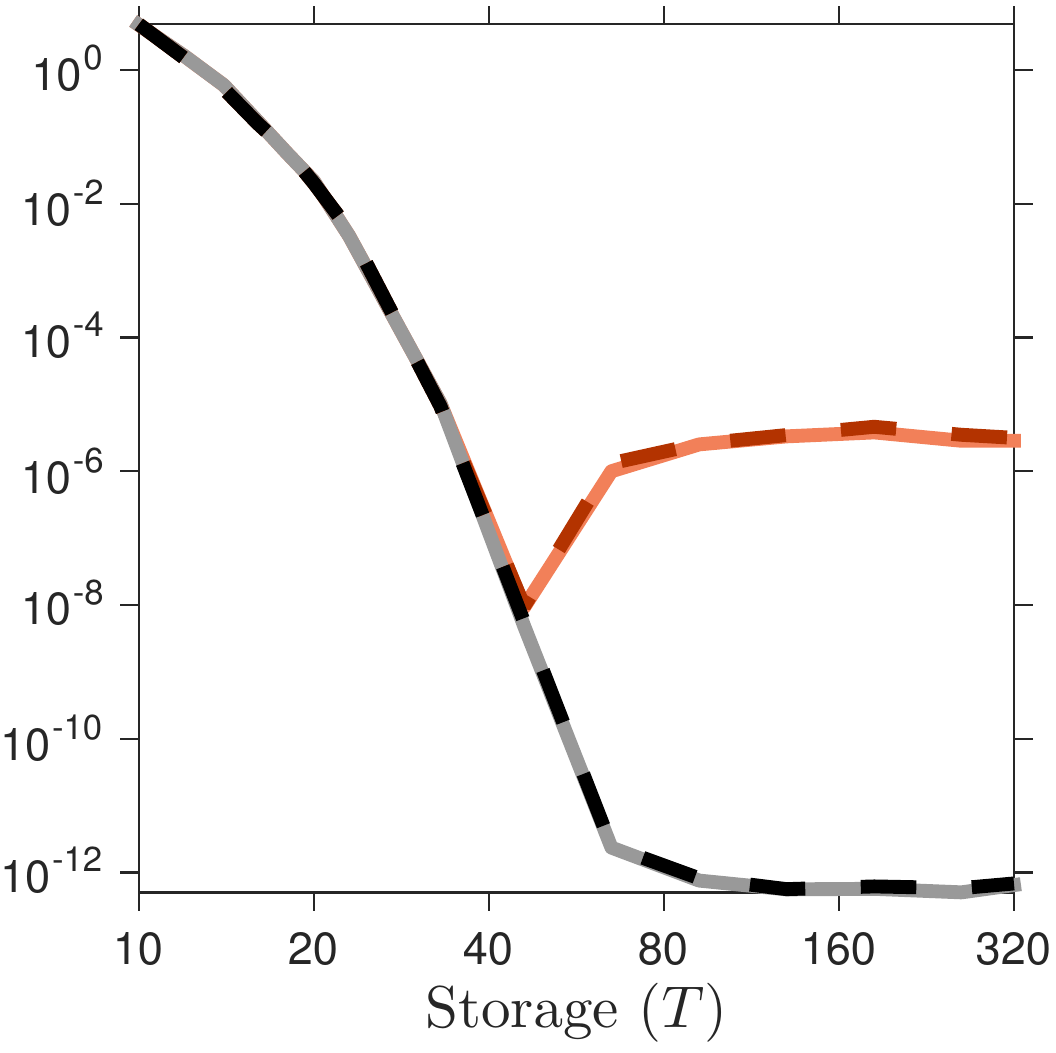}
\caption{\texttt{ExpDecayMed}, $R = 5$}
\end{center}
\end{subfigure}
\begin{subfigure}{.325\textwidth}
\begin{center}
\includegraphics[height=1.5in]{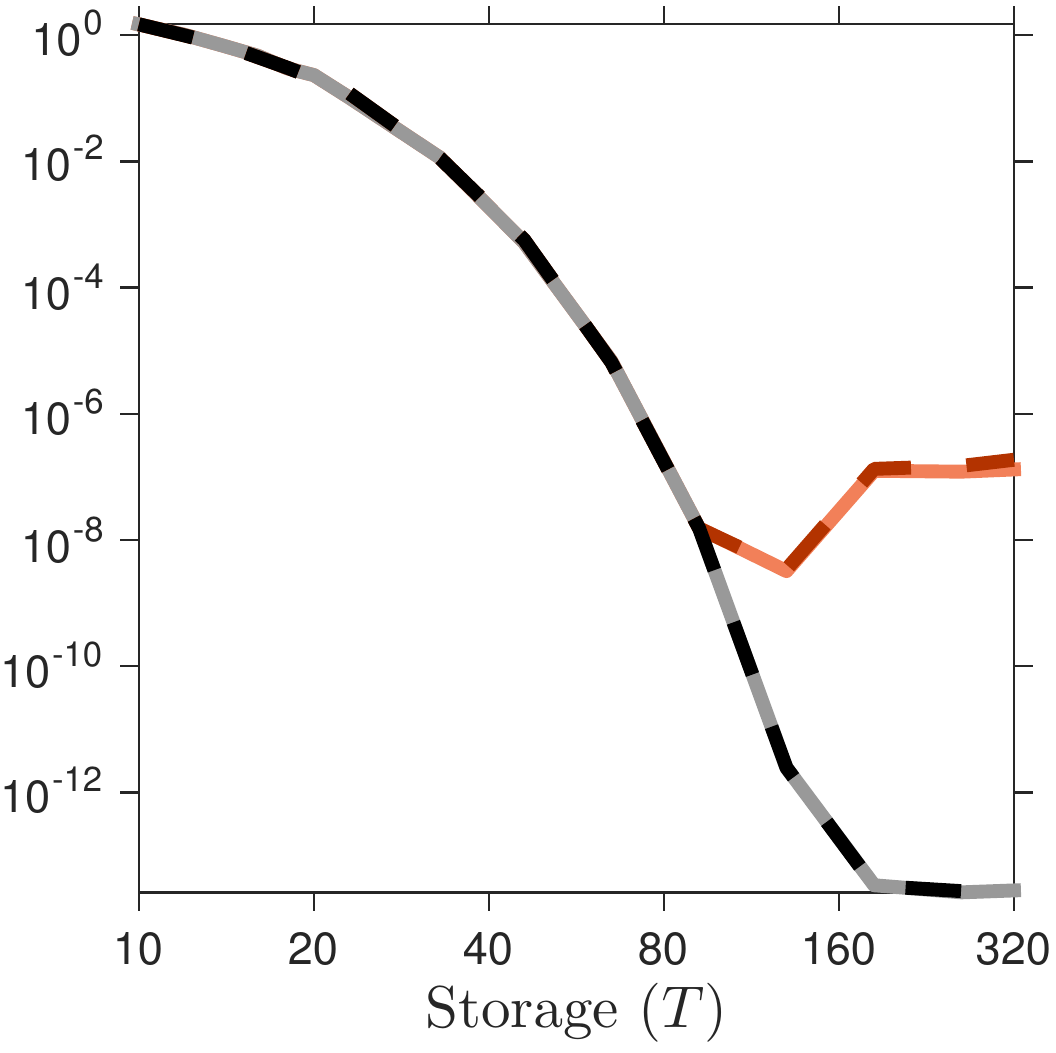}
\caption{\texttt{ExpDecaySlow}, $R = 5$}
\end{center}
\end{subfigure}
\end{center}

\vspace{.5em}

\begin{center}
\begin{subfigure}{.325\textwidth}
\begin{center}
\includegraphics[height=1.5in]{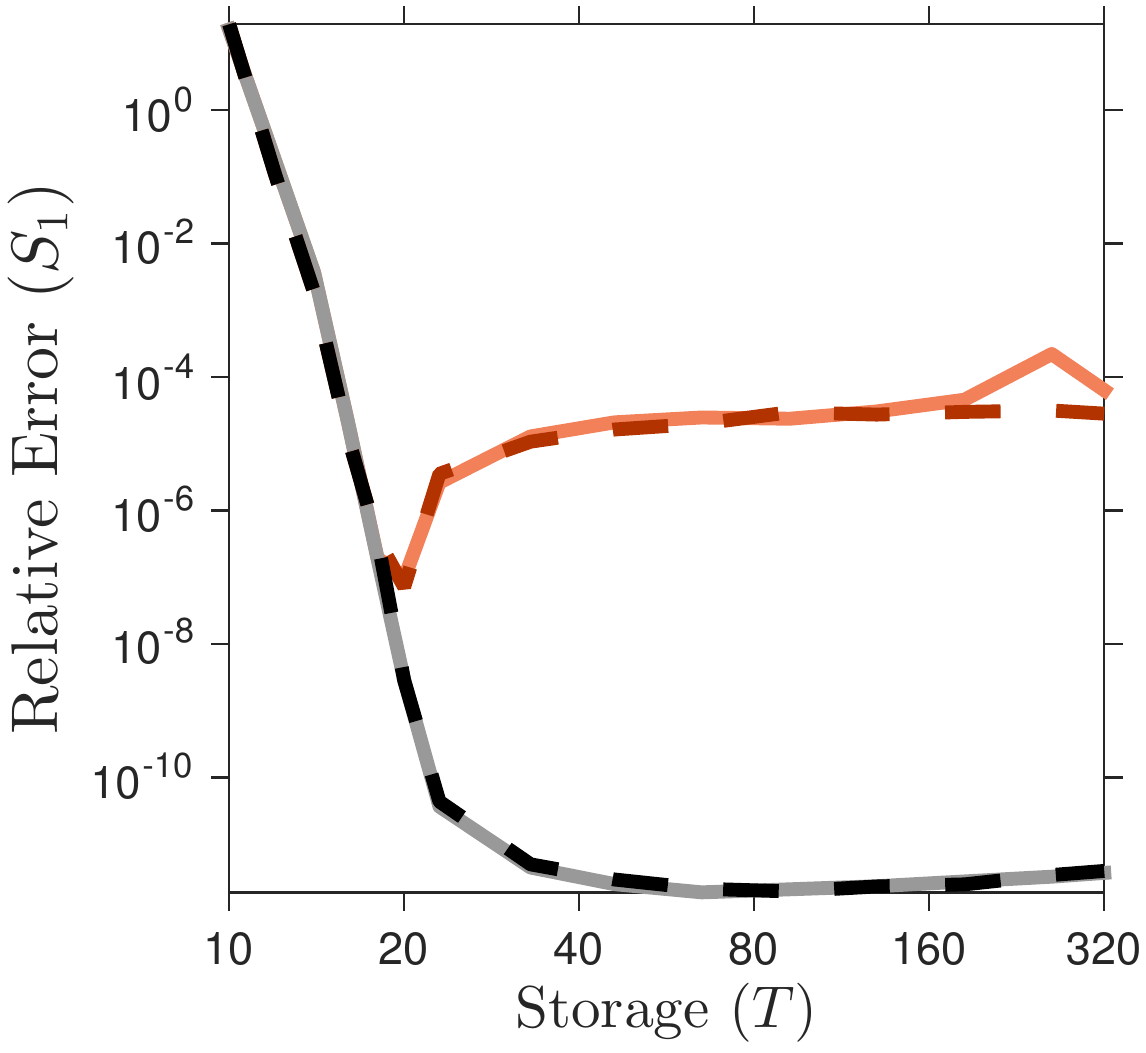}
\caption{\texttt{ExpDecayFast}, $R = 10$}
\end{center}
\end{subfigure}
\begin{subfigure}{.325\textwidth}
\begin{center}
\includegraphics[height=1.5in]{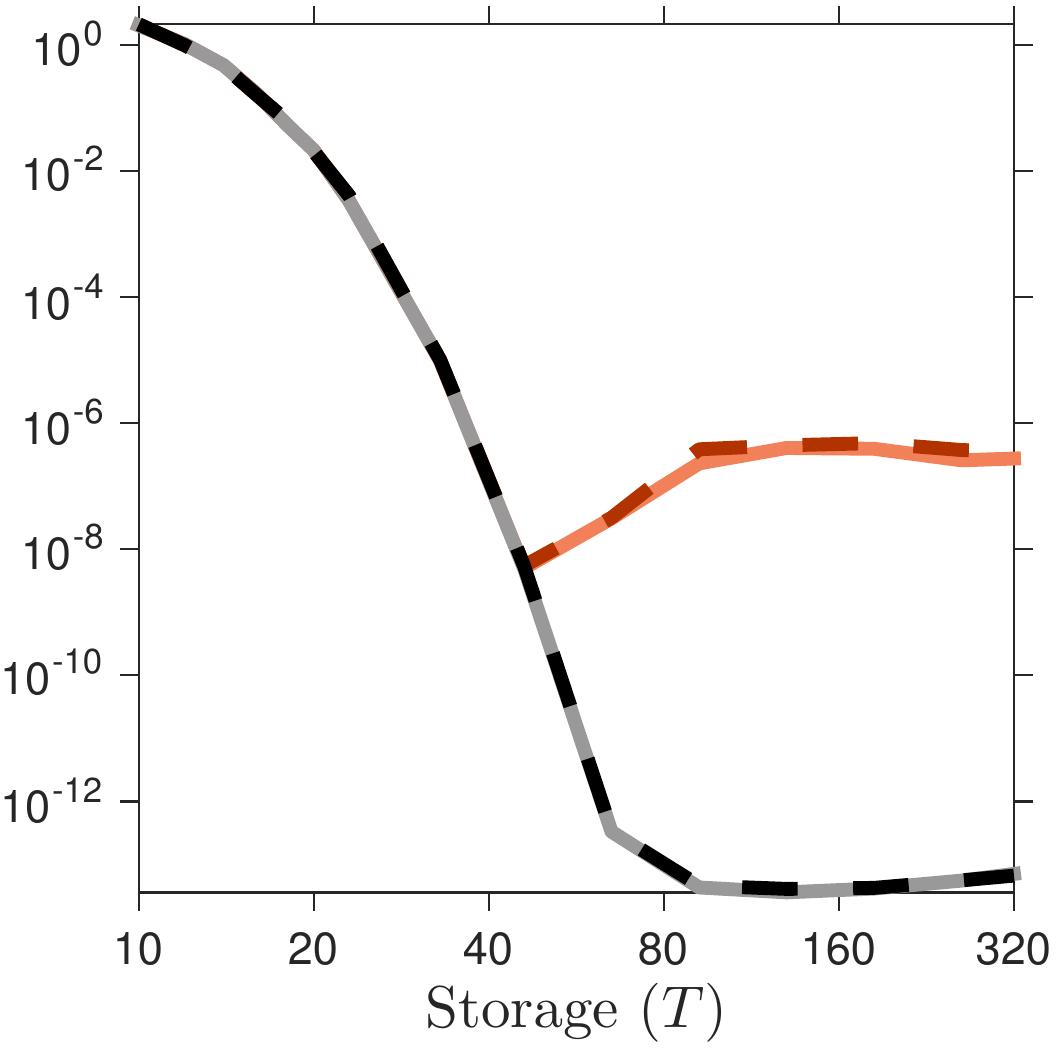}
\caption{\texttt{ExpDecayMed}, $R = 10$}
\end{center}
\end{subfigure}
\begin{subfigure}{.325\textwidth}
\begin{center}
\includegraphics[height=1.5in]{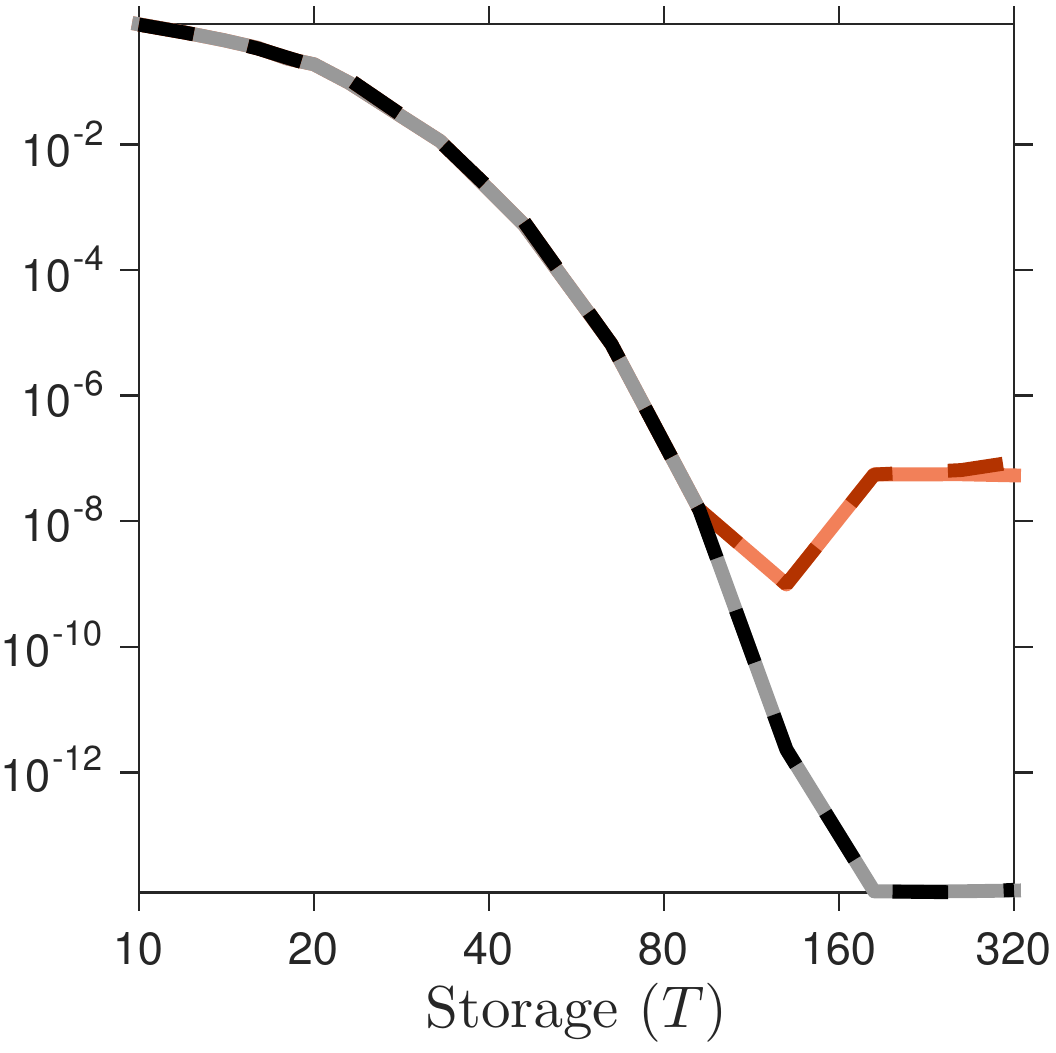}
\caption{\texttt{ExpDecaySlow}, $R = 10$}
\end{center}
\end{subfigure}
\end{center}

\vspace{0.5em}

\begin{center}
\begin{subfigure}{.325\textwidth}
\begin{center}
\includegraphics[height=1.5in]{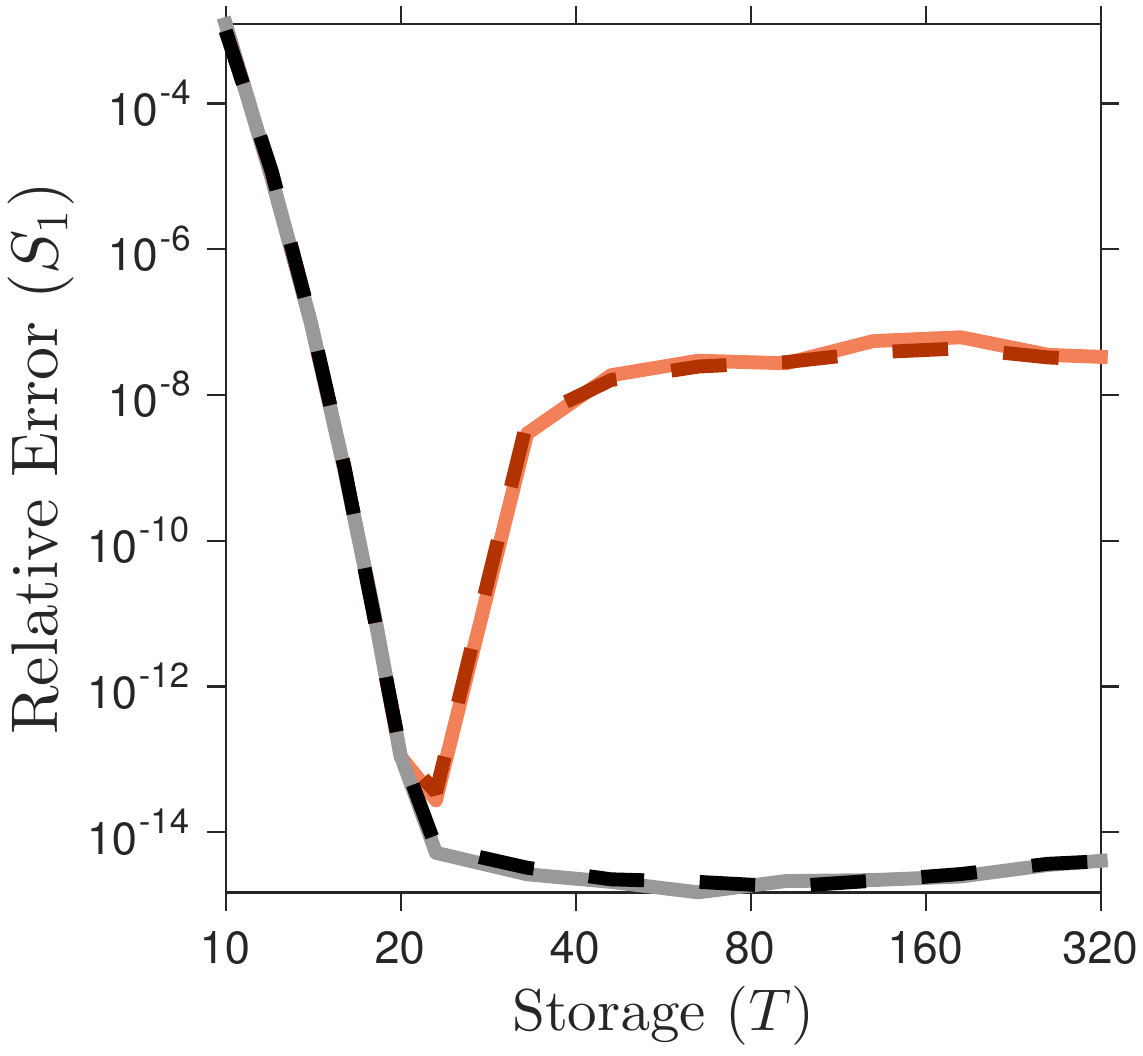}
\caption{\texttt{ExpDecayFast}, $R = 20$}
\end{center}
\end{subfigure}
\begin{subfigure}{.325\textwidth}
\begin{center}
\includegraphics[height=1.5in]{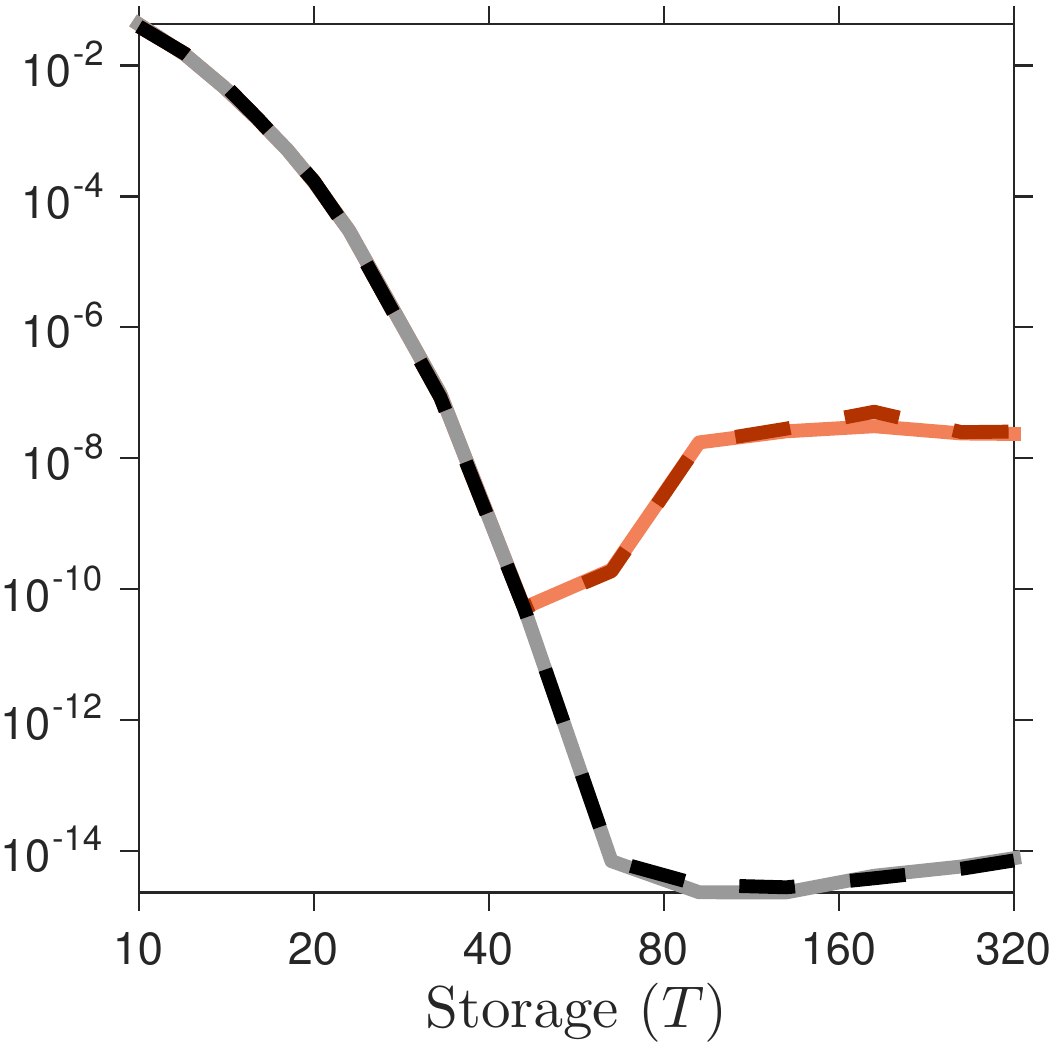}
\caption{\texttt{ExpDecayMed}, $R = 20$}
\end{center}
\end{subfigure}
\begin{subfigure}{.325\textwidth}
\begin{center}
\includegraphics[height=1.5in]{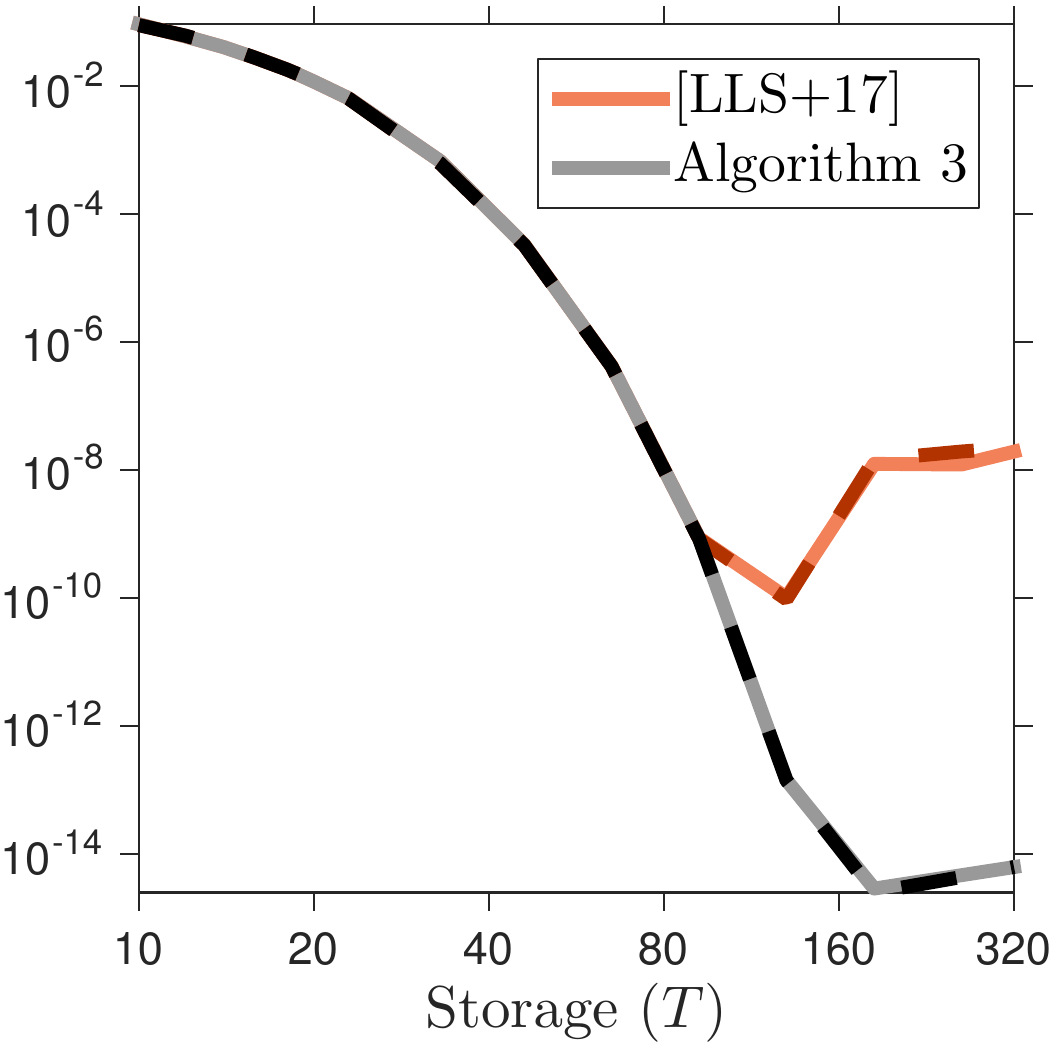}
\caption{\texttt{ExpDecaySlow}, $R = 20$}
\end{center}
\end{subfigure}
\end{center}

\caption{\textbf{Bad Numerics, Approximation Rank $r = 10$, Schatten $1$-Norm Error.} The series are generated by two implementations of the
fixed-rank psd approximation~\eqref{eqn:Ahat-fixed}.
We compare Algorithm~\ref{alg:low-rank-recon} with
another approach [LLS+17] proposed in~\cite[Eqn.~(13)]{LLS+17:Algorithm-971}.
\textbf{Solid lines} are generated from the Gaussian sketch;
\textbf{dashed lines} are from the SSFT sketch.
Each panel displays the Schatten 1-norm relative error~\eqref{eqn:relative-error}
as a function of storage cost $T$.  See App.~\ref{sec:numerics-extra}
for details.}
\label{fig:bad-numerics}
\end{figure}

\clearpage

\fi}

\bibliographystyle{plainnat}

{\small \begin{thebibliography}{42}
\providecommand{\natexlab}[1]{#1}
\providecommand{\url}[1]{\texttt{#1}}
\expandafter\ifx\csname urlstyle\endcsname\relax
  \providecommand{\doi}[1]{doi: #1}\else
  \providecommand{\doi}{doi: \begingroup \urlstyle{rm}\Url}\fi

\bibitem[Ailon and Chazelle(2009)]{AC09:Fast-Johnson-Lindenstrauss}
N.~Ailon and B.~Chazelle.
\newblock The fast {J}ohnson-{L}indenstrauss transform and approximate nearest
  neighbors.
\newblock \emph{SIAM J. Comput.}, 39\penalty0 (1):\penalty0 302--322, 2009.

\bibitem[Bhatia(1997)]{Bha97:Matrix-Analysis}
R.~Bhatia.
\newblock \emph{Matrix analysis}.
\newblock Springer-Verlag, New York, 1997.

\bibitem[Boutsidis and Gittens(2013)]{BG13:Improved-Matrix}
C.~Boutsidis and A.~Gittens.
\newblock Improved matrix algorithms via the subsampled randomized {H}adamard
  transform.
\newblock \emph{SIAM J. Matrix Anal. Appl.}, 34\penalty0 (3):\penalty0
  1301--1340, 2013.

\bibitem[Boutsidis et~al.(2015)Boutsidis, Garber, Karnin, and
  Liberty]{BGKL15:Online-Principal}
C.~Boutsidis, D.~Garber, Z.~Karnin, and E.~Liberty.
\newblock Online principal components analysis.
\newblock In \emph{Proc. 26th Ann. ACM-SIAM Symp. Discrete Algorithms (SODA)},
  pages 887--901, 2015.

\bibitem[Boutsidis et~al.(2016)Boutsidis, Woodruff, and
  Zhong]{BWZ16:Optimal-Principal-STOC}
C.~Boutsidis, D.~Woodruff, and P.~Zhong.
\newblock Optimal principal component analysis in distributed and streaming
  models.
\newblock In \emph{Proc. 48th ACM Symp. Theory of Computing (STOC)}, 2016.

\bibitem[Chiu and Demanet(2013)]{CD13:Sublinear-Randomized}
J.~Chiu and L.~Demanet.
\newblock Sublinear randomized algorithms for skeleton decompositions.
\newblock \emph{SIAM J. Matrix Anal. Appl.}, 34\penalty0 (3):\penalty0
  1361--1383, 2013.

\bibitem[Clarkson and Woodruff(2017)]{CW17:Low-Rank-PSD}
K.~Clarkson and D.~Woodruff.
\newblock Low-rank {PSD} approximation in input-sparsity time.
\newblock In \emph{Proc. 28th Ann. ACM-SIAM Symp. Discrete Algorithms (SODA)},
  pages 2061--2072, Jan. 2017.

\bibitem[Clarkson and Woodruff(2009)]{CW09:Numerical-Linear}
K.~L. Clarkson and D.~P. Woodruff.
\newblock Numerical linear algebra in the streaming model.
\newblock In \emph{Proc. 41st ACM Symp. Theory of Computing (STOC)}, 2009.

\bibitem[Cohen et~al.(2015)Cohen, Elder, Musco, Musco, and
  Persu]{CEMMP15:Dimensionality-Reduction}
M.~B. Cohen, S.~Elder, C.~Musco, C.~Musco, and M.~Persu.
\newblock Dimensionality reduction for k-means clustering and low rank
  approximation.
\newblock In \emph{Proc. 47th {ACM} {S}ymp. {T}heory of {C}omputing (STOC)},
  pages 163--172. ACM, New York, 2015.

\bibitem[Cohen et~al.(2016)Cohen, Nelson, and
  Woodruff]{CNW16:Optimal-Approximate}
M.~B. Cohen, J.~Nelson, and D.~P. Woodruff.
\newblock {Optimal Approximate Matrix Product in Terms of Stable Rank}.
\newblock In \emph{43rd Int. Coll. Automata, Languages, and Programming
  (ICALP)}, volume~55, pages 11:1--11:14, 2016.

\bibitem[Davis and Hu(2011)]{DH11:University-Florida}
T.~A. Davis and Hu.
\newblock The {U}niversity of {F}lorida sparse matrix collection.
\newblock \emph{ACM Trans. Math. Softw.}, 3\penalty0 (1):\penalty0 1:1--1:25,
  2011.

\bibitem[Drineas and Mahoney(2005)]{DM05:Nystrom-Method}
P.~Drineas and M.~W. Mahoney.
\newblock On the {N}ystr\"om method for approximating a {G}ram matrix for
  improved kernel-based learning.
\newblock \emph{J. Mach. Learn. Res.}, 6:\penalty0 2153--2175, 2005.

\bibitem[Feldman et~al.(2016)Feldman, Volkov, and
  Rus]{FVR16:Dimensionality-Reduction}
D.~Feldman, M.~Volkov, and D.~Rus.
\newblock Dimensionality reduction of massive sparse datasets using coresets.
\newblock In \emph{Adv. Neural Information Processing Systems 29 (NIPS)}, 2016.

\bibitem[Fowlkes et~al.(2004)Fowlkes, Belongie, Chung, and
  Malik]{FBCM04:Spectral-Grouping}
C.~Fowlkes, S.~Belongie, F.~Chung, and J.~Malik.
\newblock Spectral grouping using the {N}ystr\"{o}m method.
\newblock \emph{IEEE Trans. Pattern Anal. Mach. Intell.}, 26\penalty0
  (2):\penalty0 214--225, Jan. 2004.

\bibitem[Ghasemi et~al.(2016)Ghasemi, Liberty, Phillips, and
  Woodruff]{GLPW16:Frequent-Directions}
M.~Ghasemi, E.~Liberty, J.~M. Phillips, and D.~P. Woodruff.
\newblock Frequent directions: Simple and deterministic matrix sketching.
\newblock \emph{SIAM J. Comput.}, 45\penalty0 (5):\penalty0 1762--1792, 2016.

\bibitem[Gilbert et~al.(2012)Gilbert, Park, and Wakin]{GPW12:Sketched-SVD}
A.~C. Gilbert, J.~Y. Park, and M.~B. Wakin.
\newblock Sketched {SVD}: Recovering spectral features from compressed
  measurements.
\newblock Available at \url{http://arXiv.org/abs/1211.0361}, Nov. 2012.

\bibitem[Gittens(2011)]{Git11:Spectral-Norm}
A.~Gittens.
\newblock The spectral norm error of the na{\"\i}ve {N}ystr{\"o}m extension.
\newblock Available at \url{http:arXiv.org/abs/1110.5305}, Oct. 2011.

\bibitem[Gittens(2013)]{Git13:Topics-Randomized}
A.~Gittens.
\newblock \emph{Topics in {R}andomized {N}umerical {L}inear {A}lgebra}.
\newblock PhD thesis, California Institute of Technology, 2013.

\bibitem[Gittens and Mahoney(2013)]{GM13:Revisiting-Nystrom}
A.~Gittens and M.~W. Mahoney.
\newblock Revisiting the {N}ystr\"om method for improved large-scale machine
  learning.
\newblock Available at \url{http://arXiv.org/abs/1303.1849}, Mar. 2013.

\bibitem[Gittens and Mahoney(2016)]{GM16:Revisiting-Nystrom-JMLR}
A.~Gittens and M.~W. Mahoney.
\newblock Revisiting the {N}ystr\"om method for improved large-scale machine
  learning.
\newblock \emph{J. Mach. Learn. Res.}, 17:\penalty0 Paper No. 117, 65, 2016.

\bibitem[Goemans and Williamson(1995)]{GW95:Improved-Approximation}
M.~X. Goemans and D.~P. Williamson.
\newblock Improved approximation algorithms for maximum cut and satisfiability
  problems using semidefinite programming.
\newblock \emph{J. Assoc. Comput. Mach.}, 42\penalty0 (6):\penalty0 1115--1145,
  1995.

\bibitem[Gu(2015)]{Gu15:Subspace-Iteration}
M.~Gu.
\newblock Subspace iteration randomization and singular value problems.
\newblock \emph{SIAM J. Sci. Comput.}, 37\penalty0 (3):\penalty0 A1139--A1173,
  2015.

\bibitem[Halko et~al.(2011)Halko, Martinsson, and
  Tropp]{HMT11:Finding-Structure}
N.~Halko, P.~G. Martinsson, and J.~A. Tropp.
\newblock Finding structure with randomness: probabilistic algorithms for
  constructing approximate matrix decompositions.
\newblock \emph{SIAM Rev.}, 53\penalty0 (2):\penalty0 217--288, 2011.

\bibitem[Higham(1989)]{Hig89:Matrix-Nearness}
N.~J. Higham.
\newblock Matrix nearness problems and applications.
\newblock In \emph{Applications of matrix theory ({B}radford, 1988)}, pages
  1--27. Oxford Univ. Press, New York, 1989.

\bibitem[Jain et~al.(2016)Jain, Jin, Kakade, Netrapalli, and
  Sidford]{JJK+16:Streaming-PCA}
P.~Jain, C.~Jin, S.~M. Kakade, P.~Netrapalli, and A.~Sidford.
\newblock Streaming {PCA}: Matching matrix {B}ernstein and near-optimal finite
  sample guarantees for {O}ja's algorithm.
\newblock In \emph{29th Ann. Conf. Learning Theory (COLT)}, pages 1147--1164,
  2016.

\bibitem[Kumar et~al.(2012)Kumar, Mohri, and Talwalkar]{KMT12:Sampling-Methods}
S.~Kumar, M.~Mohri, and A.~Talwalkar.
\newblock Sampling methods for the {N}ystr{\"o}m method.
\newblock \emph{J. Mach. Learn. Res.}, 13:\penalty0 981--1006, Apr. 2012.

\bibitem[Li et~al.(2017)Li, Linderman, Szlam, Stanton, Kluger, and
  Tygert]{LLS+17:Algorithm-971}
H.~Li, G.~C. Linderman, A.~Szlam, K.~P. Stanton, Y.~Kluger, and M.~Tygert.
\newblock Algorithm 971: An implementation of a randomized algorithm for
  principal component analysis.
\newblock \emph{ACM Trans. Math. Softw.}, 43\penalty0 (3):\penalty0
  28:1--28:14, Jan. 2017.

\bibitem[Li et~al.(2014)Li, Nguyen, and Woodruff]{LNW14:Turnstile-Streaming}
Y.~Li, H.~L. Nguyen, and D.~P. Woodruff.
\newblock Turnstile streaming algorithms might as well be linear sketches.
\newblock In \emph{Proc. 2014 {ACM} {S}ymp. {T}heory of {C}omputing (STOC)},
  pages 174--183. ACM, 2014.

\bibitem[Liberty(2009)]{Lib09:Accelerated-Dense}
E.~Liberty.
\newblock \emph{Accelerated dense random projections}.
\newblock PhD thesis, Yale Univ., New Haven, 2009.

\bibitem[Mahoney(2011)]{Mah11:Randomized-Algorithms}
M.~W. Mahoney.
\newblock Randomized algorithms for matrices and data.
\newblock \emph{Found. Trends Mach. Learn.}, 3\penalty0 (2):\penalty0 123--224,
  2011.

\bibitem[Martinsson et~al.(2011)Martinsson, Rokhlin, and
  Tygert]{MRT11:Randomized-Algorithm}
P.-G. Martinsson, V.~Rokhlin, and M.~Tygert.
\newblock A randomized algorithm for the decomposition of matrices.
\newblock \emph{Appl. Comput. Harmon. Anal.}, 30\penalty0 (1):\penalty0 47--68,
  2011.

\bibitem[Mitliagkas et~al.(2013)Mitliagkas, Caramanis, and
  Jain]{MCJ13:Memory-Limited}
I.~Mitliagkas, C.~Caramanis, and P.~Jain.
\newblock Memory limited, streaming {PCA}.
\newblock In \emph{Adv. Neural Information Processing Systems 26 (NIPS)}, pages
  2886--2894, 2013.

\bibitem[Musco and Woodruff(2017)]{MW17:Sublinear-Time}
C.~Musco and D.~Woodruff.
\newblock Sublinear time low-rank approximation of positive semidefinite
  matrices.
\newblock Available at \url{http://arXiv.org/abs/1704.03371}, Apr. 2017.

\bibitem[Platt(2005)]{Pla05:FastMap-MetricMap}
J.~C. Platt.
\newblock {FastMap}, {MetricMap}, and {Landmark MDS} are all {N}ystr{\"o}m
  algorithms.
\newblock In \emph{Proc. 10th Int. Workshop Artificial Intelligence and
  Statistics (AISTATS)}, pages 261--268, 2005.

\bibitem[Tropp(2011)]{Tro11:Improved-Analysis}
J.~A. Tropp.
\newblock Improved analysis of the subsampled randomized {H}adamard transform.
\newblock \emph{Adv. Adapt. Data Anal.}, 3\penalty0 (1-2):\penalty0 115--126,
  2011.

\bibitem[Tropp et~al.(2017)Tropp, Yurtsever, Udell, and
  Cevher]{TYUC17:Randomized-Single-View-TR}
J.~A. Tropp, A.~Yurtsever, M.~Udell, and V.~Cevher.
\newblock Randomized single-view algorithms for low-rank matrix approximation.
\newblock ACM Report 2017-01, Caltech, Pasadena, Jan. 2017.
\newblock Available at \url{http://arXiv.org/abs/1609.00048}, v1.

\bibitem[Tygert(2014)]{Tyg14:Matlab-Routines}
M.~Tygert.
\newblock Beta versions of {M}atlab routines for principal component analysis.
\newblock Available at \url{http://tygert.com/software.html}, 2014.

\bibitem[Williams and Seeger(2000)]{WS01:Using-Nystrom}
C.~K.~I. Williams and M.~Seeger.
\newblock Using the {N}ystr{\"o}m method to speed up kernel machines.
\newblock In \emph{Adv. Neural Information Processing Systems 13 (NIPS)}, 2000.

\bibitem[Woodruff(2014)]{Woo14:Sketching-Tool}
D.~P. Woodruff.
\newblock Sketching as a tool for numerical linear algebra.
\newblock \emph{Found. Trends Theor. Comput. Sci.}, 10\penalty0 (1-2):\penalty0
  iv+157, 2014.

\bibitem[Woolfe et~al.(2008)Woolfe, Liberty, Rokhlin, and
  Tygert]{WLRT08:Fast-Randomized}
F.~Woolfe, E.~Liberty, V.~Rokhlin, and M.~Tygert.
\newblock A fast randomized algorithm for the approximation of matrices.
\newblock \emph{Appl. Comput. Harmon. Anal.}, 25\penalty0 (3):\penalty0
  335--366, 2008.

\bibitem[Yang et~al.(2012)Yang, Li, Mahdavi, Jin, and
  Zhou]{YLM+12:Nystrom-Method}
T.~Yang, Y.-F. Li, M.~Mahdavi, R.~Jin, and Z.-H. Zhou.
\newblock Nystr\"{o}m method vs random {F}ourier features: A theoretical and
  empirical comparison.
\newblock In \emph{Adv. Neural Information Processing Systems 25 (NIPS)}, pages
  476--484, 2012.

\bibitem[Yurtsever et~al.(2017)Yurtsever, Udell, Tropp, and
  Cevher]{YUTC17:Sketchy-Decisions}
A.~Yurtsever, M.~Udell, J.~A. Tropp, and V.~Cevher.
\newblock Sketchy decisions: Convex low-rank matrix optimization with optimal
  storage.
\newblock In \emph{Proc. 20th Int. Conf. Artificial Intelligence and Statistics
  (AISTATS)}, Fort Lauderdale, May 2017.

\end{thebibliography}
 }

\end{document}